\documentclass[oneside,english,a4paper]{scrbook}
\usepackage[T1]{fontenc}
\usepackage[latin9]{inputenc}
\usepackage{babel}
\usepackage{verbatim}
\usepackage{prettyref}
\usepackage{mathtools}
\usepackage{amsthm}
\usepackage{amsmath}
\usepackage{amssymb}
\usepackage{setspace}
\usepackage{nomencl}
\providecommand{\printnomenclature}{\printglossary}
\providecommand{\makenomenclature}{\makeglossary}
\makenomenclature
\usepackage[unicode=true,pdfusetitle,
 bookmarks=true,bookmarksnumbered=false,bookmarksopen=false,
 breaklinks=false,pdfborder={0 0 1},backref=false,colorlinks=false]
 {hyperref}

\makeatletter

\let\pr@chap=\pr@cha

 \theoremstyle{definition}
 \newtheorem*{defn*}{\protect\definitionname}
\theoremstyle{plain}
\newtheorem{thm}{\protect\theoremname}[section]
  \theoremstyle{remark}
  \newtheorem{rem}[thm]{\protect\remarkname}
  \theoremstyle{plain}
  \newtheorem{lem}[thm]{\protect\lemmaname}
  \theoremstyle{plain}
  \newtheorem{prop}[thm]{\protect\propositionname}
  \theoremstyle{plain}
  \newtheorem{cor}[thm]{\protect\corollaryname}
  \theoremstyle{definition}
  \newtheorem{condition}[thm]{\protect\conditionname}
  \theoremstyle{definition}
  \newtheorem{example}[thm]{\protect\examplename}

\@ifundefined{date}{}{\date{}}

\usepackage[T1]{fontenc}    
\usepackage{enumerate}

\usepackage{helvet}
\usepackage[T1]{fontenc}
\usepackage{scrpage2}
\usepackage{amsmath,amsthm}
\usepackage{setspace}
\usepackage{prettyref}
\usepackage{faktor}
\usepackage{enumerate}
\usepackage{url}
\usepackage{chngcntr}
\usepackage{doi}
\usepackage{ifthen}
\makeatletter
\allowdisplaybreaks

\usepackage[T1]{fontenc}    

\usepackage{a4wide}        
\usepackage{amsmath}

\usepackage{euscript}
\allowdisplaybreaks[1] 
\usepackage{amsfonts}

\usepackage{color}

\newcommand*{\trace}{\operatorname{trace}}

\newcommand*{\dive}{\operatorname{div}}

\newcommand*{\curl}{\operatorname{curl}}

\newcommand*{\Grad}{\operatorname{Grad}}
\newcommand*{\Div}{\operatorname{Div}}

\newcommand*{\grad}{\operatorname{grad}}

\newcommand*{\diam}{\operatorname{diam}}

\newcommand*{\lin}{\operatorname{lin}}
\newcommand*{\e}{\operatorname{e}}
\renewcommand*{\i}{\operatorname{i}}

\newcommand{\m}{\operatorname{m}}

\newcommand{\spt}[0]{\operatorname{spt}}

\DeclareMathAccent{\Circ}{\mathalpha}{operators}{"17}
\newcommand{\interior}[1]{\Circ{#1}}

\newcommand{\dd}{ \mathrm{d}}
\renewcommand{\Re}{\operatorname{Re}}
\renewcommand{\Im}{\operatorname{Im}}
\renewcommand{\hat}{\widehat}
\renewcommand{\tilde}{\widetilde}
\renewcommand*{\epsilon}{\varepsilon}
\renewcommand*{\rho}{\varrho}
\newcommand{\logeq}{\mathrel{\vcentcolon\Leftrightarrow}}

\newrefformat{prop}{Proposition \ref{#1}}
\newrefformat{lem}{Lemma \ref{#1}}
\newrefformat{thm}{Theorem \ref{#1}}
\newrefformat{cor}{Corollary \ref{#1}}
\newrefformat{rem}{Remark \ref{#1}}
\newrefformat{eq}{(\ref{#1})}
\newrefformat{item}{(\ref{#1})}
\newrefformat{cond}{Condition \ref{#1}}
\newrefformat{sec}{Section \ref{#1}}
\newrefformat{sub}{Subsection \ref{#1}}
\newrefformat{chap}{Chapter \ref{#1}}
\newrefformat{exa}{Example \ref{#1}}
\makeatother

\usepackage{babel}

\renewcommand{\nomgroup}[1]{%
\ifthenelse{\equal{#1}{A}}{\item[\textbf{{\large Sets}}]}{%
\ifthenelse{\equal{#1}{B}}{\item[\textbf{{\large Spaces}}]}{%
\ifthenelse{\equal{#1}{O}}{\item[\textbf{{\large Operators}}]}{%
\ifthenelse{\equal{#1}{Z}}{\item[\textbf{{\large Further Symbols}}]}{%
}
}
}
}
}

\makeatother

  \providecommand{\conditionname}{Condition}
  \providecommand{\corollaryname}{Corollary}
  \providecommand{\definitionname}{Definition}
  \providecommand{\examplename}{Example}
  \providecommand{\lemmaname}{Lemma}
  \providecommand{\propositionname}{Proposition}
  \providecommand{\remarkname}{Remark}
\providecommand{\theoremname}{Theorem}

\begin{document}
\begin{center}
\textbf{\LARGE{}Exponential Stability and Initial Value Problems}\\
\textbf{\LARGE{}for Evolutionary Equations}
\par\end{center}{\LARGE \par}

\vspace{3cm}

 \begin{center}
{\large{}Sascha Trostorff}\end{center}\thispagestyle{empty}

\newpage{}

$\:$

\thispagestyle{empty}\newpage{}\tableofcontents{}\newpage{}

$\:$

\thispagestyle{empty}

\newpage{}

\markright{Nomenclature}

\printnomenclature[4cm]{}\newpage{}

$\:$

\thispagestyle{empty}

\newpage{}

\markright{Introduction}

\chapter*{Introduction}

\addcontentsline{toc}{chapter}{Introduction}

In this thesis we deal with so-called \emph{linear evolutionary equations,
}a class of partial differential equations which was introduced by
Picard in 2009, \cite{Picard}. These equations can be written as
\[
\partial_{t}w+Au=f,
\]
where $A$ is a densely defined closed linear operator on a Hilbert
space $H$, serving as the state space, and $\partial_{t}$ denotes
the time derivative. The right-hand side $f$ is given and $w$ and
$u$ are to be determined. Of course, in order to obtain a solvable
problem, we need to link the unknowns $w$ and $u.$ This is done
by a so-called \emph{material law $\mathcal{M}$, }a suitable linear
operator acting in space-time. The relation is then given as 
\[
w=\mathcal{M}u
\]
and thus, eliminating $w$ from our abstract differential equation
via the material relation, we arrive at a problem of the form 
\begin{equation}
\left(\partial_{t}\mathcal{M}+A\right)u=f,\label{eq:evo}
\end{equation}
where now $u$ is the unknown of interest. Clearly, $w$ can be re-constructed
from $u$ via the material relation.

Well-posedness in the sense of Hadamard of problems of the form \prettyref{eq:evo}
is our first topic of interest. This encompasses existence and uniqueness
of a solution $u$ and its continuous dependence on the data $f$.
For attack this question, we need to first define, what we mean by
a solution of \prettyref{eq:evo}. In the case $\mathcal{M}=1$, the
well established theory semigroups could be applied. There it is common
to write the problem as an ordinary differential equation $\partial_{t}u=-Au+f$
in an infinite dimensional state space and to define so-called \emph{mild
solutions} $u.$ These are continuous functions satisfying the equation
in an integrated sense. The existence and uniqueness of mild solutions
is then equivalent to the existence of a $C_{0}$-semigroup generated
by \textbf{$-A$} (for the theory of $C_{0}$-semigroups we refer
to the monographs \cite{Engel1998,ABHN_2011,Pazy1983}).\textbf{ }

Here, however, we follow a different approach with the aim to obtain
well-posedness of equations of the general form \prettyref{eq:evo}.
The main idea is to establish the operator sum $\partial_{t}\mathcal{M}+A$
as a suitable operator acting in space-time. Then, the uniqueness
and existence of solutions is equivalent to the bijectivity of this
operator and to obtain continuous dependence of $u$ on the data requires
for the continuity of the inverse operator $(\partial_{t}\mathcal{M}+A)^{-1}$.
In order to establish the operator sum properly, we need to define
the underlying function space and the operators $\partial_{t}$ and
$\mathcal{M}.$ As the underlying space we will use the Hilbert space
$H_{\rho}(\mathbb{R};H)$ consisting of (equivalence classes of) measurable
functions $f:\mathbb{R}\to H$, which are square integrable with respect
to the exponentially weighted Lebesgue measure $\e^{-2\rho t}\mbox{ d}t$
with $\rho\in\mathbb{R}$, equipped with its canonical inner product.
Then, the temporal derivative $\partial_{t}$ can be defined as the
closure of the derivative of test functions $C_{c}^{\infty}(\mathbb{R};H)$.
The so obtained operator, which will be denoted by $\partial_{0,\rho}$,
turns out to be a normal operator which is continuously invertible
if and only if $\rho\neq0.$ The idea to define the derivative in
that way goes back to \cite{picard1989hilbert}. As a normal operator,
$\partial_{0,\rho}$ has a spectral representation, which is given
by the so-called Fourier-Laplace transformation. Using this spectral
representation, we will define $\mathcal{M}$ to be an operator-valued
function of $\partial_{0,\rho}$ denoted by $M(\partial_{0,\rho}).$
Hence, \prettyref{eq:evo} takes the form 
\[
\left(\partial_{0,\rho}M(\partial_{0,\rho})+A\right)u=f
\]
on $H_{\rho}(\mathbb{R};H)$ and well-posedness of this problem can
be reformulated as the continuous invertibility of $\partial_{0,\rho}M(\partial_{0,\rho})+A$
on $H_{\rho}(\mathbb{R};H).$ We emphasize that the so obtained solutions
are not continuous in general and the so introduced notion of solutions
is weaker than the notion of mild solutions in the theory of $C_{0}$-semigroups.

The above mentioned perspective of looking at differential equations
as operator equations involving a sum of two unbounded operators,
is by no means new and was indeed successfully employed earlier and
also in the more general setting of Banach space theory. In particular,
when dealing for example with the issue of so-called maximal regularity
of evolution equations, this perspective appears quite natural. We
refer here to the seminal paper of da Prato and Grisvard \cite{daPrato1975},
where the closedness, closability and continuous invertibility of
operator sums are studied in a general Banach space setting. Since
we are, however, restricting ourselves to the Hilbert space setting
and since the operators involved have the above mentioned special
structure, these properties can be obtained easily and we do not need
to employ the deep and intricate results of \cite{daPrato1975}. A
further reference is \cite{Brezis1976}, where the authors derive
abstract results for the range of the sum of two binary relations
and apply them to derive existence results for differential equations
and inclusions. In \cite{Favini1999} an abstract operator equation
of the form $(TM-L)u=f$ is considered, and conditions on the operators
involved are derived to obtain closability and bounded invertibility
of this operator sum. In particular, setting $T=\partial_{0,\rho},M=M(\partial_{0,\rho})$
and $L=-A$, we end up with an evolutionary equation. However, in
general these operators do not satisfy the constraints imposed on
the operators $T,M$ and $L$ in \cite{Favini1999}.

Beyond the well-posedness in the sense of Hadamard, causality, as
a distinguishing property of evolutionary problems, needs to be addressed.
By causality we mean, roughly speaking, that the solution vanishes
as long as the source term does. More precisely, if $f$ is zero on
$]-\infty,a]$ for some $a\in\mathbb{R},$ then also $u$ vanishes
on the same interval $]-\infty,a].$ This can be seen as a characteristic
property of time evolution. The requirement of causality imposes a
strong constraint on admissible material laws, namely that the operator-valued
function describing the material properties needs to be analytic.
This is due to a theorem by Paley and Wiener, which characterizes
those $L_{2}(\mathbb{R})$-functions which are supported on the positive
real axis by properties of their Laplace transform (see \cite{Paley_Wiener}
and \prettyref{thm:Paley-Wiener} of this thesis). For a discussion
of causality and the concept of causal differential equations we refer
to the monograph \cite{lakshmikantham2010theory}. 

As it turns out, the framework of evolutionary equations covers a
broad class of linear time-shift invariant (i.e. autonomous) partial
differential equations. For instance, most (if not all) of the linear
differential equations of classical mathematical physics can be written
as evolutionary equations in our sense. In particular, coupling phenomena
can be conveniently incorporated and analyzed (see e.g. \cite{Mulholland2016,Mukhopadhyay2014_thermoelast,Mukhopadhyay2015_2temp,Picard2010_poroelastic,Picard2015_micropoloar}).
Moreover, other types of differential equations can be written as
evolutionary problems such as integro-differential equations (see
\cite{Trostorff2012_integro}), fractional differential equations
(see \cite{Picard2013_fractional}) or delay differential equations
(see \cite{Kalauch2011}). Also problems with transmission conditions
and complicated boundary conditions, such as impedance boundary conditions,
can be treated within this framework (see \cite{Picard2016_graddiv,Trostorff2013_bd_maxmon}).
We also mention that certain generalizations of evolutionary problems
were also already investigated. For instance, replacing the material
law operator $M(\partial_{0,\rho})$ by more general operators, which
fail to commute with the derivative $\partial_{0,\rho}$, allows for
the treatment of certain non-autonomous problems (see \cite{Picard2013_nonauto,Waurick2015_nonauto}).
Another direction of generalization is to incorporate nonlinear terms
by replacing the linear operator $A$ by a nonlinear operator, or
even more general, by a binary relation. The problem class obtained
is referred to as evolutionary inclusions (see \cite{Trostorff2012_NA,Trostorff2012_nonlin_bd,Trostorff2013_nonautoincl}).
A further direction in current research employing the framework of
evolutionary problems is the homogenization of partial differential
equations, or, to put it in another perspective, the continuous dependence
of the solution $u$ on the coefficient operators involved (see \cite{Waurick2012,Waurick2013,Waurick2013_fractional,Wauricl2014,Waurick2016}
and in particular the habilitation thesis \cite{Waurick_habil}).

Having secured well-posedness of an evolutionary problem, it is natural
to investigate qualitative properties of its solution $u$. As a first
property of interest of such a solution $u$, we mention its long-time
asymptotics. More precisely, we ask whether the solution decays to
zero with an exponential rate, provided that the given right-hand
side $f$ does. This property is known as \emph{exponential stability}
and was studied in \cite{Trostorff2013_stability,Trostorff2014_PAMM,Trostorff2015_secondorder}.
The first difficulty is that the solutions of evolutionary problems
are not continuous, so a classical point-wise estimate cannot be used.
The main idea to overcome this is to relax the notion of exponential
stability by the constraint that the solution should be square integrable
with respect to $\e^{2\mu t}\mbox{ d}t$ for some $\mu>0,$ that is
$u\in H_{-\mu}(\mathbb{R};H).$ Hence, the question of exponential
stability reduces to the question whether the solution operator $(\partial_{0,\rho}M(\partial_{0,\rho})+A)^{-1}$
, which is independent of the parameter $\rho$ for large enough $\rho,$
can be extended to $H_{-\mu}(\mathbb{R};H)$ for some positive $\mu$.
Similar ideas were already used in the study of exponential stability
for evolution equations via semigroups, where the exponential stability
can be defined in the point-wise sense. As it turns out, at least
in the Hilbert space setting, the extension of the solution operator
indeed yields the exponential decay of solutions in the point-wise
sense by the famous Gearhart-Prüß Theorem (see \cite{Pruss1984,Gearhart_1978,Herbst1983}). 

It should be mentioned that in the framework of semigroups a number
of other stability results are well-known. There is for example the
celebrated Arendt-Batty-Lyubich-V\~{u} Theorem \cite{Arendt1988,Lyubich1988},
giving a criterion for the strong convergence of a $C_{0}$-semigroup
to $0$, a recent result by Borichev and Tomilov \cite{Borichev2010},
providing a characterization of polynomial decay of solutions of evolution
equations in Hilbert spaces, or the Theorem of Batty-Duyckaerts \cite{Batty2008}
for further decay estimates of $C_{0}$-semigroups, which are obtained
by Laplace transform techniques. 

We recall that evolutionary equations, i.e. equations of the form
\prettyref{eq:evo}, are operator equations on $H_{\rho}(\mathbb{R};H).$
Hence, the functions involved are supported on the whole real line
in general. However, in the classical theory for evolution equations
it is common to consider functions supported on the positive real
line and to add an initial condition for the unknown $u$. Such initial
value problems can also be incorporated into the framework of evolutionary
problems. First, we note that due to the causality of the solution
operator, right-hand sides $f$ supported on $\mathbb{R}_{\geq0}$
also yield solutions supported on $\mathbb{R}_{\geq0}.$ So, the question
is how to incorporate an initial condition for $u$. Following the
ideas developed in \cite{Picard_McGhee}, this can be done by adding
suitable distributional source terms on the right-hand side of \prettyref{eq:evo}.
More precisely, an initial condition turns out to be encoded as an
impulse at time zero by a Dirac-delta distribution at the initial
time zero. Indeed, employing a Sobolev embedding theorem and the causality
of the solution operator, one can show that for this modified right-hand
side the solution $u$ indeed attains the initial value. However,
in order to do so, one needs to extend the solution operator to distributional
right-hand sides. This can be done easily, using the extrapolation
spaces for the time derivative operator $\partial_{0,\rho}$. For
the theory of extrapolation spaces and their application to partial
differential equations we refer to \cite{Picard_McGhee,Picard2000}
and to \cite{Nagel1997} with a focus on $C_{0}$-semigroups. 

Since evolutionary equations also cover purely algebraic equations
(choose for a simple example $\mathcal{M}=\partial_{0,\rho}^{-1}$)
or, more generally, differential-algebraic equations, an initial condition
for the unknown $u$ does not make sense for every choice of $\mathcal{M}.$
Thus, one has to address the question what meaningful initial values
are in presence of more complicated material laws. The discussion
of admissible spaces of initial conditions is well-known in the theory
of differential-algebraic equations, see e.g. \cite{Mehrmann2006,Reis2007}.
Moreover, since also certain classes of delay equations are covered
by our approach, where it seems to be more natural to prescribe whole
histories instead of initial values at zero, the question arises how
those given histories could be incorporated as data for evolutionary
equations.

Once initial value problems for evolutionary equations are settled,
one can go further and ask for more regular solutions, or, more specifically,
when and where solutions are continuous in time. This will be done
by associating a $C_{0}$-semigroup to the given initial value problem,
which of course requires some additional constraints on the operators
involved. Having the $C_{0}$-semigroup at hand, it is then possible
to compare the exponential stability results for evolutionary equations
with the classical exponential stability for the associated semigroup
and to provide a generalization of the Gearhart-Prüß Theorem. 

This thesis has three chapters and an appendix. In the first chapter
we introduce the framework of evolutionary equations and study their
well-posedness. This includes the classical Hadamard requirements
of uniqueness, existence and continuous dependence of solutions as
well as the causality of the solution operator. In Section 1.3 we
discuss several examples of partial differential equations and show
that these problems can be written as evolutionary equations and their
well-posedness can be shown using the results provided in Chapter
1. These examples cover classical linear partial differential equations
from mathematical physics as well as classes of delay equations and
integro-differential equations. The second chapter is devoted to the
exponential stability of evolutionary problems. There, we introduce
the notion of exponential stability and derive criteria for the operators
involved to obtain exponentially stable evolutionary problems. In
particular, we are focusing on second-order problems, where we provide
a way to reformulate the second order problem as a suitable first
order problem, for which the exponential stability can be shown. Again,
at the end of Chapter 2, we discuss several examples of different
classes of partial differential equations, whose asymptotics are studied
with the help of the abstract results of the previous sections. Finally,
in Chapter 3 we address the problem of admissible initial values and
histories for a given evolutionary problem and discuss how these given
data could be incorporated into the setting of evolutionary problems.
For doing so, we introduce the extrapolation space $H_{\rho}^{-1}(\mathbb{R};H)$
and generalized cut-off operators on that space. These operators will
then be used to formulate initial value problems and problems with
prescribed histories in Section 3.2. Section 3.3 is devoted to the
regularity of pure initial value problems and provides a Hille-Yosida
type result for the existence of an associated $C_{0}$-semigroup
on a suitable Hilbert space. Moreover, we present a generalization
of the Gearhart-Prüß Theorem to a class of evolutionary problems.
In the last section in Chapter 3 we discuss several examples to illustrate
the results. In the appendix we recall some well-known theorems, which
will be needed in particular in Chapter 3. 

We assume that the reader is familiar with basic functional analysis,
in particular with Hilbert space theory. For these topics we refer
to the monographs \cite{Yosida,Weidmann,Phillips1959,Rudin_fa}. Throughout,
all Hilbert spaces are assumed to be complex and the inner product
$\langle\cdot|\cdot\rangle$ is assumed to be linear in the second
and conjugate linear in the first argument.

\subsection*{Acknowledgement}

First of all, I like to express my gratitude to my PhD-supervisor,
co-author and friend Prof. Dr. Rainer Picard. His clear and mind-opening
perspective on partial differential equations and mathematics in general
build the foundation of this thesis. Second, I would like to thank
my friend Dr. Marcus Waurick for carefully reading this manuscript
and for uncountable hours (at least, I can not count them) of scientific
discussions and pure nonsense. It is a pleasure and a great luck to
have colleagues and friends like them. Furthermore, I thank Prof.
Dr. Ralph Chill, my mentor for this thesis, for several valuable and
helpful comments, remarks and fruitful discussions.

Moreover, I am grateful to Prof. Dr. Jürgen Voigt, who was my first
analysis teacher and my diploma supervisor. Almost everything I know
in analysis, I have learned from him. I thank Prof. Dr. Stefan Siegmund
for supporting my academic career in many ways and the members of
the Institute for Analysis of the TU Dresden for providing such a
beautiful and familiar working atmosphere: Dr. Anke Kalauch, Dr. Eva
Fašangová, Dr. Norbert Koksch, Jeremias Epperlein, Helena Malinowski,
Sebastian Mildner, Reinhard Stahn, Lars Perlich and Markus Hartlapp.
Many thanks go to Christine Heinke, Monika Gaede-Samat and Karola
Schreiter for their support and in particular for keeping away from
me as much bureaucracy as possible. 

Finally, I like to thank my parents and my whole family for their
life-long support and last but not least my (almost) wife Anja and
our daughter Nele for enriching every aspect of my life in so many
ways.

\vspace{3cm}

\begin{center}
- Warning: May contain nuts%
\footnote{This warning is stated for the case, that this thesis will be distributed
in Ankh-Morpork, see \cite{Pratchett2002}.%
} -
\par\end{center}

\newpage{}

$\,$\thispagestyle{empty}

\chapter{The framework of evolutionary problems}

In this chapter we introduce the framework of evolutionary equations.
The class was introduced by R. Picard in \cite{Picard} and it turned
out that it covers a broad class of different types of differential
equations. The key feature of the framework presented here is the
realization of the time-derivative as an accretive operator on a suitable
Hilbert space, which can be seen as an $L_{2}$-analogue of the continuous
functions equipped with the Morgenstern-norm for the treatment of
ordinary differential equations (see \cite{Morgenstern}). The idea
of introducing the time derivative in an exponentially weighted $L_{2}$-space
originates from \cite{picard1989hilbert}, where certain integral
transformations on Hilbert spaces were considered as functions of
suitable differential operators. We begin by introducing this Hilbert
space and the time derivative established on it. After that, we define
what we mean by an evolutionary problem and address the question of
well-posedness and causality for this class of problems. In the last
section of this chapter we present some examples.

\section{The Hilbert space setting and the time derivative}

Throughout this section let $H$ be a Hilbert space. 
\begin{defn*}
Let $\rho\in\mathbb{R}.$ We define 
\[
\mathcal{L}_{2,\rho}(\mathbb{R};H)\coloneqq\left\{ f:\mathbb{R}\to H\,;\, f\mbox{ measurable, }\intop_{\mathbb{R}}|f(t)|_{H}^{2}\e^{-2\rho t}\mbox{ d}t<\infty\right\} .
\]
Moreover, we introduce the equivalence relation $\cong$ on $\mathcal{L}_{2,\rho}(\mathbb{R};H)$
by $f\cong g$ if $f(t)=g(t)$ for a.e. $t\in\mathbb{R}$ with respect
to the Lebesgue measure on $\mathbb{R}.$ As usual we define 
\[
H_{\rho}(\mathbb{R};H)\coloneqq\faktor{\mathcal{L}_{2,\rho}(\mathbb{R};H)}{\cong}
\]
which is a Hilbert space with respect to the inner product 
\[
\langle f|g\rangle_{\rho}\coloneqq\intop_{\mathbb{R}}\langle f(t)|g(t)\rangle_{H}\e^{-2\rho t}\mbox{ d}t\quad(f,g\in H_{\rho}(\mathbb{R};H)).
\]
\end{defn*}
\begin{rem}
$\,$

\begin{enumerate}[(a)] \item If we choose $\rho=0$ in the latter
definition, $H_{0}(\mathbb{R};H)$ coincides with the usual Bochner
space $L_{2}(\mathbb{R};H)$. 

\item For $\rho\in\mathbb{R}$ the operator 
\begin{align*}
\e^{-\rho m}:H_{\rho}(\mathbb{R};H) & \to L_{2}(\mathbb{R};H)\\
f & \mapsto\left(t\mapsto\e^{-\rho t}f(t)\right)
\end{align*}
is obviously unitary. In particular, since $\e^{-\rho m}$ is a bijection
on $C_{c}^{\infty}(\mathbb{R};H)$\nomenclature[B_010]{$C(\mathbb{R};H)$}{the space of continuous functions $f:\mathbb{R}\to H$. }\nomenclature[B_040]{$C_{c}(\mathbb{R};H)$}{the space of continuous functions $f:\mathbb{R}\to H$ having compact support. }\nomenclature[B_020]{$C^k(\mathbb{R};H)$}{the space of $k$ times continuously differentiable functions $f:\mathbb{R}\to H$. }\nomenclature[B_030]{$C^{\infty}(\mathbb{R};H)$}{$\bigcap_{k\in \mathbb{N}} C^k(\mathbb{R};H)$.}\nomenclature[B_050]{$C^\infty_c(\mathbb{R};H)$}{$C^\infty(\mathbb{R};H)\cap C_c(\mathbb{R};H)$.},
we derive that $C_{c}^{\infty}(\mathbb{R};H)$ is dense in $H_{\rho}(\mathbb{R};H).$ 

\end{enumerate}
\end{rem}
In fact, we can also show a slightly stronger result than the density
of $C_{c}^{\infty}(\mathbb{R};H)$ in $H_{\rho}(\mathbb{R};H)$.
\begin{lem}
\label{lem:density test fct}Let $\rho,\mu\in\mathbb{R}$ and $f\in H_{\rho}(\mathbb{R};H)\cap H_{\mu}(\mathbb{R};H).$
Then there exists a sequence $(\varphi_{n})_{n\in\mathbb{N}}$ in
$C_{c}^{\infty}(\mathbb{R};H)$ such that $\varphi_{n}\to f$ in $H_{\rho}(\mathbb{R};H)$
and $H_{\mu}(\mathbb{R};H).$ \end{lem}
\begin{proof}
Without loss of generality let $\rho<\mu$. For $n\in\mathbb{N}$
we define $f_{n}(t)\coloneqq\chi_{[-n,n]}(t)f(t)$ for $t\in\mathbb{R}$
and get $f_{n}\to f$ in $H_{\rho}(\mathbb{R};H)$ and $H_{\mu}(\mathbb{R};H)$
as $n\to\infty$ by dominated convergence. For $n\in\mathbb{N}$ we
find $\varphi_{n}\in C_{c}^{\infty}(\mathbb{R};H)$ with $\spt\varphi_{n}\subseteq[-n,n]$
such that $|\varphi_{n}-f_{n}|_{\mu}<\frac{1}{2n}\e^{(\rho-\mu)n}.$
Hence, 
\[
|\varphi_{n}-f_{n}|_{\rho}=\left(\,\intop_{-n}^{n}|\varphi_{n}(t)-f_{n}(t)|_{H}^{2}\e^{-2\rho t}\mbox{ d}t\right)^{\frac{1}{2}}\leq|\varphi_{n}-f_{n}|_{\mu}\e^{(\mu-\rho)n}<\frac{1}{2n}
\]
and thus, we derive 
\begin{align*}
|\varphi_{n}-f|_{\mu}<\frac{1}{2n}\e^{(\rho-\mu)n}+|f_{n}-f|_{\mu} & \to0,\\
|\varphi_{n}-f|_{\rho}<\frac{1}{2n}+|f_{n}-f|_{\rho} & \to0
\end{align*}
as $n\to\infty.$
\end{proof}
We now define the time derivative as an operator acting on $H_{\rho}(\mathbb{R};H).$
\begin{lem}
Let $\rho\in\mathbb{R}.$ Then the operator 
\begin{align*}
\partial_{0,\rho,c}:C_{c}^{\infty}(\mathbb{R};H)\subseteq H_{\rho}(\mathbb{R};H) & \to H_{\rho}(\mathbb{R};H)\\
\varphi & \mapsto\varphi'
\end{align*}
is closable.\end{lem}
\begin{proof}
Let $\varphi,\psi\in C_{c}^{\infty}(\mathbb{R};H)$. Then we compute,
by using integration by parts, 
\begin{align*}
\langle\partial_{0,\rho,c}\varphi|\psi\rangle_{\rho} & =\intop_{\mathbb{R}}\langle\varphi'(t)|\psi(t)\rangle_{H}\e^{-2\rho t}\mbox{ d}t\\
 & =-\intop_{\mathbb{R}}\langle\varphi(t)|\psi'(t)\rangle_{H}\e^{-2\rho t}\mbox{ d}t+2\rho\intop_{\mathbb{R}}\langle\varphi(t)|\psi(t)\rangle_{H}\e^{-2\rho t}\mbox{ d}t\\
 & =\langle\varphi|-\partial_{0,\rho,c}\psi+2\rho\psi\rangle_{\rho},
\end{align*}
which shows $\psi\in D(\partial_{0,\rho,c}^{\ast})$ with $\partial_{0,\rho,c}^{\ast}\psi=-\partial_{0,\rho,c}\psi+2\rho\psi.$
Hence, $\partial_{0,\rho,c}\subseteq-\partial_{0,\rho,c}^{\ast}+2\rho,$
which in particular yields the closability of $\partial_{0,\rho,c}.$ \end{proof}
\begin{defn*}
Let $\rho\in\mathbb{R}.$ We define $\partial_{0,\rho}\coloneqq\overline{\partial_{0,\rho,c}}$
and $H_{\rho}^{1}(\mathbb{R};H)\coloneqq D(\partial_{0,\rho}).$ We
equip $H_{\rho}^{1}(\mathbb{R};H)$ with the inner product 
\[
\langle f|g\rangle_{\rho,1}\coloneqq\langle f|g\rangle_{\rho}+\langle\partial_{0,\rho}f|\partial_{0,\rho}g\rangle_{\rho}\quad(f,g\in H_{\rho}^{1}(\mathbb{R};H)).
\]
Note that $H_{\rho}^{1}(\mathbb{R};H)$ is a Hilbert space.
\end{defn*}
Our next goal is to show that $\partial_{0,\rho}$ is normal. For
doing so, we show that the Fourier-Laplace transformation $\mathcal{L}_{\rho}:H_{\rho}(\mathbb{R};H)\to L_{2}(\mathbb{R};H)$
defined as the unitary extension of 
\[
\left(\mathcal{L}_{\rho}\varphi\right)(t)\coloneqq\frac{1}{\sqrt{2\pi}}\intop_{\mathbb{R}}\e^{-\left(\i t+\rho\right)s}\varphi(s)\mbox{ d}s\quad(\varphi\in C_{c}^{\infty}(\mathbb{R};H),\, t\in\mathbb{R})
\]
establishes a spectral representation for $\partial_{0,\rho}.$ For
a deeper study of the Fourier-Laplace transformation we refer to Appendix
A.
\begin{prop}
\label{prop:spectral_repr_partial_0}Let $\rho\in\mathbb{R}.$ Then
we have 
\[
\partial_{0,\rho}=\mathcal{L}_{\rho}^{\ast}\left(\i\mathrm{m}+\rho\right)\mathcal{L}_{\rho},
\]
where $\m:D(\m)\subseteq L_{2}(\mathbb{R};H)\to L_{2}(\mathbb{R};H)$
is defined as multiplication with the argument with maximal domain,
i.e. $\left(\m f\right)(t)\coloneqq tf(t)$ for $t\in\mathbb{R}$
and $f\in D(\m)$ with 
\[
D(\m)\coloneqq\left\{ g\in L_{2}(\mathbb{R};H)\,;\,(t\mapsto tg(t))\in L_{2}(\mathbb{R};H)\right\} .
\]
\end{prop}
\begin{proof}
The proof will be done in three steps.\\
Firstly, we show that $\mathcal{S}_{\rho}(\mathbb{R};H)\coloneqq\left\{ f:\mathbb{R}\to H\,;\,\e^{-\rho m}f\in\mathcal{S}(\mathbb{R};H)\right\} $
is a core for $\partial_{0,\rho}$ and $\partial_{0,\rho}f=f'$ for
$f\in\mathcal{S}_{\rho}(\mathbb{R};H).$ Here, $\mathcal{S}(\mathbb{R};H)$\nomenclature[B_060]{$\mathcal{S}(\mathbb{R};H)$}{the Schwartz-space of rapidly decreasing functions $f:\mathbb{R}\to H$.}
denotes the Schwartz-space of rapidly decreasing smooth functions
with values in $H$. By definition of $\partial_{0,\rho}$, the space
$C_{c}^{\infty}(\mathbb{R};H)\subseteq\mathcal{S}_{\rho}(\mathbb{R};H)$
is a core for $\partial_{0,\rho}.$ Thus, it suffices to prove that
$\mathcal{S}_{\rho}(\mathbb{R};H)\subseteq D(\partial_{0,\rho}).$
For doing so, let $f\in\mathcal{S}_{\rho}(\mathbb{R};H)$. Moreover,
let $\psi\in C_{c}^{\infty}(\mathbb{R})$ such that $\psi=1$ on $[-1,1]$
and set $\psi_{n}(t)\coloneqq\psi(n^{-1}t)$ for $n\in\mathbb{N},t\in\mathbb{R}.$
Then, for each $n\in\mathbb{N}$ we have $\psi_{n}f\in C_{c}^{\infty}(\mathbb{R};H)$
and $\psi_{n}f\to f$ in $H_{\rho}(\mathbb{R};H)$ as $n\to\infty$
by dominated convergence. Moreover, $\partial_{0,\rho}(\psi_{n}f)=\psi_{n}'f+\psi_{n}f'\to f'$
as $n\to\infty$ in $H_{\rho}(\mathbb{R};H),$ again by dominated
convergence. This shows the claim.\\
Secondly, we show that $C_{c}^{\infty}(\mathbb{R};H)$ is a core for
$\m,$ and thus, also for $\i\m+\rho.$ Let $f\in D(\m)$ and set
$f_{n}(t)\coloneqq\chi_{[-n,n]}(t)f(t)$ for $t\in\mathbb{R},n\in\mathbb{N}.$
Then, clearly $f_{n}\to f$ and $\m f_{n}\to\m f$ in $L_{2}(\mathbb{R};H)$
as $n\to\infty$ by dominated convergence. Hence, for $\varepsilon>0$
we find $n\in\mathbb{N}$ such that 
\[
|f-f_{n}|_{L_{2}},|\m f-\m f_{n}|_{L_{2}}<\varepsilon.
\]
Choose now $\varphi\in C_{c}^{\infty}(\mathbb{R};H)$ with $\spt\varphi\subseteq\left[-n-1,n+1\right]$
\nomenclature[Z_060]{$\spt f$}{the support of a function $f$.}and
$|\varphi-f_{n}|_{L_{2}}<\frac{\varepsilon}{n+1}.$ Then 
\[
|\m f_{n}-\m\varphi|_{L_{2}}\leq(n+1)|f_{n}-\varphi|_{L_{2}}<\varepsilon
\]
and hence, 
\[
|f-\varphi|_{L_{2}},|\m f-\m\varphi|_{L_{2}}<2\varepsilon,
\]
which shows the assertion.\\
Finally, we now show the asserted statement. Let $f\in\mathcal{S}_{\rho}(\mathbb{R};H).$
Then $\mathcal{L}_{\rho}f\in\mathcal{S}(\mathbb{R};H)\subseteq D(\i\m+\rho)$
(cp. \prettyref{prop:Fourier_Schwartz}) and we compute, using the
first step 
\begin{align*}
\left(\mathcal{L}_{\rho}\partial_{0,\rho}f\right)(t) & =\frac{1}{\sqrt{2\pi}}\intop_{\mathbb{R}}\e^{-(\i t+\rho)s}f'(s)\mbox{ d}s\\
 & =(\i t+\rho)\frac{1}{\sqrt{2\pi}}\intop_{\mathbb{R}}\e^{-(\i t+\rho)s}f(s)\mbox{ d}s\\
 & =\left((\i\m+\rho)\mathcal{L}_{\rho}f\right)(t)\quad(t\in\mathbb{R}).
\end{align*}
Since $C_{c}^{\infty}(\mathbb{R};H)$ is a core for $\i\m+\rho$ according
to the second step, so is $\mathcal{S}(\mathbb{R};H)$ and hence,
$\mathcal{S}_{\rho}(\mathbb{R};H)$ is a core for $\mathcal{L}_{\rho}^{\ast}(\i\m+\rho)\mathcal{L}_{\rho},$
since $\mathcal{L}_{\rho}$ is a bijection from $\mathcal{S}_{\rho}(\mathbb{R};H)$
to $\mathcal{S}(\mathbb{R};H)$ by \prettyref{prop:Fourier_Schwartz}.
Since $\partial_{0,\rho}$ and $\mathcal{L}_{\rho}^{\ast}(\i\m+\rho)\mathcal{L}_{\rho}$
coincide on $\mathcal{S}_{\rho}(\mathbb{R};H),$ which is a core for
both operators, we derive the assertion.
\end{proof}
As an immediate consequence, we obtain the normality of $\partial_{0,\rho}.$
\begin{prop}
Let $\rho\in\mathbb{R}.$ Then $\partial_{0,\rho}$ is a normal operator
with $\partial_{0,\rho}^{\ast}=-\partial_{0,\rho}+2\rho$, $\Re\partial_{0,\rho}=\rho$
and $\Im\partial_{0,\rho}=\i(\rho-\partial_{0,\rho}).$ Moreover,
$\sigma(\partial_{0,\rho})=C\sigma(\partial_{0,\rho})=\mathbb{C}_{\Re=\rho}$.\nomenclature[O_150]{$\sigma(T)$}{the spectrum of an operator $T$.}\nomenclature[O_160]{$C\sigma(T)$}{the continuous spectrum of an operator $T$.}\nomenclature[A_020]{$\mathbb{C}_{\Re \gtreqqless \alpha}$}{the set of all complex numbers with real part greater than, equal to or less than some real number $\alpha$.}\nomenclature[O_140]{$\rho(T)$}{the resolvent set of an operator $T$.}\end{prop}
\begin{proof}
The assertion follows immediately from the respective results for
the multiplication operator $\i\m+\rho$ on $L_{2}(\mathbb{R};H)$
and \prettyref{prop:spectral_repr_partial_0}.
\end{proof}
The latter proposition especially yields that $\partial_{0,\rho}$
is continuously invertible if and only if $\rho\ne0.$ In fact, the
formula for $\partial_{0,\rho}^{-1}$ differs for $\rho>0$ or $\rho<0$
as the next proposition shows.
\begin{prop}
\label{prop:derivative_invertible}Let $\rho\ne0.$ Then $\partial_{0,\rho}$
is continuously invertible with $\|\partial_{0,\rho}^{-1}\|=\frac{1}{|\rho|}$
and for $f\in H_{\rho}(\mathbb{R};H),t\in\mathbb{R}$ we have that
\[
\left(\partial_{0,\rho}^{-1}f\right)(t)=\begin{cases}
\intop_{-\infty}^{t}f(s)\,\dd s, & \mbox{ if }\rho>0,\\
-\intop_{t}^{\infty}f(s)\,\dd s, & \mbox{ if }\rho<0.
\end{cases}
\]
\end{prop}
\begin{proof}
The bounded invertibility and the equality $\|\partial_{0,\rho}^{-1}\|=\frac{1}{|\rho|}$
follow from \prettyref{prop:spectral_repr_partial_0}. Let now $\varphi\in C_{c}^{\infty}(\mathbb{R};H)$.
If $\rho>0$ we get that $\mu\coloneqq\left(t\mapsto\intop_{-\infty}^{t}\varphi(s)\mbox{ d}s\right)\in H_{\rho}(\mathbb{R};H).$
Indeed, 
\begin{align*}
\intop_{\mathbb{R}}\left|\mu(t)\right|_{H}^{2}\e^{-2\rho t}\mbox{ d}t & =\intop_{\mathbb{R}}\left|\:\intop_{-\infty}^{t}\varphi(s)\e^{-\frac{\rho}{2}s}\e^{\frac{\rho}{2}s}\mbox{ d}s\right|_{H}^{2}\e^{-2\rho t}\mbox{ d}t\\
 & \leq\frac{1}{\rho}\intop_{\mathbb{R}}\intop_{-\infty}^{t}|\varphi(s)|_{H}^{2}\e^{-\rho s}\mbox{ d}s\,\e^{-\rho t}\mbox{ d}t\\
 & =\frac{1}{\rho}\intop_{\mathbb{R}}|\varphi(s)|_{H}^{2}\e^{-\rho s}\intop_{s}^{\infty}\e^{-\rho t}\mbox{ d}t\mbox{ d}s\\
 & =\frac{1}{\rho^{2}}\intop_{\mathbb{R}}|\varphi(s)|_{H}^{2}\e^{-2\rho s}\mbox{ d}s\\
 & =\frac{1}{\rho^{2}}|\varphi|_{\rho}^{2}.
\end{align*}
Moreover, for each $\psi\in C_{c}^{\infty}(\mathbb{R};H)$ we have
\begin{align*}
\langle\partial_{0,\rho}\psi|\mu\rangle_{\rho} & =\intop_{\mathbb{R}}\langle\psi'(t)|\mu(t)\rangle_{H}\e^{-2\rho t}\mbox{ d}t\\
 & =-\intop_{\mathbb{R}}\langle\psi(t)|\varphi(t)-2\rho\mu(t)\rangle_{H}\e^{-2\rho t}\mbox{ d}t\\
 & =\langle\psi|-\varphi+2\rho\mu\rangle_{\rho}
\end{align*}
and thus $\mu\in D(\partial_{0,\rho}^{\ast})=D(\partial_{0,\rho})$
with $-\varphi+2\rho\mu=\partial_{0,\rho}^{\ast}\mu=-\partial_{0,\rho}\mu+2\rho\mu$
which gives $\partial_{0,\rho}\mu=\varphi,$ i.e. $\mu=\partial_{0,\rho}^{-1}\varphi.$
The statement for $\rho<0$ follows by arguing analogously and thus,
we omit it. \end{proof}
\begin{rem}
As $\partial_{0,\rho}$ is continuously invertible if $\rho\ne0,$
we may also equip $H_{\rho}^{1}(\mathbb{R};H)$ with the inner product
\[
\langle f|g\rangle\coloneqq\langle\partial_{0,\rho}f|\partial_{0,\rho}g\rangle_{\rho}\quad(f,g\in H_{\rho}^{1}(\mathbb{R};H)),
\]
which yields an equivalent norm on $H_{\rho}^{1}(\mathbb{R};H).$ 
\end{rem}
We conclude this section with a version of the Sobolev embedding theorem.
\begin{prop}[Sobolev embedding]
\label{prop:Sobolev} Let $\rho\in\mathbb{R}$. We define the space\nomenclature[B_080]{$C_\rho(\mathbb{R};H)$}{space of continuous functions $f:\mathbb{R}\to H$ with $\sup_{t\in \mathbb{R}} \vert f(t)\e^{-\rho t}\vert_H<\infty$ for some $\rho\in \mathbb{R}$.}
\[
C_{\rho}(\mathbb{R};H)\coloneqq\left\{ f:\mathbb{R}\to H\,;\, f\mbox{ continuous,}\,\sup_{t\in\mathbb{R}}|f(t)\e^{-\rho t}|_{H}<\infty\right\} 
\]
equipped with the norm 
\[
|f|_{\infty,\rho}\coloneqq\sup_{t\in\mathbb{R}}|f(t)\e^{-\rho t}|_{H}.
\]
Moreover, we define $C_{\rho,0}(\mathbb{R};H)\coloneqq\left\{ f\in C_{\rho}(\mathbb{R};H)\,;\, f(t)\e^{-\rho t}\to0\quad(t\to\pm\infty)\right\} $.
Then we have that $H_{\rho}^{1}(\mathbb{R};H)\hookrightarrow C_{\rho,0}(\mathbb{R};H).$
\nomenclature[Z_040]{$\hookrightarrow$}{continuous embedding.}\end{prop}
\begin{proof}
Let $\varphi\in C_{c}^{\infty}(\mathbb{R};H)$. For $\rho>0$ we compute
\begin{align*}
|\varphi(t)|_{H} & =\left|\,\intop_{-\infty}^{t}\varphi'(s)\mbox{ d}s\right|_{H}\\
 & \leq\left(\,\intop_{-\infty}^{t}|\varphi'(s)|_{H}^{2}\e^{-2\rho s}\mbox{ d}s\right)^{\frac{1}{2}}\frac{1}{\sqrt{2\rho}}\e^{\rho t}\\
 & \leq|\varphi|_{\rho,1}\frac{1}{\sqrt{2\rho}}\e^{\rho t}\quad(t\in\mathbb{R}),
\end{align*}
which yields the continuity of the mapping $\mathrm{id}:C_{c}^{\infty}(\mathbb{R};H)\subseteq H_{\rho}^{1}(\mathbb{R};H)\to C_{\rho,0}(\mathbb{R};H).$
Noting that $C_{\rho,0}(\mathbb{R};H)$ is a Banach space, we obtain
the asserted embedding by continuous extension. If $\rho<0$, we compute
\begin{align*}
|\varphi(t)|_{H} & =\left|\intop_{t}^{\infty}\varphi'(s)\mbox{ d}s\right|_{H}\\
 & \leq\left(\intop_{t}^{\infty}|\varphi'(s)|_{H}^{2}\e^{-2\rho s}\mbox{ d}s\right)^{\frac{1}{2}}\frac{1}{\sqrt{2|\rho|}}\e^{\rho t}\\
 & \leq|\varphi|_{\rho,1}\frac{1}{\sqrt{2|\rho|}}\e^{\rho t}
\end{align*}
for each $t\in\mathbb{R}$ and so the assertion follows as above.
For $\rho=0$ we estimate 
\begin{align*}
|\varphi(t)|_{H} & \leq\intop_{t-1}^{t}|\varphi(t)-\varphi(s)|_{H}\mbox{ d}s+\intop_{t-1}^{t}|\varphi(s)|_{H}\mbox{ d}s\\
 & \leq\intop_{t-1}^{t}\left|\intop_{s}^{t}\varphi'(r)\mbox{ d}r\right|_{H}\mbox{ d}s+|\varphi|_{L_{2}}\\
 & \leq\intop_{t-1}^{t}\left(\intop_{s}^{t}|\varphi'(r)|_{H}^{2}\mbox{ d}r\right)^{\frac{1}{2}}\sqrt{t-s}\mbox{ d}s+|\varphi|_{L_{2}}\\
 & \leq|\varphi'|_{L_{2}}+|\varphi|_{L_{2}}\\
 & \leq\sqrt{2}|\varphi|_{0,1},
\end{align*}
which yields the assertion.\end{proof}
\begin{rem}
One can show a slightly stronger result than in \prettyref{prop:Sobolev}.
In fact, $H_{\rho}^{1}(\mathbb{R};H)$ is continuously embedded into
an (exponentially weighted) space of Hölder continuous functions (see
e.g. \cite[p. 282]{evans2010partial} for $\rho=0$, \cite[Lemma 3.1.59]{Picard_McGhee}
for $\rho>0$). 
\end{rem}

\section{Evolutionary problems\label{sec:Evolutionary-problems}}

With the Hilbert space setting developed in the previous section at
hand, we are now able to define so-called evolutionary problems. Evolutionary
problems, as they were introduced in \cite{Picard}, consist of two
equations with two unknowns $u$ and $w$ belonging to $H_{\rho}(\mathbb{R};H)$
for some $\rho\in\mathbb{R}$. First, 
\[
\partial_{0,\rho}w+Au=f,
\]
where $f\in H_{\rho}(\mathbb{R};H)$ is an arbitrary source term and
$A:D(A)\subseteq H\to H$ is a linear operator acting on $H$, which
is extended to $H_{\rho}(\mathbb{R};H)$ in the canonical way by $\left(Au\right)(t)\coloneqq Au(t)$
for $t\in\mathbb{R}$ and $u\in\left\{ g\in H_{\rho}(\mathbb{R};H)\,;\, g(t)\in D(A)\mbox{ for a.e. }t\in\mathbb{R},\,\left(t\mapsto Ag(t)\right)\in H_{\rho}(\mathbb{R};H)\right\} .$
The first equation is completed by a second one, which links $u$
and $w$:
\[
w=\mathcal{M}u,
\]
where $\mathcal{M}:D(\mathcal{M})\subseteq H_{\rho}(\mathbb{R};H)\to H_{\rho}(\mathbb{R};H)$
is a linear operator. Substituting the second equation into the first
one, we end up with an equation of the form 
\begin{equation}
\left(\partial_{0,\rho}\mathcal{M}+A\right)u=f\label{eq:evol}
\end{equation}
and we refer to \prettyref{eq:evol} as an \emph{evolutionary equation.
}We focus here on operators $\mathcal{M}=M(\partial_{0,\rho})$, defined
as operator-valued functions of the normal operator $\partial_{0,\rho}$,
which in particular implies that $\mathcal{M}$ commutes with translations
in time. Thus, we are dealing with autonomous equations. The precise
definition of $\mathcal{M}=M(\partial_{0,\rho})$ is as follows.
\begin{defn*}
Let $M:D(M)\subseteq\mathbb{C}\to L(H)$ \nomenclature[B_120]{$L(X;Y)$}{the space of bounded linear operators $T:X\to Y$ between two Banach spaces $X$ and $Y$.}\nomenclature[B_130]{$L(X)$}{the space $L(X;X)$.}analytic,
$D(M)$ open such that $\mathbb{C}_{\Re>\rho_{0}}\subseteq D(M)$
for some $\rho_{0}\in\mathbb{R}$. Then we call $M$ a \emph{linear
material law. }Moreover, we define the set 
\[
S_{M}\coloneqq\left\{ \rho\in\mathbb{R}\,;\,\i t+\rho\in D(M)\mbox{ for almost every }t\in\mathbb{R}\right\} 
\]
and for each $\rho\in S_{M}$ we define the operator $M(\i\m+\rho):D(M(\i\m+\rho))\subseteq L_{2}(\mathbb{R};H)\to L_{2}(\mathbb{R};H)$
as follows 
\[
\left(M(\i\m+\rho)f\right)(t)\coloneqq M(\i t+\rho)f(t)\quad(t\in\mathbb{R})
\]
for each $f\in D(M(\i\m+\rho))\coloneqq\{g\in L_{2}(\mathbb{R};H)\,;\,\left(t\mapsto M(\i t+\rho)g(t)\right)\in L_{2}(\mathbb{R};H)\}.$
Furthermore, we define the operator $M(\partial_{0,\rho}):D(M(\partial_{0,\rho}))\subseteq H_{\rho}(\mathbb{R};H)\to H_{\rho}(\mathbb{R};H)$
by 
\[
M(\partial_{0,\rho})\coloneqq\mathcal{L}_{\rho}^{\ast}M(\i\m+\rho)\mathcal{L}_{\rho}
\]
with its natural domain.
\end{defn*}
We state some elementary properties of linear material laws.
\begin{prop}
\label{prop:basic_prop_material_law} Let $M:D(M)\subseteq\mathbb{C}\to L(H)$
be a linear material law. Then for each $\rho\in S_{M}$ the operators
$M(\i\m+\rho)$ and $M(\partial_{0,\rho})$ are densely defined and
closed.\end{prop}
\begin{proof}
Let $\rho\in S_{M}.$ By unitary equivalence, it suffices to prove
the asserted properties for the operator $M(\i\m+\rho).$ For showing
the density of $D(M(\i\m+\rho)),$ we set 
\[
D\coloneqq\left\{ t\in\mathbb{R}\,;\,\i t+\rho\in D(M)\right\} ,
\]
which is an open subset of $\mathbb{R}$ with $\lambda(\mathbb{R}\setminus D)=0.$
The latter gives that $L_{2}(D;H)=L_{2}(\mathbb{R};H)$ (more precisely,
the mapping $L_{2}(\mathbb{R};H)\ni f\mapsto f|_{D}\in L_{2}(D;H)$
is unitary). Moreover, since $C_{c}^{\infty}(D;H)$ is dense in $L_{2}(D;H),$
we infer the density of $C_{c}^{\infty}(D;H)$ in $L_{2}(\mathbb{R};H).$
Finally, using that $M$ is bounded on compact subsets of $D$, we
have that $C_{c}^{\infty}(D;H)\subseteq D(M(\i\m+\rho)$), and so
the density of $D(M(\i\m+\rho))$ follows.\\
It is left to prove the closedness of $M(\i\m+\rho).$ For doing so,
let $(f_{n})_{n\in\mathbb{N}}$ in $D(M(\i\m+\rho))$ with $f_{n}\to f$
and $M(\i\m+\rho)f_{n}\to g$ in $L_{2}(\mathbb{R};H)$ for some $f,g\in L_{2}(\mathbb{R};H).$
Passing to a subsequence, we may assume without loss of generality
that $f_{n}(t)\to f(t)$ and $M(\i t+\rho)f_{n}(t)\to g(t)$ for almost
every $t\in\mathbb{R}.$ Since $M(\i t+\rho)$ is bounded, we infer
$M(\i t+\rho)f(t)=g(t)$ for almost every $t\in\mathbb{R},$ which
shows the assertion. 
\end{proof}
Using the notion of linear material laws, we consider evolutionary
equations of the form 
\[
\left(\partial_{0,\rho}M(\partial_{0,\rho})+A\right)u=f.
\]
We first address the question of well-posedness of such an equation,
which we define as follows:
\begin{defn*}
Let $M:D(M)\subseteq\mathbb{C}\to L(H)$ be a linear material law.
Moreover, let $A:D(A)\subseteq H\to H$ be a densely defined closed
linear operator. For $\rho\in S_{M}$ we call the problem of finding
a solution $u\in H_{\rho}(\mathbb{R};H)$ of 
\[
\left(\partial_{0,\rho}M(\partial_{0,\rho})+A\right)u=f
\]
for right-hand sides $f\in H_{\rho}(\mathbb{R};H)$ an \emph{evolutionary
problem associated with $M$ and $A$. }We call such a problem \emph{well-posed,
}if there exists some $\rho_{1}\in\mathbb{R}$ such that

\begin{enumerate}[(a)]

\item for each $z\in\mathbb{C}_{\Re>\rho_{1}}\cap D(M)$ the operator
$zM(z)+A$ is boundedly invertible, and

\item the mapping $\mathbb{C}_{\Re>\rho_{1}}\cap D(M)\ni z\mapsto(zM(z)+A)^{-1}\in L(H)$
possesses a bounded and analytic extension to $\mathbb{C}_{\Re>\rho_{1}}.$ 

\end{enumerate}

Moreover, we define the \emph{abscissa of boundedness }$s_{0}(M,A)$
of a well-posed problem by 
\[
s_{0}(M,A)\coloneqq\inf\left\{ \rho_{1}\in\mathbb{R}\,;\,\mbox{(a) and (b) are satisfied}\right\} .
\]

\end{defn*}
We now show that the well-posedness definition above indeed yields
the unique solvability of the evolutionary problem and the continuous
dependence of the solution on the given right-hand side.
\begin{prop}
Assume we have a well-posed evolutionary problem associated with $M$
and $A$. Then for each $\rho\in S_{M}$ with $\rho>s_{0}(M,A)$ \nomenclature[Z_090]{$s_0(M,A)$}{the abscissa of boundedness of a material law $M$ and an operator $A$.}
the operator $\partial_{0,\rho}M(\partial_{0,\rho})+A$ is closable
and $\overline{\partial_{0,\rho}M(\partial_{0,\rho})+A}$ is boundedly
invertible.\end{prop}
\begin{proof}
For $\rho\in S_{M}$ we define the multiplication operator 
\begin{align*}
S_{\rho}:D(S_{\rho})\subseteq L_{2}(\mathbb{R};H) & \to L_{2}(\mathbb{R};H)\\
u & \mapsto\left(t\mapsto\left((\i t+\rho)M(\i t+\rho)+A\right)u(t)\right)
\end{align*}
with maximal domain $D(S_{\rho})$ given by 
\[
\left\{ u\in L_{2}(\mathbb{R};H)\,;\, u(t)\in D(A)\mbox{ for a.e. }t\in\mathbb{R},\:\left((\i\cdot+\rho)M(\i\cdot+\rho)+A\right)u(\cdot)\in L_{2}(\mathbb{R};H)\right\} .
\]
Then $S_{\rho}$ is closed. Indeed, let $(u_{n})_{n\in\mathbb{N}}$
in $D(S_{\rho})$ with $u_{n}\to u$ and $S_{\rho}u_{n}\to v$ in
$L_{2}(\mathbb{R};H)$ for some $u,v\in L_{2}(\mathbb{R};H).$ Passing
to a subsequence (without re-labeling), we infer that $u_{n}(t)\to u(t)$
and $\left(S_{\rho}u_{n}\right)(t)\to v(t)$ for almost every $t\in\mathbb{R}.$
Since $M(\i t+\rho)$ is bounded, we derive that 
\[
Au_{n}(t)=\left(S_{\rho}u_{n}\right)(t)-(\i t+\rho)M(\i t+\rho)u_{n}(t)\to v(t)-\left(\i t+\rho\right)M(\i t+\rho)u(t),
\]
and thus, $u(t)\in D(A)$ for almost every $t\in\mathbb{R},$ since
$A$ is closed. Moreover, $\left(\i t+\rho\right)M(\i t+\rho)u(t)+Au(t)=v(t)$
for almost every $t\in\mathbb{R},$ which shows $u\in D(S_{\rho})$
and $S_{\rho}u=v.$ \\
Since clearly 
\[
(\i\m+\rho)M(\i\m+\rho)+A\subseteq S_{\rho},
\]
we infer the closability of $(\i\m+\rho)M(\i\m+\rho)+A$ and hence,
by unitary equivalence we get the closability of $\partial_{0,\rho}M(\partial_{0,\rho})+A.$
We show $\overline{(\i\m+\rho)M(\i\m+\rho)+A}=S_{\rho}$. By what
we have shown above, we need to verify that $S_{\rho}\subseteq\overline{(\i\m+\rho)M(\i\m+\rho)+A}$.
So let $u\in D(S_{\rho}).$ Then we set $u_{n}(t)\coloneqq\chi_{[-n,n]}(t)u(t)$
for $n\in\mathbb{N},t\in\mathbb{R}.$ By dominated convergence, $u_{n}\to u$
and $S_{\rho}u_{n}\to S_{\rho}u$ in $L_{2}(\mathbb{R};H).$ Moreover,
\begin{align*}
\intop_{\mathbb{R}}|(\i t+\rho)M(\i t+\rho)u_{n}(t)|^{2}\mbox{ d}t & \leq\sup_{s\in[-n,n]}\|M(\i s+\rho)\|\intop_{-n}^{n}(t^{2}+\rho^{2})|u(t)|^{2}\mbox{ d}t\\
 & \leq\sup_{s\in[-n,n]}\|M(\i s+\rho)\|(n^{2}+\rho^{2})|u|_{L_{2}}^{2},
\end{align*}
which yields $u_{n}\in D((\i\m+\rho)M(\i\m+\rho))$ for each $n\in\mathbb{N}.$
Since also $u_{n}\in D(S_{\rho})$ it follows that $u_{n}\in D((\i\m+\rho)M(\i\m+\rho)+A)$.
Moreover, using that 
\[
\left((\i\m+\rho)M(\i\m+\rho)+A\right)u_{n}=S_{\rho}u_{n}\to S_{\rho}u,
\]
we derive $u\in D\left(\overline{(\i\m+\rho)M(\i\m+\rho)+A}\right),$
which yields the assertion.\\
Using the conditions (a) and (b) in the definition of well-posed evolutionary
problems, we derive that $S_{\rho}=\overline{(\i\m+\rho)M(\i\m+\rho)+A}$
and hence, by unitary equivalence, $\overline{\partial_{0,\rho}M(\partial_{0,\rho})+A}$
is boundedly invertible for $\rho>s_{0}(M,A).$ 
\end{proof}
Besides the unique existence of a solution and its continuous dependence
on the given data, we want to address the causality of the associated
solution operator. Since evolutionary equations are intended to model
some physical phenomenon, where the argument of the unknown $u$ will
be interpreted as time, causality will be one of its crucial properties.
In our framework we define causality as follows.
\begin{defn*}
Let $T:D(T)\subseteq H_{\rho}(\mathbb{R};H)\to H_{\rho}(\mathbb{R};H)$
\nomenclature[B_090]{$H_\rho(\mathbb{R};H)$}{the Hilbert space of all measurable functions $f:\mathbb{R}\to H$ satisfying $\int_{\mathbb{R}} \vert f(t) \vert^2 \exp(-2\rho t) \, \dd t<\infty$.}
\nomenclature[Z_000]{$\langle \cdot \vert \cdot \rangle_H$}{inner product on $H$, linear in the second and conjugate linear in the first argument.}\nomenclature[Z_010]{$\vert \cdot  \vert_X$}{norm on $X$.}\nomenclature[Z_020]{$\Vert \cdot \Vert$}{the operator norm on a space $L(X;Y)$.}be
a mapping for some $\rho\in\mathbb{R}$. We say $T$ is \emph{(forward)
causal}, if for each $u,v\in D(T)$ with $u=v$ on $]-\infty,t]$
for some $t\in\mathbb{R}$ we have that $Tu=Tv$ on $]-\infty,t]$.
\end{defn*}
For an abstract notion of causality we refer to \cite{Saeks1970}.
Moreover, we note that there is another notion of causality introduced
in \cite{Waurick2013_causality}, which coincides with the definition
above in case of closed operators and which has the benefit that the
closure of a causal operator stays causal (which is not the case for
the definition above).\\
We begin by stating some equivalent formulations of causality for
certain classes of mappings. In order to fix some notation, we make
the following definitions.
\begin{defn*}
Let $\rho\in\mathbb{R}.$ 

\begin{enumerate}[(a)]

\item For $t\in\mathbb{R}$ we define the \emph{cut-off operator}
$\chi_{\mathbb{R}_{\leq t}}(\m):H_{\rho}(\mathbb{R};H)\to H_{\rho}(\mathbb{R};H)$
by setting $\left(\chi_{\mathbb{R}_{\leq t}}(\m)f\right)(s)\coloneqq\chi_{\mathbb{R}_{\leq t}}(s)f(s)$
for $f\in H_{\rho}(\mathbb{R};H),s\in\mathbb{R}.$ 

\item  For $h\in\mathbb{R}$ let $\tau_{h}:H_{\rho}(\mathbb{R};H)\to H_{\rho}(\mathbb{R};H)$
\nomenclature[O_180]{$\tau_h$}{the translation operator given by $\tau_h f = f(\cdot +h)$.}be
the \emph{translation operator} defined by $\left(\tau_{h}f\right)(s)\coloneqq f(s+h)$
for $f\in H_{\rho}(\mathbb{R};H),s\in\mathbb{R}.$ 

\item  A mapping $T:D(T)\subseteq H_{\rho}(\mathbb{R};H)\to H_{\rho}(\mathbb{R};H)$
is called \emph{translation-invariant, }if $\tau_{h}T=T\tau_{h}$
for each $h\in\mathbb{R}.$

\end{enumerate}\end{defn*}
\begin{lem}
\label{lem:char_causality}Let $\rho\in\mathbb{R}.$

\begin{enumerate}[(a)]

\item A mapping $T:H_{\rho}(\mathbb{R};H)\to H_{\rho}(\mathbb{R};H)$
is causal, if and only if $\chi_{\mathbb{R}_{\leq t}}(\m)T\chi_{\mathbb{R}_{\leq t}}(\m)=\chi_{\mathbb{R}_{\leq t}}(\m)T$
for each $t\in\mathbb{R}.$ 

\item A linear mapping $T:D(T)\subseteq H_{\rho}(\mathbb{R};H)\to H_{\rho}(\mathbb{R};H)$
is causal, if and only if for each $u\in D(T)$ with $u=0$ on $]-\infty,t]$
for some $t\in\mathbb{R}$ it follows that $Tu=0$ on $]-\infty,t].$

\item A translation-invariant mapping $T:D(T)\subseteq H_{\rho}(\mathbb{R};H)\to H_{\rho}(\mathbb{R};H)$
is causal, if and only if for each $u,v\in D(T)$ with $u=v$ on $]-\infty,0]$
it follows that $Tu=Tv$ on $]-\infty,0]$. If in addition $D(T)=H_{\rho}(\mathbb{R};H),$
the latter is equivalent to $\chi_{\mathbb{R}_{\leq0}}(\m)T\chi_{\mathbb{R}_{\leq0}}(\m)=\chi_{\mathbb{R}_{\leq0}}(\m)T.$

\end{enumerate}\end{lem}
\begin{proof}
\begin{enumerate}[(a)]

\item Let $T$ be causal, $u\in H_{\rho}(\mathbb{R};H)$ and $t\in\mathbb{R}.$
Then $u=\chi_{\mathbb{R}_{\leq t}}(\m)u$ on $]-\infty,t]$ and thus,
$Tu=T\chi_{\mathbb{R}_{\leq t}}(\m)u$ on $]-\infty,t]$ which gives
$\chi_{\mathbb{R}_{\leq t}}(\m)Tu=\chi_{\mathbb{R}_{\leq t}}(\m)T\chi_{\mathbb{R}_{\leq t}}(\m)u$.
Assume now that $\chi_{\mathbb{R}_{\leq s}}(\m)T\chi_{\mathbb{R}_{\leq s}}(\m)=\chi_{\mathbb{R}_{\leq s}}(\m)T$
for each $s\in\mathbb{R}$ and take $u,v\in H_{\rho}(\mathbb{R};H)$
with $u=v$ on $]-\infty,t]$ for some $t\in\mathbb{R}.$ Then, 
\begin{align*}
\chi_{\mathbb{R}_{\leq t}}(\m)Tu & =\chi_{\mathbb{R}_{\leq t}}(\m)T\chi_{\mathbb{R}_{\leq t}}(\m)u\\
 & =\chi_{\mathbb{R}_{\leq t}}(\m)T\chi_{\mathbb{R}_{\leq t}}(\m)v\\
 & =\chi_{\mathbb{R}_{\leq t}}(\m)Tv
\end{align*}
and thus $Tu=Tv$ on $]-\infty,t].$ 

\item  This is clear, since for $u,v\in D(T)$ we have $u-v\in D(T)$
and $u=v$ and $Tu=Tv$ on $]-\infty,t]$ is equivalent to $u-v=0$
and $T(u-v)=0$ on $]-\infty,t]$, respectively. 

\item Since $\chi_{\mathbb{R}_{\leq t}}(\m)=\tau_{-t}\chi_{\mathbb{R}_{\leq0}}(\m)\tau_{t}$
for each $t\in\mathbb{R},$ the first assertion follows. In the case
$D(T)=H_{\rho}(\mathbb{R};H)$ the assertion follows from (a).\qedhere

\end{enumerate}
\end{proof}
For later purposes, we need some further properties of bounded analytic
functions on half planes.
\begin{prop}
\label{prop:trace_of_H^infty}Let $T:\mathbb{C}_{\Re>\rho_{0}}\to L(H)$
be bounded and analytic for some $\rho_{0}\in\mathbb{R}.$ Then there
exists a unique operator $T_{\rho_{0}}\in L(L_{2}(\mathbb{R};H))$
such that for $\varphi\in C_{c}^{\infty}(\mathbb{R};H)$ we have that
\[
T(\i\m+\rho)\mathcal{L}_{\rho}\varphi\to T_{\rho_{0}}\mathcal{L}_{\rho_{0}}\varphi\quad(\rho\to\rho_{0})
\]
in $L_{2}(\mathbb{R};H).$ Moreover, $\|T_{\rho_{0}}\|\leq\|T\|_{\infty}.$\end{prop}
\begin{proof}
The uniqueness of such an operator is clear, since it is determined
on $\mathcal{L}_{\rho_{0}}[C_{c}^{\infty}(\mathbb{R};H)]$, which
is dense in $L_{2}(\mathbb{R};H).$ Let now $\varphi\in C_{c}^{\infty}(\mathbb{R};H)$
with $\spt\varphi\subseteq\mathbb{R}_{\geq0}$. Then we have $\widehat{\varphi}\in\mathcal{H}^{2}(\mathbb{C}_{\Re>0};H)$
by \prettyref{cor:Paley_Wiener_unitary}. Since $T$ is bounded, we
have that 
\[
(z\mapsto T(z)\widehat{\varphi}(z))\in\mathcal{H}^{2}(\mathbb{C}_{\Re>0};H)
\]
and thus, there is $f\in L_{2}(\mathbb{R}_{\geq0};H)$ such that 
\[
\mathcal{L}_{\rho}f=T(\i\m+\rho)\mathcal{L}_{\rho}\varphi\quad(\rho>\rho_{0}).
\]
Since $\mathcal{L}_{\rho}f\to\mathcal{L}_{\rho_{0}}f$ in $L_{2}(\mathbb{R}_{\geq0};H)$
as $\rho\to\rho_{0}$ by dominated convergence, the right-hand side
converges in $L_{2}(\mathbb{R};H)$ as well. Let now $\varphi\in C_{c}^{\infty}(\mathbb{R};H).$
Then we set $a\coloneqq\inf\spt\varphi$ and thus, we have that $\tau_{a}\varphi\in C_{c}^{\infty}(\mathbb{R}_{\geq0}).$
Moreover, since $\mathcal{L}_{\rho}\tau_{a}\varphi=\e^{(\i\m+\rho)a}\mathcal{L}_{\rho}\varphi,$
we infer that
\[
T(\i\m+\rho)\mathcal{L}_{\rho}\varphi=\e^{-\left(\i\m+\rho\right)a}T(\i\m+\rho)\mathcal{L}_{\rho}\tau_{a}\varphi
\]
and thus, 
\[
T_{\rho_{0}}\mathcal{L}_{\rho_{0}}\varphi\coloneqq\lim_{\rho\to\rho_{0}}T(\i\m+\rho)\mathcal{L}_{\rho}\varphi
\]
is well-defined. Then, $T_{\rho_{0}}$ is obviously linear and 
\begin{align*}
|T_{\rho_{0}}\mathcal{L}_{\rho_{0}}\varphi|_{L_{2}(\mathbb{R};H)} & \leq\|T\|_{\infty}\liminf_{\rho\to\rho_{0}}|\mathcal{L}_{\rho}\varphi|_{L_{2}(\mathbb{R};H)}\\
 & =\|T\|_{\infty}|\mathcal{L}_{\rho_{0}}\varphi|_{L_{2}(\mathbb{R};H)}
\end{align*}
and hence, it extends to a bounded operator $T_{\rho_{0}}\in L(L_{2}(\mathbb{R};H))$
with $\|T_{\rho_{0}}\|\leq\|T\|_{\infty}.$\end{proof}
\begin{defn*}
Let $T:\mathbb{C}_{\Re>\rho_{0}}\to L(H)$ be bounded and analytic
for some $\rho_{0}\in\mathbb{R}$. Then we define the operator 
\[
T(\partial_{0,\rho_{0}})\coloneqq\mathcal{L}_{\rho_{0}}^{\ast}T_{\rho_{0}}\mathcal{L}_{\rho_{0}}.
\]
\end{defn*}
\begin{rem}
We note that $T_{\rho_{0}}=T(\i\m+\rho_{0})$ for a bounded continuous
function $T:\mathbb{C}_{\Re\geq\rho_{0}}\to L(H)$, which is analytic
in $\mathbb{C}_{\Re>\rho_{0}}.$ Indeed, since 
\[
T(\i\m+\rho)f\to T(\i\m+\rho_{0})f\quad(\rho\to\rho_{0})
\]
for each $f\in L_{2}(\mathbb{R};H)$ by dominated convergence, it
follows that 
\[
T(\i\m+\rho)\mathcal{L}_{\rho}\varphi\to T(\i\m+\rho_{0})\mathcal{L}_{\rho_{0}}\varphi\quad(\rho\to\rho_{0})
\]
for each $\varphi\in C_{c}^{\infty}(\mathbb{R};H),$ which yields
$T_{\rho_{0}}=T(\i\m+\rho_{0}).$
\end{rem}
We note that the operator $T_{\rho_{0}}$ in \prettyref{prop:trace_of_H^infty}
does not need to be a multiplication operator. However, if we assume
that $H$ is separable, we get the following statement.
\begin{lem}
\label{lem:boundary_value}Let $T:\mathbb{C}_{\Re>\rho_{0}}\to L(H)$
bounded and analytic for some $\rho_{0}\in\mathbb{R}$ and assume
that $H$ is separable. Then there exists $h:\mathbb{R}\to L(H)$
strongly measurable and bounded such that 
\[
T_{\rho_{0}}=h(\m).
\]
\end{lem}
\begin{proof}
{} Choose $D\subseteq H$ dense and countable subspace. Let $y\in D$
and consider the function $f_{y}\in\bigcap_{\rho\geq\rho_{0}}H_{\rho}(\mathbb{R};H)$
given by \nomenclature[A_010]{$\mathbb{R}_{\gtreqless a}$}{the set of reals greater/less than or equal to $a\in \mathbb{R}$.}
\[
f_{y}(t)\coloneqq\sqrt{2\pi}\chi_{\mathbb{R}_{\geq0}}(t)\e^{-(1-\rho_{0})t}y\quad(t\in\mathbb{R}).
\]
Then $\hat{f}_{y}(z)=\frac{1}{z+1-\rho_{0}}y$ for $z\in\mathbb{C}_{\Re\geq\rho_{0}}$.
Then 
\begin{align*}
T(\i t+\rho)y & =\left(\i t+\rho+1-\rho_{0}\right)T(\i t+\rho)\left(\mathcal{L}_{\rho}f_{y}\right)(t)\quad(t\in\mathbb{R})
\end{align*}
and since $T(\i\m+\rho)\mathcal{L}_{\rho}f_{y}\to T_{\rho_{0}}\mathcal{L}_{\rho_{0}}f_{y}$
in $L_{2}(\mathbb{R};H)$ as $\rho\to\rho_{0}$, we find a sequence
$(\rho_{n})_{n\in\mathbb{N}}$ with $\rho_{n}\to\rho_{0}$ a Lebesgue
null-set $M_{y}\subseteq\mathbb{R}$ such that 
\[
T(\i t+\rho_{n})y\to(\i t+1)\left(T_{\rho_{0}}\mathcal{L}_{\rho_{0}}f_{y}\right)(t)\quad(n\to\infty)
\]
for $t\in\mathbb{R}\setminus M_{y}.$ We set $M\coloneqq\bigcup_{y\in D}M_{y}$
and define 
\begin{align*}
h(t)y\coloneqq & \begin{cases}
(\i t+1)\left(T_{\rho_{0}}\mathcal{L}_{\rho_{0}}f_{y}\right)(t) & \mbox{ if }t\in\mathbb{R}\setminus M,\\
0 & \mbox{ otherwise},
\end{cases}
\end{align*}
for $y\in D.$ Then $h(t):D\to H$ is linear and 
\[
|h(t)y|_{H}\leq\|T\|_{\infty}|y|_{H}\quad(y\in D)
\]
for each $t\in\mathbb{R}$ and hence, it has a unique extension $h(t)\in L(H).$
Since $h(\cdot)y$ is measurable for each $y\in D$ we infer that
$h:\mathbb{R}\to L(H)$ is strongly measurable and bounded. It is
left to show $T_{\rho}=h(\m).$ For doing so, we take $\varphi\in C_{c}^{\infty}(\mathbb{R};H).$
Moreover, we choose a subsequence (without re-labeling) $(\rho_{n})_{n\in\mathbb{N}}$
with $\rho_{n}\to\rho_{0}$ such that 
\[
T(\i t+\rho_{n})\left(\mathcal{L}_{\rho_{n}}\varphi\right)(t)\to\left(T_{\rho_{0}}\mathcal{L}_{\rho_{0}}\varphi\right)(t)
\]
and $\left(\mathcal{L}_{\rho_{n}}\varphi\right)(t)\to\left(\mathcal{L}_{\rho_{0}}\varphi\right)(t)$
for every $t\in\mathbb{R}\setminus N,$ where $N$ is a Lebesgue null-set.
Let now $\varepsilon>0,$ $t\in\mathbb{R}\setminus(M\cup N)$ and
choose $y\in D$ such that $|\mathcal{L}_{\rho_{0}}\varphi(t)-y|_{H}<\varepsilon.$
We choose $n_{0}\in\mathbb{N}$ such that 
\begin{align*}
|T(\i t+\rho_{n})y-h(t)y|_{H} & <\varepsilon,\\
|\left(\mathcal{L}_{\rho_{n}}\varphi\right)(t)-\left(\mathcal{L}_{\rho_{0}}\varphi\right)(t)|_{H} & <\varepsilon
\end{align*}
for each $n\geq n_{0}$. Then we have for $n\geq n_{0}$
\begin{align*}
 & |T(\i t+\rho_{n})\left(\mathcal{L}_{\rho_{n}}\varphi\right)(t)-h(t)\left(\mathcal{L}_{\rho_{0}}\varphi\right)(t)|_{H}\\
 & \leq\left|T(\i t+\rho_{n})\left(\left(\mathcal{L}_{\rho_{n}}\varphi\right)(t)-\left(\mathcal{L}_{\rho_{0}}\varphi\right)(t)\right)\right|_{H}+\left|T(\i t+\rho_{n})\left(\left(\mathcal{L}_{\rho_{0}}\varphi\right)(t)-y\right)\right|_{H}+\\
 & \quad+\left|\left(T(\i t+\rho_{n})-h(t)\right)y\right|_{H}+|h(t)(y-\left(\mathcal{L}_{\rho_{0}}\varphi\right)(t))|_{H}\\
 & \leq\left(3\|T\|_{\infty}+1\right)\varepsilon,
\end{align*}
which yields 
\[
\left(T_{\rho_{0}}\mathcal{L}_{\rho_{0}}\varphi\right)(t)=h(t)(\mathcal{L}_{\rho_{0}}\varphi)(t)
\]
for $t\in\mathbb{R}\setminus(N\cup M)$ and thus, $T_{\rho_{0}}=h(\m)$.\end{proof}
\begin{rem}
We note that $h\notin L_{\infty}(\mathbb{R};L(H))$ in general. The
existence of a boundary function in $L_{p}(\mathbb{R};X)$ for functions
$\mathcal{H}^{p}(\mathbb{C}_{\Re>0};X)$ for some Banach space $X$
and $1\leq p\leq\infty$ was studied in \cite{Bukhvalov1981}. Banach
spaces allowing such boundary functions are said to have the \emph{analytic
Radon-Nikodym property. }It was shown that if $c_{0}\subseteq X,$
then $X$ fails to have this property and thus, we cannot expect a
better measurability than asserted in \prettyref{lem:boundary_value}.
We note that if $X$ is a Banach lattice, $c_{0}\nsubseteq X$ is
equivalent to the analytic Radon-Nikodym property \cite[Theorem 1]{Bukhvalov1982}.
This in particular implies that the analytic Radon-Nikodym property
is weaker than the Radon-Nikodym property, since $c_{0}\nsubseteq L_{1}(0,1)$
but $L_{1}(0,1)$ does not satisfy the Radon-Nikodym property (see
\cite[p. 61]{DiestelUhl}).
\end{rem}
We need a further auxiliary result for bounded analytic functions.
\begin{lem}
\label{lem:independent of rho} Let $\mu,\nu\in\mathbb{R}$ with $\mu<\nu$
and set $U\coloneqq\{z\in\mathbb{C}\,;\,\mu<\Re z<\nu\}$. Moreover,
let $T:\overline{U}\subseteq\mathbb{C}\to L(H)$ be continuous, bounded
and analytic in $U$. Then for $f\in H_{\mu}(\mathbb{R};H)\cap H_{\nu}(\mathbb{R};H)$
we have that 
\[
\mathcal{L}_{\mu}^{\ast}T(\i\m+\mu)\mathcal{L}_{\mu}f=\mathcal{L}_{\nu}^{\ast}T(\i\m+\nu)\mathcal{L}_{\nu}f.
\]
\end{lem}
\begin{proof}
By \prettyref{lem:density test fct} there exists a sequence in $C_{c}^{\infty}(\mathbb{R};H)$
converging to $f$ in both spaces $H_{\mu}(\mathbb{R};H)$ and $H_{\nu}(\mathbb{R};H).$
Thus, due to the boundedness of $T$ it suffices to prove the assertion
for all $\varphi\in C_{c}^{\infty}(\mathbb{R};H).$ So let $\varphi\in C_{c}^{\infty}(\mathbb{R};H).$
For $t\in\mathbb{R}$ we define the function 
\[
f(z)\coloneqq\e^{zt}T(z)\hat{\varphi}(z)\quad(z\in\overline{U}).
\]
Then $f$ is analytic in $U$ and so, 
\begin{equation}
\i\intop_{-R}^{R}f(\i s+\mu)\mbox{ d}s+\intop_{\mu}^{\nu}f(\i R+\kappa)\mbox{ d}\kappa-\i\intop_{-R}^{R}f(\i s+\nu)\mbox{ d}s-\intop_{\mu}^{\nu}f(-\i R+\kappa)\mbox{ d}\kappa=0\label{eq:Cauchy_integral}
\end{equation}
for each $R>0.$ Moreover, $f$ is bounded on $\overline{U}$ since
\begin{align}
|f(z)|_{H} & \leq\|T\|_{\infty}\max\{\e^{\mu t},\e^{\nu t}\}|\hat{\varphi}(z)|_{H}\nonumber \\
 & \leq\|T\|_{\infty}\max\{\e^{\mu t},\e^{\nu t}\}C|\varphi|_{L_{1}},\label{eq:boundedness_f}
\end{align}
for each $z\in\overline{U}$, where $C\coloneqq\frac{1}{\sqrt{2\pi}}\sup\left\{ \e^{-\kappa s}\,;\,\kappa\in[\mu,\nu],s\in\spt\varphi\right\} .$
Moreover, since $\hat{\varphi}(\pm\i R+\kappa)\to0$ as $R\to\infty$
for each $\kappa\in]\mu,\nu[$ by the Riemann-Lebesgue Lemma (see
\prettyref{rem:Riemann-Lebesgue}), we get $f(\pm\i R+\kappa)\to0$
as $R\to\infty$ for each $\kappa\in]\rho,\mu[$ by the first line
of \prettyref{eq:boundedness_f}. Thus, the second and fourth term
in \prettyref{eq:Cauchy_integral} tend to $0$ as $R\to\infty$ by
dominated convergence. Moreover 
\[
\intop_{-R}^{R}f(\i s+\mu)\mbox{ d}s=\intop_{-R}^{R}\e^{(\i s+\mu)t}T(\i s+\mu)\left(\mathcal{L}_{\mu}\varphi\right)(s)\mbox{ d}s=\sqrt{2\pi}\left(\mathcal{L}_{\mu}^{\ast}\chi_{[-R,R]}(\m)T(\i\m+\mu)\mathcal{L}_{\mu}\varphi\right)(t)
\]
 and the same for $\mu$ replaced by $\nu.$ Passing now to a suitable
subsequence, \prettyref{eq:Cauchy_integral} and the last equality
yield 
\[
\left(\mathcal{L}_{\mu}^{\ast}T(\i\m+\mu)\mathcal{L}_{\mu}\varphi\right)(t)=\left(\mathcal{L}_{\nu}^{\ast}T(\i\m+\nu)\mathcal{L}_{\nu}\varphi\right)(t)\quad(t\in\mathbb{R}\mbox{ a.e.}),
\]
which gives the assertion.\end{proof}
\begin{thm}
\label{thm:material_law_good}Let $\rho_{0}\in\mathbb{R}$ and $T:\mathbb{C}_{\Re>\rho_{0}}\to L(H)$
be analytic and bounded. Then for each $\rho\geq\rho_{0}$ the operator
$T(\partial_{0,\rho})$ is translation-invariant, causal and bounded
with $\|T(\partial_{0,\rho})\|\leq\|T\|_{\infty}$. Moreover, the
operator is independent on the choice of $\rho$ in the sense that
$T(\partial_{0,\rho})f=T(\partial_{0,\nu})f$ for each $f\in H_{\rho}(\mathbb{R};H)\cap H_{\nu}(\mathbb{R};H)$
and $\rho,\nu\geq\rho_{0}$.\end{thm}
\begin{proof}
First, let $\rho>\rho_{0}.$ Since $\mathcal{L}_{\rho}\tau_{h}=\e^{(\i\m+\rho)h}\mathcal{L}_{\rho},$
we infer the translation-invariance of $T(\partial_{0,\rho}).$ For
proving the causality, it suffices to prove that for $u\in H_{\rho}(\mathbb{R};H)$
with $u=0$ on $]-\infty,0]$ it follows that $T(\partial_{0,\rho})u=0$
on $]-\infty,0]$ by \prettyref{lem:char_causality}. So let $u\in H_{\rho}(\mathbb{R};H)$
with $u=0$ on $]-\infty,0].$ By \prettyref{cor:Paley_Wiener_unitary}
we get that $\hat{u}\in\mathcal{H}^{2}(\mathbb{C}_{\Re>\rho};H)$.
As $T$ is analytic and bounded, we get that $\left(z\mapsto T(z)\hat{u}(z)\right)\in\mathcal{H}^{2}(\mathbb{C}_{\Re>\rho};H),$
which yields $T(\partial_{0,\rho})u=0$ on $]-\infty,0]$ again by
\prettyref{cor:Paley_Wiener_unitary}. The boundedness of $T(\partial_{0,\rho})$
and the norm estimate are obvious.\\
Let $f\in H_{\rho}(\mathbb{R};H)\cap H_{\nu}(\mathbb{R};H)$ and $\rho<\nu.$
If $\rho>\rho_{0}$ we infer $T(\partial_{0,\rho})f=T(\partial_{0,\nu})f$
from \prettyref{lem:independent of rho}. Assume now $\rho=\rho_{0}.$
We first prove that $T(\partial_{0,\rho_{0}})\varphi=T(\partial_{0,\nu})\varphi$
for each $\nu>\rho_{0}$ and $\varphi\in C_{c}^{\infty}(\mathbb{R};H).$
We choose a sequence $(\rho_{n})_{n\in\mathbb{N}}$ in $]\rho_{0},\nu[$
such that $\rho_{n}\to\rho_{0}$. Then we have 
\[
(T(\partial_{0,\nu})\varphi)(t)=\left(\mathcal{L}_{\rho_{n}}^{\ast}T(\i\m+\rho_{n})\mathcal{L}_{\rho_{n}}\varphi\right)(t)
\]
for every $n\in\mathbb{N}$ and almost every $t\in\mathbb{R}$ by
what we have shown above. Since, 
\begin{align*}
\left(\mathcal{L}_{\rho_{n}}^{\ast}T(\i\m+\rho_{n})\mathcal{L}_{\rho_{n}}\varphi\right)(t) & =\e^{\rho_{n}t}\left(\mathcal{F}^{\ast}T(\i\m+\rho_{n})\mathcal{L}_{\rho_{n}}\varphi\right)(t)\\
 & \to\e^{\rho_{0}t}\left(\mathcal{F}^{\ast}T_{\rho_{0}}\mathcal{L}_{\rho_{0}}\varphi\right)(t)\\
 & =\left(T(\partial_{0,\rho_{0}})\varphi\right)(t)
\end{align*}
for almost every $t\in\mathbb{R},$ we derive the assertion. Again,
by \prettyref{lem:density test fct} we obtain $T(\partial_{0,\rho_{0}})f=T(\partial_{0,\nu})f$
for each $f\in H_{\rho_{0}}(\mathbb{R};H)\cap H_{\nu}(\mathbb{R};H).$
In particular, we get for each $\varphi\in C_{c}^{\infty}(\mathbb{R};H)$
and $h\in\mathbb{R}$
\[
T(\partial_{0,\rho_{0}})\tau_{h}\varphi=T(\partial_{0,\nu})\tau_{h}\varphi=\tau_{h}T(\partial_{0,\nu})\varphi=\tau_{h}T(\partial_{0,\rho_{0}})\varphi,
\]
and thus, the translation-invariance of $T(\partial_{0,\rho_{0}})$
follows by continuous extension. Moreover, for $\varphi\in C_{c}^{\infty}(\mathbb{R}_{>0};H)$
we have that 
\[
\spt T(\partial_{0,\rho_{0}})\varphi=\spt T(\partial_{0,\nu})\varphi\subseteq\mathbb{R}_{\geq0},
\]
and hence, causality follows again by continuous extension.
\end{proof}
The latter theorem shows that analytic and bounded mappings on a right
half plane induce a family of bounded, causal and translation-invariant
operators on $H_{\rho}(\mathbb{R};H),$ which are independent of the
particular choice of $\rho$. In particular, this result applies to
the solution operator of a well-posed evolutionary problem.
\begin{cor}
\label{cor:sol_po_causal}Consider a well-posed evolutionary problem
associated with $M$ and $A$. Then for each $\rho\in S_{M}$ with
$\rho>s_{0}(M,A)$, the operator $\left(\overline{\partial_{0,\rho}M(\partial_{0,\rho})+A}\right)^{-1}$
is causal and independent of the parameter $\rho$ in the sense that
\[
\left(\overline{\partial_{0,\rho}M(\partial_{0,\rho})+A}\right)^{-1}f=\left(\overline{\partial_{0,\nu}M(\partial_{0,\nu})+A}\right)^{-1}f
\]
for each $f\in H_{\rho}(\mathbb{R};H)\cap H_{\nu}(\mathbb{R};H)$
and $\rho,\nu\in S_{M}$ with $\rho,\nu>s_{0}(M,A)$.\end{cor}
\begin{proof}
By assumption there exists a bounded analytic mapping $T:\mathbb{C}_{\Re>\rho_{0}}\to L(H)$
for each $\rho_{0}>s_{0}(M,A)$ such that $T(z)=(zM(z)+A)^{-1}$ for
each $z\in D(M).$ In particular, for $\rho\in S_{M}$ we derive 
\[
T(\partial_{0,\rho})=\left(\overline{\partial_{0,\rho}M(\partial_{0,\rho})+A}\right)^{-1}
\]
and hence, the assertion follows from \prettyref{thm:material_law_good}.
\end{proof}
Summarizing, we have shown that our definition of well-posedness yields
a bounded solution operator, which is causal and independent of the
particular choice of the parameter $\rho$ in the sense of \prettyref{thm:material_law_good}.
The latter two properties strongly rely on the assumption (b) in the
definition of well-posedness. We now show that assumption (b) is not
only sufficient for causality and independence of $\rho$, but also
necessary, that is, we prove a converse statement of \prettyref{thm:material_law_good}.
For doing so, we need the following representation result.
\begin{thm}
\label{thm:weiss}Let $T:L_{2}(\mathbb{R};H)\to L_{2}(\mathbb{R};H)$
a bounded, translation-invariant, causal linear operator. Then there
exists an analytic and bounded mapping $N:\mathbb{C}_{\Re>0}\to L(H)$
satisfying $\|N\|_{\infty}\leq\|T\|$, such that for each $f\in L_{2}(\mathbb{R}_{\geq0};H)$
one has 
\[
\widehat{Tf}(z)=N(z)\hat{f}(z)\quad(z\in\mathbb{C}_{\Re>0}).
\]

\end{thm}
\nomenclature[Z_050]{$\hat{f}$}{the Fourier transform of a function $f$.}The
latter result is a special case of \cite[Theorem 2]{Foures1955},
where also the case $L_{2}(\mathbb{R}^{n};H)$ for $n>1$ is considered
and causality is defined with respect to a closed convex cone in $\mathbb{R}^{n}.$
However, for our purposes it suffices to consider the case $n=1$.
For proving the latter theorem we follow the rationale given in \cite{Weiss1991}.
We begin with the following lemma.
\begin{lem}
\label{lem:pre_weiss}Let $f\in L_{2}(\mathbb{R}_{\geq0};H)$. We
assume there exists $z\in\mathbb{C}_{\Re>0}$ such that for each $h\geq0$
we have that 
\[
f(t+h)=\e^{-zh}f(t)\quad(t\in\mathbb{R}_{\geq0}\mbox{ a.e.}).
\]
Then $f\in L_{1}(\mathbb{R}_{\geq0};H)$ and 
\[
f(t)=\e^{-zt}z\intop_{0}^{\infty}f(s)\,\dd s\quad(t\in\mathbb{R}_{\geq0}\mbox{ a.e.}).
\]
\end{lem}
\begin{proof}
We first show that $f\in L_{1}(\mathbb{R}_{\geq0};H).$ For doing
so, we compute 
\begin{align*}
\intop_{0}^{\infty}|f(t)|_{H}\mbox{ d}t & =\sum_{n=0}^{\infty}\intop_{n}^{n+1}|f(t)|_{H}\mbox{ d}t\\
 & =\sum_{n=0}^{\infty}\intop_{0}^{1}|f(t+n)|_{H}\mbox{ d}t\\
 & =\sum_{n=0}^{\infty}\intop_{0}^{1}\e^{-\left(\Re z\right)n}|f(t)|_{H}\mbox{ d}t\\
 & \leq|f|_{L_{2}}\frac{1}{1-\e^{-\Re z}}<\infty.
\end{align*}
We define the function $F(t)\coloneqq\intop_{t}^{\infty}f(r)\mbox{ d}r$
for $t\in\mathbb{R}_{\geq0}$, which is well-defined and continuous,
since $f\in L_{1}(\mathbb{R}_{\geq0};H).$ Then we obtain for each
$t\geq0$
\begin{align*}
F(t) & =\intop_{t}^{\infty}f(s)\mbox{ d}s\\
 & =\intop_{0}^{\infty}f(s+t)\mbox{ d}s\\
 & =\e^{-zt}F(0)
\end{align*}
for each $t\in\mathbb{R}_{\geq0},$ which proves that $F$ is continuously
differentiable with $F'(t)=\e^{-zt}(-zF(0)).$ On the other hand,
\prettyref{thm:Lebesgue} gives 
\[
F'(t)=-f(t)\quad(t\in\mathbb{R}_{\geq0}\mbox{ a.e.})
\]
and hence, the assertion follows.
\end{proof}
With this preparation we are able to prove \prettyref{thm:weiss}.
\begin{proof}[Proof of \prettyref{thm:weiss}]
 Let $z\in\mathbb{C}_{\Re>0},y\in H$ and define 
\[
g_{z,y}(t)\coloneqq\chi_{\mathbb{R}_{\geq0}}(t)\e^{-zt}y\quad(t\in\mathbb{R}).
\]
Then for $h\geq0$ we have $\left(\chi_{\mathbb{R}_{\geq0}}(\m)\tau_{h}g_{z,y}\right)(t)=\chi_{\mathbb{R}_{\geq0}}(t)g_{z,y}(t+h)=\chi_{\mathbb{R}_{\geq0}}(t)\e^{-z(t+h)}y=\e^{-zh}g_{z,y}(t)$
for $t\in\mathbb{R}.$ Since $T$ is causal we have $\chi_{\mathbb{R}_{\leq0}}(\m)T=\chi_{\mathbb{R}_{\leq0}}(\m)T\chi_{\mathbb{R}_{\leq0}}(\m)$
by \prettyref{lem:char_causality} and hence, 
\[
\chi_{\mathbb{R}_{\geq0}}(\m)T\chi_{\mathbb{R}_{\geq0}}(\m)=T\chi_{\mathbb{R}_{\geq0}}(\m)-\chi_{\mathbb{R}_{\leq0}}(\m)T\chi_{\mathbb{R}_{\geq0}}(\m)=T\chi_{\mathbb{R}_{\geq0}}(\m).
\]
The latter gives $\chi_{\mathbb{R}_{\geq0}}(\m)T^{\ast}\chi_{\mathbb{R}_{\geq0}}(\m)=\chi_{\mathbb{R}_{\geq0}}(\m)T^{\ast}$
and hence we obtain, by using that $T^{\ast}$ is translation-invariant
as well, 
\begin{align*}
\chi_{\mathbb{R}_{\geq0}}(\m)\tau_{h}\chi_{\mathbb{R}_{\geq0}}(\m)T^{\ast}g_{z,y} & =\chi_{\mathbb{R}_{\geq0}}(\m)\chi_{\mathbb{R}_{\geq-h}}(\m)\tau_{h}T^{\ast}g_{z,y}\\
 & =\chi_{\mathbb{R}_{\geq0}}(\m)T^{\ast}\tau_{h}g_{z,y}\\
 & =\chi_{\mathbb{R}_{\geq0}}(\m)T^{\ast}\chi_{\mathbb{R}_{\geq0}}(\m)\tau_{h}g_{z,y}\\
 & =\chi_{\mathbb{R}_{\geq0}}(\m)T^{\ast}\e^{-zh}g_{z,y}\\
 & =\e^{-zh}\chi_{\mathbb{R}_{\geq0}}(\m)T^{\ast}g_{z,y},
\end{align*}
for each $h\geq0.$ Thus, the function $\chi_{\mathbb{R}_{\geq0}}(\m)T^{\ast}g_{z,y}$
satisfies the hypothesis of \prettyref{lem:pre_weiss} and hence,
\begin{equation}
\left(T^{\ast}g_{z,y}\right)(t)=\e^{-zt}z\intop_{0}^{\infty}\left(T^{\ast}g_{z,y}\right)(s)\mbox{ d}s\quad(t\in\mathbb{R}_{\geq0}\mbox{ a.e.}).\label{eq:repr_adjoint}
\end{equation}
For $z\in\mathbb{C}_{\Re>0}$ we consider now the mapping 
\begin{align*}
L(z):H & \to H\\
y & \mapsto z\intop_{0}^{\infty}\left(T^{\ast}g_{z,y}\right)(s)\mbox{ d}s.
\end{align*}
Then, $L(z)$ is linear, since for $w,y\in H$ and $\lambda\in\mathbb{C}$
we have that $g_{z,\lambda w+y}=\lambda g_{z,w}+g_{z,y}$ and hence,
the linearity of $T^{\ast}$ implies the linearity of $L(z).$ Furthermore
$L(z)$ is bounded with $\|L(z)\|\leq\|T\|$. Indeed, for $y\in H$
Equation \prettyref{eq:repr_adjoint} gives
\begin{align*}
|T^{\ast}g_{z,y}|_{L_{2}(\mathbb{R}_{\geq0};H)}^{2} & =\intop_{0}^{\infty}\left|\left(T^{\ast}g_{z,y}\right)(t)\right|_{H}^{2}\mbox{ d}t\\
 & =\intop_{0}^{\infty}\left|\e^{-zt}L(z)y\right|_{H}^{2}\mbox{ d}t\\
 & =|L(z)y|_{H}^{2}\frac{1}{2\Re z}
\end{align*}
and hence, 
\begin{align*}
\frac{1}{\sqrt{2\Re z}}|L(z)y|_{H} & =|T^{\ast}g_{z,y}|_{L_{2}(\mathbb{R}_{\geq0};H)}\\
 & \leq\|T^{\ast}\||g_{z,y}|_{L_{2}}\\
 & =\|T\|\frac{1}{\sqrt{2\Re z}}|y|_{H}.
\end{align*}
Summarizing, we have shown that the mapping $L:\mathbb{C}_{\Re>0}\to L(H)$
is well-defined and bounded with $\|L\|_{\infty}\leq\|T\|$. Let now
$f\in L_{2}(\mathbb{R}_{\geq0};H).$ Then for $z\in\mathbb{C}_{\Re>0}$
and $y\in H$ we have (note that $Tf=0$ on $\mathbb{R}_{\leq0}$
due to causality) 
\begin{align*}
\langle y|\widehat{Tf}(z)\rangle_{H} & =\left\langle y\left|\frac{1}{\sqrt{2\pi}}\intop_{\mathbb{R}_{\geq0}}\e^{-zs}\left(Tf\right)(s)\mbox{ d}s\right.\right\rangle _{H}\\
 & =\frac{1}{\sqrt{2\pi}}\langle g_{z^{\ast},y}|Tf\rangle_{L_{2}(\mathbb{R}_{\geq0};H)}\\
 & =\frac{1}{\sqrt{2\pi}}\langle T^{\ast}g_{z^{\ast},y}|f\rangle_{L_{2}(\mathbb{R}_{\geq0};H)}\\
 & =\frac{1}{\sqrt{2\pi}}\langle L(z^{\ast})y\e^{-z^{\ast}\m}|f\rangle_{L_{2}(\mathbb{R}_{\geq0};H)}\\
 & =\left\langle L(z^{\ast})y\left|\hat{f}(z)\right.\right\rangle _{H}\\
 & =\langle y|L(z^{\ast})^{\ast}\hat{f}(z)\rangle_{H},
\end{align*}
and thus 
\[
\widehat{Tf}(z)=N(z)\hat{f}(z)\quad(z\in\mathbb{C}_{\Re>0})
\]
with $N(z)\coloneqq L(z^{\ast})^{\ast}.$ It remains to show that
$z\mapsto N(z)\in L(H)$ is analytic. This however follows easily
by applying the latter formula to $f=\sqrt{2\pi}g_{1,y}$ for $y\in H.$
Then $\hat{f}(z)=\frac{1}{z+1}y$ for $z\in\mathbb{C}_{\Re>0}$ and
thus, 
\[
N(z)y=(z+1)\widehat{Tf}(z).
\]
As the right-hand side is analytic on $\mathbb{C}_{\Re>0}$, so is
the left-hand side and hence, the analyticity of $z\mapsto N(z)$
follows from \cite[Theorem 3.10.1]{hille1957functional}.
\end{proof}
Indeed, \prettyref{thm:weiss} yields the converse of \prettyref{thm:material_law_good}
in the following sense.
\begin{prop}
Let $\rho_{0}\in\mathbb{R}$ and $\mathcal{T}:\bigcap_{\rho\geq\rho_{0}}H_{\rho}(\mathbb{R};H)\to\bigcap_{\rho\geq\rho_{0}}H_{\rho}(\mathbb{R};H)$
such that for each $\rho\geq\rho_{0}$, $\mathcal{T}$ has a bounded,
translation-invariant and causal extension $\mathcal{T}_{\rho}\in L(H_{\rho}(\mathbb{R};H))$
and $\sup_{\rho\geq\rho_{0}}\|\mathcal{T}_{\rho}\|<\infty.$ Then
there exists $T:\mathbb{C}_{\Re>\rho_{0}}\to L(H)$ analytic and bounded
such that $\mathcal{T}_{\rho}=T(\partial_{0,\rho})$ for each $\rho\geq\rho_{0}.$ \end{prop}
\begin{proof}
Let $\rho>\rho_{0}.$ Then, $S_{\rho}\coloneqq\e^{-\rho\m}\mathcal{T}_{\rho}\left(\e^{-\rho\m}\right)^{-1}:L_{2}(\mathbb{R};H)\to L_{2}(\mathbb{R};H)$
is bounded, causal and translation-invariant, where $\e^{-\rho\m}:H_{\rho}(\mathbb{R};H)\to L_{2}(\mathbb{R};H)$
is the unitary mapping defined by $\left(\e^{-\rho\m}f\right)(t)\coloneqq\e^{-\rho t}f(t)$
for each $t\in\mathbb{R},f\in H_{\rho}(\mathbb{R};H).$ Thus, by \prettyref{thm:weiss}
there is an analytic function $N_{\rho}:\mathbb{C}_{\Re>0}\to L(H)$
with $\|N_{\rho}\|_{\infty}\leq\|S_{\rho}\|=\|\mathcal{T}_{\rho}\|$
and 
\[
\widehat{S_{\rho}f}(z)=N_{\rho}(z)\hat{f}(z)\quad(z\in\mathbb{C}_{\Re>0})
\]
for each $f\in L_{2}(\mathbb{R}_{\geq0};H).$ Let $u\in\bigcap_{\rho>\rho_{0}}H_{\rho}(\mathbb{R};H)$
with $u=0$ on $\mathbb{R}_{\leq0}$. Then we obtain 
\[
\widehat{\left(\e^{-\rho\m}\mathcal{T}u\right)}(z)=\widehat{\left(S_{\rho}\e^{-\rho\m}u\right)}(z)=N_{\rho}(z)\widehat{\left(\e^{-\rho\m}u\right)}(z)\quad(z\in\mathbb{C}_{\Re>0}),
\]
or in other words (set $z=\i t+\nu-\rho$ for $\nu>\rho$): 
\begin{equation}
\mathcal{L}_{\nu}\mathcal{T}u=N_{\rho}(\i\m+\nu-\rho)\mathcal{L}_{\nu}u,\label{eq:representation on intersection}
\end{equation}
for each $\nu>\rho.$ Let now $\rho'>\rho_{0}$ and set $u(t)\coloneqq\sqrt{2\pi}\chi_{\mathbb{R}_{\geq0}}(t)\e^{\rho_{0}t}x$
for some $x\in H$ and $t\in\mathbb{R}$. Then the above gives for
$\nu>\max\{\rho,\rho'\}$ and $t\in\mathbb{R}$ 
\begin{align*}
\frac{1}{\i t+\nu-\rho_{0}}N_{\rho}(\i t+\nu-\rho)x & =N_{\rho}(\i t+\nu-\rho)\left(\mathcal{L}_{\nu}u\right)(t)\\
 & =\left(\mathcal{L}_{\nu}\mathcal{T}u\right)(t)\\
 & =N_{\rho'}(\i t+\nu-\rho')\left(\mathcal{L}_{\nu}u\right)(t)\\
 & =\frac{1}{\i t+\nu-\rho_{0}}N_{\rho'}(\i t+\nu-\rho')x,
\end{align*}
which yields
\begin{equation}
N_{\rho}(z-\rho)=N_{\rho'}(z-\rho')\quad(z\in\mathbb{C}_{\Re>\max\{\rho,\rho'\}}).\label{eq:identitiy_th}
\end{equation}
We define the function 
\begin{align*}
T:\mathbb{C}_{\Re>\rho_{0}} & \to L(H)\\
z & \mapsto N_{\frac{1}{2}\left(\Re z+\rho_{0}\right)}\left(z-\frac{1}{2}\left(\Re z+\rho_{0}\right)\right).
\end{align*}
Then $T$ is bounded, since $|T(z)|\leq\|N_{\frac{1}{2}(\Re z+\rho_{0})}\|_{\infty}\leq\|\mathcal{T}_{\frac{1}{2}(\Re z+\rho_{0})}\|\leq\sup_{\rho>\rho_{0}}\|\mathcal{T}_{\rho}\|$
for each $z\in\mathbb{C}_{\Re>\rho_{0}}.$ Moreover, $T$ is analytic.
Indeed, for $z\in\mathbb{C}_{\Re>\rho_{0}}$ and $r\coloneqq\frac{1}{2}(\Re z-\rho_{0})$
we have that 
\[
T(z')=N_{\frac{1}{2}(\Re z'+\rho_{0})}\left(z'-\frac{1}{2}(\Re z'+\rho_{0})\right)=N_{\Re z-r}(z'-\left(\Re z-r\right))\quad(z'\in B_{\mathbb{C}}(z,r))
\]
by \prettyref{eq:identitiy_th} and so the analyticity of $T$ on
$B_{\mathbb{C}}(z,r)$ follows from the analyticity of $N_{\Re z-r}$
on $\mathbb{C}_{\Re>0}$. It remains to show $\mathcal{T}_{\rho}=T(\partial_{0,\rho})$
for each $\rho>\rho_{0}.$ So let $\rho>\rho_{0}$ and $\varphi\in C_{c}^{\infty}(\mathbb{R};H).$
We set $h\coloneqq\inf\spt\varphi$ and get $\spt\tau_{h}\varphi\subseteq\mathbb{R}_{\geq0}.$
Equality \prettyref{eq:representation on intersection} now gives
\begin{align*}
\mathcal{L}_{\rho}\mathcal{T}\tau_{h}\varphi & =N_{\frac{1}{2}(\rho+\rho_{0})}\left(\i\m+\rho-\frac{1}{2}(\rho+\rho_{0})\right)\mathcal{L}_{\rho}\tau_{h}\varphi\\
 & =T(\i\m+\rho)\mathcal{L}_{\rho}\tau_{h}\varphi.
\end{align*}
Using the translation-invariance of $\mathcal{T}$ and the fact that
$\mathcal{L}_{\rho}\tau_{h}=\e^{\left(\i t+\rho\right)h}\mathcal{L}_{\rho}$
we get 
\[
\mathcal{L}_{\rho}\mathcal{T}\varphi=T(\i\m+\rho)\mathcal{L}_{\rho}\varphi.
\]
Since $C_{c}^{\infty}(\mathbb{R};H)$ is dense in $H_{\rho}(\mathbb{R};H),$
we get $\mathcal{T}_{\rho}=T(\partial_{0,\rho})$ for each $\rho>\rho_{0}$.
Moreover, since 
\[
T(\partial_{0,\rho_{0}})\varphi=T(\partial_{0,\rho})\varphi=\mathcal{T}\varphi,
\]
we also get the assertion for $\rho=\rho_{0}.$
\end{proof}

We conclude this section by giving some examples for operators $A$
and mappings $M$ yielding a well-posed evolutionary problem. One
important class of operators, which will be frequently used in the
forthcoming sections is the class of $m$-accretive operators, which
are defined as follows.
\begin{defn*}
Let $A:D(A)\subseteq H\to H$ be linear. The operator $A$ is called
\emph{accretive, }if for each $x\in D(A)$ we have 
\[
\Re\langle Ax|x\rangle_{H}\geq0.
\]
Moreover\emph{, $A$ }is called \emph{strictly accretive, }if there
exists some $c>0$ such that $A-c$ is accretive. Finally, $A$ is
called $m$-\emph{accretive, }if $A$ is accretive and there is some
$\lambda\in\mathbb{C}_{\Re>0}$ such that $\lambda+A$ is onto.\end{defn*}
\begin{rem}
The notion of accretive operators has a natural extension to operators
on Banach spaces and also to non-linear operators, or even relations.
For a deeper study of accretive and $m$-accretive operators we refer
to the monographs \cite{Brezis,showalter_book,papageogiou}. The letter
$m$ in the notion $m$-accretive refers to the word maximal. The
reason is, that $A$ is $m$-accretive if and only if it is accretive
and has no proper accretive extension in the set of binary relations
on $H$. This equivalence is known as Minty's Theorem, \cite{Minty}.
\end{rem}
We now show some useful properties of $m$-accretive operators.
\begin{prop}
\label{prop:prop-accretive}Let $A:D(A)\subseteq H\to H$ be $m$-accretive.
Then $A$ is densely defined and closed. Moreover, $\mu+A$ is boundedly
invertible for each $\mu\in\mathbb{C}_{\Re>0}$ and $\|(\mu+A)^{-1}\|\leq\frac{1}{\Re\mu}$. \end{prop}
\begin{proof}
Let $y\in D(A)^{\bot}.$ By assumption, there is $\lambda\in\mathbb{C}_{\Re>0}$
and $x\in D(A)$ with $y=\lambda x+Ax.$ Then we get 
\[
0=\Re\langle x|y\rangle_{H}=\Re\langle x|\lambda x+Ax\rangle_{H}\geq\Re\lambda|x|_{H}^{2},
\]
due to the accretivity of $A$. The latter gives $x=0,$ and consequently
$y=0,$ which proves the density of $D(A).$ Let now $\mu\in\mathbb{C}_{\Re>0}.$
Then we estimate 
\[
\Re\langle x|\mu x+Ax\rangle_{H}\geq\Re\mu|x|_{H}^{2},
\]
and thus, $\mu+A$ is one-to-one and the inverse $(\mu+A)^{-1}:R(\mu+A)\subseteq H\to H$
is bounded with $\|(\mu+A)^{-1}\|\le\frac{1}{\Re\mu}.$ Moreover,
by assumption, there is $\lambda\in\mathbb{C}_{\Re>0}$ such that
$\lambda+A$ is onto. Hence, $\lambda\in\rho(A),$ which in particular
implies the closedness of $(\lambda+A)^{-1}$ and hence of $A$. Furthermore,
since $\lambda\in\rho(A)$ we get 
\[
B(\lambda,\Re\lambda)\subseteq B(\lambda,\|(\lambda+A)^{-1}\|^{-1})\subseteq\rho(A),
\]
which gives $\mathbb{C}_{\Re>0}\subseteq\rho(A)$ by induction. \end{proof}
\begin{lem}
\label{lem:accretive_inverse}Let $T\in L(H)$ such that 
\[
\exists c>0\,\forall x\in H:\Re\langle Tx|x\rangle_{H}\geq c|x|_{H}^{2}.
\]
Then $T$ is boundedly invertible and 
\[
\Re\langle T^{-1}x|x\rangle_{H}\geq\frac{c}{\|T\|^{2}}|x|_{H}^{2}
\]
for each $x\in H.$\end{lem}
\begin{proof}
The strict accretivity of $T$ implies the injectivity of $T$. Moreover,
the same holds for $T^{\ast}$ and hence, $T$ has dense range. Furthermore,
$\|T^{-1}\|$ is bounded by $\frac{1}{c}$ and hence, the range of
$T$ is closed, proving the bounded invertibility of $T$. Let now
$x\in H.$ Then we estimate 
\[
|x|_{H}^{2}=|TT^{-1}x|_{H}^{2}\leq\|T\|^{2}|T^{-1}x|_{H}^{2}\leq\frac{\|T\|^{2}}{c}\Re\langle TT^{-1}x|T^{-1}x\rangle_{H}=\frac{\|T\|^{2}}{c}\Re\langle T^{-1}x|x\rangle_{H}.\tag*{\qedhere}
\]

\end{proof}
We also need a simple perturbation result for $m$-accretive operators.
\begin{prop}
\label{prop:pert_accretive}Let $A:D(A)\subseteq H\to H$ $m$-accretive
and $B\in L(H)$ accretive. Then $B+A$ is $m$-accretive.\end{prop}
\begin{proof}
The accretivity of $B+A$ is obvious. Moreover, by \prettyref{prop:prop-accretive}
we have $\mathbb{C}_{\Re>0}\subseteq\rho(A).$ In particular, $\|B\|+1\in\rho(A).$
Let $y\in H.$ Then there is $x\in H$ with 
\[
(\|B\|+1+A)^{-1}\left(y-Bx\right)=x
\]
by the contraction mapping theorem (note that $\|(\|B\|+1+A)^{-1}\|\leq\frac{1}{\|B\|+1}$
by \prettyref{prop:prop-accretive}). Since the left-hand side belongs
to $D(A),$ we infer $x\in D(A)$ with 
\[
\left(\|B\|+1+A\right)x=y-Bx
\]
and thus, $\left(\|B\|+1+B+A\right)x=y,$ which shows that $\|B\|+1+\left(B+A\right)$
is onto. Thus, $B+A$ is $m$-accretive.
\end{proof}
With the help of the latter two proposition, we can provide a class
of well-posed evolutionary problems.
\begin{prop}
\label{prop:Rainer}Let $M:D(M)\subseteq\mathbb{C}\to L(H)$ be a
linear material law and $A:D(A)\subseteq H\to H$ be $m$-accretive.
Moreover, assume there is $\rho_{0}\in\mathbb{R}$ and $c>0$ such
that 
\begin{equation}
\forall z\in\mathbb{C}_{\Re>\rho_{0}}\cap D(M),x\in H:\:\Re\langle zM(z)x|x\rangle_{H}\geq c|x|_{H}^{2}.\label{eq:pos_def_M}
\end{equation}
Then the evolutionary problem associated with $M$ and $A$ is well-posed. \end{prop}
\begin{proof}
The estimate \prettyref{eq:pos_def_M} states that $zM(z)-c$ is accretive
for each $z\in\mathbb{C}_{\Re>\rho_{0}}\cap D(M).$ Then, by \prettyref{prop:pert_accretive}
$zM(z)-c+A$ is $m$-accretive for each $z\in\mathbb{C}_{\Re>\rho_{0}}\cap D(M).$
Hence, by \prettyref{prop:prop-accretive}, $zM(z)+A$ is boundedly
invertible with $\|(zM(z)+A)^{-1}\|\leq\frac{1}{c}$ for each $z\in\mathbb{C}_{\Re>\rho_{0}}\cap D(M).$
Choosing now $\rho_{1}\geq\rho_{0}$ such that $\mathbb{C}_{\Re>\rho_{1}}\subseteq D(M)$
we get that 
\[
\mathbb{C}_{\Re>\rho_{1}}\ni z\mapsto(zM(z)+A)^{-1}\in L(H)
\]
is an analytic and bounded function. Thus, both conditions for well-posedness
of the evolutionary problem hold. \end{proof}
\begin{rem}
The latter proposition was first formulated in \cite[Solution Theory]{Picard},
where the mapping $M$ was assumed to be bounded and $A$ was skew-selfadjoint. 
\end{rem}
Finally, we consider a particular class of material laws $M$, which
satisfy \prettyref{eq:pos_def_M}.
\begin{prop}
\label{prop:classical_material_law}Let $M_{0},M_{1}\in L(H)$. Assume
that $M_{0}$ is selfadjoint and that there exist $c_{0},c_{1}>0$
with 
\begin{align*}
\langle M_{0}x|x\rangle_{H} & \geq c_{0}|x|_{H}^{2}\quad(x\in R(M_{0})),\\
\Re\langle M_{1}x'|x'\rangle_{H} & \geq c_{1}|x'|_{H}^{2}\quad(x'\in N(M_{0})).
\end{align*}
Then, $M:\mathbb{C}\setminus\{0\}\to L(H)$ defined by $M(z)\coloneqq M_{0}+z^{-1}M_{1}$
is a linear material law, which satisfies \prettyref{eq:pos_def_M}.\end{prop}
\begin{proof}
First, $M$ is obviously a linear material law. We set $\rho_{0}\coloneqq\frac{1}{c_{0}}\left(\frac{c_{1}}{2}+\|M_{1}\|^{2}\frac{2}{c_{1}}\right)$
and for $y\in H$ we use the decomposition $y=x+x'$ with $x\in\overline{R(M_{0})},x'\in N(M_{0})=R(M_{0})^{\bot}.$
We note that the strict accretivity of $M_{0}$ extends to $\overline{R(M_{0})}.$
Then we get for $z=\i t+\rho$ with $t\in\mathbb{R}$ and $\rho>\rho_{0}$
\begin{align*}
\Re\langle zM(z)y|y\rangle_{H} & =\Re\langle(\i t+\rho)M_{0}x|x\rangle_{H}+\Re\langle M_{1}x|x'\rangle_{H}+\Re\langle M_{1}x'|x\rangle_{H}+\Re\langle M_{1}x'|x'\rangle_{H}\\
 & \geq\rho c_{0}|x|_{H}^{2}-2\|M_{1}\||x|_{H}|x'|_{H}+c_{1}|x'|_{H}^{2}\\
 & \geq\left(\rho c_{0}-\|M_{1}\|^{2}\frac{2}{c_{1}}\right)|x|_{H}^{2}+\frac{c_{1}}{2}|x'|_{H}^{2}\\
 & \geq\frac{c_{1}}{2}|y|_{H}^{2},
\end{align*}
which shows the assertion.
\end{proof}

\section{Examples\label{sec:Examples_1}}

This section is devoted to several examples of partial differential
equations, which fit into the scheme of evolutionary problems, introduced
in the previous section. We start with some classical partial differential
equations from mathematical physics. Then we shortly discuss a class
of delay differential equations and we conclude this section by studying
integro-differential equations with operator-valued convolution kernels.

\subsection{Classical equations from mathematical physics\label{sub:mat_phys}}

Before we can start to give some concrete examples of equations from
mathematical physics, we need to introduce some differential operators,
which will be used throughout the text.
\begin{defn*}
Let $\Omega\subseteq\mathbb{R}^{n}$ open. We define the operators
$\grad_{0},\dive_{0}$ as the closures of the operators \nomenclature[O_010]{$\grad_0$}{the gradient on $L_2(\Omega)$ with homogeneous Dirichlet boundary values.}
\nomenclature[O_030]{$\dive_0$}{the divergence on $L_2(\Omega)^n$ with homogeneous Neumann boundary values.}\nomenclature[O_040]{$\dive$}{the divergence on $L_2(\Omega)^n$. }
\nomenclature[O_020]{$\grad$}{the gradient on $L_2(\Omega)$. } 
\begin{align*}
\grad|_{C_{c}^{\infty}(\Omega)}:C_{c}^{\infty}(\Omega)\subseteq L_{2}(\Omega) & \to L_{2}(\Omega)^{n}\\
\varphi & \mapsto\left(\partial_{i}\varphi\right)_{i\in\{1,\ldots,n\}}
\end{align*}
and 
\begin{align*}
\dive|_{C_{c}^{\infty}(\Omega)^{n}}:C_{c}^{\infty}(\Omega)^{n}\subseteq L_{2}(\Omega)^{n} & \to L_{2}(\Omega)\\
(\varphi_{i})_{i\in\{1,\ldots,n\}} & \mapsto\sum_{i=1}^{n}\partial_{i}\varphi_{i},
\end{align*}
respectively. Similarly, we define the operators $\Grad_{0},\Div_{0}$
\nomenclature[O_080]{$\Grad$}{the symmetrized Jacobian on $L_2(\Omega)^n$.}\nomenclature[O_070]{$\Grad_0$}{the symmetrized Jacobian on $L_2(\Omega)^n$ with homogeneous Dirichlet boundary values. }\nomenclature[O_090]{$\Div_0$}{the row-wise divergence on $L_2(\Omega)^{n\times n}$ with homogeneous Neumann boundary values. }\nomenclature[O_100]{$\Div$}{the row-wise divergence on $L_2(\Omega)^{n\times n}$. }as
the closures of 
\begin{align*}
\Grad|_{C_{c}^{\infty}(\Omega)^{n}}:C_{c}^{\infty}(\Omega)^{n}\subseteq L_{2}(\Omega)^{n} & \to L_{2,\mathrm{sym}}(\Omega)^{n\times n}\\
(\varphi_{i})_{i\in\{1,\ldots,n\}} & \mapsto\left(\frac{1}{2}(\partial_{j}\varphi_{i}+\partial_{i}\varphi_{j})\right)_{i,j\in\{1,\ldots,n\}}
\end{align*}
 and
\begin{align*}
\Div|_{C_{c,\mathrm{sym}}^{\infty}(\Omega)^{n\times n}}:C_{c,\mathrm{sym}}^{\infty}(\Omega)^{n\times n}\subseteq L_{2,\mathrm{sym}}(\Omega)^{n\times n} & \to L_{2}(\Omega)^{n}\\
(\varphi_{ij})_{i,j\in\{1,\ldots,n\}} & \mapsto\left(\sum_{j=1}^{n}\partial_{j}\varphi_{ij}\right)_{i\in\{1,\ldots,n\}},
\end{align*}
respectively. Here, $L_{2,\mathrm{sym}}(\Omega)^{n\times n}\coloneqq\left\{ f\in L_{2}(\Omega)^{n\times n}\,;\, f(x)^{T}=f(x)\quad(x\in\Omega\mbox{ a.e.})\right\} $
endowed with the inner product 
\[
\langle f|g\rangle_{L_{2,\mathrm{sym}}(\Omega)^{n\times n}}\coloneqq\intop_{\Omega}\trace(f(x)^{\ast}g(x))\mbox{ d}x\quad(f,g\in L_{2,\mathrm{sym}}(\Omega)^{n\times n})
\]
and $C_{c,\mathrm{sym}}^{\infty}(\Omega)^{n\times n}\coloneqq C_{c}^{\infty}(\Omega)^{n\times n}\cap L_{2,\mathrm{sym}}(\Omega)^{n\times n}.$
Then, by integration by parts one gets 
\begin{align*}
\grad_{0} & \subseteq-\dive_{0}^{\ast}\eqqcolon\grad,\\
\dive_{0} & \subseteq-\grad_{0}^{\ast}\eqqcolon\dive,\\
\Grad_{0} & \subseteq-\Div_{0}^{\ast}\eqqcolon\Grad,\\
\Div_{0} & \subseteq-\Grad_{0}^{\ast}\eqqcolon\Div.
\end{align*}
Finally, if $n=3$ we define $\curl_{0}$ as the closure of 
\begin{align*}
\curl|_{C_{c}^{\infty}(\Omega)^{3}}:C_{c}^{\infty}(\Omega)^{3}\subseteq L_{2}(\Omega)^{3} & \to L_{2}(\Omega)^{3}\\
(\varphi_{i})_{i\in\{1,2,3\}} & \mapsto\left(\begin{array}{ccc}
0 & -\partial_{3} & \partial_{2}\\
\partial_{3} & 0 & -\partial_{1}\\
-\partial_{2} & \partial_{1} & 0
\end{array}\right)\left(\begin{array}{c}
\varphi_{1}\\
\varphi_{2}\\
\varphi_{3}
\end{array}\right).
\end{align*}
Then, again by integration by parts, we have that \nomenclature[O_060]{$\curl$}{the rotation on $L_2(\Omega)^3$ .}\nomenclature[O_050]{$\curl_0$}{the rotation on $L_2(\Omega)^3$ with homogeneous electrical boundary condition.}
\[
\curl_{0}\subseteq\curl_{0}^{\ast}\eqqcolon\curl.
\]
\end{defn*}
\begin{rem}
By the definitions above, the domain of $\grad_{0}$ coincides with
the classical Sobolev space $H_{0}^{1}(\Omega),$ while $D(\grad)$
is given by $H^{1}(\Omega).$ Hence, the elements in the domains of
the differential operators indexed with $0$ satisfy an additional
boundary condition, if the boundary of $\Omega$ is smooth enough.
These conditions are given as follows:
\begin{align*}
u\in D(\grad_{0}) & \Rightarrow u=0\mbox{ on }\partial\Omega\\
u\in D(\dive_{0}) & \Rightarrow u\cdot\nu=0\mbox{ on }\partial\Omega\\
u\in D(\Grad_{0}) & \Rightarrow u=0\mbox{ on }\partial\Omega\\
u\in D(\Div_{0}) & \Rightarrow u\nu=0\mbox{ on }\partial\Omega\\
u\in D(\curl_{0}) & \Rightarrow u\times\nu=0\mbox{ on }\partial\Omega,
\end{align*}
where $\nu$ denotes the unit outward normal vector field on $\partial\Omega.$
However, we will not discuss the case of smooth boundaries and use
the domain description as a suitable generalization of those boundary
conditions, which has the advantage that we can deal with arbitrary
open sets $\Omega$.
\end{rem}

\subsubsection*{The heat equation}

Let $\Omega\subseteq\mathbb{R}^{3}$ open. The classical heat equation
consists of two equations. First, the balance of momentum, given by\nomenclature[O_130]{$\partial_{0,\rho}$}{the derivative on $H_\rho(\mathbb{R};H)$, the time-derivative.}
\[
\partial_{0,\rho}\vartheta+\dive q=f.
\]
Here, $\vartheta\in H_{\rho}(\mathbb{R};L_{2}(\Omega))$ describes
the heat density of the medium $\Omega$, $q\in H_{\rho}(\mathbb{R};L_{2}(\Omega){}^{3})$
stands for the heat flux, and $f\in H_{\rho}(\mathbb{R};L_{2}(\Omega))$
is an external heat source. The equation is completed by Fourier's
law, given by 
\[
q=-k\grad\vartheta,
\]
where $k:L_{2}(\Omega)^{3}\to L_{2}(\Omega)^{3}$ is a bounded, strictly
accretive operator, modelling the heat conductivity of the underlying
medium $\Omega.$ As $k$ is strictly accretive and bounded, so is
$k^{-1}$ by \prettyref{lem:accretive_inverse}. Thus, we may rewrite
the two equations as a system of the form 
\[
\left(\partial_{0,\rho}\left(\begin{array}{cc}
1 & 0\\
0 & 0
\end{array}\right)+\left(\begin{array}{cc}
0 & 0\\
0 & k^{-1}
\end{array}\right)+\left(\begin{array}{cc}
0 & \dive\\
\grad & 0
\end{array}\right)\right)\left(\begin{array}{c}
\vartheta\\
q
\end{array}\right)=\left(\begin{array}{c}
f\\
0
\end{array}\right).
\]
If we now impose some boundary conditions, say homogeneous Dirichlet
boundary conditions, the operator $\left(\begin{array}{cc}
0 & \dive\\
\grad & 0
\end{array}\right)$ will be replaced by $\left(\begin{array}{cc}
0 & \dive\\
\grad_{0} & 0
\end{array}\right),$ which is a skew-selfadjoint and hence, $m$-accretive operator. Thus,
we are in the situation of \prettyref{prop:Rainer} with a material
law as in \prettyref{prop:classical_material_law}.

\subsubsection*{The wave equation}

Similar to the heat equation, we can provide a formulation of the
wave equation 
\[
\partial_{0,\rho}^{2}u-\Delta u=f
\]
within the framework of evolutionary equations. We define $v\coloneqq\partial_{0,\rho}u$
and $q\coloneqq-\grad u$ and obtain a first order formulation of
the form 
\[
\left(\partial_{0,\rho}\left(\begin{array}{cc}
1 & 0\\
0 & 1
\end{array}\right)+\left(\begin{array}{cc}
0 & \dive\\
\grad & 0
\end{array}\right)\right)\left(\begin{array}{c}
v\\
q
\end{array}\right)=\left(\begin{array}{c}
f\\
0
\end{array}\right).
\]
Again, if we choose suitable boundary conditions, yielding an $m$-accretive
realization of the operator $\left(\begin{array}{cc}
0 & \dive\\
\grad & 0
\end{array}\right)$, we obtain en evolutionary problem considered in \prettyref{prop:Rainer}
with a material law of the form considered in \prettyref{prop:classical_material_law}.

\subsubsection*{Maxwell's equation}

Maxwell's equations of electro-magnetism consist of two equations
linking the electric field $E\in H_{\rho}(\mathbb{R};L_{2}(\Omega)^{3})$
and the magnetic field $H\in H_{\rho}(\mathbb{R};L_{2}(\Omega)^{3})$
in an open domain $\Omega\subseteq\mathbb{R}^{3}$ in the following
way
\begin{align*}
\partial_{0,\rho}\varepsilon E+\sigma E-\curl H & =f,\\
\partial_{0,\rho}\mu H+\curl_{0}E & =0,
\end{align*}
where $\varepsilon,\mu\in L(L_{2}(\Omega)^{3})$ are selfadjoint and
model the electric permetivity and the magnetic permeability of the
medium $\Omega$, respectively. Moreover, $\sigma\in L(L_{2}(\Omega)^{3})$
stands for the conductivity of $\Omega$ and $f\in H_{\rho}(\mathbb{R};H)$
is an external current. The first equation results from Ampere's law
combined with Ohm's law, while the second equation is Faraday's law.
Writing the two equations as a system, we end up with 
\[
\left(\partial_{0,\rho}\left(\begin{array}{cc}
\varepsilon & 0\\
0 & \mu
\end{array}\right)+\left(\begin{array}{cc}
\sigma & 0\\
0 & 0
\end{array}\right)+\left(\begin{array}{cc}
0 & -\curl\\
\curl_{0} & 0
\end{array}\right)\right)\left(\begin{array}{c}
E\\
H
\end{array}\right)=\left(\begin{array}{c}
f\\
0
\end{array}\right),
\]
which is again of the form studied in \prettyref{prop:Rainer} with
a material law as in \prettyref{prop:classical_material_law}. Hence,
the well-posedness of the problem follows if $\mu$ is strictly accretive,
$\varepsilon$ is strictly accretive on its range and $\sigma$ is
strictly accretive on the kernel of $\varepsilon.$ In particular,
we can allow certain regions where the electric permetivity vanishes
which provides a way to deal with the so-called eddy-current approximation
(see e.g. \cite{Ana2010,Pauly2016}).

\subsubsection*{The equations of visco-elasticity}

We denote by $u\in H_{\rho}(\mathbb{R};L_{2}(\Omega))$ the displacement
of an elastic body $\Omega\subseteq\mathbb{R}^{3}$ and by $\sigma\in H_{\rho}(\mathbb{R};L_{2,\mathrm{sym}}(\Omega)^{3\times3})$
the stress tensor. Then, the balance of momentum yields 
\[
\tilde{\rho}\partial_{0,\rho}^{2}u-\Div\sigma=f,
\]
where $\tilde{\rho}\in L(L_{2}(\Omega)^{3})$ selfadjoint and strictly
accretive is the density of the medium and $f\in H_{\rho}(\mathbb{R};H)$
is an external source term. The equation is completed by a constitutive
relation. In visco-elasticity a common relation is given by the Kelvin-Voigt
model 
\begin{equation}
(\partial_{0,\rho}C+D)\Grad u=\sigma,\label{eq:stress-strain}
\end{equation}
with $C,D\in L(L_{2,\mathrm{sym}}(\Omega)^{3\times3})$ modeling the
viscosity and the modulus of elasticity, respectively. Defining $v\coloneqq\partial_{0,\rho}u$
as new unknown, the latter equation can be written as 
\[
\left(C+\partial_{0,\rho}^{-1}D\right)\Grad v=\sigma.
\]
If we assume that $C$ is strictly accretive, we obtain the bounded
invertibility of $C+\partial_{0,\rho}^{-1}D$ if we choose $\rho>0$
large enough. Indeed 
\[
C+\partial_{0,\rho}^{-1}D=C(1+\partial_{0,\rho}^{-1}C^{-1}D),
\]
and since $\|\partial_{0,\rho}^{-1}C^{-1}D\|\leq\frac{1}{\rho}\|C^{-1}D\|,$
we obtain the bounded invertibility for $\rho>0$ large enough by
the Neumann series. Thus, we can write the equations of visco-elasticity
in the following way 
\[
\left(\partial_{0,\rho}\left(\begin{array}{cc}
\tilde{\rho} & 0\\
0 & \partial_{0,\rho}^{-1}(C+\partial_{0,\rho}^{-1}D)^{-1}
\end{array}\right)+\left(\begin{array}{cc}
0 & \Div\\
\Grad & 0
\end{array}\right)\right)\left(\begin{array}{c}
v\\
\sigma
\end{array}\right)=\left(\begin{array}{c}
f\\
0
\end{array}\right).
\]
If we choose suitable boundary conditions, say for simplicity homogeneous
Neumann boundary conditions for $\sigma,$ we obtain an evolutionary
equation of the form considered in \prettyref{prop:Rainer} with 
\[
M(z)=\left(\begin{array}{cc}
\tilde{\rho} & 0\\
0 & z^{-1}(C+z^{-1}D)^{-1}
\end{array}\right),\, A=\left(\begin{array}{cc}
0 & \Div_{0}\\
\Grad & 0
\end{array}\right).
\]
Moreover, the material law $M$ satisfies the well-posedness condition
\prettyref{eq:pos_def_M} of \prettyref{prop:Rainer}, since $\tilde{\rho}$
is assumed to be selfadjoint and strictly accretive, $C^{-1}$ is
strictly accretive and 
\[
\Re\langle(C+z^{-1}D)^{-1}x|x\rangle_{L_{2,\mathrm{sym}}(\Omega)^{3\times3}}\geq\Re\langle C^{-1}x|x\rangle_{L_{2,\mathrm{sym}}(\Omega)^{3\times3}}-\frac{\frac{1}{\rho}\|C^{-1}\|^{2}\|D\|}{1-\frac{1}{\rho}\|C^{-1}\|\|D\|}|x|_{L_{2,\mathrm{sym}}(\Omega)^{3\times3}}^{2}
\]
for $z\in\mathbb{C}_{\Re>\rho},x\in L_{2,\mathrm{sym}}(\Omega)^{3\times3}$,
yielding that $(C+z^{-1}D)^{-1}$ is strictly accretive uniformly
in $z\in\mathbb{C}_{\Re>\rho}$ for $\rho$ large enough (note that
the last summand tends to $0$ as $\rho$ tends to infinity). 
\begin{rem}
We note that besides the Kelvin-Voigt model there exist other models
for visco-elasticity. For instance, a common model uses convolution
terms in \prettyref{eq:stress-strain} (see e.g. \cite{Dafermos1970_abtract_Volterra,Dafermos1970_asymp_stab,Trostorff2012_integro}).
More recently, fractional derivatives were used to model elasticity
and we refer to \cite{Podlubny1999,Waurick2013_fractional} for that
topic.
\end{rem}

\subsubsection*{The equations of poro-elastic deformation}

To illustrate how systems of coupled partial differential equations
can be written as evolutionary equations, we treat the equations of
poro-elastic deformations, where a diffusion equation is coupled with
the equations of linear elasticity. We discuss the equations of poro-elasticity
as they were proposed in \cite{Murad1996} and mathematically studied
in \cite{Showalter2000,Picard2010_poroelastic} given by 
\begin{align}
\tilde{\rho}\partial_{0,\rho}^{2}u-\grad\partial_{0,\rho}\lambda\dive u-\Div C\Grad u+\grad\alpha^{\ast}p & =f,\label{eq:elastic}\\
\partial_{0,\rho}(c_{0}p+\alpha\dive u)-\dive k\grad p & =g.\label{eq:diffusion}
\end{align}
Here, $u\in H_{\rho}(\mathbb{R};L_{2}(\Omega)^{3})$ describes the
displacement field of an elastic body $\Omega\subseteq\mathbb{R}^{3}$
and $p\in H_{\rho}(\mathbb{R};L_{2}(\Omega))$ is the pressure of
a fluid diffusing through $\Omega.$ The bounded operators $C\in L(L_{2,\mathrm{sym}}(\Omega)^{3\times3}),k\in L(L_{2}(\Omega)^{3})$
stand for the elasticity tensor and the hydraulic conductivity of
the medium, respectively. The function $\tilde{\rho}\in L_{\infty}(\Omega)$
describes the density of the medium and the operator $\alpha\in L(L_{2}(\Omega))$
generalizes the so-called Biot-Willis constant. Finally, let $c_{0},\lambda\in L(L_{2}(\Omega))$,
where $c_{0}$ models the porosity of the medium and the compressibility
of the fluid. We consider the operator 
\begin{align*}
\trace:L_{2,\mathrm{sym}}(\Omega)^{3\times3} & \to L_{2}(\Omega)\\
(\Psi_{ij})_{i,j\in\{1,2,3\}} & \mapsto\sum_{i=1}^{3}\Psi_{ii}
\end{align*}
and its adjoint given by $\trace^{\ast}f=\left(\begin{array}{ccc}
f & 0 & 0\\
0 & f & 0\\
0 & 0 & f
\end{array}\right)$ for $f\in L_{2}(\Omega).$ Then, using the relations $\trace\Grad\subseteq\dive$
and $\grad=\Div\trace^{\ast},$ we can rewrite \prettyref{eq:elastic}
and \prettyref{eq:diffusion} as 
\begin{align*}
\tilde{\rho}\partial_{0,\rho}^{2}u-\Div\left(\left(\partial_{0,\rho}\trace^{\ast}\lambda\trace+C\right)\Grad u-\trace^{\ast}\alpha^{\ast}p\right) & =f,\\
\partial_{0,\rho}\left(c_{0}p+\alpha\trace\Grad u\right)-\dive k\grad p & =g.
\end{align*}
We define 
\begin{align*}
v & \coloneqq\partial_{0,\rho}u,\\
T & \coloneqq C\Grad u,\\
\omega & \coloneqq\lambda\trace\Grad v-\alpha^{\ast}p,\\
q & \coloneqq-k\grad p
\end{align*}
as new unknowns, which yields, assuming that $\lambda$ is continuously
invertible, 
\[
\trace\Grad v=\lambda^{-1}\omega+\lambda^{-1}\alpha^{\ast}p.
\]
Hence, we end up with the following equations 
\begin{align*}
\partial_{0,\rho}\tilde{\rho}v-\Div\left(T+\trace^{\ast}\omega\right) & =f,\\
\partial_{0,\rho}c_{0}p+\alpha\lambda^{-1}\omega+\alpha\lambda^{-1}\alpha^{\ast}p+\dive q & =g,\\
\lambda^{-1}\omega+\lambda^{-1}\alpha^{\ast}p-\trace\Grad v & =0,\\
\partial_{0,\rho}C^{-1}T-\Grad v & =0,\\
k^{-1}q+\grad p & =0,
\end{align*}
where we assume that $C$ and $k$ are boundedly invertible. 

Thus, as a system, the equations of poro-elasticity have the form
\[
\left(\partial_{0,\rho}M_{0}+M_{1}+U^{\ast}\left(\begin{array}{ccccc}
0 & 0 & 0 & -\Div & 0\\
0 & 0 & 0 & 0 & \dive\\
0 & 0 & 0 & 0 & 0\\
-\Grad & 0 & 0 & 0 & 0\\
0 & \grad & 0 & 0 & 0
\end{array}\right)U\right)\left(\begin{array}{c}
v\\
p\\
\omega\\
T\\
q
\end{array}\right)=\left(\begin{array}{c}
f\\
g\\
0\\
0\\
0
\end{array}\right),
\]
where $U\coloneqq\left(\begin{array}{ccccc}
1 & 0 & 0 & 0 & 0\\
0 & 1 & 0 & 0 & 0\\
0 & 0 & 1 & 0 & 0\\
0 & 0 & \trace^{\ast} & 1 & 0\\
0 & 0 & 0 & 0 & 1
\end{array}\right)$ and 
\[
M_{0}\coloneqq\left(\begin{array}{ccccc}
\tilde{\rho} & 0 & 0 & 0 & 0\\
0 & c_{0} & 0 & 0 & 0\\
0 & 0 & 0 & 0 & 0\\
0 & 0 & 0 & C^{-1} & 0\\
0 & 0 & 0 & 0 & 0
\end{array}\right),\: M_{1}\coloneqq\left(\begin{array}{ccccc}
0 & 0 & 0 & 0 & 0\\
0 & \alpha\lambda^{-1}\alpha^{\ast} & \alpha\lambda^{-1} & 0 & 0\\
0 & \lambda^{-1}\alpha^{\ast} & \lambda^{-1} & 0 & 0\\
0 & 0 & 0 & 0 & 0\\
0 & 0 & 0 & 0 & k^{-1}
\end{array}\right).
\]
 By choosing suitable boundary conditions, say homogeneous Dirichlet
boundary conditions for $v$ and $p$, we end up with an evolutionary
equation of the form discussed in \prettyref{prop:Rainer} with a
material law of the form given in \prettyref{prop:classical_material_law}.
Hence, the well-posedness can be derived, by imposing suitable constraints
on the coefficients involved in order to satisfy the hypothesis of
\prettyref{prop:classical_material_law} for $M_{0}$ and $M_{1}$.

\subsection{Differential equations with delay}

Let $H$ be a Hilbert space, $A:D(A)\subseteq H\to H$ be $m$-accretive,
$M_{0},M_{1}\in L(H)$ satisfying the assumptions of \prettyref{prop:classical_material_law}.
Moreover, let $\left(h_{k}\right)_{k\in\mathbb{N}}$ be a strictly
monotone increasing sequence in $\mathbb{R}_{>0}$ with $\eta\coloneqq\inf\{|h_{k+1}-h_{k}|\,;\, k\in\mathbb{N}\}>0$
and $(N_{k})_{k\in\mathbb{N}}$ a sequence in $L(H)$ with $\sup_{k\in\mathbb{N}}\|N_{k}\|<\infty.$
We consider a delay equation of the form 
\begin{equation}
\left(\partial_{0,\rho}M_{0}+M_{1}+\sum_{k\in\mathbb{N}}N_{k}\tau_{-h_{k}}+A\right)u=f.\label{eq:delay}
\end{equation}
This is indeed an evolutionary equation with a linear material law
given by 
\[
M(z)=M_{0}+z^{-1}M_{1}+z^{-1}\sum_{k\in\mathbb{N}}N_{k}\e^{-h_{k}z}\quad(z\in\mathbb{C}\setminus\{0\}).
\]
We note that the series $\sum_{k\in\mathbb{N}}N_{k}\e^{-h_{k}z}$
converges absolutely for each $z\in\mathbb{C}_{\Re>0}.$ Indeed, we
have that 
\begin{align*}
\sum_{k\in\mathbb{N}}\left|\e^{-h_{k}z}\right| & \leq\e^{-h_{0}\Re z}+\sum_{k=1}^{\infty}\frac{1}{h_{k}-h_{k-1}}\intop_{h_{k-1}}^{h_{k}}\e^{-\Re zs}\mbox{ d}s\\
 & \leq\e^{-h_{0}\Re z}+\frac{1}{\eta}\intop_{h_{0}}^{\infty}\e^{-\Re zs}\mbox{ d}s\\
 & =\left(1+\frac{1}{\eta\Re z}\right)\e^{-h_{0}\Re z},
\end{align*}
and since the sequence $(N_{k})_{k\in\mathbb{N}}$ is bounded, we
derive the absolute convergence of the series with 
\[
\left\Vert \sum_{k\in\mathbb{N}}N_{k}\e^{-h_{k}z}\right\Vert \leq\sup_{k\in\mathbb{N}}\|N_{k}\|\left(1+\frac{1}{\eta\Re z}\right)\e^{-h_{0}\Re z}\quad(z\in\mathbb{C}_{\Re>0}).
\]
In particular, the norm of $\sum_{k\in\mathbb{N}}N_{k}\e^{-h_{k}z}$
tends to $0$ as $\Re z\to\infty.$ Moreover, we recall that by \prettyref{prop:classical_material_law},
the material law $N(z)\coloneqq M_{0}+z^{-1}M_{1}$ satisfies \prettyref{eq:pos_def_M}
on some half plane $\mathbb{C}_{\Re>\rho}$. Hence, choosing $\rho$
large enough, we derive that also $M$ satisfies \prettyref{eq:pos_def_M}
and hence, the well-posedness of \prettyref{eq:delay} follows from
\prettyref{prop:Rainer}.

\subsection{Integro-differential equations\label{sub:Integro-differential-equations}}

In this subsection we study integro-differential equations. The results
presented here are based on \cite{Trostorff2012_integro}. We focus
on equations of the form\nomenclature[O_190]{$k\ast$}{the convolution operator with a kernel $k$.}
\begin{equation}
\left(\partial_{0,\rho}(1+k\ast)+A\right)u=f,\label{eq:integro}
\end{equation}
as well as 
\begin{equation}
\left(\partial_{0,\rho}(1-k\ast)^{-1}+A\right)u=f,\label{eq:integro_resol}
\end{equation}
where in both cases $A:D(A)\subseteq H\to H$ is an $m$-accretive
operator on some Hilbert space $H$. The kernel $k$ is a suitable
operator-valued function in the space $L_{1,\rho}(\mathbb{R}_{\geq0};L(H))$
defined as follows.
\begin{defn*}
Let $\rho\in\mathbb{R}.$ We call a function $k:\mathbb{R}_{\geq0}\to L(H)$
\emph{admissible, }if $k$ is weakly measurable (i.e., for each $x,y\in H$
the mapping $t\mapsto\langle k(t)x|y\rangle_{H}$ is measurable) and
$t\mapsto\|k(t)\|$ is measurable%
\footnote{We note that the measurability of $t\mapsto\|k(t)\|$ follows from
the weak measurability of $k$, if $H$ is separable.%
}. We define 
\[
\mathcal{L}_{1,\rho}(\mathbb{R}_{\geq0};L(H))\coloneqq\left\{ k:\mathbb{R}_{\geq0}\to L(H)\,;\, k\mbox{ admissible, }\intop_{0}^{\infty}\|k(t)\|\e^{-\rho t}\mbox{ d}t<\infty\right\} 
\]
as well as 
\[
L_{1,\rho}(\mathbb{R}_{\geq0};L(H))\coloneqq\faktor{\mathcal{L}_{1,\rho}(\mathbb{R}_{\geq0};L(H))}{\cong},
\]
where the relation $\cong$ is the usual equality almost everywhere.
We equip $L_{1,\rho}(\mathbb{R};L(H))$ with the usual norm defined
by 
\[
|k|_{L_{1,\rho}}\coloneqq\intop_{0}^{\infty}\|k(t)\|\e^{-\rho t}\mbox{ d}t\quad(k\in L_{1,\rho}(\mathbb{R};L(H))).
\]
\end{defn*}
\begin{rem}
We note that $L_{1,\rho}(\mathbb{R}_{\geq0};L(H))\hookrightarrow L_{1,\mu}(\mathbb{R}_{\geq0};L(H))$
for $\rho\leq\mu$ with $|k|_{L_{1,\mu}}\leq|k|_{L_{1,\rho}}$ for
$k\in L_{1,\rho}(\mathbb{R}_{\geq0};L(H)).$\end{rem}
\begin{lem}
\label{lem:convolution_well-defined} Let $\rho_{0}\in\mathbb{R}$
and $k\in L_{1,\rho_{0}}(\mathbb{R}_{\geq0};L(H)).$ Let $I\subseteq\mathbb{R}$
a bounded interval and $x\in H.$ Then the mapping 
\[
H\ni y\mapsto\intop_{0}^{\infty}\langle k(s)x|y\rangle_{H}\chi_{I}(t-s)\,\dd s\in\mathbb{C}
\]
is a bounded linear functional for each $t\in\mathbb{R}$ and we define
$\left(k\ast\left(\chi_{I}x\right)\right)(t)\in H$ as the element
satisfying 
\[
\langle\left(k\ast\left(\chi_{I}x\right)\right)(t)|y\rangle_{H}=\intop_{0}^{\infty}\langle k(s)x|y\rangle_{H}\chi_{I}(t-s)\,\dd s\quad(y\in H).
\]
The so defined mapping $k\ast\left(\chi_{I}x\right):\mathbb{R}\to H$
is continuous and 
\begin{equation}
\left|\left(k\ast\chi_{I}x\right)(t)\right|_{H}\leq\intop_{0}^{\infty}\|k(s)\||\chi_{I}(t-s)x|_{H}\,\dd s\label{eq:convolution_ptwise}
\end{equation}
for each $t\in\mathbb{R}$.\end{lem}
\begin{proof}
For $t\in\mathbb{R}$ and $y\in H$ we have that 
\begin{align*}
\left|\intop_{0}^{\infty}\langle k(s)x|y\rangle_{H}\chi_{I}(t-s)\mbox{ d}s\right|_{H} & \leq\intop_{0}^{\infty}\|k(s)\||\chi_{I}(t-s)x|_{H}\mbox{ d}s|y|_{H}
\end{align*}
which proves that the functional is indeed bounded. The linearity
is trivial. Hence, by the Riesz representation theorem there exists
a unique element $\left(k\ast\left(\chi_{I}x\right)\right)(t)\in H$
with 
\[
\langle\left(k\ast\left(\chi_{I}x\right)\right)(t)|y\rangle_{H}=\intop_{0}^{\infty}\langle k(s)x|y\rangle_{H}\chi_{I}(t-s)\mbox{ d}s
\]
for each $y\in H$ and the asserted estimate holds. Moreover, for
$t,t'\in\mathbb{R}$ we have 
\begin{align*}
\left|\left(k\ast(\chi_{I}x)\right)(t)-\left(k\ast(\chi_{I}x)\right)(t')\right|_{H} & =\sup_{y\in H,|y|_{H}=1}\langle\left(k\ast(\chi_{I}x)\right)(t)-\left(k\ast(\chi_{I}x)\right)(t')|y\rangle_{H}\\
 & =\sup_{y\in H,|y|_{H}=1}\intop_{0}^{\infty}\langle k(s)x|y\rangle_{H}\left(\chi_{I}(t-s)-\chi_{I}(t'-s)\right)\,\dd s\\
 & \leq\intop_{0}^{\infty}\|k(s)\||\chi_{I}(t-s)-\chi_{I}(t'-s)|\,\dd s|x|_{H}\\
 & \to0\quad(t'\to t),
\end{align*}
by dominated convergence.\end{proof}
\begin{lem}
\label{lem:convolution_bd}Let $\rho_{0}\in\mathbb{R}$ and $k\in L_{1,\rho_{0}}(\mathbb{R}_{\geq0};L(H)).$
Then the mapping 
\begin{align*}
k\ast:\mathrm{Sim}(\mathbb{R};H)\subseteq H_{\rho}(\mathbb{R};H) & \to H_{\rho}(\mathbb{R};H)\\
\sum_{i=1}^{n}\chi_{I_{i}}x_{i} & \mapsto\left(\sum_{i=1}^{n}k\ast\left(\chi_{I_{i}}x_{i}\right)\right),
\end{align*}
where $\mathrm{Sim}(\mathbb{R};H)$ denotes the space of simple functions
with values in $H$, is well-defined. Moreover, $k\ast$ extends to
a bounded linear operator on $H_{\rho}(\mathbb{R};H)$ for each $\rho\geq\rho_{0}$
with $\|k\ast\|_{L(H_{\rho})}\leq|k|_{L_{1,\rho}}.$ In particular
$\|k\ast\|_{L(H_{\rho})}\to0$ as $\rho\to\infty.$\end{lem}
\begin{proof}
By \prettyref{lem:convolution_well-defined} we have defined $k\ast$
for functions $\varphi=\chi_{I}x$ where $I\subseteq\mathbb{R}$ is
a bounded interval and $x\in H.$ Thus, for $\varphi\in\mathrm{Sim}(\mathbb{R};H)$,
$k\ast\varphi$ is a continuous function, in particular, it is measurable.
Moreover, choosing pairwise disjoint intervals, we derive from \prettyref{eq:convolution_ptwise}
that 
\[
|\left(k\ast\varphi\right)(t)|_{H}\leq\intop_{0}^{\infty}\|k(s)\||\varphi(t-s)|_{H}\mbox{ d}s\quad(\varphi\in\mathrm{Sim}(\mathbb{R};H),t\in\mathbb{R}).
\]
Hence, for $\rho\geq\rho_{0}$ and $\varphi\in\mathrm{Sim}(\mathbb{R};H)$
we have that 
\begin{align*}
\intop_{\mathbb{R}}\left|\left(k\ast\varphi\right)(t)\right|_{H}^{2}\e^{-2\rho t}\mbox{ d}t & \leq\intop_{\mathbb{R}}\left(\intop_{0}^{\infty}\|k(s)\||\varphi(t-s)|_{H}\right)^{2}\e^{-2\rho t}\mbox{ d}t\\
 & \leq\intop_{\mathbb{R}}\left(\intop_{0}^{\infty}\|k(s)\|\e^{-\rho s}\mbox{ d}s\right)\left(\intop_{0}^{\infty}\|k(s)\||\varphi(t-s)|_{H}^{2}\e^{\rho s}\mbox{ d}s\right)\e^{-2\rho t}\mbox{ d}t\\
 & =|k|_{L_{1,\rho}}\intop_{0}^{\infty}\|k(s)\|\e^{-\rho s}\intop_{\mathbb{R}}|\varphi(t-s)|_{H}^{2}\ \e^{-2\rho(t-s)}\mbox{ d}t\mbox{ d}s\\
 & =|k|_{L_{1,\rho}}^{2}|\varphi|_{\rho}^{2},
\end{align*}
which shows the first assertion. The second assertion follows from
$\|k\ast\|_{L(H_{\rho})}\leq|k|_{L_{1,\rho}}$ and $|k|_{L_{1,\rho}}\to0$
for $\rho\to\infty,$ by monotone convergence.
\end{proof}
In case of a separable Hilbert space $H$, we have the usual integral
expression for the function $k\ast f$ as the next lemma shows. 
\begin{lem}
\label{lem:integral_formula_conv}Let $\rho_{0}\in\mathbb{R}$ and
$k\in L_{1,\rho_{0}}(\mathbb{R}_{\geq0};L(H))$, $H$ separable. Then
for $f\in H_{\rho}(\mathbb{R};H)$ with $\rho\geq\rho_{0}$ we have
that 
\[
\left(k\ast f\right)(t)=\intop_{0}^{\infty}k(s)f(t-s)\,\dd s\quad(t\in\mathbb{R}\mbox{ a.e.}).
\]
\end{lem}
\begin{proof}
We first prove that 

\[
\mathbb{R}_{\geq0}\ni t\mapsto k(t)f(t)\in H
\]
is measurable. Indeed, first we note that for all $x,y\in H$ we have
that 
\[
t\mapsto\langle x|k(t)^{\ast}y\rangle_{H}=\langle k(t)x|y\rangle_{H}
\]
is measurable, i.e. $t\mapsto k(t)^{\ast}y$ is weakly measurable
for each $y\in H$. By the Theorem of Pettis (see \prettyref{thm:Pettis}),
we infer that $t\mapsto k(t)^{\ast}y$ is measurable for each $y\in H.$
Hence, 
\[
t\mapsto\langle k(t)f(t)|y\rangle_{H}=\langle f(t)|k(t)^{\ast}y\rangle_{H}
\]
is measurable and thus, again by \prettyref{thm:Pettis} we derive
the measurability of $t\mapsto k(t)f(t).$ In particular, the function
\[
\mathbb{R}_{\geq0}\ni s\mapsto k(s)f(t-s)\in H
\]
is measurable for each $t\in\mathbb{R}$. Moreover, $\intop_{0}^{\infty}|k(s)f(t-s)|_{H}\mbox{ d}s<\infty$
for almost every $t\in\mathbb{R},$ since 
\begin{align*}
\intop_{\mathbb{R}}\left(\intop_{0}^{\infty}|k(s)f(t-s)|_{H}\mbox{ d}s\right)^{2}\e^{-2\rho t}\mbox{ d}t & \leq\intop_{\mathbb{R}}\left(\intop_{0}^{\infty}\|k(s)e^{-\rho s}\||f(t-s)\e^{-\rho(t-s)}|_{H}\mbox{ d}s\right)^{2}\mbox{ d}t\\
 & \leq|k|_{L_{1,\rho}}^{2}|f|_{\rho}^{2}
\end{align*}
by Young's inequality. Consider now the function $g:\mathbb{R}\to H$
defined by 
\[
g(t)\coloneqq\begin{cases}
\intop_{0}^{\infty}k(s)f(t-s)\mbox{ d}s, & \mbox{ if }\intop_{0}^{\infty}|k(s)f(t-s)|_{H}\mbox{ d}s<\infty,\\
0, & \mbox{ otherwise}.
\end{cases}
\]
Then $g$ is measurable. Indeed, for $y\in H$ we have that 
\begin{align*}
\intop_{\mathbb{R}}\left|\intop_{0}^{\infty}\langle k(s)f(t-s)|y\rangle_{H}\mbox{ d}s\right|^{2}\e^{-2\rho t}\mbox{ d}t & \leq\intop_{\mathbb{R}}\left(\intop_{0}^{\infty}|k(s)f(t-s)|_{H}\mbox{ d}s\right)^{2}\e^{-2\rho t}\mbox{ d}t|y|_{H}^{2}\\
 & \leq|k|_{L_{1,\rho}}^{2}|y|_{H}^{2}|f|_{\rho}^{2}
\end{align*}
and since 
\[
\mathbb{R}_{\geq0}\times\mathbb{R}\ni(s,t)\mapsto\langle k(s)f(t-s)|y\rangle_{H}
\]
is measurable, we obtain the measurability of 
\[
t\mapsto\langle g(t)|y\rangle_{H}
\]
by Fubini's Theorem. Again by \prettyref{thm:Pettis}, the measurability
of $g$ follows and by the estimate shown above, we have that 
\[
\intop_{\mathbb{R}}\left|\intop_{0}^{\infty}k(s)f(t-s)\mbox{ d}s\right|^{2}\e^{-2\rho t}\mbox{ d}t\leq|k|_{L_{1,\rho}}^{2}|f|_{\rho}^{2}.
\]
Summarizing we have shown that 
\[
\tilde{k\ast}:f\mapsto\left(t\mapsto\intop_{0}^{\infty}k(s)f(t-s)\mbox{ d}s\right)
\]
is a well-defined and bounded operator on $H_{\rho}(\mathbb{R};H)$.
Let now $I\subseteq\mathbb{R}$ be a bounded interval and $x\in H.$
Then 
\[
\intop_{0}^{\infty}\langle k(s)x|y\rangle_{H}\chi_{I}(t-s)\,\dd s=\left\langle \left.\intop_{0}^{\infty}k(s)\chi_{I}(t-s)x\mbox{ d}s\right|y\right\rangle _{H}\quad(y\in H)
\]
which proves that 
\[
\left(k\ast(\chi_{I}x)\right)(t)=\intop_{0}^{\infty}k(s)\chi_{I}(t-s)x\mbox{ d}s
\]
for each $t\in\mathbb{R}.$ Consequently, we have for each $\varphi\in\mathrm{Sim}(\mathbb{R};H)$
that 
\[
\left(k\ast\varphi\right)(t)=\intop_{0}^{\infty}k(s)\varphi(t-s)\mbox{ d}s\quad(t\in\mathbb{R}),
\]
i.e. $k\ast$ and $\tilde{k\ast}$ coincide on the dense set $\mathrm{Sim}(\mathbb{R};H),$
which yields the assertion. \end{proof}
\begin{defn*}
Let $\rho_{0}\in\mathbb{R}$ and $k\in L_{1,\rho_{0}}(\mathbb{R}_{\geq0};L(H)).$
Then we define $\hat{k}(z)\in L(H)$ for $z\in\mathbb{C}_{\Re\geq\rho_{0}}$
by 
\[
\langle x|\hat{k}(z)y\rangle_{H}\coloneqq\frac{1}{\sqrt{2\pi}}\intop_{\mathbb{R}_{\geq0}}\e^{-zs}\langle x|k(s)y\rangle_{H}\mbox{ d}s\quad(x,y\in H).
\]
\end{defn*}
\begin{rem}
We note that $\hat{k}(z)$ is well-defined by the Riesz representation
theorem. Moreover, the mapping $\mathbb{C}_{\Re>\rho_{0}}\ni z\mapsto\hat{k}(z)\in L(H)$
is bounded by $\frac{1}{\sqrt{2\pi}}|k|_{L_{1,\rho_{0}}}$ and analytic
according to \cite[Theorem 3.10.1]{hille1957functional}.\end{rem}
\begin{lem}
\label{lem:kernel_Fourier}Let $\rho_{0}\in\mathbb{R}$ and $k\in L_{1,\rho_{0}}(\mathbb{R}_{\geq0};L(H)).$
Then, $k\ast=\sqrt{2\pi}\hat{k}(\partial_{0,\rho})$ for each $\rho>\rho_{0}.$ \end{lem}
\begin{proof}
Since $\hat{k}$ is bounded on $\mathbb{C}_{\Re>\rho_{0}},$ the operator
$\hat{k}(\partial_{0,\rho})$ is bounded for each $\rho>\rho_{0}$
as well. Moreover, by \prettyref{lem:convolution_bd} the operator
$k\ast$ is bounded on $H_{\rho}(\mathbb{R};H),$ too. Thus, it suffices
to show $k\ast\varphi=\sqrt{2\pi}\hat{k}(\partial_{0,\rho})\varphi$
for $\varphi=\chi_{I}y$ for some bounded interval $I\subseteq\mathbb{R},y\in H$
and $\rho>\rho_{0}.$ So, let $\rho>\rho_{0}$ and $\varphi=\chi_{I}y$
for some interval $I\subseteq\mathbb{R},y\in H.$ Using that $k\ast\varphi\in L_{1,\rho}(\mathbb{R};H)$
by \prettyref{eq:convolution_ptwise}, we can compute
\begin{align*}
\langle x|\sqrt{2\pi}\hat{k}(\i t+\rho)\left(\mathcal{L}_{\rho}\varphi\right)(t)\rangle_{H} & =\intop_{\mathbb{R}}\e^{-(\i t+\rho)s}\langle x|k(s)\left(\mathcal{L}_{\rho}\varphi\right)(t)\rangle_{H}\mbox{ d}s\\
 & =\frac{1}{\sqrt{2\pi}}\intop_{\mathbb{R}}\intop_{\mathbb{R}}\e^{-(\i t+\rho)(s+r)}\langle x|k(s)\varphi(r)\rangle_{H}\mbox{ d}r\mbox{ d}s\\
 & =\frac{1}{\sqrt{2\pi}}\intop_{\mathbb{R}}\intop_{\mathbb{R}}\e^{-(\i t+\rho)r}\langle x|k(s)y\rangle_{H}\chi_{I}(r-s)\mbox{ d}r\mbox{ d}s\\
 & =\frac{1}{\sqrt{2\pi}}\intop_{\mathbb{R}}\e^{-(\i t+\rho)r}\langle x|\left(k\ast\varphi\right)(r)\rangle_{H}\mbox{ d}r\\
 & =\langle x|\mathcal{L}_{\rho}(k\ast\varphi)(t)\rangle_{H}
\end{align*}
for each $x\in H,t\in\mathbb{R}.$ That shows the assertion.
\end{proof}
The latter lemma gives that \prettyref{eq:integro} and \prettyref{eq:integro_resol}
are indeed evolutionary problems with $M(z)=1+\sqrt{2\pi}\hat{k}(z)$
and $M(z)=\left(1+\sqrt{2\pi}\hat{k}(z)\right)$, respectively. We
now address the well-posedness of the problems \prettyref{eq:integro}
and \prettyref{eq:integro_resol}. For doing so, we formulate the
following conditions.
\begin{condition}
\label{cond:kernel}Let $\rho_{0}\in\mathbb{R}$ and $k\in L_{1,\rho_{0}}(\mathbb{R}_{\geq0};L(H)).$
We say that $k$ satisfies the condition (a),(b) and (c), respectively,
if 

\begin{enumerate}[(a)]

\item \label{item: selfadjoint}For almost every $t\in\mathbb{R},$
the operator $k(t)$ is selfadjoint.

\item \label{item: commute}For almost every $t,s\in\mathbb{R}$
we have $k(t)k(s)=k(s)k(t).$

\item \label{item:im_d}There exists $\rho_{1}\geq\rho_{0}$ and
$d\leq0$ such that 
\[
t\Im\langle\hat{k}(\i t+\rho_{1})x|x\rangle_{H}\geq d|x|_{H}^{2}
\]
for each $t\in\mathbb{R},x\in H.$

\end{enumerate}
\end{condition}
We first show, that a kernel $k$ satisfying \prettyref{cond:kernel}
(a) and (c), also satisfies a similar inequality like in (c) for all
$\rho>\rho_{1}.$ The precise statement is as follows.
\begin{lem}
\label{lem:for_one_rho_for_all_rho}Let $\rho_{0}\in\mathbb{R}$ and
$k\in L_{1,\rho_{0}}(\mathbb{R}_{\geq0};L(H))$ satisfying \prettyref{cond:kernel}
\prettyref{item: selfadjoint} and \prettyref{item:im_d}. Then for
each $t\in\mathbb{R},\rho\geq\rho_{1}$ (where $\rho_{1}$ is chosen
according to \prettyref{cond:kernel} (c)) and $x\in H$ one has
\[
t\Im\langle\hat{k}(\i t+\rho)x|x\rangle_{H}\geq4d|x|_{H}^{2}.
\]
\end{lem}
\begin{proof}
The proof is based on the proof presented in \cite[Lemma 3.4]{Cannarsa2003}
in case of scalar-valued kernels. For $x\in H$ we define 
\[
f(t)\coloneqq\langle k(t)x|x\rangle_{H}\quad(t\in\mathbb{R}_{\geq0})
\]
and get $f\in L_{1,\rho_{0}}(\mathbb{R}_{\geq0};\mathbb{R})$ by the
selfadjointness of $k(t).$ Moreover, we get 
\begin{align*}
\hat{f}(\i t+\rho) & =\frac{1}{\sqrt{2\pi}}\intop_{0}^{\infty}\e^{-(\i t+\rho)s}\langle k(s)x|x\rangle_{H}\mbox{ d}s\\
 & =\langle\hat{k}(-\i t+\rho)x|x\rangle_{H}
\end{align*}
for each $t\in\mathbb{R},\rho\geq\rho_{1}.$ Hence, we have to show
that $t\Im\hat{f}(-\i t+\rho)\geq4d|x|_{H}^{2}$ for each $t\in\mathbb{R},x\in H.$
Moreover, we note that since $f$ is real-valued we have that $\Im\hat{f}(-\i t+\rho)=-\Im\hat{f}(\i t+\rho)$
and hence, we need to prove $t\Im\hat{f}(\i t+\rho)\leq-4d|x|_{H}^{2}.$
Using that $\Im\hat{f}$ is a harmonic function, we employ the Poisson
formula for the half-plane (see e.g. \cite[p.149]{stein2003fourier})
and get 
\begin{align*}
 & \Im\hat{f}(\i t+\rho)\\
 & =\frac{1}{\pi}\intop_{\mathbb{R}}\frac{\rho-\rho_{1}}{(t-s)^{2}+(\rho-\rho_{1})^{2}}\Im\hat{f}(\i s+\rho_{1})\mbox{ d}s\\
 & =\frac{\rho-\rho_{1}}{\pi}\left(\intop_{0}^{\infty}\frac{1}{(t-s)^{2}+(\rho-\rho_{1})^{2}}\Im\hat{f}(\i s+\rho_{1})\mbox{ d}s+\right.\\
 & \phantom{aaaaaaaaaaaaa}\left.+\intop_{0}^{\infty}\frac{1}{(t+s)^{2}+(\rho-\rho_{1})^{2}}\Im\hat{f}(-\i s+\rho_{1})\mbox{ d}s\right)\\
 & =\frac{\rho-\rho_{1}}{\pi}\intop_{0}^{\infty}\frac{4st}{\left((t-s)^{2}+(\rho-\rho_{1})^{2}\right)\left((t+s)^{2}+(\rho-\rho_{1})^{2}\right)}\Im\hat{f}(\i s+\rho_{1})\mbox{ d}s,
\end{align*}
where we have used $\Im\hat{f}(z^{\ast})=-\Im\hat{f}(z)$ for $z\in\mathbb{C}_{\Re\geq\rho_{1}}.$
Using \prettyref{cond:kernel} \prettyref{item:im_d} we estimate
\begin{align*}
t\Im\hat{f}(\i t+\rho) & \leq-4t^{2}\frac{\rho-\rho_{1}}{\pi}d|x|_{H}^{2}\intop_{0}^{\infty}\frac{1}{\left((t-s)^{2}+(\rho-\rho_{1})^{2}\right)\left((t+s)^{2}+(\rho-\rho_{1})^{2}\right)}\mbox{ d}s\\
 & \leq-4d|x|_{H}^{2}\frac{\rho-\rho_{1}}{\pi}\intop_{0}^{\infty}\frac{1}{(t-s)^{2}+(\rho-\rho_{1})^{2}}\mbox{ d}s\\
 & \leq-4d|x|_{H}^{2}.\tag*{\qedhere}
\end{align*}

\end{proof}
With this result at hand, we are able to prove the well-posedness
of the integro-differential equations. In fact, we show that the material
laws satisfy \prettyref{eq:pos_def_M} and thus, the well-posedness
follows from \prettyref{prop:Rainer}. We start with \prettyref{eq:integro}.
\begin{prop}
Let $\rho_{0}\in\mathbb{R}$ and $k\in L_{1,\rho_{0}}(\mathbb{R}_{\geq0};H).$
Moreover, assume that $k$ satisfies \prettyref{cond:kernel} \prettyref{item: selfadjoint}
and \prettyref{item:im_d}. Then, the material law $M$ defined by
$M(z)\coloneqq1+\sqrt{2\pi}\hat{k}(z)$ satisfies \prettyref{eq:pos_def_M}
on $\mathbb{C}_{\Re>\rho}$ for some $\rho\geq\max\{0,\rho_{0}\}$.\end{prop}
\begin{proof}
For $x\in H$ and $t\in\mathbb{R},\rho\geq\max\{0,\rho_{1}\}$ where
$\rho_{1}$ is chosen according to \prettyref{cond:kernel} \prettyref{item:im_d}
we estimate 
\begin{align*}
\Re\langle(\i t+\rho)M(\i t+\rho)x|x\rangle_{H} & =\rho|x|_{H}^{2}+\sqrt{2\pi}\Re\langle(\i t+\rho)\hat{k}(\i t+\rho)x|x\rangle_{H}\\
 & =\rho|x|_{H}^{2}+\sqrt{2\pi}\left(\rho\Re\langle\hat{k}(\i t+\rho)x|x\rangle_{H}+t\Im\langle\hat{k}(\i t+\rho)x|x\rangle_{H}\right)\\
 & \geq\left(\rho(1-|k|_{L_{1,\rho}})+\sqrt{2\pi}4d\right)|x|_{H}^{2},
\end{align*}
where we have used \prettyref{lem:for_one_rho_for_all_rho}. Thus,
choosing $\rho$ large enough, the assertion follows, since $|k|_{L_{1,\rho}}\to0$
as $\rho\to\infty$ by monotone convergence.
\end{proof}
To deal with \prettyref{eq:integro_resol}, we additionally need to
impose \prettyref{cond:kernel} \prettyref{item: commute}. We start
with the following observation.
\begin{lem}
\label{lem:comp_realpart_kernel} Let $\rho_{0}\in\mathbb{R}$ and
$k\in L_{1,\rho_{0}}(\mathbb{R}_{\geq0};H)$ such that $|k|_{L_{1,\rho_{0}}}<1$
and assume that $k$ satisfies \prettyref{cond:kernel} \prettyref{item: selfadjoint}
and \prettyref{item: commute}. Let $\rho>\rho_{0}$ and $t\in\mathbb{R}.$
Then the operator $|1-\sqrt{2\pi}\hat{k}(\i t+\rho)|$ is boundedly
invertible and for each $x\in H$ we have 
\begin{align*}
 & \Re\langle(\i t+\rho)(1-\sqrt{2\pi}\hat{k}(\i t+\rho))^{-1}x|x\rangle_{H}\\
 & =\rho|D(\i t+\rho)x|_{H}^{2}-\sqrt{2\pi}\rho\Re\langle\hat{k}(-\i t+\rho)D(\i t+\rho)x|D(\i t+\rho)x\rangle_{H}+\\
 & \quad-\sqrt{2\pi}t\Im\langle\hat{k}(-\i t+\rho)D(\i t+\rho)x|D(\i t+\rho)x\rangle_{H},
\end{align*}
where $D(\i t+\rho)\coloneqq|1-\sqrt{2\pi}\hat{k}(\i t+\rho)|^{-1}.$ \end{lem}
\begin{proof}
We have $\|\sqrt{2\pi}\hat{k}(\i t+\rho)\|\leq\|k\ast\|_{L(H_{\rho}(\mathbb{R};H))}\leq|k|_{L_{1,\rho_{0}}}<1$
and hence, $1-\sqrt{2\pi}\hat{k}(\i t+\rho)$ is boundedly invertible
due to the Neumann series. Moreover, by \prettyref{cond:kernel} \prettyref{item: selfadjoint}
we have that $\hat{k}(\i t+\rho)^{\ast}=\hat{k}(-\i t+\rho)$ and
thus, we have that $(1-\sqrt{2\pi}\hat{k}(\i t+\rho))^{\ast}$ is
boundedly invertible, too. The latter gives that $|1-\sqrt{2\pi}\hat{k}(\i t+\rho)|$
is boundedly invertible. Moreover, \prettyref{cond:kernel} \prettyref{item: commute}
yields that $\hat{k}(\i t+\rho)$ is normal and so is $1-\sqrt{2\pi}\hat{k}(\i t+\rho).$
This implies that $1-\sqrt{2\pi}\hat{k}(\i t+\rho)$, $\left(1-\sqrt{2\pi}\hat{k}(\i t+\rho)\right)^{\ast}$
and $D(\i t+\rho)$ pairwise commute. Thus, we have 
\begin{align*}
(1-\sqrt{2\pi}\hat{k}(\i t+\rho))^{-1} & =(1-\sqrt{2\pi}\hat{k}(\i t+\rho))^{\ast}D(\i t+\rho)^{2}\\
 & =D(\i t+\rho)(1-\sqrt{2\pi}\hat{k}(-\i t+\rho))D(\i t+\rho),
\end{align*}
where we have used $\hat{k}(\i t+\rho)^{\ast}=\hat{k}(-\i t+\rho)$
again. Hence, we compute 
\begin{align*}
 & \Re\langle(\i t+\rho)(1-\sqrt{2\pi}\hat{k}(\i t+\rho))^{-1}x|x\rangle_{H}\\
 & =\Re(-\i t+\rho)\langle(1-\sqrt{2\pi}\hat{k}(-\i t+\rho))D(\i t+\rho)x|D(\i t+\rho)x\rangle_{H}\\
 & =\rho|D(\i t+\rho)x|_{H}^{2}-\sqrt{2\pi}\rho\Re\langle\hat{k}(-\i t+\rho)D(\i t+\rho)x|D(\i t+\rho)x\rangle_{H}+\\
 & \quad-\sqrt{2\pi}t\Im\langle\hat{k}(-\i t+\rho)D(\i t+\rho)x|D(\i t+\rho)x\rangle_{H}
\end{align*}
for each $x\in H.$ \end{proof}
\begin{prop}
\label{prop:kernel_resolvent}Let $\rho_{0}\in\mathbb{R}$ and $k\in L_{1,\rho_{0}}(\mathbb{R}_{\geq0};H).$
Then there is $\rho\geq\rho_{0},$ such that $1-\sqrt{2\pi}\hat{k}(z)$
is boundedly invertible for each $z\in\mathbb{C}_{\Re>\rho}.$ If
$k$ satisfies \prettyref{cond:kernel} \prettyref{item: selfadjoint}
- \prettyref{item:im_d}, then there is some $\rho\geq\max\{0,\rho_{0}\}$,
such that $M(z)\coloneqq(1-\sqrt{2\pi}\hat{k}(z))^{-1}$ satisfies
\prettyref{eq:pos_def_M} on $\mathbb{C}_{\Re>\rho}.$\end{prop}
\begin{proof}
First we choose $\rho_{2}\geq\rho_{1}$, where $\rho_{1}$ is chosen
according to \prettyref{cond:kernel} \prettyref{item:im_d}, such
that $|k|_{L_{1,\rho_{2}}}<1$. Hence, by \prettyref{lem:comp_realpart_kernel}
we have for each $x\in H,t\in\mathbb{R}$ and $\rho>\rho_{2}$ 
\begin{align*}
\Re\langle(\i t+\rho)M(\i t+\rho)x|x\rangle_{H} & =\Re\left\langle \left.(\i t+\rho)\left(1-\sqrt{2\pi}\hat{k}(-\i t+\rho)\right)^{-1}x\right|x\right\rangle _{H}\\
 & =\rho|D(\i t+\rho)x|_{H}^{2}-\sqrt{2\pi}\rho\Re\langle\hat{k}(-\i t+\rho)D(\i t+\rho)x|D(\i t+\rho)x\rangle_{H}+\\
 & \quad-\sqrt{2\pi}t\Im\langle\hat{k}(-\i t+\rho)D(\i t+\rho)x|D(\i t+\rho)x\rangle_{H},
\end{align*}
where $D(\i t+\rho)\coloneqq|1-\sqrt{2\pi}\hat{k}(\i t+\rho)|^{-1}$.
Choosing $\rho>\max\{0,\rho_{2}\}$ and using \prettyref{lem:for_one_rho_for_all_rho}
we infer that 
\[
\Re\langle(\i t+\rho)M(\i t+\rho)x|x\rangle_{H}\geq\left(\rho(1-|k|_{L_{1,\rho}})+\sqrt{2\pi}4d\right)|D(\i t+\rho)x|_{H}^{2}
\]
Moreover, we have that 
\[
|x|_{H}=|D(\i t+\rho)^{-1}D(\i t+\rho)x|_{H}=\left|(1-\sqrt{2\pi}\hat{k}(\i t+\rho))D(\i t+\rho)x\right|_{H}\leq(1+|k|_{L_{1,\rho}})|D(\i t+\rho)x|_{H},
\]
and hence, 
\[
\Re\langle(\i t+\rho)M(\i t+\rho)x|x\rangle_{H}\geq\frac{\rho\left(1-|k|_{L_{1,\rho}}\right)+\sqrt{2\pi}4d}{\left(1+|k|_{L_{1,\rho}}\right)^{2}}|x|_{H}^{2}.
\]
Taking into account that $|k|_{L_{1,\rho}}\to0$ as $\rho\to\infty$,
we derive the assertion for large enough $\rho.$ 
\end{proof}
We conclude this section with some classical examples of kernels,
satisfying \prettyref{cond:kernel} \prettyref{item: selfadjoint}-\prettyref{item:im_d}.
\begin{example}
\label{exa:kernels}Let $k:\mathbb{R}_{\geq0}\to\mathbb{R}$ measurable
such that $\intop_{0}^{\infty}|k(t)|\e^{-\rho_{0}t}\mbox{ d}t<\infty$
for some $\rho_{0}\in\mathbb{R}.$ Then clearly, $k$ satisfies \prettyref{cond:kernel}
\prettyref{item: selfadjoint} and \prettyref{item: commute}.

\begin{enumerate}[(a)]

\item Assume that $k$ is absolutely continuous with $\intop_{0}^{\infty}|k'(t)|\e^{-\rho_{1}t}\mbox{ d}t<\infty$
for some $\rho_{1}\in\mathbb{R}.$ Then, $k$ satisfies \prettyref{cond:kernel}
\prettyref{item:im_d}. Indeed, for $t\in\mathbb{R},\rho>\max\{\rho_{0},\rho_{1},0\}$
we have that
\begin{align*}
\sqrt{2\pi}\hat{k}(\i t+\rho) & =\intop_{0}^{\infty}\e^{-(\i t+\rho)s}k(s)\mbox{ d}s\\
 & =\intop_{0}^{\infty}\e^{-(\i t+\rho)s}\intop_{0}^{s}k'(r)\mbox{ d}r\mbox{ d}s+k(0)\intop_{0}^{\infty}\e^{-(\i t+\rho)s}\mbox{ d}s\\
 & =\frac{1}{\i t+\rho}\sqrt{2\pi}\hat{k'}(\i t+\rho)+\frac{1}{\i t+\rho}k(0)
\end{align*}
 and thus, for $x\in H$ we can estimate 
\begin{align*}
t\Im\langle\hat{k}(\i t+\rho)x|x\rangle_{H} & =-t\Im\hat{k}(\i t+\rho)|x|_{H}^{2}\\
 & \geq-|t||\hat{k}(\i t+\rho)||x|_{H}^{2}\\
 & =-|t|\frac{1}{|\i t+\rho|}\left(|\hat{k'}(\i t+\rho)|+\frac{1}{\sqrt{2\pi}}|k(0)|\right)|x|_{H}^{2}\\
 & \geq-\frac{1}{\sqrt{2\pi}}(|k'|_{L_{1,\rho_{0}}}+|k(0)|)|x|_{H}^{2}.
\end{align*}

\item In \cite{Pruss2009} the kernel is assumed to be non-negative
and non-increasing. In fact, this also yields that $k$ satisfies
\prettyref{cond:kernel} \prettyref{item:im_d} with $d=0.$ Indeed,
we can even generalize this fact to kernels $k\in L_{1,\rho_{0}}(\mathbb{R}_{\geq0};L(H))$
satisfying \prettyref{cond:kernel}\prettyref{item: selfadjoint}
such that 
\[
\langle k(t)x|x\rangle_{H}\geq0
\]
and 
\[
\langle\left(k(t)-k(s)\right)x|x\rangle_{H}\leq0,
\]
for each $x\in H$ and almost every $t,s\in\mathbb{R}$ with $s\leq t.$
First we note that this implies 
\[
\langle\left(\e^{-\rho t}k(t)-\e^{-\rho s}k(s)\right)x|x\rangle_{H}=\e^{-\rho t}\langle\left(k(t)-k(s)\right)x|x\rangle_{H}+(\e^{-\rho t}-\e^{-\rho s})\langle k(s)x|x\rangle_{H}\leq0
\]
for each $\rho\geq0,x\in H$ and almost every $t,s\in\mathbb{R}$
with $s\leq t.$ Hence, for $\rho>\max\{0,\rho_{0}\},t\in\mathbb{R}_{>0}$
and $x\in H$ we obtain 
\begin{align*}
 & \Im\langle\hat{k}(\i t+\rho)x|x\rangle_{H}\\
 & =\Im\frac{1}{\sqrt{2\pi}}\intop_{0}^{\infty}\e^{\i ts}\langle\e^{-\rho s}k(s)x|x\rangle_{H}\mbox{ d}s\\
 & =\frac{1}{\sqrt{2\pi}}\intop_{0}^{\infty}\sin(ts)\langle\e^{-\rho s}k(s)x|x\rangle_{H}\mbox{ d}s\\
 & =\frac{1}{\sqrt{2\pi}}\sum_{k=0}^{\infty}\left(\intop_{2k\frac{\pi}{t}}^{(2k+1)\frac{\pi}{t}}\sin(ts)\langle\e^{-\rho s}k(s)x|x\rangle_{H}\mbox{ d}s+\intop_{(2k+1)\frac{\pi}{t}}^{2(k+1)\frac{\pi}{t}}\sin(ts)\langle\e^{-\rho s}k(s)x|x\rangle_{H}\mbox{ d}s\right)\\
 & =\frac{1}{\sqrt{2\pi}}\sum_{k=0}^{\infty}\left(\intop_{2k\frac{\pi}{t}}^{(2k+1)\frac{\pi}{t}}\sin(ts)\langle\e^{-\rho s}k(s)x|x\rangle_{H}\mbox{ d}s+\right.\\
 & \phantom{aaaaaaaaaaaa}+\left.\intop_{2k\frac{\pi}{t}}^{(2k+1)\frac{\pi}{t}}\sin\left(t\left(s+\frac{\pi}{t}\right)\right)\langle\e^{-\rho\left(s+\frac{\pi}{t}\right)}k(s+\frac{\pi}{t})x|x\rangle_{H}\mbox{ d}s\right)\\
 & =\frac{1}{\sqrt{2\pi}}\sum_{k=0}^{\infty}\intop_{2k\frac{\pi}{t}}^{(2k+1)\frac{\pi}{t}}\sin(ts)\langle\left(\e^{-\rho s}k(s)-\e^{-\rho(s+\frac{\pi}{t})}k(s+\frac{\pi}{t})\right)x|x\rangle_{H}\mbox{ d}s\geq0,
\end{align*}
which yields 
\[
t\Im\langle\hat{k}(\i t+\rho)x|x\rangle_{H}\geq0
\]
for $t>0.$ For $t=0$ the inequality holds trivially and for $t<0$
we use the fact that $\hat{k}(\i t+\rho)=(\hat{k}(-\i t+\rho))^{\ast}$
by \prettyref{cond:kernel}\prettyref{item: selfadjoint} and hence,
\[
t\Im\langle\hat{k}(\i t+\rho)x|x\rangle_{H}=t\Im\langle x|\hat{k}(-\i t+\rho)x\rangle_{H}=-t\Im\langle\hat{k}(-\i t+\rho)x|x\rangle_{H}\geq0.
\]

\end{enumerate}
\end{example}

\section{Notes}

The main idea of evolutionary problems is to look at partial differential
equations as operator equations in time and space. So, the crucial
point is to realize the temporal derivative as an operator itself.
Thus, we are actually dealing with the sum of two unbounded operators
$B=\partial_{0,\rho}M(\partial_{0,\rho})$ and $A$. However, since
we restrict ourselves to the Hilbert space setting and the operator
$B$ has this special form, so that Laplace transform techniques are
employable, the conditions on the closability of $B+A$ and its continuous
invertibility are rather easy to verify. If we are leaving the Hilbert
space context or deal with more general operators, the situation gets
much more involved and we refer to the famous paper \cite{daPrato1975}
for that topic.

Moreover, we remark that the well-posedness of an evolutionary problem,
as it is defined above, does not imply the existence of a $C_{0}$-semigroup,
even in the case of Cauchy problems. Indeed, if $A:D(A)\subseteq H\to H$
is a closed densely defined linear operator and $M=1,$ then the associated
evolutionary problem is well-posed in the sense above, if there is
$\rho_{1}\in\mathbb{R}$ such that $\mathbb{C}_{\Re>\rho_{1}}\subseteq\rho(-A)$
and 
\[
\mathbb{C}_{\Re>\rho_{1}}\ni z\mapsto(z+A)^{-1}\in L(H)
\]
is bounded. Hence, if $-A$ generates a $C_{0}$-semigroup, then the
associated evolutionary problem is well-posed by the Hille-Yosida
theorem (cf. \cite[Theorem 3.8]{engel2000one}). However, the converse
is false in general. Indeed, if we choose $H\coloneqq L_{2}(\mathbb{R}_{>0})\oplus L_{2}(\mathbb{R}_{>0})$
and set 
\[
A\coloneqq\left(\begin{array}{cc}
-\i\m & -\i\m\\
0 & -\i\m
\end{array}\right)
\]
on $H$ with maximal domain, then $\mathbb{C}_{\Re>0}\subseteq\rho(-A)$
with 
\[
(z+A)^{-1}=\left(\begin{array}{cc}
(z-\i\m)^{-1} & \i\m(z-\i\m)^{-2}\\
0 & (z-\i\m)^{-1}
\end{array}\right)\quad(z\in\mathbb{C}_{\Re>0})
\]
and thus, the associated evolutionary problem is well-posed. Moreover,
\[
(z+A)^{-k}=\left(\begin{array}{cc}
(z-\i\m)^{-k} & \i\m k(z-\i\m)^{-(k+1)}\\
0 & (z-\i\m)^{-k}
\end{array}\right)\quad(z\in\mathbb{C}_{\Re>0})
\]
for each $k\in\mathbb{N}$ and hence, 
\[
k\|\m(z-\i\m)^{-(k+1)}\|\leq\|(z+A)^{-k}\|\quad(z\in\mathbb{C}_{\Re>0},k\in\mathbb{N}).
\]
If $-A$ would generate a $C_{0}$-semigroup, there would exist $M\geq1$
and $\omega\geq0$ such that 
\[
\|(\lambda+A)^{-k}\|\leq\frac{M}{(\lambda-\omega)^{k}}\quad(\lambda>\omega,k\in\mathbb{N})
\]
and consequently 
\[
(\lambda-\omega)^{k}k\|\m(\lambda-\i\m)^{-(k+1)}\|\leq M\quad(\lambda>\omega,k\in\mathbb{N}).
\]
Estimating the norm of the multiplication operator from below by 
\[
\frac{\lambda}{\sqrt{k}}\lambda^{-(k+1)}|1-\i\frac{1}{\sqrt{k}}|^{-(k+1)}
\]
we infer that
\[
\left(\frac{\lambda-\omega}{\lambda}\right)^{k}\sqrt{k}\frac{1}{\sqrt{(1+\frac{1}{k})^{k+1}}}\leq M
\]
for each $\lambda>\omega$ and $k\in\mathbb{N}.$ Letting $\lambda\to\infty,$
we get $\sqrt{k}\frac{1}{\sqrt{(1+\frac{1}{k})^{k+1}}}\leq M$ for
each $k\in\mathbb{N},$ which yields a contradiction as the left-hand
side tends to infinity as $k\to\infty.$ This shows, that $-A$ is
not the generator of a $C_{0}$-semigroup. Hence, in the framework
of evolutionary problems, the property of being the generator of a
$C_{0}$-semigroup is more a regularity property than a well-posedness
condition. We will study $C_{0}$-semigroups associated with evolutionary
problems in \prettyref{chap:Initial-conditions-for}.

Besides the examples treated in \prettyref{sec:Examples_1}, the framework
of evolutionary problems was used in the study for a broad class of
partial differential equations, in particular for coupled systems.
We refer to \cite{Mukhopadhyay2014_thermoelast,Mukhopadhyay2015_2temp}
for systems occurring in thermo-elasticity, to \cite{Mulholland2016}
for thermo-piezo-electricity, to \cite{Picard2010_poroelastic} for
poro-elastic deformations and to \cite{Picard2015_micropoloar} for
so-called micro-polar elasticity models. Moreover, we refer to \cite{Kalauch2011,Picard2012_delay_BS}
for an approach to delay equations, to \cite{Picard2013_fractional}
for fractional differential equations and to \cite{Picard2012_conservative,Picard2012_comprehensive_control,Picard2012_boundary_control,Trostorff2015_syst_theo}
for an approach to control systems. Moreover, more complicated boundary
conditions could be treated within the framework of evolutionary equations,
see \cite{Picard2012_Impedance,Trostorff2013_bd_maxmon,Picard2016_graddiv}.
Finally, we remark that the general structure of the material law
$M(\partial_{0,\rho})$ allows for the treatment of partial differential
equations of mixed type. These are equations which are elliptic in
one part of the underlying domain, parabolic in another one and hyperbolic
in a third one. In particular, one does not need to impose transmission
conditions, as these are automatically satisfied by solutions. There
exist several approaches for such mixed type problems beginning with
the early work of Friedrichs \cite{Friedrichs1958}, which introduces
a framework which is nowadays usually referred to as a Friedrichs
system. Other approaches, even in a Banach space setting, were proposed
by da Prato and Grisvard \cite{daPrato1975}, Colli and Favini \cite{Colli1995,Colli1996}
for parabolic-hyperbolic problems and by Favini and Yagi \cite{Favini1999}.
We also refer to However, all these approaches require certain constraints
on the operators involved, like a Hille-Yosida type condition or certain
decay rates for their resolvents, which do not have to be satisfied
by our material laws in general. 

We note that for the examples treated in \prettyref{sec:Examples_1},
other approaches can be found in the literature. For instance, we
refer to \cite{Hale1971,Webb1976,Batkai_2005,Batkai2001} for semigroup
approaches to delay differential equations and to \cite{Gripenberg1990_Volterra,Pruss2009}
for different approaches to integral and integro-differential equations. 

The framework of evolutionary problems as it is introduced in the
previous sections is not restricted to the autonomous case. In fact,
in \cite{Picard2013_nonauto} a non-autonomous version of \prettyref{prop:classical_material_law}
was proved, where the operators $M_{0},M_{1}$ are replaced by operator-valued
functions $M_{0},M_{1}:\mathbb{R}\to L(H).$ Hence, the corresponding
evolutionary equations takes the form 
\begin{equation}
\left(\partial_{0,\rho}M_{0}(\m)+M_{1}(\m)+A\right)u=f,\label{eq:non-auto}
\end{equation}
where $\left(M_{i}(\m)u\right)(t)\coloneqq M_{i}(t)u(t)$ for $t\in\mathbb{R}$,
which covers a class of non-autonomous problems. Later on, this result
was generalized in \cite{Waurick2015_nonauto}, where abstract operator
equations of the form 
\[
(\partial_{0,\rho}\mathcal{M}+\mathcal{N}+\mathcal{A})u=f
\]
where considered. Here the operators involved act on $H_{\rho}(\mathbb{R};H)$
and do not need to commute with the translation-operator $\tau_{h}.$
A further generalization in a different direction is the study of
so-called evolutionary inclusions, that is, one replaces the evolutionary
equation by an inclusion of the form 
\[
(u,f)\in\left(\partial_{0,\rho}M(\partial_{0,\rho})+A\right),
\]
where $A$ is no longer an operator, but a so-called maximal monotone
relation $A\subseteq H\oplus H,$ which in particular does not need
to be linear. For the topic of maximal monotone relations and differential
inclusions we refer to the monographs \cite{Brezis,papageogiou,showalter_book}.
Inclusions of the above form were studied in \cite{Trostorff2012_NA,Trostorff2012_nonlin_bd}
and in \cite{Trostorff2013_nonautoincl}, where a nonlinear analogue
of the non-autonomous problem \prettyref{eq:non-auto} was studied.

\chapter{Exponential stability for evolutionary problems\label{chap:Exponential-stability-for}}

This chapter is devoted to the topic of exponential stability of evolutionary
problems. We start to introduce the notion of exponential stability
and prove a useful characterization result. In the second part of
this chapter, we focus on a certain class of second-order problems,
and discuss the exponential stability for this class, by using a suitable
reformulation as a first-order problem. We conclude this chapter by
studying several examples. The results of this chapter are based on
\cite{Trostorff2013_stability,Trostorff2014_PAMM,Trostorff2015_secondorder}.

\section{Exponential stability}

There is a well-established definition of exponential stability in
the framework of strongly continuous semigroups. It is simply defined
as the property that for each initial value the corresponding trajectory,
which is a continuous function, should decay with a certain exponential
rate. Using the variation of constant formula, this yields that exponentially
decaying right-hand sides yield exponentially decaying solutions.
However, in the framework of evolutionary equations, as it was introduced
in the previous chapter, we cannot expect to have continuous solutions.
Indeed, as evolutionary equations also cover elliptic-type problems,
the solution can not be more regular in time than the given right-hand
side. Thus, we need to introduce an adapted notion of exponential
stability. Throughout, let $H$ be a Hilbert space, $A:D(A)\subseteq H\to H$
a densely defined closed linear operator and $M:D(M)\subseteq\mathbb{C}\to L(H)$
a linear material law.
\begin{defn*}
Assume hat the evolutionary problem associated with $M$ and $A$
is well-posed. We call the problem \emph{exponentially stable with
decay rate $\nu_{0}>0$, }if for all $\rho>s_{0}(M,A)$ with $\rho\in S_{M}$
and $f\in H_{\rho}(\mathbb{R};H)\cap H_{-\nu}(\mathbb{R};H)$ for
some $0\leq\nu<\nu_{0},$ it follows that 
\[
\left(\overline{\partial_{0,\rho}M(\partial_{0,\rho})+A}\right)^{-1}f\in H_{-\nu}(\mathbb{R};H).
\]

\end{defn*}
The definition of exponential stability states that the causal solution
operator of an evolutionary problem leaves the spaces $H_{-\nu}(\mathbb{R};H)$
invariant. Thus, the exponential decay of a function is replaced by
the condition that the function should belong to an exponentially
weighted $L_{2}$-space with a positive weighting factor. We emphasize,
that it is not enough to assume that the evolutionary equation is
well-posed on the space $H_{-\nu}(\mathbb{R};H)$ in the sense that
there exists a unique solution, which depends continuously on the
given right-hand side, since then the causality of the solution operator
may not hold as the next simple example will show. 
\begin{example}
We choose $H=\mathbb{C},$ $A=0$ and $M(z)=1$. Then the corresponding
evolutionary problem reads as 
\[
\partial_{0,\rho}u=f.
\]
This problem is solvable in $H_{\rho}(\mathbb{R};H)$ for each $\rho\ne0.$
However, the solution operator $\partial_{0,\rho}^{-1}$ is causal
if and only if $\rho>0$ (cp. \prettyref{prop:derivative_invertible}).
And indeed, this problem is not exponentially stable in the above
sense, since for $f=\chi_{[0,1]}\in\bigcap_{\mu\in\mathbb{R}}H_{\mu}(\mathbb{R};H)$
we have that for $\rho>s_{0}(M,A)=0$ 
\[
u(t)=\left(\partial_{0,\rho}^{-1}f\right)(t)=\begin{cases}
0 & \mbox{ if }t<0,\\
t & \mbox{ if }0\leq t\leq1,\\
1 & \mbox{ if }t>1
\end{cases}
\]
and thus, $u\notin H_{\mu}(\mathbb{R};H)$ for any $\mu\leq0.$ 
\end{example}
Although the solutions of evolutionary problems are not continuous
in general, one easy way to obtain continuity is to deal with more
regular right-hand sides. Moreover, it turns out that then the exponential
stability in the sense above really yields an exponential decay of
the solution, as the next lemma shows.
\begin{lem}
Let the evolutionary problem associated with $M$ and $A$ be well-posed.
Then, if $f\in H_{\rho}^{1}(\mathbb{R};H)$ for some $\rho>s_{0}(M,A),\rho\in S_{M}$,
we have that $u\coloneqq\left(\overline{\partial_{0,\rho}M(\partial_{0,\rho})+A}\right)^{-1}f\in H_{\rho}^{1}(\mathbb{R};H)$
with $\partial_{0,\rho}u=\left(\overline{\partial_{0,\rho}M(\partial_{0,\rho})+A}\right)^{-1}\partial_{0,\rho}f.$
Moreover, if the evolutionary problem is exponentially stable with
decay rate $\nu_{0}>0$ and $f\in H_{\rho}^{1}(\mathbb{R};H)\cap H_{-\nu}^{1}(\mathbb{R};H)$
for some $\rho>s_{0}(M,A),\rho\in S_{M}$ and $0\leq\nu<\nu_{0},$
then 
\[
|u(t)|_{H}\e^{\nu t}\to0\quad(t\to\infty).
\]
\end{lem}
\begin{proof}
Let $\rho>s_{0}(M,A)$ and $f\in H_{\rho}^{1}(\mathbb{R};H).$ By
$\left(\partial_{0,\rho}M(\partial_{0,\rho})+A\right)\partial_{0,\rho}\subseteq\partial_{0,\rho}\left(\partial_{0,\rho}M(\partial_{0,\rho})+A\right)$
it follows that $\partial_{0,\rho}^{-1}\left(\overline{\partial_{0,\rho}M(\partial_{0,\rho})+A}\right)^{-1}=\left(\overline{\partial_{0,\rho}M(\partial_{0,\rho})+A}\right)^{-1}\partial_{0,\rho}^{-1}.$
Thus, 
\[
u=\left(\overline{\partial_{0,\rho}M(\partial_{0,\rho})+A}\right)^{-1}f=\partial_{0,\rho}^{-1}\left(\overline{\partial_{0,\rho}M(\partial_{0,\rho})+A}\right)^{-1}\partial_{0,\rho}f\in H_{\rho}^{1}(\mathbb{R};H),
\]
which shows the first assertion. Assume now that the evolutionary
problem is exponentially stable with decay rate $\nu_{0}>0$. Then
for $f\in H_{\rho}^{1}(\mathbb{R};H)\cap H_{-\nu}^{1}(\mathbb{R};H),$
we have that 
\[
u\in H_{-\nu}(\mathbb{R};H)\cap H_{\rho}^{1}(\mathbb{R};H).
\]
Moreover, 
\[
\partial_{0,\rho}u=\left(\overline{\partial_{0,\rho}M(\partial_{0,\rho})+A}\right)^{-1}\partial_{0,\rho}f\in H_{-\nu}(\mathbb{R};H).
\]
The two conditions imply $u\in H_{-\nu}^{1}(\mathbb{R};H).$ Indeed,
for $\varphi\in C_{c}^{\infty}(\mathbb{R};H)$ we compute 
\begin{align*}
\langle u|\partial_{0,-\nu,c}\varphi\rangle_{-\nu} & =\intop_{\mathbb{R}}\langle u(t)|\varphi'(t)\rangle_{H}\e^{2\nu t}\mbox{ d}t\\
 & =\intop_{\mathbb{R}}\langle u(t)|\varphi'(t)\e^{2(\nu+\rho)t}\rangle_{H}\e^{-2\rho t}\mbox{ d}t\\
 & =\left\langle u\left|\partial_{0,\rho,c}\left(\varphi\e^{2(\nu+\rho)\cdot}\right)-2(\nu+\rho)\varphi\e^{2(\nu+\rho)\cdot}\right.\right\rangle _{\rho}\\
 & =\langle\partial_{0,\rho}^{\ast}u-2(\nu+\rho)u|\varphi\e^{2(\nu+\rho)\cdot}\rangle_{\rho}\\
 & =\langle-\partial_{0,\rho}u-2\nu u|\varphi\rangle_{-\nu}
\end{align*}
which yields $u\in H_{-\nu}^{1}(\mathbb{R};H)$ with $\partial_{0,-\nu}u=\partial_{0,\rho}u.$
The assertion now follows from \prettyref{prop:Sobolev}.
\end{proof}
We now come to the main result of this section, the characterization
of exponential stability for a class of well-posed evolutionary problems.
\begin{thm}
\label{thm:char_exp_stab}Let the evolutionary problem associated
with $M$ and $A$ be well-posed. Moreover, let $\nu_{0}>0$ be such
that $]-\nu_{0},0]\subseteq S_{M}$ and assume that $\mathbb{C}_{\Re>-\nu}\cap D(M)$
is connected for each $\nu_{0}-\varepsilon<\nu<\nu_{0}$ and each
$\varepsilon>0$. Then the evolutionary problem is exponentially stable
with decay rate $\nu_{0}$ if and only if $s_{0}(M,A)\leq-\nu_{0}.$\end{thm}
\begin{proof}
We first show that if $s_{0}(M,A)\leq-\nu_{0}$, then the evolutionary
problem is exponentially stable with decay rate $\nu_{0}.$ So let
$0\leq\nu<\nu_{0},\rho>s_{0}(M,A)$ with $\rho\in S_{M}$ and $f\in H_{\rho}(\mathbb{R};H)\cap H_{-\nu}(\mathbb{R};H).$
Since $-\nu\in S_{M}$ and $-\nu>s_{0}(M,A)$ by assumption, \prettyref{cor:sol_po_causal}
yields 
\[
\left(\overline{\partial_{0,\rho}M(\partial_{0,\rho})+A}\right)^{-1}f=\left(\overline{\partial_{0,-\nu}M(\partial_{0,-\nu})+A}\right)^{-1}f\in H_{-\nu}(\mathbb{R};H),
\]
which shows the assertion.\\
Assume now that the evolutionary problem is exponentially stable with
decay rate $\nu_{0}.$ Let $\varepsilon>0$. We show that $s_{0}(M,A)\leq-\nu$
for each $\nu_{0}-\varepsilon<\nu<\nu_{0}$, which would yield the
assertion. So, let $\nu_{0}-\varepsilon<\nu<\nu_{0}$ and choose $\rho>\max\{s_{0}(M,A),0\}$
such that $\mathbb{C}_{\Re\geq\rho}\subseteq D(M).$ We consider the
operator 
\[
S:L_{2}(\mathbb{R}_{\geq0};H)\to L_{2}(\mathbb{R}_{\geq0};H)
\]
given by $S\coloneqq\e^{\nu\m}\left(\overline{\partial_{0,\rho}M(\partial_{0,\rho})+A}\right)^{-1}\left(\e^{\nu\m}\right)^{-1}.$
Indeed, this operator is well-defined since for $f\in L_{2}(\mathbb{R}_{\geq0};H)$
we have that $\left(\e^{\nu\m}\right)^{-1}f\in H_{-\nu}(\mathbb{R}_{\geq0};H)\subseteq H_{\rho}(\mathbb{R}_{\geq0};H)$
and hence, $Sf\in L_{2}(\mathbb{R}_{\geq0};H)$ by exponential stability
and causality of $\left(\overline{\partial_{0,\rho}M(\partial_{0,\rho})+A}\right)^{-1}$.
Now we show that $S$ is closed. For doing so, let $(f_{n})_{n\in\mathbb{N}}$
in $L_{2}(\mathbb{R}_{\geq0};H)$ such that $f_{n}\to f$ and $Sf_{n}\to g$
in $L_{2}(\mathbb{R}_{\geq0};H)$ for some $f,g\in L_{2}(\mathbb{R}_{\geq0};H).$
We derive that $\left(\e^{\nu\m}\right)^{-1}f_{n}\to\left(\e^{\nu\m}\right)^{-1}f$
in $H_{-\nu}(\mathbb{R}_{\geq0};H)$ and consequently in $H_{\rho}(\mathbb{R}_{\geq0};H).$
By continuity of $\left(\overline{\partial_{0,\rho}M(\partial_{0,\rho})+A}\right)^{-1}$
and $\e^{\nu\m}$ we infer $Sf_{n}\to Sf$ in $H_{\rho+\nu}(\mathbb{R}_{\geq0};H).$
However, since $L_{2}(\mathbb{R}_{\geq0};H)\hookrightarrow H_{\rho+\nu}(\mathbb{R}_{\geq0};H),$
it follows that $g=Sf$ and hence, $S$ is closed. Thus, by the closed
graph theorem, $S\in L(L_{2}(\mathbb{R}_{\geq0};H)).$ \\
Now, we consider the following operator 
\[
T:C_{c}^{\infty}(\mathbb{R};H)\subseteq L_{2}(\mathbb{R};H)\to L_{2}(\mathbb{R};H)
\]
again defined by $T\coloneqq\e^{\nu\m}\left(\overline{\partial_{0,\rho}M(\partial_{0,\rho})+A}\right)^{-1}\left(\e^{\nu\m}\right)^{-1}.$
The operator is well-defined, since for $\varphi\in C_{c}^{\infty}(\mathbb{R};H)$
and $h\coloneqq\inf\spt\varphi$ we get $\tau_{-h}\varphi\in L_{2}(\mathbb{R}_{\geq0};H)$
and $T\varphi=\tau_{h}S\tau_{-h}\varphi$ by the translation invariance
of $\e^{\nu\m}\left(\overline{\partial_{0,\rho}M(\partial_{0,\rho})+A}\right)^{-1}\left(\e^{\nu\m}\right)^{-1}$.
The latter yields that $T$ is bounded, since $S$ is bounded. Moreover,
$T$ is causal and hence, there is $N:\mathbb{C}_{\Re>0}\to L(H)$
bounded and analytic such that
\[
\widehat{Tf}(z)=N(z)\hat{f}(z)\quad(z\in\mathbb{C}_{\Re>0})
\]
for each $f\in L_{2}(\mathbb{R}_{\geq0};H)$ by \prettyref{thm:weiss}.
Next, we show that 
\begin{equation}
N(z+\nu)=(zM(z)+A)^{-1}\quad(z\in\mathbb{C}_{\Re>\rho}).\label{eq:resolvent_extension}
\end{equation}
For doing so, let $x\in H$ and set $f(t)\coloneqq\sqrt{2\pi}\chi_{\mathbb{R}_{\geq0}}(t)\e^{-t}x.$
Then $\hat{f}(z)=\frac{1}{z+1}x$ for $z\in\mathbb{C}_{\Re>0}.$ Thus,
we have for $t\in\mathbb{R},\mu>\rho+\nu$ 
\begin{align*}
N(\i t+\mu)x & =(\i t+\mu+1)N(\i t+\mu)\hat{f}(\i t+\mu)\\
 & =(\i t+\mu+1)\widehat{Tf}(\i t+\mu)\\
 & =(\i t+\mu+1)\left(\mathcal{L}_{\mu}\e^{\nu\m}\left(\overline{\partial_{0,\rho}M(\partial_{0,\rho})+A}\right)^{-1}\left(\e^{\nu\m}\right)^{-1}f\right)(t)\\
 & =(\i t+\mu+1)\left(\mathcal{L}_{\mu-\nu}\left(\overline{\partial_{0,\mu-\nu}M(\partial_{0,\mu-\nu})+A}\right)^{-1}\left(\e^{\nu\m}\right)^{-1}f\right)(t),
\end{align*}
 where we have used $\left(\e^{\nu\m}\right)^{-1}f\in H_{\mu-\nu}(\mathbb{R};H)$
and \prettyref{cor:sol_po_causal}. The latter gives 
\begin{align*}
N(\i t+\mu)x & =(\i t+\mu+1)\left((\i t+\mu-\nu)M(\i t+\mu-\nu)+A\right)^{-1}\left(\mathcal{L}_{\mu-\nu}\left(\e^{\nu\m}\right)^{-1}f\right)(t)\\
 & =\left((\i t+\mu-\nu)M(\i t+\mu-\nu)+A\right)^{-1}x,
\end{align*}
which shows \prettyref{eq:resolvent_extension}. So far, we have shown
that $\mathbb{C}_{\Re>\rho}\ni z\mapsto(zM(z)+A)^{-1}$ has an analytic
and bounded extension on $\mathbb{C}_{\Re>-\nu}$ given by $N(\cdot+\nu)$.
Thus, it is left to show that $zM(z)+A$ is boundedly invertible for
each $z\in D(M)\cap\mathbb{C}_{\Re>-\nu}$ . Note that then $(zM(z)+A)^{-1}=N(z+\nu)$
for each $z\in D(M)\cap\mathbb{C}_{\Re>-\nu}$ by the identity theorem.
We define the set 
\[
\Omega\coloneqq\left\{ z\in D(M)\cap\mathbb{C}_{\Re>-\nu}\,;\, zM(z)+A\mbox{ is boundedly invertible}\right\} \subseteq D(M)\cap\mathbb{C}_{\Re>-\nu}.
\]
We show that $\Omega$ is open. Let $z\in\Omega$ and choose $\delta>0$
such that $B(z,\delta)\subseteq D(M)\cap\mathbb{C}_{\Re>-\nu}$ \nomenclature[A_030]{$B(x,r)$}{the open ball with center $x$ and radius $r$.}\nomenclature[A_040]{$B[x,r]$}{the closed ball with center $x$ and radius $r$.}and
\[
\left\Vert (z'M(z')-zM(z))(zM(z)+A)^{-1}\right\Vert <1
\]
for each $z'\in B(z,\delta).$ Then, for $z'\in B(z,\delta)$ it follows
that 
\begin{align}
z'M(z')+A & =z'M(z')-zM(z)+zM(z)+A\nonumber \\
 & =\left(\left(z'M(z')-zM(z)\right)\left(zM(z)+A\right)^{-1}+1\right)(zM(z)+A)\label{eq:omega_open}
\end{align}
is boundedly invertible by the Neumann series. So, $\Omega$ is open.
Consider now the component $C$ in $\Omega$ containing the point
$\rho+1.$ Since $\Omega$ is open, so is $C$. Moreover, by \prettyref{eq:resolvent_extension}
and the identitiy theorem we have that $N(z+\nu)=\left(zM(z)+A\right)^{-1}$
for each $z\in C.$ Let now $(z_{n})_{n\in\mathbb{N}}$ be a sequence
in $C$ with $z_{n}\to z\in D(M)\cap\mathbb{C}_{\Re>-\nu}.$ Since
\[
\sup_{n\in\mathbb{N}}\|(z_{n}M(z_{n})+A)^{-1}\|=\sup_{n\in\mathbb{N}}\|N(z_{n}+\nu)\|<\infty,
\]
we infer, using that for each $n\in\mathbb{N}$ 
\[
zM(z)+A=\left(\left(zM(z)-z_{n}M(z_{n})\right)\left(z_{n}M(z_{n})+A\right)^{-1}+1\right)(z_{n}M(z_{n})+A),
\]
that $z\in\Omega$. Since $\Omega$ is open, we also get $z\in C$,
showing that $C$ is closed. Hence, $C=D(M)\cap\mathbb{C}_{\Re>-\nu}$
and consequently $\Omega=D(M)\cap\mathbb{C}_{\Re>-\nu},$ which completes
the proof.\end{proof}
\begin{rem}
The latter theorem especially applies to the case of evolution equations,
i.e. equations of the form 
\[
\left(\overline{\partial_{0,\rho}+A}\right)u=f,
\]
where $-A$ is the generator of a strongly-continuous semigroup. Indeed,
here $M(z)=1$ for each $z\in\mathbb{C}$, and so the assumptions
in \prettyref{thm:char_exp_stab} on the material law $M$ are trivially
satisfied. Hence, exponential stability in the sense defined in this
section is equivalent to $s_{0}(1,A)=s_{0}(A)<0$. Using now the Theorem
of Gearhart-Prüß (see \cite{Pruss1984} and \prettyref{thm:G-P} in
this thesis), which states that the growth bound of the associated
semigroup equals $s_{0}(A),$ we derive the exponential stability
of the evolution equation in the classical sense. Hence, our notion
of exponential stability coincides with the classical ones for evolution
equations. We will address the relation between semigroups and general
evolutionary problems in the next chapter. 
\end{rem}
We conclude this section by discussing conditions on the linear material
law, which yield exponential stability of the associated evolutionary
problem.
\begin{prop}
\label{prop:exp_stab_pos_def}Let $A:D(A)\subseteq H\to H$ be $m$-accretive
and $M:D(M)\subseteq\mathbb{C}\to L(H)$ be a linear material law.
Let $\nu_{0}>0$ such that $\mathbb{C}_{\Re>-\nu_{0}}\setminus D(M)$
is discrete and 
\[
\exists c>0\:\forall z\in D(M)\cap\mathbb{C}_{\Re>-\nu_{0}},x\in H:\Re\langle zM(z)x|x\rangle_{H}\geq c|x|_{H}^{2}.
\]
Then the evolutionary problem associated with $M$ and $A$ is well-posed
and exponentially stable with decay rate $\nu_{0}.$ \end{prop}
\begin{proof}
The well-posedness of the evolutionary problem follows from \prettyref{prop:Rainer}.
Moreover, the material law $M$ satisfies the assumptions in \prettyref{thm:char_exp_stab},
since $\mathbb{C}_{\Re>-\nu_{0}}\setminus D(M)$ is at most countable
and so $]-\nu_{0},0]\subseteq S_{M}$ as well as $D(M)\cap\mathbb{C}_{\Re>-\nu}=\mathbb{C}_{\Re>-\nu}\setminus\left(\mathbb{C}_{\Re>-\nu_{0}}\setminus D(M)\right)$
is connected for each $\nu<\nu_{0}$. Moreover, for each $z\in D(M)\cap\mathbb{C}_{\Re>-\nu_{0}}$
the operator $zM(z)+A$ is boundedly invertible with $\left\Vert (zM(z)+A)^{-1}\right\Vert \leq\frac{1}{c}$
by \prettyref{prop:prop-accretive} and \prettyref{prop:pert_accretive}.
Thus, the mapping 
\[
D(M)\cap\mathbb{C}_{\Re>-\nu_{0}}\ni z\mapsto(zM(z)+A)^{-1}\in L(H)
\]
is bounded and analytic. Since $\mathbb{C}_{\Re>-\nu_{0}}\setminus D(M)$
is discrete, we find a bounded and analytic extension to the whole
$\mathbb{C}_{\Re>-\nu_{0}}$ by Riemann's Theorem on removable singularities.
Thus, $s_{0}(M,A)\leq-\nu_{0}$ and hence, the problem is exponentially
stable with decay rate $\nu_{0}$ by \prettyref{thm:char_exp_stab}. \end{proof}
\begin{prop}
\label{prop:exp_stab_A_inv}Let $A:D(A)\subseteq H\to H$ be $m$-accretive,
$M:D(M)\subseteq\mathbb{C}\to L(H)$ be a linear-material law and
assume that $A$ is boundedly invertible. Let $\nu_{0}>0$ such that
$\mathbb{C}_{\Re>-\nu_{0}}\setminus D(M)$ is discrete and $\delta>0$
such that 
\[
K\coloneqq\sup_{z\in D(M)\cap B[0,\delta]}\|zM(z)\|<\|A^{-1}\|^{-1}
\]
and 
\[
\exists c>0\:\forall z\in\left(D(M)\cap\mathbb{C}_{\Re>-\nu_{0}}\right)\setminus B[0,\delta],x\in H:\Re\langle zM(z)x|x\rangle_{H}\geq c|x|_{H}^{2}.
\]
Then, the evolutionary problem associated with $M$ and $A$ is well-posed
and exponentially stable with stability rate $\nu_{0}.$ \end{prop}
\begin{proof}
Again, the well-posedness follows from \prettyref{prop:Rainer}. Moreover,
for $z\in D(M)\cap B[0,\delta]$ the operator 
\[
zM(z)+A=\left(zM(z)A^{-1}+1\right)A
\]
is boundedly invertible with 
\[
\|\left(zM(z)+A\right)^{-1}\|\leq\frac{\|A^{-1}\|}{1-K\|A^{-1}\|}
\]
by the Neumann series and for $z\in\left(D(M)\cap\mathbb{C}_{\Re>-\nu_{0}}\right)\setminus B[0,\delta]$
the invertibility follows from \prettyref{prop:prop-accretive} and
\prettyref{prop:pert_accretive} with 
\[
\|\left(zM(z)+A\right)^{-1}\|\leq\frac{1}{c}.
\]
Hence, the assertion follows again by Riemann's Theorem for removable
singularities and \prettyref{thm:char_exp_stab}.
\end{proof}

\section{Exponential stability for a class of second order evolutionary problems\label{sec:Exponential-stability-sec_order}}

In this section, we consider second-order differential equations of
the following form 
\begin{equation}
\left(\partial_{0,\rho}^{2}M(\partial_{0,\rho})+C^{\ast}C\right)u=f,\label{eq:second_order}
\end{equation}
where $C:D(C)\subseteq H_{0}\to H_{1}$ is a densely defined, closed
linear operator between two Hilbert spaces $H_{0},H_{1}$, which is
assumed to be boundedly invertible, and $M:D(M)\subseteq\mathbb{C}\to L(H)$
is a linear material law, which is given by 
\[
M(z)\coloneqq M_{0}(z)+z^{-1}M_{1}(z),
\]
where $M_{0},M_{1}:D(M)\subseteq\mathbb{C}\to L(H_{0})$ are analytic
and bounded. We want to reformulate this problem as a first-order
evolutionary problem, which allows us to study the exponential stability.
For doing so, we choose $d>0$ and define the new unknowns $v\coloneqq\partial_{0,\rho}u+du$
and $q\coloneqq-Cu.$ Then, we obtain 
\[
\partial_{0,\rho}q=-C\partial_{0,\rho}u=-Cv+dCu=-Cv-dq
\]
and 
\begin{align*}
\partial_{0,\rho}M(\partial_{0,\rho})v & =\partial_{0,\rho}^{2}M(\partial_{0,\rho})u+d\partial_{0,\rho}M(\partial_{0,\rho})u\\
 & =f+C^{\ast}q+dM_{0}(\partial_{0,\rho})\partial_{0,\rho}u+dM_{1}(\partial_{0,\rho})u\\
 & =f+C^{\ast}q+dM_{0}(\partial_{0,\rho})v+d\left(M_{1}(\partial_{0,\rho})-dM_{0}(\partial_{0,\rho})\right)u\\
 & =f+C^{\ast}q+dM_{0}(\partial_{0,\rho})v-d\left(M_{1}(\partial_{0,\rho})-dM_{0}(\partial_{0,\rho})\right)C^{-1}q,
\end{align*}
which can be written as a system of the form 
\begin{align}
\left(\partial_{0,\rho}\left(\begin{array}{cc}
M(\partial_{0,\rho}) & 0\\
0 & 1
\end{array}\right)+d\left(\begin{array}{cc}
-M_{0}(\partial_{0,\rho}) & \left(M_{1}(\partial_{0,\rho})-dM_{0}(\partial_{0,\rho})\right)C^{-1}\\
0 & 1
\end{array}\right)+\right.\nonumber \\
\left.+\left(\begin{array}{cc}
0 & -C^{\ast}\\
C & 0
\end{array}\right)\right)\left(\begin{array}{c}
v\\
q
\end{array}\right) & =\left(\begin{array}{c}
f\\
0
\end{array}\right)\label{eq:first-order}
\end{align}
Thus, we arrive at an evolutionary problem with a $d$-dependent material
law 
\begin{equation}
M_{d}(z)\coloneqq\left(\begin{array}{cc}
M(z) & 0\\
0 & 1
\end{array}\right)+z^{-1}d\left(\begin{array}{cc}
-M_{0}(z) & \left(M_{1}(z)-dM_{0}(z)\right)C^{-1}\\
0 & 1
\end{array}\right)\quad(z\in D(M)).\label{eq:M_d}
\end{equation}
Our main goal is to derive conditions on the material law $M$, which
yield the exponential stability of \prettyref{eq:first-order} for
some $d>0.$ 
\begin{rem}
We note that, if \prettyref{eq:first-order} is exponentially stable
and $v,q\in H_{-\nu}(\mathbb{R};H)$ for some $\nu>0,$ that also
$u\in H_{-\nu}^{1}(\mathbb{R};H)\hookrightarrow C_{-\nu}(\mathbb{R};H),$
yielding the exponential stability of the original second-order problem.
Indeed, since $u=-C^{-1}q$ we derive $u\in H_{-\nu}(\mathbb{R};H)$
and hence, $\partial_{0,\rho}u=v-du\in H_{-\nu}(\mathbb{R};H).$ 
\end{rem}
We first discuss, how condition \prettyref{eq:pos_def_M} for the
material law $M$ carries over to an analogous estimate for $M_{d}.$
\begin{lem}
\label{lem:estimate_M_d}Let $z\in D(M)$ and $c>0$ such that 
\[
\Re\langle zM(z)u|u\rangle_{H_{0}}\geq c|u|_{H_{0}}^{2}\quad(u\in H_{0}).
\]
Then for each $d>0$, it follows that 
\[
\Re\langle zM_{d}(z)(v,q)|(v,q)\rangle_{H_{0}\oplus H_{1}}\geq\min\left\{ c-dK(d),\frac{3}{4}d+\Re z\right\} |(v,q)|_{H_{0}\oplus H_{1}}^{2}\quad(v\in H_{0},q\in H_{1}),
\]
where $K(d)\coloneqq\|M_{0}\|_{\infty}+\left(d\|M_{0}\|_{\infty}+\|M_{1}\|_{\infty}\right)^{2}\|C^{-1}\|^{2}$
and $M_{d}$ is given by \prettyref{eq:M_d}.\end{lem}
\begin{proof}
Let $v\in H_{0},q\in H_{1}.$ Then we estimate 
\begin{align*}
 & \Re\langle zM_{d}(z)(v,q)|(v,q)\rangle_{H_{0}\oplus H_{1}}\\
 & =\Re\langle zM(z)v-dM_{0}(z)v+d\left(M_{1}(z)-dM_{0}(z)\right)C^{-1}q|v\rangle_{H_{0}}+\Re\langle zq+dq|q\rangle_{H_{1}}\\
 & \geq(c-d\|M_{0}(z)\|)|v|_{H_{0}}^{2}-d\left\Vert \left(M_{1}(z)-dM_{0}(z)\right)C^{-1}\right\Vert |q|_{H_{1}}|v|_{H_{0}}+\left(\Re z+d\right)|q|_{H_{1}}^{2}\\
 & \geq\left(c-d\|M_{0}\|_{\infty}-\frac{1}{4\varepsilon}d^{2}\left(d\|M_{0}\|_{\infty}+\|M_{1}\|_{\infty}\right)^{2}\|C^{-1}\|^{2}\right)|v|_{H_{0}}^{2}+(\Re z+d-\varepsilon)|q|_{H_{1}}^{2},
\end{align*}
for each $\varepsilon>0.$ If we choose $\varepsilon=\frac{d}{4},$
we obtain the assertion. 
\end{proof}
First, we want to check, whether there is some $d>0$ such that $M_{d}$
satisfies the assumptions of \prettyref{prop:exp_stab_pos_def}, if
$M$ does. We begin with the following proposition.
\begin{prop}
\label{prop:pos_def_M_d}Let $c>0$ such that for each $z\in D(M)$
and $u\in H_{0}$ we have that 
\[
\Re\langle zM(z)u|u\rangle_{H_{0}}\geq c|u|_{H_{0}}^{2}.
\]
Then, there exist $\tilde{c},d_{0}>0$ and $\rho_{0}>0$, such that
for each $z\in D(M)\cap\mathbb{C}_{\Re>-\rho_{0}},v\in H_{0}$ and
$q\in H_{1}$ 
\[
\Re\langle zM_{d_{0}}(z)(v,q)|(v,q)\rangle_{H_{0}\oplus H_{1}}\geq\tilde{c}|(v,q)|_{H_{0}\oplus H_{1}}^{2},
\]
where $M_{d_{0}}$ is given by \prettyref{eq:M_d}.\end{prop}
\begin{proof}
Let $v\in H_{0},q\in H_{1}$. By \prettyref{lem:estimate_M_d} we
have 
\[
\Re\langle zM_{d}(z)(v,q)|(v,q)\rangle_{H_{0}\oplus H_{1}}\geq\min\left\{ c-dK(d),\frac{3}{4}d+\Re z\right\} |(v,q)|_{H_{0}\oplus H_{1}}^{2}
\]
for each $z\in D(M)$ and $d>0,$ where $K(d)=\|M_{0}\|_{\infty}+\left(d\|M_{0}\|_{\infty}+\|M_{1}\|_{\infty}\right)^{2}\|C^{-1}\|^{2}.$
Since $dK(d)\to0$ as $d\to0$, we may choose $d_{0}>0$ such that
$d_{0}K(d_{0})<c.$ Moreover, we choose $\rho_{0}<\frac{3}{4}d_{0}$.
Then, the above estimate yields that 
\[
\Re\langle zM_{d_{0}}(z)(v,q)|(v,q)\rangle_{H_{0}\oplus H_{1}}\geq\min\left\{ c-d_{0}K(d_{0}),\frac{3}{4}d_{0}-\rho_{0}\right\} |(v,q)|_{H_{0}\oplus H_{1}}^{2}
\]
for each $z\in D(M)\cap\mathbb{C}_{\Re>-\rho_{0}},$ which proves
the assertion.\end{proof}
\begin{cor}
\label{cor:exp_stab_sec_posdef}Assume that there exist $\nu_{0},c>0$
such that $\mathbb{C}_{\Re>-\nu_{0}}\setminus D(M)$ is discrete and
for each $z\in D(M)\cap\mathbb{C}_{\Re>-\nu_{0}}$ and $u\in H_{0}$
we have that 
\[
\Re\langle zM(z)u|u\rangle_{H_{0}}\geq c|u|_{H_{0}}^{2}.
\]
Then there exists $d>0$ such that the evolutionary problem given
by \prettyref{eq:first-order} is well-posed and exponentially stable.\end{cor}
\begin{proof}
By \prettyref{prop:pos_def_M_d} there exist $d,\tilde{c},\rho_{0}>0$
such that 
\[
\Re\langle zM_{d}(z)(v,q)|(v,q)\rangle_{H_{0}\oplus H_{1}}\geq\tilde{c}|(v,q)|_{H_{0}\oplus H_{1}}^{2}
\]
for each $v\in H_{0},q\in H_{1}$ and $z\in D(M)\cap\mathbb{C}_{\Re>-\min\{\nu_{0},\rho_{0}\}}.$
Then the assertion follows from \prettyref{prop:exp_stab_pos_def}.
\end{proof}
We now come to our second exponential stability result \prettyref{prop:exp_stab_A_inv}
and again we want to find out, how the assumptions on $M$ carry over
to $M_{d}$ for some $d>0$. 
\begin{prop}
\label{prop:exp_stable_without_ball}We assume that
\begin{equation}
\forall\delta>0\,\exists\rho_{0},c>0\,\forall z\in D(M)\cap\mathbb{C}_{\Re>-\rho_{0}}\setminus B[0,\delta],u\in H_{0}:\Re\langle zM(z)u|u\rangle_{H_{0}}\geq c|u|_{H_{0}}^{2},\label{eq:pos_def_M_withoutball}
\end{equation}
and $\lim_{z\to0}M_{1}(z)=0$. Then, there exist $\tilde{c},d_{0},\delta_{0}>0$
and $\tilde{\rho_{0}}>0$, such that 
\[
\sup_{z\in D(M)\cap B[0,\delta_{0}]}\|zM_{d_{0}}(z)\|<\|A^{-1}\|^{-1},
\]
and 
\[
\forall z\in D(M)\cap\mathbb{C}_{\Re>-\tilde{\rho_{0}}}\setminus B[0,\delta_{0}],\, v\in H_{0},q\in H_{1}:\Re\langle zM_{d_{0}}(z)(v,q)|(v,q)\rangle_{H_{0}\oplus H_{1}}\geq\tilde{c}|(v,q)|_{H_{0}\oplus H_{1}}^{2},
\]
where $A\coloneqq\left(\begin{array}{cc}
0 & -C^{\ast}\\
C & 0
\end{array}\right)$ and $M_{d_{0}}$ is given by \prettyref{eq:M_d}. If, additionally
$\mathbb{C}_{\Re>-\nu_{0}}\setminus D(M)$ is discrete for some $\nu_{0}>0$,
we get that the evolutionary problem given by \prettyref{eq:first-order}
is exponentially stable for $d=d_{0}.$\end{prop}
\begin{proof}
Let $d>0$ and $z\in D(M).$ Then we have 
\begin{align*}
\|zM_{d}(z)\| & \leq\max\left\{ \|zM(z)\|,|z|\right\} +G(d)\\
 & \leq\max\left\{ |z|\|M_{0}(z)\|+\|M_{1}(z)\|,|z|\right\} +G(d)
\end{align*}
where 
\[
G(d)\coloneqq\sup_{z\in D(M)}\left\Vert d\left(\begin{array}{cc}
-M_{0}(z) & \left(M_{1}(z)-dM_{0}(z)\right)C^{-1}\\
0 & 1
\end{array}\right)\right\Vert .
\]
Since $G(d)\to0$ as $d\to0$ and $M_{1}(z)\to0$ as $z\to0,$ we
may choose $d_{1}>0$ and $\delta_{0}>0$ small enough, such that
\[
\sup_{z\in D(M)\cap B[0,\delta_{0}]}\|zM_{d}(z)\|<\|A^{-1}\|^{-1}
\]
for each $0<d<d_{1}.$ Moreover, by \prettyref{lem:estimate_M_d}
we have that 
\[
\Re\langle zM_{d}(z)(v,q)|(v,q)\rangle_{H_{0}\oplus H_{1}}\geq\min\left\{ c-dK(d),\frac{3}{4}d+\Re z\right\} |(v,q)|_{H_{0}\oplus H_{1}}^{2},
\]
for each $z\in D(M)\cap\mathbb{C}_{\Re>-\rho_{0}}\setminus B[0,\delta],d>0$
and $v\in H_{0},q\in H_{1},$ where $c,\rho_{0}>0$ are chosen according
to \prettyref{eq:pos_def_M_withoutball} for $\delta=\delta_{0}$
and $K(d)=\|M_{0}\|_{\infty}+\left(d\|M_{0}\|_{\infty}+\|M_{1}\|_{\infty}\|C^{-1}\|\right)^{2}$.
Let now $0<d_{0}<d_{1}$ such that $d_{0}K(d_{0})<c$ and $0<\tilde{\rho_{0}}\leq\rho_{0}$
such that $\tilde{\rho_{0}}<\frac{3}{4}d_{0},$ the latter estimate
gives 
\[
\Re\langle zM_{d_{0}}(z)(v,q)|(v,q)\rangle_{H_{0}\oplus H_{1}}\geq\min\left\{ c-d_{0}K(d_{0}),\frac{3}{4}d-\tilde{\rho_{0}}\right\} |(v,q)|_{H_{0}\oplus H_{1}}^{2}
\]
for each $z\in D(M)\cap\mathbb{C}_{\Re>-\rho_{0}}\setminus B[0,\delta_{0}],d>0.$
The last assertion is a direct consequence of \prettyref{prop:exp_stab_A_inv}.
\end{proof}

\section{Examples\label{sec:Examples_exp_decay}}

In this section we illustrate our previous findings by applying them
to concrete examples, which were also studied in the literature but
using different methods. We start with a class of abstract parabolic-type
equations.

\subsection{Exponential stability for equations of parabolic type}

We begin by proving a criterion for exponential stability for a class
of parabolic-type problems.
\begin{prop}
\label{prop:exp_decay_para}Let $H_{0},H_{1}$ be Hilbert spaces and
$C:D(C)\subseteq H_{0}\to H_{1}$ densely defined closed linear and
boundedly invertible. Moreover, let $M_{0}\in L(H_{0})$ selfadjoint
and strictly positive definite and $M_{1}:D(M_{1})\subseteq\mathbb{C}\to L(H_{1})$
a bounded linear material law with $\mathbb{C}_{\Re>-\nu_{0}}\setminus D(M_{1})$
discrete for some $\nu_{0}>0$, such that 
\[
\exists c>0\,\forall z\in D(M_{1}),x\in H_{1}:\Re\langle M_{1}(z)x|x\rangle_{H_{1}}\geq c|x|_{H_{1}}^{2}.
\]
Then the evolutionary problem associated with 
\[
M(z)\coloneqq\left(\begin{array}{cc}
M_{0} & 0\\
0 & 0
\end{array}\right)+z^{-1}\left(\begin{array}{cc}
0 & 0\\
0 & M_{1}(z)
\end{array}\right)\quad(z\in D(M_{1})\setminus\{0\})
\]
and 
\[
A\coloneqq\left(\begin{array}{cc}
0 & -C^{\ast}\\
C & 0
\end{array}\right)
\]
is well-posed and exponentially stable with decay rate $\nu_{1}$,
where $\nu_{1}\coloneqq\min\{\nu_{0},\frac{c}{\|M_{1}\|_{\infty}^{2}\|M_{0}\|\|C^{-1}\|^{2}}\}$.\end{prop}
\begin{proof}
The well-posedness follows from \prettyref{prop:Rainer}. For showing
the exponential stability we apply \prettyref{thm:char_exp_stab}.
We note that $M$ satisfies the assumptions of \prettyref{thm:char_exp_stab},
since $\mathbb{C}_{\Re>-\nu_{0}}\setminus D(M)$ is discrete. Thus,
it suffices to prove $s_{0}(M,A)\leq-\nu_{1}$. For doing so, let
$z\in D(M)\cap\mathbb{C}_{\Re>-\rho}$ with $0<\rho<\nu_{1}$ . We
need to show that the operator 
\[
z\left(\begin{array}{cc}
M_{0} & 0\\
0 & 0
\end{array}\right)+\left(\begin{array}{cc}
0 & 0\\
0 & M_{1}(z)
\end{array}\right)+\left(\begin{array}{cc}
0 & -C^{\ast}\\
C & 0
\end{array}\right)
\]
is boundedly invertible and the norm of its inverse is bounded in
$z$. We start by showing that the operator is onto. So, let $f\in H_{0},g\in H_{1}.$
We consider the operator 
\[
z\left(C^{-1}\right)^{\ast}M_{0}C^{-1}+M_{1}(z)^{-1}\in L(H_{1}).
\]
This operator is continuously invertible with 
\[
\left\Vert \left(z\left(C^{-1}\right)^{\ast}M_{0}C^{-1}+M_{1}(z)^{-1}\right)^{-1}\right\Vert \leq\frac{1}{\frac{c}{\|M_{1}\|_{\infty}^{2}}-\rho\|M_{0}\|\|C^{-1}\|^{2}}\eqqcolon\mu.
\]
Indeed, we have 
\begin{align*}
\Re\langle z(C^{-1})^{\ast}M_{0}C^{-1}x+M_{1}(z)^{-1}x|x\rangle_{H_{1}} & =\Re\langle zM_{0}C^{-1}x|C^{-1}x\rangle_{H_{0}}+\Re\langle M_{1}(z)^{-1}x|x\rangle_{H_{1}}\\
 & \geq-\rho\|M_{0}\|\|C^{-1}\|^{2}|x|_{H_{1}}^{2}+\frac{c}{\|M_{1}(z)\|^{2}}|x|_{H_{1}}^{2}\\
 & \geq\left(\frac{c}{\|M_{1}\|_{\infty}^{2}}-\rho\|M_{0}\|\|C^{-1}\|^{2}\right)|x|_{H_{1}}^{2}
\end{align*}
for each $x\in H_{1},$ where we have used \prettyref{lem:accretive_inverse}.
We define 
\begin{align}
u\coloneqq & C^{-1}\left(z\left(C^{-1}\right)^{\ast}M_{0}C^{-1}+M_{1}(z)^{-1}\right)^{-1}\left(\left(C^{\ast}\right)^{-1}f+M_{1}(z)^{-1}g\right),\nonumber \\
v\coloneqq & M_{1}(z)^{-1}(g-Cu),\label{eq:u and v}
\end{align}
and obtain the estimates 
\begin{align}
|u|_{H_{0}} & \leq\mu\|C^{-1}\|\left(\|C^{-1}\||f|_{H_{0}}+\frac{1}{c}|g|_{H_{1}}\right)\nonumber \\
|v|_{H_{1}} & \leq\frac{1}{c}\left(|g|_{H_{1}}+\mu\left(\|C^{-1}\||f|_{H_{0}}+\frac{1}{c}|g|_{H_{1}}\right)\right).\label{eq:norm_estimate}
\end{align}
We show that 
\begin{equation}
\left(z\left(\begin{array}{cc}
M_{0} & 0\\
0 & 0
\end{array}\right)+\left(\begin{array}{cc}
0 & 0\\
0 & M_{1}(z)
\end{array}\right)+\left(\begin{array}{cc}
0 & -C^{\ast}\\
C & 0
\end{array}\right)\right)\left(\begin{array}{c}
u\\
v
\end{array}\right)=\left(\begin{array}{c}
f\\
g
\end{array}\right).\label{eq:parabolic}
\end{equation}
Indeed, we have $u\in D(C)$ by definition and $M_{1}(z)v+Cu=g.$
Moreover, 
\begin{align*}
g-Cu & =g-\left(z\left(C^{-1}\right)^{\ast}M_{0}C^{-1}+M_{1}(z)^{-1}\right)^{-1}\left(\left(C^{-1}\right)^{\ast}f+M_{1}(z)^{-1}g\right)\\
 & =\left(z\left(C^{-1}\right)^{\ast}M_{0}C^{-1}+M_{1}(z)^{-1}\right)^{-1}\left(z\left(C^{-1}\right)^{\ast}M_{0}C^{-1}g-\left(C^{-1}\right)^{\ast}f\right)
\end{align*}
and hence, 
\[
\left(z\left(C^{-1}\right)^{\ast}M_{0}C^{-1}+M_{1}(z)^{-1}\right)\left(g-Cu\right)=\left(C^{-1}\right)^{\ast}\left(zM_{0}C^{-1}g-f\right).
\]
Thus, we read off that 
\begin{align*}
v & =M_{1}(z)^{-1}(g-Cu)\\
 & =\left(C^{-1}\right)^{\ast}\left(zM_{0}C^{-1}g-f\right)-z\left(C^{-1}\right)^{\ast}M_{0}C^{-1}(g-Cu)\\
 & =\left(C^{-1}\right)^{\ast}(zM_{0}u-f),
\end{align*}
which shows $v\in D\left(C^{\ast}\right)$ and $zM_{0}u-C^{\ast}v=f.$
Hence, the operator 
\[
z\left(\begin{array}{cc}
M_{0} & 0\\
0 & 0
\end{array}\right)+\left(\begin{array}{cc}
0 & 0\\
0 & M_{1}(z)
\end{array}\right)+\left(\begin{array}{cc}
0 & -C^{\ast}\\
C & 0
\end{array}\right)
\]
is onto.\\
Moreover, it is one-to-one, since for $\left(u,v\right)\in D(C)\times D(C^{\ast})$
with \prettyref{eq:parabolic} it immediately follows that $u$ and
$v$ are given by \prettyref{eq:u and v}. Summarizing, we have proved
that $z\left(\begin{array}{cc}
M_{0} & 0\\
0 & 0
\end{array}\right)+\left(\begin{array}{cc}
0 & 0\\
0 & M_{1}(z)
\end{array}\right)+\left(\begin{array}{cc}
0 & -C^{\ast}\\
C & 0
\end{array}\right)$ is boundedly invertible for each $z\in D(M)\cap\mathbb{C}_{\Re>-\rho}$
and the norm of the inverse is uniformly bounded by \prettyref{eq:norm_estimate}.
Hence, since $\mathbb{C}_{\Re>-\rho}\setminus D(M)$ is discrete,
we derive that 
\[
D(M)\cap\mathbb{C}_{\Re>-\rho}\ni z\mapsto\left(z\left(\begin{array}{cc}
M_{0} & 0\\
0 & 0
\end{array}\right)+\left(\begin{array}{cc}
0 & 0\\
0 & M_{1}(z)
\end{array}\right)+\left(\begin{array}{cc}
0 & -C^{\ast}\\
C & 0
\end{array}\right)\right)^{-1}
\]
has a holomorphic and bounded extension to $\mathbb{C}_{\Re>-\rho}.$
Hence, $s_{0}(M,A)\leq-\rho$ and since $0<\rho<\nu_{1}$ was arbitrary,
we infer $s_{0}(M,A)\leq-\nu_{1}.$ 
\end{proof}
In the forthcoming examples we apply this result to two concrete models
for heat conduction.

\subsubsection*{The classical heat equation}

We recall from \prettyref{sub:mat_phys} the classical heat equation
with homogeneous Dirichlet boundary conditions, which can be written
as an evolutionary equation of the form 
\[
\left(\partial_{0,\rho}\left(\begin{array}{cc}
1 & 0\\
0 & 0
\end{array}\right)+\left(\begin{array}{cc}
0 & 0\\
0 & k^{-1}
\end{array}\right)+\left(\begin{array}{cc}
0 & \dive\\
\grad_{0} & 0
\end{array}\right)\right)\left(\begin{array}{c}
\vartheta\\
q
\end{array}\right)=\left(\begin{array}{c}
f\\
0
\end{array}\right),
\]
where $k\in L(L_{2}(\Omega)^{3})$ is a strictly accretive operator
modeling the heat conductivity of a medium $\Omega\subseteq\mathbb{R}^{3}$.
We assume that $\grad_{0}$ satisfies the Poincaré inequality, i.e.
\begin{equation}
\exists c>0\,\forall u\in D(\grad_{0}):|u|_{L_{2}(\Omega)}\leq c|\grad_{0}u|_{L_{2}(\Omega)^{3}},\label{eq:Poincare}
\end{equation}
which for instance is satisfied, if $\Omega$ is bounded (see e.g.
\cite[p. 290, Corollary 9.19]{Brezis2011}). Note that \prettyref{eq:Poincare}
especially implies that $\grad_{0}$ is injective and $R(\grad_{0})$
is closed. We define 
\begin{align*}
\iota_{R(\grad_{0})}:R(\grad_{0}) & \to L_{2}(\Omega)^{3}\\
f & \mapsto f.
\end{align*}
An easy computation then yields that $\iota_{R(\grad_{0})}\iota_{R(\grad_{0})}^{\ast}:L_{2}(\Omega)^{3}\to L_{2}(\Omega)^{3}$
is the orthogonal projector on $R(\grad_{0}).$ We define $C\coloneqq\iota_{R(\grad_{0})}^{\ast}\grad_{0}:D(\grad_{0})\subseteq L_{2}(\Omega)\to R(\grad_{0})$
and obtain a boundedly invertible closed operator, due to \prettyref{eq:Poincare}.
Moreover, we have $C^{\ast}=-\dive\iota_{R(\grad_{0})},$ since $\iota_{R(\grad_{0})}^{\ast}$
is bounded. Using these operators, we derive from the second line
of the heat equation 
\[
\iota_{R(\grad_{0})}^{\ast}q=-\iota_{R(\grad_{0})}^{\ast}k\iota_{R(\grad_{0})}C\vartheta.
\]
Hence, defining $\tilde{k}\coloneqq\iota_{R(\grad_{0})}^{\ast}k\iota_{R(\grad_{0})}$
and $\tilde{q}=\iota_{R(\grad_{0})}^{\ast}q,$ we infer $\tilde{k}^{-1}\tilde{q}+C\vartheta=0.$
Moreover, since $R(\grad_{0})^{\bot}=N(\dive),$ we obtain $\dive q=-C^{\ast}\tilde{q}$
and hence, we may modify the heat equation by writing 
\[
\left(\partial_{0,\rho}\left(\begin{array}{cc}
1 & 0\\
0 & 0
\end{array}\right)+\left(\begin{array}{cc}
0 & 0\\
0 & \tilde{k}^{-1}
\end{array}\right)+\left(\begin{array}{cc}
0 & -C^{\ast}\\
C & 0
\end{array}\right)\right)\left(\begin{array}{c}
\vartheta\\
\tilde{q}
\end{array}\right)=\left(\begin{array}{c}
f\\
0
\end{array}\right),
\]
which is now of the form discussed in \prettyref{prop:exp_decay_para}.
Hence, we derive that the heat equation is exponentially stable with
decay rate $\frac{c_{1}c^{2}}{\|\tilde{k}^{-1}\|^{2}},$ where $c_{1}$
is the accretivity constant of $k$. This yields that $\vartheta$
and $\tilde{q}$ decay exponentially, if the right-hand side does.
We note that we can also allow a non-vanishing source term in the
second coordinate in the modified heat equation.

\subsubsection*{Heat conduction with an additional delay term}

As a slight generalization of \cite{Khusainov2015} we replace Fourier's
law by the following expression 
\[
q=-k\grad_{0}\vartheta-\tilde{k}\tau_{-h}\grad_{0}\vartheta,
\]
for some operators $k,\tilde{k}\in L(L_{2}(\Omega)^{3})$ and some
$h>0.$ We assume that 
\[
\exists d>0\,\forall p\in L_{2}(\Omega)^{3}:\Re\langle kp|p\rangle_{L_{2}(\Omega)^{3}}\geq d|p|_{L_{2}(\Omega)^{3}}^{2},
\]
$\|\tilde{k}\|<d$ and that the Poincaré inequality \prettyref{eq:Poincare}
is satisfied. 
\begin{lem}
Under the above conditions the operator $k+\tilde{k}\e^{-hz}$ is
uniformly strictly accretive for each $z\in\mathbb{C}_{\Re>-\rho_{1}}$
with $\rho_{1}>\frac{1}{h}\log\frac{\|\tilde{k}\|}{d}.$ \end{lem}
\begin{proof}
For $z\in\mathbb{C}$ we estimate
\[
\Re\langle\left(k+\tilde{k}e^{-hz}\right)p|p\rangle_{L_{2}(\Omega)^{3}}\geq\left(d-\|\tilde{k}\|e^{-h\Re z}\right)|p|_{L_{2}(\Omega)^{3}}^{2}
\]
for each $p\in L_{2}(\Omega)^{3}.$ Since $d-\|\tilde{k}\|e^{-h\rho_{1}}>0$,
the assertion follows. 
\end{proof}
Using the latter lemma and the operator $\iota_{R(\grad_{0})}$ from
the previous example, we end up with an evolutionary equation of the
form 
\[
\left(\partial_{0,\rho}\left(\begin{array}{cc}
1 & 0\\
0 & 0
\end{array}\right)+\left(\begin{array}{cc}
0 & 0\\
0 & M_{1}(\partial_{0,\rho})
\end{array}\right)+\left(\begin{array}{cc}
0 & -C^{\ast}\\
C & 0
\end{array}\right)\right)\left(\begin{array}{c}
\vartheta\\
\tilde{q}
\end{array}\right)=\left(\begin{array}{c}
f\\
0
\end{array}\right),
\]
where $M_{1}(z)\coloneqq\iota_{R(\grad_{0})}^{\ast}\left(k+\tilde{k}e^{-hz}\right)^{-1}\iota_{R(\grad)}$
for $z\in\mathbb{C}_{\Re>-\rho_{1}}$ with $\rho_{1}>\frac{1}{h}\log\frac{\|\tilde{k}\|}{d}$
and $C\coloneqq\iota_{R(\grad_{0})}^{\ast}\grad_{0}.$ Hence, \prettyref{prop:exp_decay_para}
is applicable and we derive the exponential stability of the latter
evolutionary problem.

\subsubsection*{Integro-differential equations\label{sub:Integro-differential-equations-1}}

We consider the following evolutionary problem 
\begin{equation}
\left(\partial_{0,\rho}+(1-k\ast)A\right)u=f,\label{eq:integro-1}
\end{equation}
where $A:D(A)\subseteq H\to H$ is strictly $m$-accretive (e.g. the
Dirichlet Laplacian on a bounded domain) and $k\in L_{1,-\mu}(\mathbb{R}_{\geq0};L(H))$
for some $\mu>0$ with $|k|_{L_{1,-\mu}}<1$ (cp. \prettyref{sub:Integro-differential-equations}
for the definition). Similar to \prettyref{sub:Integro-differential-equations}
we require that $k$ satisfies \prettyref{cond:kernel} \prettyref{item: selfadjoint}
and \prettyref{item: commute} and additionally 

\begin{enumerate}

\item[(c')] \label{Im_0} There exists $0<\mu_{0}<\mu$ such that
\[
t\Im\langle\hat{k}(\i t-\mu_{0})x|x\rangle_{H}\geq0
\]
for each $t\in\mathbb{R},x\in H$.

\end{enumerate}
\begin{rem}
Note that condition (c') is \prettyref{cond:kernel} \prettyref{item:im_d}
with $d=0$. 
\end{rem}
Since $|k|_{L_{1,\mu}}<1$ we can employ the Neumann series (note
that for each $\rho>0$, $\|k\ast\|_{L(H_{\rho})}\leq|k|_{L_{1,\rho}}\leq|k|_{L_{1,-\mu}}$
by \prettyref{lem:convolution_bd}) and rewrite \prettyref{eq:integro-1}
as 
\begin{equation}
\left(\partial_{0,\rho}(1-k\ast)^{-1}+A\right)u=(1-k\ast)^{-1}f.\label{eq:integro-modified}
\end{equation}

\begin{prop}
Let $c>0$ and $A:D(A)\subseteq H\to H$ such that $A-c$ is $m$-accretive
and $k\in L_{1,-\mu}(\mathbb{R}_{\geq0};L(H))$ for some $\mu>0$
with $|k|_{L_{1,-\mu}}<1.$ If $k$ satisfies \prettyref{cond:kernel}
\prettyref{item: selfadjoint} and \prettyref{item: commute} and
condition (c'), then the evolutionary problem associated with $M(z)\coloneqq\left(1-\sqrt{2\pi}\hat{k}(z)\right)^{-1}$
and $A$ is well-posed and exponentially stable with decay rate 
\[
\mu_{1}\coloneqq\sup\left\{ 0\leq\nu\leq\mu_{0}\,;\,\nu\left(1-|k|_{L_{1,-\nu}}\right)^{-1}\leq c\right\} >0.
\]
\end{prop}
\begin{proof}
Note that the evolutionary problem associated with $N(z)\coloneqq M(z)+z^{-1}c$
for $z\in\mathbb{C}_{\Re>-\mu}\setminus\{0\}$ and $A-c$ is the same
as the evolutionary problem associated with $M$ and $A.$ We want
to apply \prettyref{prop:exp_stab_pos_def} for the material law $N$.
For doing so, let $t\in\mathbb{R}$ and $\rho>-\mu_{1}$. We note
that by \prettyref{lem:for_one_rho_for_all_rho} we have that 
\[
t\Im\langle\hat{k}(\i t+\rho)x|x\rangle_{H}\geq0
\]
for each $x\in H$. We define $D(\i t+\rho)\coloneqq|1-\sqrt{2\pi}\hat{k}(\i t+\rho)|^{-1}$
and estimate by using \prettyref{lem:comp_realpart_kernel} 
\begin{align*}
 & \Re\langle(\i t+\rho)N(\i t+\rho)x|x\rangle_{H}\\
 & =\Re\langle(\i t+\rho)M(\i t+\rho)x|x\rangle_{H}+c|x|_{H}^{2}\\
 & =\Re\langle(\i t+\rho)(1-\sqrt{2\pi}\hat{k}(-\i t+\rho))D(\i t+\rho)x|D(\i t+\rho)x\rangle_{H}+c|x|_{H}^{2}\\
 & =\rho\Re\langle(1-\sqrt{2\pi}\hat{k}(-\i t+\rho))D(\i t+\rho)x|D(\i t+\rho)x\rangle_{H}+\\
 & \quad-\sqrt{2\pi}t\Im\langle\hat{k}(-\i t+\rho)D(\i t+\rho)x|D(\i t+\rho)x\rangle_{H}+c|x|_{H}^{2}\\
 & \geq\rho\Re\langle(1-\sqrt{2\pi}\hat{k}(-\i t+\rho))D(\i t+\rho)x|D(\i t+\rho)x\rangle_{H}+c|x|_{H}^{2}
\end{align*}
for each $x\in H.$ If $\rho\geq0,$ the latter term can be estimated
by $c|x|^{2}.$ For $\rho<0,$ we compute 
\begin{align*}
\Re\langle(1-\sqrt{2\pi}\hat{k}(-\i t+\rho))D(\i t+\rho)x|D(\i t+\rho)x\rangle_{H} & =\Re\langle(1-\sqrt{2\pi}\hat{k}(\i t+\rho))^{-1}x|x\rangle_{H}\\
 & \leq\|(1-\sqrt{2\pi}\hat{k}(\i t+\rho))^{-1}\||x|_{H}^{2}\\
 & =\frac{1}{1-\sqrt{2\pi}\|\hat{k}(\i t+\rho)\|}|x|_{H}^{2}\\
 & \leq\frac{1}{1-|k|_{L_{1,\rho}}}|x|_{H}^{2}\\
 & \leq\frac{1}{1-|k|_{L_{1,-\mu_{1}}}}|x|_{H}^{2}
\end{align*}
and thus, 
\[
\Re\langle(\i t+\rho)N(\i t+\rho)x|x\rangle_{H}\geq\left(c+\frac{\rho}{1-|k|_{L_{1,-\mu_{1}}}}\right)|x|_{H}^{2}.
\]
Since $c+\frac{\rho}{1-|k|_{L_{1,\mu_{1}}}}>c+\frac{-\mu_{1}}{1-|k|_{L_{1,-\mu_{1}}}}\geq0,$
the assertion follows by \prettyref{prop:exp_stab_pos_def}.
\end{proof}
The latter proposition shows the exponential stability of \prettyref{eq:integro-modified}.
However, this also yields the exponential stability of our original
problem \prettyref{eq:integro-1} with the same decay rate, since
$(1-k\ast)^{-1}f\in H_{-\nu}(\mathbb{R};H)\cap H_{\rho}(\mathbb{R};H)$
for $f\in H_{-\nu}(\mathbb{R};H)\cap H_{\rho}(\mathbb{R};H)$ for
$0<\nu\leq\mu_{1}$ and $\rho>0$ by \prettyref{thm:material_law_good}.

\subsection{Exponential stability for equations of hyperbolic type}

In this section we study different examples of differential equations
of second order with respect to time. We begin with another variant
of the heat equation.

\subsubsection*{Dual phase lag heat conduction}

Let $\Omega\subseteq\mathbb{R}^{3}$ such that Poincaré's inequality
\prettyref{eq:Poincare} holds. In the theory of thermodynamics with
dual phase lags, we have the usual balance of momentum equation 
\[
\partial_{0,\rho}\vartheta+\dive q=f,
\]
where $\vartheta\in H_{\rho}(\mathbb{R};L_{2}(\Omega))$ is the temperature
density and $q\in H_{\rho}(\mathbb{R};L_{2}(\Omega)^{3})$ denotes
the heat flux, but Fourier's law is replaced by 
\[
(1+\tau_{q}\partial_{0,\rho}+\frac{1}{2}\tau_{q}^{2}\partial_{0,\rho}^{2})q=-\grad(1+\tau_{\vartheta}\partial_{0,\rho})\vartheta,
\]
where $\tau_{q},\tau_{\vartheta}>0$ are the so-called phases (see
\cite{tzou1995unified}). For $\rho$ large enough, we infer that
\[
1+\tau_{q}\partial_{0,\rho}+\frac{1}{2}\tau_{q}^{2}\partial_{0,\rho}^{2}=\partial_{0,\rho}^{2}\left(\partial_{0,\rho}^{-2}+\tau_{q}\partial_{0,\rho}^{-1}+\frac{1}{2}\tau_{q}\right)
\]
is boundedly invertible, due to the Neumann series. Hence, we can
rewrite the modified Fourier law as 
\[
q=-(1+\tau_{q}\partial_{0,\rho}+\frac{1}{2}\tau_{q}^{2}\partial_{0,\rho}^{2})^{-1}(1+\tau_{\vartheta}\partial_{0,\rho})\grad\vartheta.
\]
Hence, the balance of momentum equation gives 
\[
\partial_{0,\rho}\vartheta-\dive(1+\tau_{q}\partial_{0,\rho}+\frac{1}{2}\tau_{q}^{2}\partial_{0,\rho}^{2})^{-1}(1+\tau_{\vartheta}\partial_{0,\rho})\grad\vartheta=f.
\]
Assuming that $f\in D(\partial_{0,\rho})$, the latter equation gives
\[
\partial_{0,\rho}(1+\tau_{q}\partial_{0,\rho}+\frac{1}{2}\tau_{q}^{2}\partial_{0,\rho}^{2})(1+\tau_{\vartheta}\partial_{0,\rho})^{-1}\vartheta-\dive\grad\vartheta=\tilde{f},
\]
where $\tilde{f}=(1+\tau_{q}\partial_{0,\rho}+\frac{1}{2}\tau_{q}^{2}\partial_{0,\rho}^{2})(1+\tau_{\vartheta}\partial_{0,\rho})^{-1}f.$
Assuming homogeneous Dirichlet boundary conditions, we end up with
the following problem 
\begin{equation}
\partial_{0,\rho}^{2}(\partial_{0,\rho}^{-1}+\tau_{q}+\frac{1}{2}\tau_{q}^{2}\partial_{0,\rho})(1+\tau_{\vartheta}\partial_{0,\rho})^{-1}\vartheta-\dive\iota_{R(\grad_{0})}\iota_{R(\grad_{0})}^{\ast}\grad_{0}\vartheta=\tilde{f},\label{eq:dual phase lag}
\end{equation}
which is of the form \prettyref{eq:second_order} with $C\coloneqq\iota_{R(\grad_{0})}^{\ast}\grad_{0}:D(\grad_{0})\subseteq L_{2}(\Omega)\to R(\grad_{0})$
and $M(z)=\frac{z^{-1}+\tau_{q}+\frac{1}{2}\tau_{q}^{2}z}{1+\tau_{\vartheta}z}=\tau_{q}\frac{1+\frac{1}{2}\tau_{q}z}{1+\tau_{\vartheta}z}+z^{-1}\frac{1}{1+\tau_{\vartheta}z}$
for $z\in\mathbb{C}\setminus\{0,-\frac{1}{\tau_{\vartheta}}\}.$ Thus,
in the framework of \prettyref{sec:Exponential-stability-sec_order}
we have that 
\begin{align*}
M_{0}(z) & =\tau_{q}\frac{1+\frac{1}{2}\tau_{q}z}{1+\tau_{\vartheta}z},\\
M_{1}(z) & =\frac{1}{1+\tau_{\vartheta}z},
\end{align*}
which are bounded, if we restrict the domain of $M$ to $\mathbb{C}_{\Re>-\frac{1}{\tau_{\vartheta}}+\varepsilon}\setminus\{0\}$
for some $\varepsilon>0.$ We want to apply \prettyref{cor:exp_stab_sec_posdef}
for showing the exponential stability of the dual-phase lag model.
For doing so, we need to check the uniform accretivity of $zM(z)$
on a suitable right half plane. 
\begin{lem}
Assume that $\tau_{\vartheta},\tau_{q}>0$ such that $\frac{\tau_{q}}{\tau_{\vartheta}}<2$.
Then there exists $0<\nu_{0}<\frac{1}{\tau_{\vartheta}}$ and $c>0$
such that 
\[
\Re\langle zM(z)x|x\rangle_{L_{2}(\Omega)}\geq c|x|_{L_{2}(\Omega)}^{2}
\]
for all $x\in L_{2}(\Omega),z\in\mathbb{C}_{\Re>-\nu_{0}}\setminus\{0\}.$\end{lem}
\begin{proof}
Since $M(z)$ is just the multiplication with a complex number, it
suffices to compute $\Re zM(z)$ for $z\in\mathbb{C}\setminus\{0,-\frac{1}{\tau_{\vartheta}}\}.$
Setting $\mu\coloneqq\frac{\tau_{q}}{\tau_{\vartheta}}$ we compute
\[
zM(z)=\frac{1+\tau_{q}z+\frac{1}{2}\tau_{q}^{2}z^{2}}{1+\tau_{\vartheta}z}=\frac{1}{2}z\tau_{q}\mu+\mu(1-\frac{1}{2}\mu)+\frac{1-\mu(1-\frac{1}{2}\mu)}{1+\tau_{\vartheta}z}
\]
and thus, 
\[
\Re zM(z)=\frac{1}{2}\tau_{q}\mu\Re z+\mu(1-\frac{1}{2}\mu)+\frac{\left(1-\mu(1-\frac{1}{2}\mu)\right)(1+\tau_{\vartheta}\Re z)}{|1+\tau_{\vartheta}z|^{2}}.
\]
We note that by assumption $0<\mu(1-\frac{1}{2}\mu)\leq\frac{1}{2}$
and thus, 
\[
\Re zM(z)\geq\frac{1}{2}\tau_{q}\mu\Re z+\mu(1-\frac{1}{2}\mu)
\]
for each $z\in\mathbb{C}_{\Re>-\frac{1}{\tau_{\vartheta}}}\setminus\{0\}.$
Hence, for $0<\nu_{0}<\min\{\frac{2-\mu}{\tau_{q}},\frac{1}{\tau_{\vartheta}}\}$
we derive the assertion with $c=\mu(1-\frac{1}{2}\mu)-\frac{1}{2}\tau_{q}\mu\nu_{0}>0.$ 
\end{proof}
As an immediate consequence of the latter Lemma and \prettyref{cor:exp_stab_sec_posdef},
we derive the following stability result for the dual phase lag model.
\begin{prop}
Let $\Omega\subseteq\mathbb{R}^{3}$ open such that Poincaré's inequality
\prettyref{eq:Poincare} holds. Let $\tau_{q},\tau_{\vartheta}>0$
such that $\frac{\tau_{q}}{\tau_{\vartheta}}<2.$ Then the second
order problem \prettyref{eq:dual phase lag} is well-posed and exponentially
stable.\end{prop}
\begin{rem}
We note that these conditions on $\tau_{q},\tau_{\vartheta}$ coincide
with those imposed in \cite{Quintanilla2002}, where the same exponential
stability result is stated. For further results on the well-posedness
and asymptotic behaviour for phase lag models in heat conduction we
refer to \cite{Borgmeyer2014,Quintanilla2008,Quintanilla2007} and
the references therein.
\end{rem}

\subsubsection*{Abstract damped wave equation}

Let $H_{0},H_{1}$ be Hilbert spaces, $C:D(C)\subseteq H_{0}\to H_{1}$
a densely defined closed linear operator, which is assumed to be boundedly
invertible and $M_{0},M_{1}\in L(H_{0})$ such that $M_{0}$ is selfadjoint
and accretive and $M_{1}$ is strictly accretive. We consider the
second order problem of the form 
\[
\left(\partial_{0,\rho}^{2}M_{0}+\partial_{0,\rho}M_{1}+C^{\ast}C\right)u=f.
\]
Obviously, this is a problem of the form \prettyref{eq:second_order}
with $M(z)=M_{0}+z^{-1}M_{1}$ for $z\in\mathbb{C}\setminus\{0\}$,
and hence, 
\[
M_{0}(z)\coloneqq M_{0},\quad M_{1}(z)\coloneqq M_{1}\quad(z\in\mathbb{C}\setminus\{0\}).
\]
Moreover, we have 
\[
\Re\langle zM(z)x|x\rangle_{H_{0}}=\Re z\langle M_{0}x|x\rangle_{H_{0}}+\Re\langle M_{1}x|x\rangle_{H_{0}}\quad(z\in\mathbb{C}\setminus\{0\},x\in H_{0})
\]
and hence, for $0<\nu_{0}<\frac{c}{\|M_{0}\|}$, where $c>0$ denotes
the accretivity constant of $M_{1},$ we have that 
\[
\Re\langle zM(z)x|x\rangle_{H_{0}}\geq\left(c-\nu_{0}\|M_{0}\|\right)|x|_{H_{0}}^{2}\quad(z\in\mathbb{C}_{\Re>-\nu_{0}}\setminus\{0\},x\in H_{0}),
\]
which yields the exponential stability according to \prettyref{cor:exp_stab_sec_posdef}.\\
A particular case of the latter abstract equation is the damped wave
equation. Indeed, let $\Omega\subseteq\mathbb{R}^{n}$ open, such
that the Poincaré inequality \prettyref{eq:Poincare} holds, $H_{0}=L_{2}(\Omega),H_{1}=L_{2}(\Omega)^{n},$
$M_{0}=1$ and choose $C\coloneqq\iota_{R(\grad_{0})}^{\ast}\grad_{0}.$
Then $C$ is invertible and 
\[
\left(\partial_{0,\rho}^{2}+\partial_{0,\rho}M_{1}+C^{\ast}C\right)u=\left(\partial_{0,\rho}^{2}+\partial_{0,\rho}M_{1}-\dive\grad_{0}\right)u,
\]
which is the classical wave equation. However, we emphasize that the
abstract equation covers a broader class, since $M_{0}$ is allowed
to have a non-trivial kernel. In particular, we also recover the heat
equation by setting $M_{0}=0.$ But also mixed type problems are accessible.
Indeed, for instance let $M_{1}=1$ and $M_{0}=\chi_{\Omega_{0}}$
for some open subset $\Omega_{0}\subseteq\Omega.$ Then the corresponding
equation reads as 
\[
\left(\partial_{0,\rho}^{2}\chi_{\Omega_{0}}+\partial_{0,\rho}-\dive\grad_{0}\right)u=f,
\]
which is a damped wave equation on $\Omega_{0}$ and the heat equation
on $\Omega\setminus\Omega_{0}.$ Note in particular, that we do not
need to impose any transmission condition explicitly, as these are
encoded in the constraint that $u$ belongs to the domain of $\overline{\partial_{0,\rho}^{2}\chi_{\Omega_{0}}+\partial_{0,\rho}-\dive\grad_{0}}.$

\subsubsection*{Integro-differential equations}

We conclude this section by discussing the exponential stability of
integro-differential equations of the form 
\begin{equation}
\left(\partial_{0,\rho}^{2}(1-k\ast)^{-1}+C^{\ast}C\right)u=f,\label{eq:second_order_integro}
\end{equation}
where $C:D(C)\subseteq H_{0}\to H_{1}$ is a densely defined closed
linear operator between Hilbert spaces $H_{0},H_{1}$, which is assumed
to be boundedly invertible and $k\in L_{1,-\tilde{\mu}}(\mathbb{R}_{\geq0};L(H_{0}))$
for some $\tilde{\mu}>0$. We formulate the following conditions for
the kernel $k$: \newpage{}
\begin{condition}
\label{cond:cond_kernel_sec_order} $\,$ \begin{enumerate}[(a)]

\item \label{item:invertible} There exists $0<\mu\le\tilde{\mu}$
such that for all $z\in\mathbb{C}_{\Re\geq-\mu}$ the operator $(1-\sqrt{2\pi}\hat{k}(z))$
is boundedly invertible and 
\[
\mathbb{C}_{\Re\geq-\mu}\ni z\mapsto(1-\sqrt{2\pi}\hat{k}(z))^{-1}\in L(H_{0})
\]
is bounded.

\item \label{item:pos_def_Re} For each $\rho_{0}>0$ there is $c_{\rho_{0}}>0$
such that 
\[
\forall t\in\mathbb{R},\rho\geq\rho_{0},x\in H_{0}:\Re\langle(\i t+\rho)(1-\sqrt{2\pi}\hat{k}(\i t+\rho)^{\ast})x|x\rangle_{H_{0}}\geq c_{\rho_{0}}|x|_{H_{0}}^{2}.
\]

\item\label{item:selfadjoint} For almost every $t\in\mathbb{R}_{\geq0}$
the operator $k(t)$ is selfadjoint.

\item \label{item:commute} For almost every $t,s\in\mathbb{R}_{\geq0}$
we have $k(t)k(s)=k(s)k(t)$. 

\item \label{item:pos_def} For each $\delta>0$ there exists $0<\rho_{0}\leq\tilde{\mu}$
and a function $g_{\delta}:\mathbb{R}_{\geq-\rho_{0}}\to\mathbb{R}_{\geq0}$
continuous in $0$ with $g_{\delta}(0)>0$ such that 
\begin{equation}
\forall|t|>\delta,\rho\geq-\rho_{0},x\in H_{0}:t\Im\langle\hat{k}(\i t+\rho)x|x\rangle_{H_{0}}\geq g_{\delta}(\rho)|x|_{H_{0}}^{2}.\label{eq:pos_def_im}
\end{equation}

\end{enumerate}\end{condition}
\begin{rem}
In \cite{Trostorff2015_secondorder} we did not impose the conditions
\prettyref{item:invertible} and \prettyref{item:pos_def_Re}, but
require the somehow stronger condition $|k|_{L_{1}}<1$ (see the next
lemma). However, since we also want to cover the kernels considered
by the authors in \cite{Cannarsa2011}, we use these slightly weaker
conditions here. \end{rem}
\begin{lem}
\label{lem:nice_kernels}Let $k\in L_{1,-\tilde{\mu}}(\mathbb{R}_{\geq0};L(H))$
for some $\tilde{\mu}>0$. 

\begin{enumerate}[(a)]

\item If $|k|_{L_{1}}<1$ and $k$ satisfies \prettyref{cond:cond_kernel_sec_order}
\prettyref{item:pos_def}, then $k$ satisfies \prettyref{cond:cond_kernel_sec_order}
\prettyref{item:invertible} and \prettyref{item:pos_def_Re}.

\item If $k$ satisfies \prettyref{cond:cond_kernel_sec_order} \prettyref{item:selfadjoint},
then \prettyref{eq:pos_def_im} is equivalent to 
\[
\forall t>\delta,\rho\geq-\rho_{0},x\in H_{0}:t\Im\langle\hat{k}(\i t+\rho)x|x\rangle_{H_{0}}\geq g_{\delta}(\rho)|x|_{H_{0}}^{2}.
\]

\item For each $\rho>-\tilde{\mu}$ we have that $\left(t\mapsto tk(t)\right)\in L_{1,\rho}(\mathbb{R}_{\geq0};L(H_{0})).$

\end{enumerate}\end{lem}
\begin{proof}
\begin{enumerate}[(a)]

\item Let $|k|_{L_{1}}<1$ and assume $k$ satisfies \prettyref{cond:cond_kernel_sec_order}
\prettyref{item:pos_def}. Since the mapping 
\[
\rho\mapsto|k|_{L_{1,\rho}}
\]
is continuous, we find $0<\mu\leq\tilde{\mu}$ such that $|k|_{L_{1,-\mu}}<1.$
Hence, \prettyref{cond:cond_kernel_sec_order} \prettyref{item:invertible}
follows, since for $z\in\mathbb{C}_{\Re\geq-\mu}$ we have $\sqrt{2\pi}\|\hat{k}(z)\|\leq|k|_{L_{1,-\mu}}$
and thus, the assertion follows by using the Neumann series. Moreover,
\prettyref{cond:cond_kernel_sec_order} \prettyref{item:pos_def}
gives that for all $x\in H_{0},\rho\geq0$ and $t\in\mathbb{R}$ we
have that 
\[
t\Im\langle\hat{k}(\i t+\rho)x|x\rangle_{H_{0}}\geq0.
\]
Indeed, for $t=0,$ this inequality is trivial and for $t\ne0,$ the
term on the left-hand side can be estimated from below by $g_{\frac{|t|}{2}}(\rho)|x|_{H_{0}}^{2}\geq0.$
Let now $\rho_{0}>0$. Then, for $\rho\geq\rho_{0},t\in\mathbb{R},x\in H_{0}$
we estimate 
\begin{align*}
 & \Re\langle(\i t+\rho)(1-\sqrt{2\pi}\hat{k}(\i t+\rho)^{\ast})x|x\rangle_{H_{0}}\\
 & =\rho\Re\langle(1-\sqrt{2\pi}\hat{k}(\i t+\rho)^{\ast})x|x\rangle_{H_{0}}-\sqrt{2\pi}t\Im\langle\hat{k}(\i t+\rho)^{\ast}x|x\rangle_{H_{0}}\\
 & \geq\rho\left(1-|k|_{L_{1}}\right)|x|_{H_{0}}^{2}+\sqrt{2\pi}t\Im\langle\hat{k}(\i t+\rho)x|x\rangle_{H_{0}}\\
 & \geq\rho_{0}\left(1-|k|_{L_{1}}\right)|x|_{H_{0}}^{2}.
\end{align*}
\item Assume that $k(t)$ is selfadjoint for almost every $t\geq0.$
Then $\hat{k}(-\i t+\rho)=\hat{k}(\i t+\rho)^{\ast}$ and hence, the
assertion follows.

\item This is clear, since $t\e^{-\rho t}\leq\frac{1}{\rho+\tilde{\mu}}\e^{-1}\e^{\tilde{\mu}t}$
for each $t\geq0$.\qedhere 

\end{enumerate}
\end{proof}
Before we come to examples of kernels satisfying \prettyref{cond:cond_kernel_sec_order},
we prove the exponential stability of the corresponding second-order
problem \prettyref{eq:second_order_integro}.
\begin{prop}
Let $C:D(C)\subseteq H_{0}\to H_{1}$ a densely defined closed linear
operator between Hilbert spaces $H_{0},H_{1}$, which is boundedly
invertible and $k\in L_{1,-\tilde{\mu}}(\mathbb{R}_{\geq0};L(H_{0}))$
for some $\tilde{\mu}>0$. Moreover, assume that $k$ satisfies \prettyref{cond:cond_kernel_sec_order}.
Then the evolutionary problem given by \prettyref{eq:second_order_integro}
is well-posed and exponentially stable.\end{prop}
\begin{proof}
We  apply \prettyref{prop:exp_stable_without_ball}. First, we choose
$0<\mu\leq\tilde{\mu}$ according to \prettyref{item:invertible}.
Hence, $M(z)\coloneqq M_{0}(z)\coloneqq(1-\sqrt{2\pi}\hat{k}(z))^{-1}$
for $z\in\mathbb{C}_{\Re\geq-\mu}$ is well-defined and bounded. Moreover,
$\lim_{z\to0}M_{1}(z)=0$ as $M_{1}=0$. So it suffices to check \prettyref{eq:pos_def_M_withoutball}.
For doing so, let $\delta>0$. Moreover, let $0<\rho_{0}\leq\min\{\mu,\frac{\delta}{2}\}$
to be specified later. Then, for $z\in\mathbb{C}_{\Re>-\rho_{0}}\setminus B[0,\delta]$
and $x\in H_{0}$ we estimate, using the representation $z=\i t+\rho$
for $\rho>-\rho_{0},t\in\mathbb{R}$ 
\begin{align*}
\Re\langle(\i t+\rho)M(\i t+\rho)x|x\rangle_{H_{0}} & =\Re\langle\left(\i t+\rho\right)(1-\sqrt{2\pi}\hat{k}(\i t+\rho))^{-1}x|x\rangle_{H_{0}}\\
 & =\Re\langle(\i t+\rho)(1-\sqrt{2\pi}\hat{k}(\i t+\rho)^{\ast})D(\i t+\rho)x|D(\i t+\rho)x\rangle_{H_{0}}\\
 & =\Re\langle\rho(1-\sqrt{2\pi}\hat{k}(-\i t+\rho))D(\i t+\rho)x|D(\i t+\rho)x\rangle_{H_{0}}+\\
 & \quad-\sqrt{2\pi}t\Im\langle\hat{k}(-\i t+\rho)D(\i t+\rho)x|D(\i t+\rho)x\rangle_{H_{0}}
\end{align*}
where $D(\i t+\rho)\coloneqq|1-\sqrt{2\pi}\hat{k}(\i t+\rho)|^{-1}.$
We estimate the latter expression in two steps: first for $\rho\geq\rho_{0}$
and second, for $\rho\in]-\rho_{0},\rho_{0}[.$ So let first, $\rho\geq\rho_{0}.$
Then, by \prettyref{item:pos_def_Re} there is some $c_{\rho_{0}}>0$
such that 
\begin{align*}
\Re\langle(\i t+\rho)M(\i t+\rho)x|x\rangle_{H_{0}} & \geq c_{\rho_{0}}|D(\i t+\rho)x|_{H_{0}}^{2}\quad(\rho\geq\rho_{0}).
\end{align*}
For $\rho\in]-\rho_{0},\rho_{0}[,$ we infer that $|t|>\delta-\rho_{0}\geq\frac{\delta}{2},$
since $|\i t+\rho|>\delta$ by assumption. Hence, we obtain 
\begin{align*}
\Re\langle(\i t+\rho)M(\i t+\rho)x|x\rangle_{H_{0}} & \geq\left(-\rho(1+|k|_{L_{1,-\rho}})+\sqrt{2\pi}g_{\frac{\delta}{2}}(\rho)\right)|D(\i t+\rho)x|_{H_{0}}^{2}\\
 & \geq\left(-\rho_{0}(1+|k|_{L_{1,-\rho_{0}}})+\sqrt{2\pi}\inf_{\nu\in[-\rho_{0},\rho_{0}]}g_{\frac{\delta}{2}}(\nu)\right)|D(\i t+\rho)x|_{H_{0}}^{2}.
\end{align*}
Since $g_{\frac{\delta}{2}}$ is continuous in $0$ we infer 
\[
\left(-\rho_{0}(1+|k|_{L_{1,-\rho_{0}}})+\sqrt{2\pi}\inf_{\nu\in[-\rho_{0},\rho_{0}]}g_{\frac{\delta}{2}}(\nu)\right)\to\sqrt{2\pi}g_{\frac{\delta}{2}}(0)>0\quad(\rho_{0}\to0)
\]
and thus, we may choose $0<\rho_{0}\leq\min\{\mu,\frac{\delta}{2}\}$
small enough, such that 
\[
\tilde{c}\coloneqq-\rho_{0}(1+|k|_{L_{1,-\rho_{0}}})+\sqrt{2\pi}\inf_{\nu\in[-\rho_{0},\rho_{0}]}g_{\frac{\delta}{2}}(\nu)>0.
\]
Summarizing, we have shown that
\[
\Re\langle(\i t+\rho)M(\i t+\rho)x|x\rangle_{H_{0}}\geq\min\{c_{\rho_{0}},\tilde{c}\}|D(\i t+\rho)x|_{H_{0}}^{2}
\]
for each $\rho>-\rho_{0}.$ Using now $|x|_{H_{0}}\leq(1+|k|_{L_{1,-\rho_{0}}})|D(\i t+\rho)x|_{H_{0}}$
for $t\in\mathbb{R},\rho>-\rho_{0}$, we derive 
\[
\Re\langle zM(z)x|x\rangle_{H_{0}}\geq\min\{c_{\rho_{0}},\tilde{c}\}\frac{1}{\left(1+|k|_{L_{1,-\rho_{0}}}\right)^{2}}|x|_{H_{0}}^{2}
\]
for each $x\in H_{0}$ and $z\in\mathbb{C}_{\Re>-\rho_{0}}\setminus B[0,\delta].$
Hence, the assertion follows from \prettyref{prop:exp_stable_without_ball}.
\end{proof}
We conclude this subsection by providing two examples for classes
of kernels, which satisfy \prettyref{cond:cond_kernel_sec_order}.
We start with a class of kernels considered in \cite{Pruss2009},
where the exponential and polynomial stability of hyperbolic integro-differential
equations is studied. 
\begin{prop}
\label{prop:kernels_Pr=0000FCss} Let $k\in W_{1,\mathrm{loc}}^{1}(\mathbb{R}_{>0};\mathbb{R})\setminus\{0\}$
such that $k\geq0,$ $k'\leq0$, $\intop_{0}^{\infty}k(s)\,\dd s<1$
and $k\in L_{1,-\mu}(\mathbb{R}_{\geq0};\mathbb{R})$ for some $\mu>0$.
Then $k$ satisfies \prettyref{cond:cond_kernel_sec_order}.\end{prop}
\begin{proof}
Since $k$ is assumed to be real-valued, it trivially satisfies \prettyref{item:selfadjoint}
and \prettyref{item:commute}. Moreover, since $|k|_{L_{1}}<1,$ it
suffices to prove \prettyref{item:pos_def} by \prettyref{lem:nice_kernels}
(a). First, we claim that 
\[
k(0+)>0
\]
where $k(0+)=\infty$ is allowed. The limit exists, since 
\[
k(t)=-\intop_{t}^{1}k'(s)\mbox{ ds}+k(1)\quad(t\in]0,1[)
\]
and since the right-hand side converges as $t$ tends to zero by monotone
convergence, so does the left-hand side. Moreover, $k(0+)\geq k(t)$
for each $t\in\mathbb{R}_{>0}$ and thus, $k(0+)>0$ since otherwise
$k=0$. For showing \prettyref{item:pos_def} we fix $\delta>0$ and
consider the function 
\begin{align*}
f:\mathbb{R}_{\geq\delta}\times\mathbb{R}_{\geq-\mu} & \to\mathbb{R}\\
(t,\rho) & \mapsto t\intop_{0}^{\infty}\sin(ts)\e^{-\rho s}k(s)\mbox{ d}s,
\end{align*}
which is continuous by dominated convergence. Note that 
\begin{align}
t\Im\langle\hat{k}(\i t+\rho)x|x\rangle_{H_{0}} & =-t|x|_{H_{0}}^{2}\Im\hat{k}(\i t+\rho)\nonumber \\
 & =t|x|_{H_{0}}^{2}\frac{1}{\sqrt{2\pi}}\intop_{0}^{\infty}\sin(ts)\e^{-\rho s}k(s)\mbox{ d}s\nonumber \\
 & =\frac{1}{\sqrt{2\pi}}f(t,\rho)|x|_{H_{0}}^{2}\label{eq:kernel_f}
\end{align}
for $t\geq\delta,\rho\geq-\mu$ and $x\in H_{0}.$ We follow the strategy
presented in \cite[Section 5]{Pruss2009} to show that there exists
$0<\rho_{0}\leq\mu,$ such that

\[
\inf_{t\geq\delta}\inf_{\rho\in[-\rho_{0},\rho_{0}]}f(t,\rho)>0\mbox{ and }\inf_{t\geq\delta}f(t,\rho)>0\quad(\rho\geq-\rho_{0}).
\]
Since $k\in L_{1,-\mu}(\mathbb{R}_{\geq0})$ there exist sequences
$(a_{n})_{n\in\mathbb{N}}$ and $(b_{n})_{n\in\mathbb{N}}$ in $\mathbb{R}_{>0}$
such that $a_{n}\to0$ and $b_{n}\to\infty$ and $\e^{\mu b_{n}}k(b_{n})\to0$
as well as $a_{n}\e^{\mu a_{n}}k(a_{n})\to0$ as $n\to\infty$. The
latter gives 
\[
\intop_{\varepsilon}^{\infty}|k'(s)|\e^{\mu s}\mbox{ d}s=\lim_{n\to\infty}\intop_{\varepsilon}^{b_{n}}-k'(s)\e^{\mu s}\mbox{ d}s=\mu\intop_{\varepsilon}^{\infty}k(s)\e^{\mu s}\mbox{ d}s+k(\varepsilon)\e^{\mu\varepsilon}<\infty,
\]
for each $\varepsilon>0$ by monotone convergence. Hence, $k'\in L_{1,-\mu}(\mathbb{R}_{\geq\varepsilon})$
for each $\varepsilon>0$. Let $t\geq\delta,\rho\geq-\mu$. Then we
have 
\begin{align*}
 & \intop_{a_{n}}^{b_{n}}(\cos(ts)-1)\left(k'(s)-\rho k(s)\right)\e^{-\rho s}\mbox{ d}s\\
 & =t\intop_{a_{n}}^{b_{n}}\sin(ts)\e^{-\rho s}k(s)\mbox{ d}s+(\cos(tb_{n})-1)\e^{-\rho b_{n}}k(b_{n})-(\cos(ta_{n})-1)\e^{-\rho a_{n}}k(a_{n})
\end{align*}
for each $n\in\mathbb{N}.$ Since $\e^{\mu b_{n}}k(b_{n})\to0$, we
infer $\left(\cos(tb_{n})-1\right)\e^{-\rho b_{n}}k(b_{n})\to0$ as
$n\to\infty.$ Similarly 
\[
\left|\cos(ta_{n})-1\right|\e^{-\rho a_{n}}k(a_{n})\leq t\e^{-\rho a_{n}}a_{n}k(a_{n})\to0\quad(n\to\infty)
\]
and thus, monotone and dominated convergence gives 
\begin{align}
f(t,\rho) & =t\intop_{0}^{\infty}\sin(ts)\e^{-\rho s}k(s)\mbox{ d}s\nonumber \\
 & =\intop_{0}^{\infty}(\cos(ts)-1)\left(k'(s)-\rho k(s)\right)\e^{-\rho s}\mbox{ d}s\nonumber \\
 & =\intop_{0}^{\infty}\left(\cos(ts)-1\right)k'(s)\e^{-\rho s}\mbox{ d}s+\rho\intop_{0}^{\infty}(1-\cos(ts))k(s)\e^{-\rho s}\mbox{ d}s.\label{eq:repr_f(t,rho)}
\end{align}
First, we prove that $\inf_{t\geq\delta}f(t,\rho)>0$ for $\rho\geq-\rho_{0}$
for some $\rho_{0}>0$. By the latter equality we see that $f(t,\rho)\geq0$
for $\rho\geq0,t\geq\delta.$ Indeed, it even holds $f(t,\rho)>0$
for $t\geq\delta,\rho\geq0,$ since otherwise $k'=0$ which contradicts
$k\in L_{1}(\mathbb{R}_{\geq0})\setminus\{0\}.$ Moreover, we observe
that for $\rho\geq-\mu,t\geq\delta$ 
\[
f(t,\rho)\geq\intop_{\varepsilon}^{\infty}\left(\cos(ts)-1\right)k'(s)\e^{-\rho s}\mbox{ d}s+\rho\intop_{0}^{\infty}(1-\cos(ts))k(s)\e^{-\rho s}\mbox{ d}s,
\]
for each $\varepsilon>0$ since $k'\leq0.$ Since $(t\mapsto k'(t)\chi_{[\varepsilon,\infty[}(t))\in L_{1}(\mathbb{R}),$
the Lemma of Riemann-Lebesgue (\prettyref{rem:Riemann-Lebesgue})
yields 
\[
\liminf_{t\to\infty}f(t,\rho)\geq-\intop_{\varepsilon}^{\infty}k'(s)\e^{-\rho s}\mbox{ ds}+\rho\intop_{0}^{\infty}\e^{-\rho s}k(s)\mbox{ d}s=k(\varepsilon)\e^{-\rho\varepsilon}+\rho\intop_{0}^{\varepsilon}k(s)\e^{-\rho s}\mbox{ d}s
\]
for each $\varepsilon>0$. Thus, $\liminf_{t\to\infty}f(t,\rho)\geq k(0+)>0$
for each $\rho\geq-\mu$. Moreover, the same argumentation yields
that for each $K\subseteq\mathbb{R}_{>-\mu}$ compact there is $M>\delta$
such that 
\begin{equation}
\forall t\geq M,\,\rho\in K:f(t,\rho)\geq\frac{k(0+)}{2},\label{eq:uniform_Riemann_Lebesgue}
\end{equation}
according to \prettyref{lem:nice_kernels} (c) and \prettyref{prop:Riemann-Lebesgue_professional}
applied to $t\mapsto\e^{-\nu t}(k(t)+\chi_{[\varepsilon,\infty[}(t)k'(t))$
for $\nu\coloneqq\min K$. Since $f(t,\rho)>0$ for each $t\geq\delta,\rho\geq0$
and $\liminf_{t\to\infty}f(t,\rho)>0$, we derive, using the continuity
of $f(t,\cdot)$, that 
\[
\inf_{t\geq\delta}f(t,\rho)>0\quad(\rho\geq0).
\]
Furthermore, for $\rho<0$ \prettyref{eq:repr_f(t,rho)} yields 
\begin{align*}
f(t,\rho) & \geq f(t,0)+2\rho|k|_{L_{1,\rho}}\geq\inf_{s\geq\delta}f(s,0)+2\rho|k|_{L_{1,\rho}},
\end{align*}
and since $\inf_{s\geq\delta}f(s,0)>0$, we find $0<\rho_{0}\leq\mu$
such that 
\begin{equation}
\inf_{t\geq\delta}f(t,\rho)>0\quad(\rho\geq-\rho_{0}).\label{eq:f_positive}
\end{equation}
Moreover, we note that by \prettyref{eq:uniform_Riemann_Lebesgue}
we find $M>\delta$ such that 
\[
\inf_{\rho\in[-\rho_{0},\rho_{0}]}\inf_{t\geq M}f(t,\rho)>0.
\]
Since $f$ is continuous and attains positive values on $[\delta,M]\times[-\rho_{0},\rho_{0}]$
by \prettyref{eq:f_positive}, we get 
\[
c\coloneqq\inf_{t\geq\delta}\inf_{\rho\in[-\rho_{0},\rho_{0}]}f(t,\rho)>0.
\]
Hence, setting $g_{\delta}\coloneqq\chi_{[-\rho_{0},\rho_{0}]}\frac{c}{\sqrt{2\pi}},$
we infer from \prettyref{eq:kernel_f} and \prettyref{eq:f_positive}
\[
t\Im\langle\hat{k}(\i t+\rho)x|x\rangle_{H_{0}}\geq g_{\delta}(\rho)|x|_{H_{0}}^{2}\quad(\rho\geq-\rho_{0},t\geq\delta,x\in H_{0}).
\]
This yields \prettyref{item:pos_def} by \prettyref{lem:nice_kernels}
(b).\end{proof}
\begin{rem}
By following the lines of the proof, the latter proposition can easily
be generalized to operator-valued kernels $k\in L_{1,-\mu}(\mathbb{R}_{\geq0};L(H_{0}))$
with $|k|_{L_{1}}<1$ by assuming that $k$ satisfies \prettyref{cond:cond_kernel_sec_order}
\prettyref{item:selfadjoint} and \prettyref{item:commute}, $k$
is locally absolutely continuous on $\mathbb{R}_{>0}$ and $k(t)$
and $-k'(t)$ are accretive for almost every $t\in\mathbb{R}_{>0}$. 
\end{rem}
A second example for kernels covered by our approach are so called
kernels of positive type.
\begin{defn*}
Let $k\in L_{1,\mathrm{loc}}(\mathbb{R}_{>0};\mathbb{R})$. Then $k$
is said to be of \emph{positive type, }if for all $t>0$ and $f\in L_{2,\mathrm{loc}}(\mathbb{R}_{>0})$
we have that 
\[
\Re\intop_{0}^{t}\left(\left(k\hat{\ast}f\right)(s)\right)^{\ast}f(s)\mbox{ d}s\geq0,
\]
where 
\[
\left(k\hat{\ast}f\right)(s)\coloneqq\intop_{0}^{s}k(r)f(s-r)\mbox{ d}r\quad(s>0).
\]
Moreover, $k$ is said to be \emph{of strict positive type, }if there
exists $\varepsilon>0$ such that $t\mapsto k(t)-\varepsilon\e^{-t}$
is of positive type.
\end{defn*}
We recall the following characterization for kernels of positive type
originally presented in \cite{Nohel1976} (in fact for measures).
However, we follow the proof given in \cite[p. 494]{Gripenberg1990_Volterra}.
\begin{prop}
\label{prop:positive_kernels}Let $k\in\bigcap_{\rho>0}L_{1,\rho}(\mathbb{R}_{\geq0};\mathbb{R})$.
Then $k$ is of positive type if and only if $\Re\,\hat{k}(z)\geq0$
for each $z\in\mathbb{C}_{\Re>0}.$ Furthermore, if $k\in L_{1}(\mathbb{R}_{\geq0};\mathbb{R})$
then $k$ is of positive type if and only if $\Re\,\hat{k}(\i t)\geq0$
for each $t\in\mathbb{R}.$ \end{prop}
\begin{proof}
Assume that $k$ is of positive type. Let $t>0,z\in\mathbb{C}_{\Re>0}$
and define $f(s)\coloneqq\e^{-zs}$ for $s\in\mathbb{R}_{\geq0}.$
Then we have 
\begin{align*}
0 & \leq\Re\intop_{0}^{t}\left(\left(k\hat{\ast}f\right)(s)\right)^{\ast}f(s)\mbox{ d}s\\
 & =\Re\intop_{0}^{t}\left(\intop_{0}^{s}k(r)\e^{-z(s-r)}\mbox{ d}r\right)^{\ast}\e^{-zs}\mbox{ d}s\\
 & =\Re\intop_{0}^{t}\intop_{0}^{s}k(r)\e^{-z^{\ast}(s-r)}\mbox{ d}r\:\e^{-zs}\mbox{ d}s\\
 & =\Re\intop_{0}^{t}k(r)\e^{z^{\ast}r}\intop_{r}^{t}\e^{-2s\Re z}\mbox{ d}s\mbox{ d}r.
\end{align*}
Letting $t$ tend to infinity gives 
\begin{align*}
0 & \leq\Re\intop_{0}^{\infty}k(r)\e^{z^{\ast}r}\intop_{r}^{\infty}\e^{-2s\Re z}\mbox{ d}s\mbox{ d}r\\
 & =\frac{1}{2\Re z}\Re\intop_{0}^{\infty}k(r)\e^{z^{\ast}r}\e^{-2r\Re z}\mbox{ d}r\\
 & =\frac{1}{2\Re z}\Re\intop_{0}^{\infty}k(r)\e^{-zr}\mbox{ d}r\\
 & =\frac{\sqrt{2\pi}}{2\Re z}\Re\,\hat{k}(z),
\end{align*}
which shows the first implication. Assume now that $\Re\,\hat{k}(z)\geq0$
for each $z\in\mathbb{C}_{\Re>0}.$ We note that this implies that
$k\ast$ is an accretive operator on $H_{\rho}(\mathbb{R})$ for each
$\rho>0$ by \prettyref{lem:kernel_Fourier}. Let $t>0$ and $f\in L_{2,\mathrm{loc}}(\mathbb{R}_{\geq0}).$
Then $g\coloneqq\chi_{[0,t]}f\in\bigcap_{\rho\geq0}H_{\rho}(\mathbb{R})$
and hence, 
\begin{align*}
0 & \leq\Re\langle k\ast g|g\rangle_{H_{\rho}}\\
 & =\Re\intop_{0}^{t}\left(\left(k\ast g\right)(s)\right)^{\ast}f(s)\e^{-2\rho s}\mbox{ d}s\\
 & =\Re\intop_{0}^{t}\left(\left(k\hat{\ast}f\right)(s)\right)^{\ast}f(s)\e^{-2\rho s}\mbox{ d}s
\end{align*}
for each $\rho>0.$ Letting $\rho\to0$ we derive the assertion by
dominated convergence.\\
Assume now that $k\in L_{1}(\mathbb{R}_{\geq0};\mathbb{R}).$ If $k$
is of positive type, we have $\Re\hat{k}(\i t+\rho)\geq0$ for each
$t\in\mathbb{R},\rho>0$ by what we have shown above. Letting $\rho$
tend to $0$, we get $\Re\hat{k}(\i t)\geq0$ for $t\in\mathbb{R}$.
If conversely, $\Re\hat{k}(\i t)\geq0$ for each $t\in\mathbb{R},$
we may use \prettyref{lem:kernel_Fourier} to derive that $k\ast$
defines an accretive operator on $L_{2}(\mathbb{R}).$ Hence, for
$f\in L_{2,\mathrm{loc}}(\mathbb{R}_{\geq0})$ and $t\geq0$ we derive
that 
\begin{align*}
0 & \leq\Re\langle k\ast\chi_{[0,t]}f|\chi_{[0,t]}f\rangle_{L_{2}(\mathbb{R})}\\
 & =\Re\intop_{0}^{t}\left(\left(k\hat{\ast}f\right)(s)\right)^{\ast}f(s)\mbox{ d}s,
\end{align*}
which shows that $k$ is of positive type.  
\end{proof}
In \cite{Cannarsa2011} the authors consider kernels $k\in L_{1}(\mathbb{R}_{\geq0};\mathbb{R}),$
such that the function $\mathbb{R}_{\geq0}\ni t\mapsto\intop_{t}^{\infty}k(s)\mbox{ d}s$
is a kernel of positive type. For those kernels we have the following
result.
\begin{lem}[{\cite[Proposition 2.4]{Cannarsa2008}}]
\label{lem:positive kernels primitive} Let $k\in L_{1}(\mathbb{R}_{\geq0};\mathbb{R})$.
We set 
\[
K(t)\coloneqq\intop_{t}^{\infty}k(s)\,\dd s\quad(t\geq0).
\]
Then, if $K$ is of positive type, we have that 
\[
t\Im\,\hat{k}(\i t+\rho)\leq0\quad(t\in\mathbb{R},\rho\geq0).
\]
If $K\in L_{1}(\mathbb{R}_{\geq0})$, then also the converse implication
holds true. \end{lem}
\begin{proof}
Assume that $K$ is of positive type. As $K\in L_{\infty}(\mathbb{R}_{\geq0};\mathbb{R})$
we have that $\Re\:\hat{K}(z)\geq0$ for each $z\in\mathbb{C}_{\Re>0}$
by \prettyref{prop:positive_kernels}. For $z\in\mathbb{C}_{\Re>0}$
we have that 
\begin{align}
\hat{K}(z) & =\frac{1}{\sqrt{2\pi}}\intop_{0}^{\infty}K(s)\e^{-zs}\mbox{ d}s\nonumber \\
 & =\frac{1}{\sqrt{2\pi}}\intop_{0}^{\infty}\intop_{s}^{\infty}k(r)\mbox{ d}r\,\e^{-zs}\mbox{ d}s\nonumber \\
 & =\frac{1}{\sqrt{2\pi}}\frac{1}{z}\left(-\intop_{0}^{\infty}k(s)\,\e^{-zs}\mbox{ d}s+\intop_{0}^{\infty}k(r)\mbox{ d}r\right)\nonumber \\
 & =\frac{1}{z}\left(\frac{1}{\sqrt{2\pi}}\intop_{0}^{\infty}k(s)\mbox{ d}s-\hat{k}(z)\right),\label{eq:K_and_k}
\end{align}
which yields that 
\begin{align*}
0 & \leq\Re\:\hat{K}(\i t+\rho)\\
 & =\frac{1}{|\i t+\rho|^{2}}\left(\rho\frac{1}{\sqrt{2\pi}}\intop_{0}^{\infty}k(s)\mbox{ d}s-\rho\Re\,\hat{k}(\i t+\rho)-t\Im\,\hat{k}(\i t+\rho)\right),
\end{align*}
for each $t\in\mathbb{R},\rho>0$. Letting $\rho$ tend to $0$, we
infer 
\[
0\leq-\frac{1}{t}\Im\,\hat{k}(\i t)
\]
which yields gives $t\Im\hat{k}(\i t)\leq0$ for each $t\in\mathbb{R}.$
The assertion now follows from \prettyref{lem:for_one_rho_for_all_rho}
with $d=0$. Assume now that $K\in L_{1}(\mathbb{R}_{\geq0})$ and
\[
t\Im\hat{k}(\i t+\rho)\leq0\quad(t\in\mathbb{R},\rho\geq0).
\]
In particular we get $t\Im\hat{k}(\i t)\leq0$ and hence, using \prettyref{eq:K_and_k},
we obtain 
\[
\Re\hat{K}(\i t)=-\Re\frac{1}{\i t}\hat{k}(\i t)=-\frac{1}{t}\Im\hat{k}(\i t)\geq0\quad(t\in\mathbb{R}).
\]
Thus, the assertion follows from \prettyref{prop:positive_kernels}.
\end{proof}
In \cite{Cannarsa2011} the exponential stability of a class of hyperbolic
semilinear integro-differential equations is studied for kernels $k\in L_{1,-\alpha}(\mathbb{R}_{\geq0};\mathbb{R})$
for some $\alpha>0,$ satisfying $\intop_{0}^{\infty}k(t)\mbox{ d}t<1$
and $t\mapsto\intop_{t}^{\infty}\e^{\alpha s}k(s)\mbox{ d}s$ is a
kernel of strict positive type. We show that these kernels are also
covered by our approach. For doing so, we need the following auxiliary
results.
\begin{lem}
\label{lem:strict_positive_kernel}Let $k\in L_{1,-\alpha}(\mathbb{R}_{\geq0};\mathbb{R})$
for some $\alpha>0$, such that \textup{$K_{\alpha}(t)\coloneqq\intop_{t}^{\infty}\e^{\alpha s}k(s)\mbox{ d}s$}
for $t\geq0$ is a kernel of strict positive type. Then there is $\varepsilon>0$
such that 
\[
t\Im\hat{k}(\i t+\rho)\leq-\varepsilon\frac{t^{2}}{t^{2}+(\rho+\alpha+1)^{2}}\quad(t\in\mathbb{R},\rho\geq-\alpha).
\]
\end{lem}
\begin{proof}
Since $K_{\alpha}$ is of strict positive type, we find $\varepsilon>0,$
such that $g(t)\coloneqq K_{\alpha}(t)-\varepsilon\e^{-t}=\intop_{t}^{\infty}k_{\alpha}(s)-\varepsilon\e^{-s}\mbox{ d}s$
is of positive type, where we set $k_{\alpha}(s)\coloneqq\e^{\alpha s}k(s)$.
Hence, by \prettyref{lem:positive kernels primitive} we have 
\[
t\Im\left(\hat{k_{\alpha}}(\i t+\rho)-\varepsilon\frac{1}{\i t+\rho+1}\right)\leq0\quad(t\in\mathbb{R},\rho\geq0).
\]
Moreover, using $\hat{k_{\alpha}}(z)=\hat{k}(z-\alpha)$ for each
$z\in\mathbb{C}_{\Re\geq0}$, we derive 
\[
t\Im\hat{k}(\i t+\rho-\alpha)\leq-\varepsilon\frac{t^{2}}{t^{2}+(\rho+1)^{2}}\quad(t\in\mathbb{R},\rho\geq0),
\]
which gives the assertion. \end{proof}
\begin{lem}
\label{lem:positive_kernels_invertible}Let $k\in L_{1}(\mathbb{R}_{\geq0};\mathbb{R})$
such that $K(t)\coloneqq\intop_{t}^{\infty}k(s)\,\dd s$ for $t\geq0$
is a kernel of positive type. Then\nomenclature[O_170]{$A^\ast$}{the adjoint of an operator $A$.}\nomenclature[Z_000]{$z^\ast$}{the conjugate of a complex number $z$.}
\[
\Re\left(z^{\ast}\hat{k}(z)\right)\leq\frac{\Re z}{\sqrt{2\pi}}\intop_{0}^{\infty}k(s)\,\dd s\quad(z\in\mathbb{C}_{\Re>0}).
\]
If in addition, $\intop_{0}^{\infty}k(s)\,\dd s<1,$ then for each
$\rho>0$ there exists $c>0$ such that 
\[
|1-\sqrt{2\pi}\hat{k}(z)|\geq c\quad(z\in\mathbb{C}_{\Re\geq\rho}).
\]
\end{lem}
\begin{proof}
Since 
\[
\hat{K}(z)=\frac{1}{z}\left(\frac{1}{\sqrt{2\pi}}\intop_{0}^{\infty}k(s)\mbox{ d}s-\hat{k}(z)\right)\quad(z\in\mathbb{C}_{\Re>0}),
\]
we derive from \prettyref{prop:positive_kernels} 
\begin{align*}
0 & \leq\Re\frac{1}{z}\left(\frac{1}{\sqrt{2\pi}}\intop_{0}^{\infty}k(s)\mbox{ d}s-\hat{k}(z)\right)\\
 & =\frac{1}{|z|^{2}}\left(\frac{\Re z}{\sqrt{2\pi}}\intop_{0}^{\infty}k(s)\mbox{ d}s-\Re\left(z^{\ast}\hat{k}(z)\right)\right)
\end{align*}
for $z\in\mathbb{C}_{\Re>0},$ which shows the first assertion. Assume
now that $\intop_{0}^{\infty}k(s)\mbox{ d}s<1$ and let $\rho>0$.
We claim that $\sqrt{2\pi}\hat{k}(z)\ne1$ for each $z\in\mathbb{C}_{\Re>0}.$
Indeed, if $\sqrt{2\pi}\hat{k}(z)=1$ for some $z\in\mathbb{C}_{\Re>0},$
we would have 
\[
\frac{\Re z}{\sqrt{2\pi}}=\Re\left(z^{\ast}\hat{k}(z)\right)\leq\frac{\Re z}{\sqrt{2\pi}}\intop_{0}^{\infty}k(s)\mbox{ d}s,
\]
according to what we have shown above. This however contradicts $\intop_{0}^{\infty}k(s)\mbox{ d}s<1$
and thus, $\sqrt{2\pi}\hat{k}(z)\ne1$ for $z\in\mathbb{C}_{\Re>0}.$
Since $|k|_{L_{1,\nu}}\to0$ as $\nu\to\infty,$ we find $\nu>0$
such that 
\[
|1-\sqrt{2\pi}\hat{k}(z)|\geq1-|k|_{L_{1,\nu}}\geq\frac{1}{2}\quad(z\in\mathbb{C}_{\Re>\nu}).
\]
Moreover, by \prettyref{prop:Riemann-Lebesgue_professional} there
exists some $M>0$ such that 
\[
|1-\sqrt{2\pi}\hat{k}(\i t+\mu)|\geq\frac{1}{2}\quad(|t|>M,\rho\leq\mu\leq\nu).
\]
By the continuity of $\hat{k}$ and the fact that $\sqrt{2\pi}\hat{k}(z)\ne1$
for each $z\in\mathbb{C}_{\Re>0},$ we infer that there is some $\tilde{c}>0$
such that 
\[
|1-\sqrt{2\pi}\hat{k}(\i t+\mu)|\geq\tilde{c}\quad(|t|\leq M,\rho\leq\mu\leq\nu).
\]
Thus, the second assertion follows by setting $c\coloneqq\min\{\tilde{c},\frac{1}{2}\}$. 
\end{proof}
We can now prove that the kernels considered in \cite{Cannarsa2011}
satisfy \prettyref{cond:cond_kernel_sec_order}.
\begin{prop}
\label{prop:kernel_sforza}Let $k\in L_{1,-\alpha}(\mathbb{R}_{\geq0};\mathbb{R})$
such that $K_{\alpha}(t)\coloneqq\intop_{t}^{\infty}\e^{\alpha s}k(s)\,\dd s$
for $t\geq0$ defines a kernel of strict positive type. Moreover,
assume that $\intop_{0}^{\infty}k(s)\,\dd s<1$. Then $k$ satisfies
\prettyref{cond:cond_kernel_sec_order}.\end{prop}
\begin{proof}
Obviously, $k$ satisfies \prettyref{item:selfadjoint} and \prettyref{item:commute}
since it is assumed to be real-valued. By \prettyref{lem:strict_positive_kernel}
there is some $\varepsilon>0$ such that 
\[
t\Im\hat{k}(\i t+\rho)\leq-\varepsilon\frac{t^{2}}{t^{2}+(\rho+\alpha+1)^{2}}\quad(t\in\mathbb{R},\rho\geq-\alpha).
\]
Hence, setting $g_{\delta}(\rho)\coloneqq\varepsilon\frac{\delta^{2}}{\delta^{2}+(\rho+\alpha+1)^{2}}$
for $\rho\geq-\alpha$ and $\delta>0,$ we infer 
\[
t\Im\langle\hat{k}(\i t+\rho)x|x\rangle_{H_{0}}=-t\Im\hat{k}(\i t+\rho)|x|_{H_{0}}^{2}\geq g_{\delta}(\rho)|x|_{H_{0}}^{2}\quad(x\in H_{0},|t|\geq\delta,\rho\geq-\alpha),
\]
which is \prettyref{item:pos_def}. Thus, it remains to show \prettyref{item:invertible}
and \prettyref{item:pos_def_Re}. For doing so, we show that for each
$\mu>-\alpha$ the kernel $K_{\mu}(t)\coloneqq\intop_{t}^{\infty}\e^{-\mu s}k(s)\mbox{ d}s$
is of positive type. First we note that 
\[
t\Im\hat{k}(\i t+\rho)\leq0\quad(\rho\geq-\mu).
\]
Moreover, $K_{\mu}\in L_{1}(\mathbb{R}_{\geq0})$ since $k\in L_{1,-\alpha}(\mathbb{R}_{\geq0})$
and hence, $K_{\mu}$ is of positive type according to \prettyref{lem:positive kernels primitive}.
Using the continuity of 
\[
\varepsilon\mapsto\intop_{0}^{\infty}\e^{\varepsilon t}k(t)\mbox{ d}t,
\]
 we find some $0<\varepsilon<\alpha$ such that $\intop_{0}^{\infty}\e^{\varepsilon t}k(t)\mbox{ d}t<1$.
Let $0<\mu<\varepsilon.$ Since $K_{\varepsilon}$ is of positive
type, we can apply \prettyref{lem:positive_kernels_invertible} to
find $c>0$ such that 
\[
|1-\sqrt{2\pi}\hat{k}(z)|\geq c\quad(z\in\mathbb{C}_{\Re\geq-\mu}),
\]
which implies \prettyref{item:invertible}. Finally, using that $K_{0}$
is of positive type, we obtain by \prettyref{lem:positive_kernels_invertible}
\[
\Re\left(z^{\ast}\hat{k}(z)\right)\leq\frac{\Re z}{\sqrt{2\pi}}\intop_{0}^{\infty}k(s)\mbox{ d}s\quad(z\in\mathbb{C}_{\Re>0}).
\]
 Thus, for $\rho_{0}>0$ we have that 
\begin{align*}
\Re\langle(\i t+\rho)(1-\sqrt{2\pi}\hat{k}(\i t+\rho)^{\ast})x|x\rangle_{H_{0}} & =\left(\rho-\sqrt{2\pi}\Re\left((-\i t+\rho)\hat{k}(\i t+\rho)\right)\right)|x|_{H_{0}}^{2}\\
 & \geq\left(\rho-\rho\intop_{0}^{\infty}k(s)\mbox{ d}s\right)|x|_{H_{0}}^{2}\\
 & \geq\rho_{0}\left(1-\intop_{0}^{\infty}k(s)\mbox{ d}s\right)|x|_{H_{0}}^{2},
\end{align*}
for each $x\in H_{0},t\in\mathbb{R},\rho\geq\rho_{0},$ which shows
\prettyref{item:pos_def_Re}. 
\end{proof}

\section{Notes}

The main idea for the exponential stability of evolutionary problems
is to study the Fourier-Laplace transformed solution operator, which
is an $\mathcal{H}^{\infty}$-function on some right half plane and
to look for analytic and bounded extensions on a bigger right half
plane containing the imaginary axis. This idea is not new and was
broadly applied, especially in the framework of $C_{0}$-semigroups.
We just mention the famous Gearhart-Prüß Theorem \cite{Pruss1984},
which states that for a $C_{0}$-semigroup on a Hilbert space the
growth bound and the abscissa of boundedness coincide. The main reason,
why this approach works is that we are dealing with Hilbert space
valued functions, since only in this case the Fourier transform (or
the Fourier-Laplace transform) becomes unitary, see \cite{Kwapien1972}.
Indeed, the Gearhart-Prüß Theorem is false for general Banach spaces,
see for instance the example in \cite{Greiner1981}, where a semigroup
with growth bound $0$ and abscissa of boundedness $-\infty$ is given.
However, under additional assumptions on the semigroup and/or the
underlying Banach space, one can show the equality of growth and abscissa
of boundedness. Examples are: eventually norm continuous semigroups
\cite[Chapter V, Theorem 1.10]{engel2000one}, positive semigroups
on ordered Banach spaces with normal cone \cite{Neubrander1986} (see
also \cite[Theorem 5.3.1]{ABHN_2011}) and positive semigroups on
$L_{p}(\Omega)$ for a $\sigma$-finite measure space $\Omega$ \cite{Weis1995}
(see also \cite[Theorem 5.3.6]{ABHN_2011}). We emphasize once again
that our notion of exponential stability (which means that the abscissa
of boundedness is negative) does not yield an exponential decay for
the solutions of evolutionary problems unless the right-hand side
is more regular. So, the main point in the cited theorems is that
for suitable Cauchy problems, where one can associate a $C_{0}$-semigroup,
our notion of exponential stability indeed yields the exponential
decay of the solution. We will address the question whether a similar
result holds for a class of evolutionary problems in the next chapter.
It should be noted that our notion of exponential stability defined
by the invariance of suitable spaces has a counterpart in the theory
of $C_{0}$-semigroups, namely Datko's Theorem (\cite{Datko1972},
see also \cite[Theorem 5.1.2]{ABHN_2011}). This theorem states that
a $C_{0}$-semigroup with generator $A$ on a Banach space $X$ is
exponentially stable if and only if all solutions $u$ of the corresponding
inhomogeneous Cauchy-Problem 
\[
\partial_{t}u=Au+f,\quad u(0)=0,
\]
belong to $L_{p}(\mathbb{R}_{\geq0};X)$ providing that $f\in L_{p}(\mathbb{R}_{\geq0};X)$
for some (or equivalently all) $p\in[1,\infty[.$ In other words,
the associated solution operator leaves the space $L_{p}(\mathbb{R}_{\geq0};X)$
invariant.

We remark that exponential stability is the perhaps strongest and
simplest notion of stability for evolutionary problems. It would be
interesting to study other notions of stability like polynomial decay
of solutions or simply their convergence to $0$. In the framework
of $C_{0}$-semigroups, results in these directions are already known.
We mention the famous Arendt-Batty-Lyubich-V\~{u} Theorem (see \cite{Arendt1988}
and \cite{Lyubich1988}) for the strong stability of semigroups, i.e.
$T(t)\to0$ strongly as $t\to\infty.$ For the polynomial stability
of semigroups on Hilbert spaces we refer to \cite{Borichev2010},
where the polynomial stability is characterized in terms of resolvent
estimates for the generator. A similar result linking resolvent estimates
of the generator and the asymptotic behaviour of the semigroup in
a Banach space setting is the Theorem of Batty-Duyckarts, \cite{Batty2008}.
The question, which arises naturally is whether those or similar results
can be carried over to evolutionary problems. This is not known so
far and worthy to be studied in future.

\newpage{}

$\:$

\thispagestyle{empty}

\chapter{Initial conditions for evolutionary problems\label{chap:Initial-conditions-for}}

In this last chapter we address the question of how to formulate initial
value problems in the framework of evolutionary problems. As evolutionary
problems are operator equations, which should hold on the whole real
line as ``time horizon'', we need to inspect how initial conditions
could be imposed in this setting. It turns out that particular distributional
right-hand sides and the causality of the solution operators can be
used to formulate initial conditions (see e.g. \cite{Picard_McGhee}).
The distributions, which are needed to formulate initial value problems,
will be introduced in the first section of this chapter. Furthermore,
in the theory of delay equations, it is common to impose histories
instead of initial values (see e.g. \cite{Hale1971,Webb1976,Batkai_2005}).
Since problems with memories are also covered by evolutionary problems,
one needs to inspect how histories can be formulated within the present
framework and what is a suitable class of histories. We answer these
questions in the second section of this chapter. The third section
is devoted to the regularity of initial value problems. In particular,
we derive conditions on the material law $M$ and the operator $A$
that allows us to associate a strongly continuous semigroup to the
problem under consideration. For the theory of strongly continuous
semigroups we refer to the monographs \cite{engel2000one,ABHN_2011}.
Moreover, we show that for those evolutionary problems, where a semigroup
can be associated with, the exponential stability presented in \prettyref{chap:Exponential-stability-for}
yields the classical exponential stability, as it is introduced in
the theory of semigroups under suitable assumptions on the material
law. In the last section we discuss several examples and study their
regularity and exponential stability.

\section{Extrapolation spaces }

Throughout this section let $H$ be a Hilbert space. In this section
we recall the concept of extrapolation spaces, define the space $H_{\rho}^{-1}(\mathbb{R};H)$
and discuss some of its properties. We begin with the definition of
extrapolation spaces associated with an operator $C$.
\begin{defn*}
Let $C:D(C)\subseteq H\to H$ be a densely defined closed linear operator.
We assume that $C$ is boundedly invertible. Then we define 
\[
H^{1}(C)\coloneqq\left(D(C),|C\cdot|_{H}\right),
\]
which, due to the closedness of $C$, is again a Hilbert space with
inner product
\[
\langle x|y\rangle_{H^{1}(C)}\coloneqq\langle Cx|Cy\rangle_{H}.
\]
Moreover, we define 
\[
H^{-1}(C)\coloneqq\widetilde{(H,|C^{-1}\cdot|_{H})},
\]
i.e. the completion of $H$ with respect to the norm $x\mapsto|C^{-1}x|_{H}.$
We call $(H^{1}(C),H,H^{-1}(C))$ the \emph{Gelfand triple associated
with $C$.}\end{defn*}
\begin{rem}
Note that by definition $C:H^{1}(C)\to H$ is an isometry and onto
by assumption, hence a unitary operator. Moreover 
\[
C:D(C)\subseteq H\to H^{-1}(C)
\]
is isometric as well and has a dense range by assumption. Thus, there
is a unique extension of $C$, which we will again denote by $C$,
such that 
\[
C:H\to H^{-1}(C)
\]
is unitary. We note that we can continue the process of defining extrapolation
spaces by using powers of $C$. This results in so-called Sobolev
chains, or more general, Sobolev lattices, see \cite[Chapter 2]{Picard_McGhee}. 
\end{rem}
The following lemma provides another way of defining the space $H^{-1}(C)$.
\begin{lem}
\label{lem:H^-1 as dual} Let $C:D(C)\subseteq H\to H$ densely defined
closed linear and boundedly invertible. Then so is $C^{\ast}$ and
\nomenclature[B_140]{$X'$}{the dual space of a space $X$.} 
\begin{align*}
U:H^{-1}(C) & \to H^{1}(C^{\ast})'\\
x & \mapsto\left(y\mapsto\langle C^{-1}x|C^{\ast}y\rangle_{H}\right)
\end{align*}
is unitary, if we equip $H^{1}(C^{\ast})$ with the linear structure
\[
(\lambda\varphi+\psi)(y)\coloneqq\lambda^{\ast}\varphi(y)+\psi(y)\quad(\lambda\in\mathbb{C},\varphi,\psi\in H^{1}(C^{\ast})',y\in H^{1}(C^{\ast})).
\]
\end{lem}
\begin{proof}
It is clear that $C^{\ast}$ is densely defined closed and linear.
Moreover, $\left(C^{\ast}\right)^{-1}=\left(C^{-1}\right)^{\ast},$
which yields that $C^{\ast}$ is boundedly invertible. Let now $w,x\in H^{-1}(C)$
and $\lambda\in\mathbb{C}.$ Then 
\begin{align*}
U(\lambda x+w)(y) & =\langle C^{-1}(\lambda x+w)|C^{\ast}y\rangle_{H}\\
 & =\lambda^{\ast}U(x)(y)+U(w)(y)\\
 & =(\lambda U(x)+U(w))(y),
\end{align*}
for each $y\in H^{1}(C^{\ast}),$ which shows the linearity of $U$.
Moreover, we estimate 
\[
\|U(x)\|=\sup_{|y|_{H^{1}(C^{\ast})}=1}\left|\langle C^{-1}x|C^{\ast}y\rangle\right|_{H}\leq|C^{-1}x|_{H}=|x|_{H^{-1}(C)}
\]
and 
\begin{align*}
|x|_{H^{-1}(C)}^{2} & =\langle C^{-1}x|C^{-1}x\rangle_{H}=U(x)(\left(C^{\ast}\right)^{-1}C^{-1}x)\leq\|U(x)\||C^{-1}x|_{H}=\|U(x)\||x|_{H^{-1}(C)},
\end{align*}
which gives the isometry of $U$. For showing that $U$ is onto, let
$\varphi\in H^{1}(C^{\ast})'.$ By the Riesz representation theorem
there is $u\in H^{1}(C^{\ast})$ such that 
\[
\varphi(y)=\langle u|y\rangle_{H^{1}(C^{\ast})}=\langle C^{\ast}u|C^{\ast}y\rangle_{H}\quad(y\in H^{1}(C^{\ast})).
\]
Hence, setting $x\coloneqq CC^{\ast}u\in H^{-1}(C),$ we obtain $U(x)=\varphi.$\end{proof}
\begin{rem}
The latter lemma also gives that the inner product in $H$ has a natural
extension to a continuous sesquilinearform on $H^{-1}(C)\times H^{1}(C^{\ast}).$ 
\end{rem}
Further on, we do not distinguish between the spaces $H^{-1}(C)$
and $H^{1}(C^{\ast})'.$
\begin{lem}
\label{lem:H^-1 independent}Let $C:D(C)\subseteq H\to H$ densely
defined closed linear and let $\lambda,\mu\in\rho(C).$ Then 
\begin{align*}
H^{1}(C-\lambda) & \cong H^{1}(C-\mu)\\
H^{-1}(C-\lambda) & \cong H^{-1}(C-\mu).
\end{align*}
\end{lem}
\begin{proof}
Obviously $H^{1}(C-\lambda)=D(C)=H^{1}(C-\mu)$ as a set. We need
show that $\mathrm{id}:H^{1}(C-\lambda)\to H^{1}(C-\mu)$ is continuous.
The latter however is clear, since for $x\in D(C)$ we have that 
\[
|x|_{H^{1}(C-\mu)}=|(C-\mu)x|_{H}\leq|(C-\lambda)x|_{H}+|\lambda-\mu||x|_{H}\leq(1+|\lambda-\mu|\|(C-\lambda)^{-1}\|)|x|_{H^{1}(C-\lambda)}.
\]
By interchanging $\lambda$ and $\mu$ in the latter computation,
we infer $H^{1}(C-\lambda)\cong H^{1}(C-\mu)$. By what we have shown
so far, we get that $H^{1}(C^{\ast}-\lambda^{\ast})\cong H^{1}(C^{\ast}-\mu^{\ast}),$
since $\lambda^{\ast},\mu^{\ast}\in\rho(C^{\ast}).$ Thus, $H^{1}(C^{\ast}-\lambda^{\ast})'\cong H^{1}(C^{\ast}-\mu^{\ast})'$
and hence, $H^{-1}(C-\lambda)\cong H^{-1}(C-\mu)$ by \prettyref{lem:H^-1 as dual}. \end{proof}
\begin{example}
Let $\rho\in\mathbb{R}\setminus\{0\}$ and set $C=\partial_{0,\rho},$
which is a densely defined closed linear boundedly invertible operator
on $H_{\rho}(\mathbb{R};H)$. We set 
\begin{align*}
H_{\rho}^{1}(\mathbb{R};H) & \coloneqq H^{1}(\partial_{0,\rho}),\\
H_{\rho}^{-1}(\mathbb{R};H) & \coloneqq H^{-1}(\partial_{0,\rho}).
\end{align*}
Moreover, we set 
\begin{align*}
H_{0}^{1}(\mathbb{R};H) & \coloneqq H^{1}(\partial_{0,0}+1),\\
H_{0}^{-1}(\mathbb{R};H) & \coloneqq H^{-1}(\partial_{0,0}+1).
\end{align*}
For $\rho>0$ we aim to compute $\partial_{0,\rho}\chi_{\mathbb{R}_{\geq t}}x\in H_{\rho}^{-1}(\mathbb{R};H)$
for $t\in\mathbb{R},x\in H.$ For doing so, let $y\in C_{c}^{\infty}(\mathbb{R};H)\subseteq H^{1}(\partial_{0,\rho}^{\ast}).$
Then 
\begin{align*}
\langle\partial_{0,\rho}\chi_{\mathbb{R}_{\geq t}}x|y\rangle_{H^{-1}(\partial_{0,\rho})\times H^{1}(\partial_{0,\rho}^{\ast})} & =\langle\chi_{\mathbb{R}_{\geq t}}x|\partial_{0,\rho}^{\ast}y\rangle_{H_{\rho}(\mathbb{R};H)}\\
 & =\intop_{t}^{\infty}\langle x|\left(-y'(s)+2\rho y(s)\right)\rangle_{H}\e^{-2\rho s}\mbox{ d}s\\
 & =\langle x|y(t)\rangle_{H}\e^{-2\rho t}
\end{align*}
and thus, by the density of $C_{c}^{\infty}(\mathbb{R};H)$ in $H^{1}(\partial_{0,\rho}^{\ast})$
we get 
\[
\partial_{0,\rho}\chi_{\mathbb{R}_{\geq t}}x=\e^{-2\rho t}\delta_{t}x,
\]
where $\delta_{t}$ denotes the Dirac-Delta distribution in $t$.
\end{example}
Next, we extend the Fourier-Laplace transform to the space $H_{\rho}^{-1}(\mathbb{R};H).$
\begin{prop}
Let $\rho\in\mathbb{R}\setminus\{0\}$. Then $\i\m+\rho:D(\m)\subseteq L_{2}(\mathbb{R};H)\to L_{2}(\mathbb{R};H)$
is boundedly invertible and 
\[
\mathcal{L}_{\rho}:H_{\rho}(\mathbb{R};H)\subseteq H_{\rho}^{-1}(\mathbb{R};H)\to H^{-1}(\i\m+\rho)
\]
has a unitary extension. Moreover, the operator 
\[
\mathcal{F}:L_{2}(\mathbb{R};H)\subseteq H_{0}^{-1}(\mathbb{R};H)\to H^{-1}(\i\m+1)
\]
has a unitary extension. \end{prop}
\begin{proof}
Let $\rho\ne0$. Since $\i\m$ is skew-selfadjoint, we infer the bounded
invertibility of $\i\m+\rho.$ Moreover, we have that 
\begin{align*}
|\mathcal{L}_{\rho}f|_{H^{-1}(\i\m+\rho)} & =|(\i\m+\rho)^{-1}\mathcal{L}_{\rho}f|_{L_{2}(\mathbb{R};H)}\\
 & =|\mathcal{L}_{\rho}\partial_{0,\rho}^{-1}f|_{L_{2}(\mathbb{R};H)}\\
 & =|\partial_{0,\rho}^{-1}f|_{H_{\rho}(\mathbb{R};H)}\\
 & =|f|_{H_{\rho}^{-1}(\mathbb{R};H)}
\end{align*}
for each $f\in H_{\rho}(\mathbb{R};H).$ Since $\mathcal{L}_{\rho}\left[H_{\rho}(\mathbb{R};H)\right]=L_{2}(\mathbb{R};H)\subseteq H^{-1}(\i\m+\rho)$
is dense, the assertion follows. The proof for $\mathcal{F}$ is completely
analogous and is therefore omitted.\end{proof}
\begin{prop}
Let $A:D(A)\subseteq H_{0}\to H_{1}$ be a densely defined closed
and linear operator between two Hilbert spaces $H_{0}$ and $H_{1}$.
Then 
\[
A:D(A)\subseteq H_{0}\to H^{-1}(|A^{\ast}|+1)
\]
is bounded and thus, it has a unique continuous extension to $H_{0}$. \end{prop}
\begin{proof}
We use the relation (see e.g. \cite[p. 197 ff.]{Weidmann}) 
\[
A|A|x=|A^{\ast}|Ax
\]
for all elements $x\in D(A^{\ast}A).$ Then we have for $x\in D(A^{\ast}A)$
\begin{align*}
|Ax|_{H^{-1}(|A^{\ast}|+1)} & =\left|(|A^{\ast}|+1)^{-1}Ax\right|_{H_{1}}\\
 & =\left|A\left(|A|+1\right)^{-1}x\right|_{H_{1}}\\
 & =\left||A|\left(|A|+1\right)^{-1}x\right|_{H_{0}}\\
 & \leq|x|_{H_{0}}.
\end{align*}
This gives the assertion, since $D(A^{\ast}A)$ is dense in $H_{0}$. \end{proof}
\begin{lem}
\label{lem:support_H^-1}Let $\rho>0,t\in\mathbb{R}$ and $f\in H_{\rho}^{-1}(\mathbb{R};H).$
Then $\spt f\subseteq\mathbb{R}_{\geq t}$ (here we mean the usual
support of a distribution) if and only if $\spt\partial_{0,\rho}^{-1}f\subseteq\mathbb{R}_{\geq t}.$\end{lem}
\begin{proof}
Assume that $\spt f\subseteq\mathbb{R}_{\geq t}.$ By continuous extension
we obtain 
\[
\langle f|g\rangle_{H^{-1}(\partial_{0,\rho})\times H^{1}(\partial_{0,\rho}^{\ast})}=0
\]
for each $g\in H^{1}(\partial_{0,\rho})$ with $\spt g\subseteq\mathbb{R}_{<t}.$
Let now $\varphi\in C_{c}^{\infty}(\mathbb{R};H)$ with $\spt\varphi\subseteq\mathbb{R}_{<t}.$
Then an easy computation yields that 
\[
\left(\left(\partial_{0,\rho}^{\ast}\right)^{-1}\varphi\right)(s)=\intop_{s}^{\infty}\varphi(r)\e^{-2\rho(r-s)}\mbox{ d}r\quad(s\in\mathbb{R})
\]
and consequently $\spt\left(\partial_{0,\rho}^{\ast}\right)^{-1}\varphi\subseteq\mathbb{R}_{<t}.$
Hence, we obtain 
\[
\langle\partial_{0,\rho}^{-1}f|\varphi\rangle_{H_{\rho}(\mathbb{R};H)}=\langle f|\left(\partial_{0,\rho}^{\ast}\right)^{-1}\varphi\rangle_{H^{-1}(\partial_{0,\rho})\times H^{1}(\partial_{0,\rho}^{\ast})}=0,
\]
which gives $\spt\partial_{0,\rho}^{-1}f\subseteq\mathbb{R}_{\geq t}.$
If on the other hand $\spt\partial_{0,\rho}^{-1}f\subseteq\mathbb{R}_{\geq t},$
we compute for each $\varphi\in C_{c}^{\infty}(\mathbb{R};H)$ with
$\spt\varphi\subseteq\mathbb{R}_{<t}$
\[
\langle f|\varphi\rangle_{H^{-1}(\partial_{0,\rho})\times H^{1}(\partial_{0,\rho}^{\ast})}=\langle\partial_{0,\rho}^{-1}f|\partial_{0,\rho}^{\ast}\varphi\rangle_{H_{\rho}(\mathbb{R};H)}=0,
\]
 since $\spt\partial_{0,\rho}^{\ast}\varphi\subseteq\spt\varphi\subseteq\mathbb{R}_{<t}.$
\end{proof}
Our goal is to extend the solution theory for evolutionary problems
to the extrapolation space $H_{\rho}^{-1}(\mathbb{R};H).$ For doing
so, we need the following auxiliary results.
\begin{prop}
\label{prop:continuation_fcts_of_partial_0}Let $\rho\in\mathbb{R}$
and $F:\left\{ \i t+\rho\,;\, t\in\mathbb{R}\right\} \to L(H)$ be
bounded and strongly measurable. Then $F(\partial_{0,\rho})$ leaves
the spaces $H_{\rho}^{1}(\mathbb{R};H)$ invariant and $F(\partial_{0,\rho})\in L(H_{\rho}^{1}(\mathbb{R};H))$.
Moreover, 
\[
\partial_{0,\rho}F(\partial_{0,\rho})g=F(\partial_{0,\rho})\partial_{0,\rho}g\quad(g\in H_{\rho}^{1}(\mathbb{R};H))
\]
and the operator 
\[
F(\partial_{0,\rho}):H_{\rho}(\mathbb{R};H)\subseteq H_{\rho}^{-1}(\mathbb{R};H)\to H_{\rho}^{-1}(\mathbb{R};H)
\]
is bounded and thus, it has a unique bounded extension to $H_{\rho}^{-1}(\mathbb{R};H).$
Moreover, we have that 
\[
(\partial_{0,\rho}-c)^{-1}F(\partial_{0,\rho})g=F(\partial_{0,\rho})\left(\partial_{0,\rho}-c\right)^{-1}g
\]
for each $g\in H_{\rho}^{-1}(\mathbb{R};H)$ and $c\in\rho(\partial_{0,\rho}).$ \end{prop}
\begin{proof}
Let $g\in H_{\rho}^{1}(\mathbb{R};H)$ and set $u\coloneqq F(\partial_{0,\rho})g.$
We need to show that $u\in H_{\rho}^{1}(\mathbb{R};H),$ which is
equivalent to 
\[
t\mapsto(\i t+\rho)\left(\mathcal{L}_{\rho}u\right)(t)\in L_{2}(\mathbb{R};H).
\]
Since 
\[
\left(\mathcal{L}_{\rho}u\right)(t)=F(\i t+\rho)\left(\mathcal{L}_{\rho}g\right)(t)
\]
for almost every $t\in\mathbb{R},$ we infer that 
\begin{equation}
\left(\i t+\rho\right)\left(\mathcal{L}_{\rho}u\right)(t)=F(\i t+\rho)(\i t+\rho)\left(\mathcal{L}_{\rho}g\right)(t)\label{eq:H_rho^1 invariant-1}
\end{equation}
for almost every $t\in\mathbb{R}$ and thus, 
\[
\intop_{\mathbb{R}}\left|\left(\i t+\rho\right)\left(\mathcal{L}_{\rho}u\right)(t)\right|^{2}\mbox{ d}t\leq\sup_{s\in\mathbb{R}}\left\Vert F(\i s+\rho)\right\Vert ^{2}\intop_{\mathbb{R}}\left|\left(\i t+\rho\right)\mathcal{L}_{\rho}g(t)\right|^{2}\mbox{ d}t<\infty,
\]
since $g\in H_{\rho}^{1}(\mathbb{R};H).$ Moreover, applying the inverse
Fourier-Laplace transform $\mathcal{L_{\rho}^{\ast}}$ to both sides
of \prettyref{eq:H_rho^1 invariant-1}, we obtain 
\[
\partial_{0,\rho}F(\partial_{0,\rho})g=F(\partial_{0,\rho})\partial_{0,\rho}g.
\]
Consider now $F(\partial_{0,\rho}):H_{\rho}(\mathbb{R};H)\subseteq H_{\rho}^{-1}(\mathbb{R};H)\to H_{\rho}^{-1}(\mathbb{R};H)$
and let $g\in H_{\rho}(\mathbb{R};H),c\in\rho(\partial_{0,\rho}).$
Since 
\[
F(\partial_{0,\rho})g=F(\partial_{0,\rho})\left(\partial_{0,\rho}-c\right)\left(\partial_{0,\rho}-c\right)^{-1}g=\left(\partial_{0,\rho}-c\right)F(\partial_{0,\rho})\left(\partial_{0,\rho}-c\right)^{-1}g,
\]
by what we have shown above, we get 
\begin{equation}
\left(\partial_{0,\rho}-c\right)^{-1}F(\partial_{0,\rho})g=F(\partial_{0,\rho})\left(\partial_{0,\rho}-c\right)^{-1}g.\label{eq:partial_0^-1}
\end{equation}
Thus, using \prettyref{lem:H^-1 independent} we estimate 
\begin{align*}
|F(\partial_{0,\rho})g|_{H_{\rho}^{-1}(\mathbb{R};H)} & \leq C|\left(\partial_{0,\rho}-c\right)^{-1}F(\partial_{0,\rho})g|_{H_{\rho}(\mathbb{R};H)}\\
 & =C|F(\partial_{0,\rho})\left(\partial_{0,\rho}-c\right)^{-1}g|_{H_{\rho}(\mathbb{R};H)}\\
 & \leq C\|F(\partial_{0,\rho})\||\left(\partial_{0,\rho}-c\right)^{-1}g|_{H_{\rho}(\mathbb{R};H)}\\
 & \leq\tilde{C}\|F(\partial_{0,\rho})\||g|_{H_{\rho}^{-1}(\mathbb{R};H)},
\end{align*}
for suitable constants $C,\tilde{C}>0,$ which shows the boundedness
of $F(\partial_{0,\rho})$ on $H_{\rho}^{-1}(\mathbb{R};H)$. The
last equality now follows from \prettyref{eq:partial_0^-1} by continuous
extension. \end{proof}
\begin{rem}
The latter proposition particularly applies to solution operators
of well-posed evolutionary problems. Indeed, let $A:D(A)\subseteq H\to H$
densely defined closed and linear and $M:D(M)\subseteq\mathbb{C}\to L(H)$
a linear material law. Assume that the associated evolutionary problem
is well-posed. Then, for $\rho>s_{0}(M,A)$ we have that 
\[
\left(\overline{\partial_{0,\rho}M(\partial_{0,\rho})+A}\right)^{-1}=F(\partial_{0,\rho}),
\]
where $F(\i t+\rho)\coloneqq\left(\left(\i t+\rho\right)M(\i t+\rho)+A\right)^{-1}$
for $t\in\mathbb{R}$ and hence, \prettyref{prop:continuation_fcts_of_partial_0}
applies.\\
Moreover, the translation operator $\tau_{h}$ on $H_{\rho}(\mathbb{R};H)$
for $\rho\in\mathbb{R}$ is another example, since $\tau_{h}=F(\partial_{0,\rho})$
for $F(\i t+\rho)\coloneqq\e^{(\i t+\rho)h}.$ 
\end{rem}
For later purposes we need a way to compare elements of $H_{\rho}^{-1}(\mathbb{R};H)$
and $H_{\mu}^{-1}(\mathbb{R};H)$ for different values $\mu,\rho\in\mathbb{R}.$
This will be done by the following equivalence relation (see also
\cite{Picard2012_delay_BS}).
\begin{defn*}
Let $f\in H_{\rho}^{-1}(\mathbb{R};H),g\in H_{\mu}^{-1}(\mathbb{R};H)$
for some $\rho,\mu\in\mathbb{R}.$ Then we set 
\[
f\sim g\logeq\exists c<\min\{\rho,\mu\}:\left(\partial_{0,\rho}-c\right)^{-1}f=\left(\partial_{0,\mu}-c\right)^{-1}g.
\]

\end{defn*}
We state the following useful observations.
\begin{lem}
\label{lem:identity on H^-1}Let $\mu,\rho\in\mathbb{R}$ and $f\in H_{\rho}^{-1}(\mathbb{R};H),g\in H_{\mu}^{-1}(\mathbb{R};H).$ 

\begin{enumerate}[(a)]

\item The following statements are equivalent:

\begin{enumerate}[(i)]

\item $f\sim g,$

\item there exists a sequence $(\varphi_{n})_{n\in\mathbb{N}}$ in
$C_{c}^{\infty}(\mathbb{R};H)$ such that $\varphi_{n}\to f$ in $H_{\rho}^{-1}(\mathbb{R};H)$
and $\varphi_{n}\to g$ in $H_{\mu}^{-1}(\mathbb{R};H)$ as $n\to\infty$,

\item For all $c<\min\{\rho,\mu\}$ we have $(\partial_{0,\rho}-c)^{-1}f=(\partial_{0,\mu}-c)^{-1}g.$

\end{enumerate}

\item If $f\in H_{\rho}(\mathbb{R};H),g\in H_{\mu}(\mathbb{R};H)$
then $f\sim g$ if and only if $f=g.$ 

\end{enumerate}\end{lem}
\begin{proof}
\begin{enumerate}[(a)]

\item (i) $\Rightarrow$ (ii): Assume $f\sim g$, i.e. there is some
$c<\min\{\rho,\mu\}$ such that $(\partial_{0,\rho}-c)^{-1}f=(\partial_{0,\mu}-c)^{-1}g.$
In particular, $\left(\partial_{0,\rho}-c\right)^{-1}f\in H_{\rho}(\mathbb{R};H)\cap H_{\mu}(\mathbb{R};H)$
and thus, there exists a sequence $(\psi_{n})_{n\in\mathbb{N}}$ in
$C_{c}^{\infty}(\mathbb{R};H)$ such that $\psi_{n}\to\left(\partial_{0,\rho}-c\right)^{-1}f$
in $H_{\rho}(\mathbb{R};H)$ and $H_{\mu}(\mathbb{R};H)$ as $n\to\infty$
according to \prettyref{lem:density test fct}. We set $\varphi_{n}\coloneqq\psi_{n}'-c\psi_{n}\in C_{c}^{\infty}(\mathbb{R};H).$
Then we have, using \prettyref{lem:H^-1 independent}, 
\begin{align*}
|\varphi_{n}-f|_{H_{\rho}^{-1}(\mathbb{R};H)} & \leq C|(\partial_{0,\rho}-c)^{-1}\varphi_{n}-\left(\partial_{0,\rho}-c\right)^{-1}f|_{H_{\rho}(\mathbb{R};H)}\\
 & =C|\psi_{n}-\left(\partial_{0,\rho}-c\right)^{-1}f|_{H_{\rho}(\mathbb{R};H)}\to0
\end{align*}
and 
\begin{align*}
|\varphi_{n}-g|_{H_{\mu}^{-1}(\mathbb{R};H)} & \leq C|(\partial_{0,\mu}-c)^{-1}\varphi_{n}-\left(\partial_{0,\mu}-c\right)^{-1}g|_{H_{\mu}(\mathbb{R};H)}\\
 & =C|\psi_{n}-\left(\partial_{0,\rho}-c\right)^{-1}f|_{H_{\mu}(\mathbb{R};H)}\to0
\end{align*}
as $n\to\infty$.\\
(ii) $\Rightarrow$ (iii): Let $(\varphi_{n})_{n\in\mathbb{N}}$ in
$C_{c}^{\infty}(\mathbb{R};H)$ be a sequence as in (ii) and $c<\min\{\rho,\mu\}$.
Note that by \prettyref{lem:H^-1 independent} we have that 
\begin{align*}
(\partial_{0,\rho}-c)^{-1}\varphi_{n} & \to\left(\partial_{0,\rho}-c\right)^{-1}f\mbox{ in }H_{\rho}(\mathbb{R};H)\\
(\partial_{0,\mu}-c)^{-1}\varphi_{n} & \to\left(\partial_{0,\mu}-c\right)^{-1}g\mbox{ in }H_{\mu}(\mathbb{R};H)
\end{align*}
as $n\to\infty.$ Choosing a suitable subsequence, we assume without
loss of generality that the convergence holds pointwise almost everywhere.
Since the function $T:\mathbb{C}_{\Re>\min\{\rho,\mu\}}\to\mathbb{C}$
defined by $T(z)\coloneqq(z-c)^{-1}$ is analytic and bounded, we
get by \prettyref{thm:material_law_good} that 
\[
(\partial_{0,\rho}-c)^{-1}\varphi_{n}=(\partial_{0,\mu}-c)^{-1}\varphi_{n}
\]
for each $n\in\mathbb{N}$ and thus, also 
\[
\left(\partial_{0,\rho}-c\right)^{-1}f=\left(\partial_{0,\mu}-c\right)^{-1}g.
\]

(iii) $\Rightarrow$ (i): This is trivial.

\item Let $f\in H_{\rho}(\mathbb{R};H)$ and $g\in H_{\mu}(\mathbb{R};H).$
If $f\sim g$ then there is some $c<\min\{\rho,\mu\}$ such that $(\partial_{0,\rho}-c)^{-1}f=\left(\partial_{0,\mu}-c\right)^{-1}g.$
We prove the following auxiliary result. For $h\in H_{\mu}^{1}(\mathbb{R};H)\cap H_{\rho}^{1}(\mathbb{R};H)$
we have $\partial_{0,\mu}h=\partial_{0,\rho}h.$ Indeed, for $\varphi\in C_{c}^{\infty}(\mathbb{R};H)$
we have that 
\begin{align*}
\intop_{\mathbb{R}}\langle\partial_{0,\rho}h(t)|\varphi(t)\rangle\mbox{ d}t & =\langle\partial_{0,\rho}h|\e^{2\rho\m}\varphi\rangle_{H_{\rho}(\mathbb{R};H)}\\
 & =\langle h|\partial_{0,\rho}^{\ast}\left(\e^{2\rho\m}\varphi\right)\rangle_{H_{\rho}(\mathbb{R};H)}\\
 & =\langle h|-\e^{2\rho\m}\varphi'\rangle_{H_{\rho}(\mathbb{R};H)}\\
 & =\intop_{\mathbb{R}}\langle h(t)|-\varphi'(t)\rangle\mbox{ dt}
\end{align*}
and the same computation for $\mu$ instead of $\rho$ yields 
\[
\intop_{\mathbb{R}}\langle\partial_{0,\rho}h(t)|\varphi(t)\rangle\mbox{ d}t=\intop_{\mathbb{R}}\langle\partial_{0,\mu}h(t)|\varphi(t)\rangle\mbox{ d}t.
\]
Since this holds for each $\varphi\in C_{c}^{\infty}(\mathbb{R};H),$
we infer $\partial_{0,\rho}h=\partial_{0,\mu}h.$ Applying this result
to $h\coloneqq(\partial_{0,\rho}-c)^{-1}f\in H_{\rho}^{1}(\mathbb{R};H)\cap H_{\mu}^{1}(\mathbb{R};H),$
we obtain 
\[
f=(\partial_{0,\rho}-c)h=\left(\partial_{0,\mu}-c\right)h=g.
\]
If on the other hand $f=g,$ then \prettyref{thm:material_law_good}
applied to $T(z)\coloneqq(z-c)^{-1}$ for $z\in\mathbb{C}_{\Re>\max\{\rho,\mu\}}$
and some $c<\max\{\rho,\mu\}$ yields $\left(\partial_{0,\rho}-c\right)^{-1}f=(\partial_{0,\mu}-c)^{-1}g$,
i.e., $f\sim g.$ 

\end{enumerate}\end{proof}
\begin{rem}
The latter lemma in particular implies that $\sim$ is an equivalence
relation on the set $\bigcup_{\rho\in\mathbb{R}}H_{\rho}^{-1}(\mathbb{R};H).$
Indeed, the reflexivity and symmetry are obvious and the transitivity
can easily be obtained by using the equivalence (i)$\Leftrightarrow$(iii)
in \prettyref{lem:identity on H^-1} (a). 
\end{rem}
With the relation $\sim$ at hand, we are able to prove the independence
on the parameter $\rho$ of the operator $(\partial_{0,\rho}M(\partial_{0,\rho})+A)^{-1},$
established as a bounded operator on $H_{\rho}^{-1}(\mathbb{R};H).$
As in \prettyref{sec:Evolutionary-problems}, we formulate this result
for analytic and bounded operator-valued functions.
\begin{prop}
\label{prop:independence of rho H^-1}Let $\rho_{0}\in\mathbb{R}$
and $T:\mathbb{C}_{\Re>\rho_{0}}\to L(H)$ be analytic and bounded.
Let $\rho,\mu\geq\rho_{0}$ and $f\in H_{\rho}^{-1}(\mathbb{R};H),g\in H_{\mu}^{-1}(\mathbb{R};H)$
with $f\sim g$. Then, $T(\partial_{0,\rho})f\sim T(\partial_{0,\mu})g.$ \end{prop}
\begin{proof}
Let $c<\rho_{0}$. Then we have $h\coloneqq(\partial_{0,\rho}-c)^{-1}f=(\partial_{0,\mu}-c)^{-1}g\in H_{\rho}(\mathbb{R};H)\cap H_{\mu}(\mathbb{R};H)$
according to \prettyref{lem:identity on H^-1} (a). Hence, by \prettyref{prop:continuation_fcts_of_partial_0}
and \prettyref{thm:material_law_good} we have that 
\[
(\partial_{0,\rho}-c)^{-1}T(\partial_{0,\rho})f=T(\partial_{0,\rho})h=T(\partial_{0,\mu})h=(\partial_{0,\mu}-c)^{-1}T(\partial_{0,\mu})g,
\]
which shows the assertion. 
\end{proof}
Classical initial value problems are formulated on $\mathbb{R}_{>0}$
as ``time horizon''. Since evolutionary problems use the whole real
line, we need to find a proper way to restrict the ``time horizon''
to the positive reals. The key for doing this is the following lemma.
\begin{lem}
\label{lem:cut_off_mal_anders}Let $\rho>0$ and $t\in\mathbb{R}.$
Consider the operators 
\begin{align*}
\chi_{\mathbb{R}_{\geq t}}(\m):H_{\rho}(\mathbb{R};H) & \to H_{\rho}(\mathbb{R};H)\\
f & \mapsto\left(s\mapsto\chi_{\mathbb{R}_{\geq t}}(s)f(s)\right)
\end{align*}
and 
\begin{align*}
\chi_{\mathbb{R}_{\leq t}}(\m):H_{\rho}(\mathbb{R};H) & \to H_{\rho}(\mathbb{R};H)\\
f & \mapsto\left(s\mapsto\chi_{\mathbb{R}_{\leq t}}(s)f(s)\right).
\end{align*}
Then 
\begin{align*}
\chi_{\mathbb{R}_{\geq t}}(\m)f & =\partial_{0,\rho}\chi_{\mathbb{R}_{\geq t}}(\m)\partial_{0,\rho}^{-1}f-\e^{-2\rho t}\left(\partial_{0,\rho}^{-1}f\right)(t+)\delta_{t},\\
\chi_{\mathbb{R}_{\leq t}}(\m)f & =\partial_{0,\rho}\chi_{\mathbb{R}_{\leq t}}(\m)\partial_{0,\rho}^{-1}f+\e^{-2\rho t}\left(\partial_{0,\rho}^{-1}f\right)(t-)\delta_{t}
\end{align*}
for each $f\in H_{\rho}(\mathbb{R};H).$ Note here, that $\partial_{0,\rho}^{-1}f$
has a continuous representative by \prettyref{prop:Sobolev}.\end{lem}
\begin{proof}
Let $f\in H_{\rho}(\mathbb{R};H)$ and set $F(t)\coloneqq(\partial_{0,\rho}^{-1}f)(t)=\intop_{-\infty}^{t}f(s)\mbox{ d}s$
for $t\in\mathbb{R}.$ For $g\in C_{c}^{\infty}(\mathbb{R};H)$ we
then compute using integration by parts 
\begin{align*}
 & \langle\partial_{0,\rho}\chi_{\mathbb{R}_{\geq t}}(\m)\partial_{0,\rho}^{-1}f|g\rangle_{H^{-1}(\partial_{0,\rho})\times H^{1}(\partial_{0,\rho}^{\ast})}\\
 & =\langle\chi_{\mathbb{R}_{\geq t}}(\m)F|\partial_{0,\rho}^{\ast}g\rangle_{H_{\rho}(\mathbb{R};H)}\\
 & =\intop_{t}^{\infty}\langle F(s)|-g'(s)+2\rho g(s)\rangle_{H}\e^{-2\rho s}\mbox{ d}s\\
 & =\intop_{t}^{\infty}\langle f(s)|g(s)\rangle_{H}\e^{-2\rho s}\mbox{ d}s+\e^{-2\rho t}F(t)\, g(t)\\
 & =\langle\chi_{\mathbb{R}_{\geq t}}(\m)f|g\rangle_{H_{\rho}(\mathbb{R};H)}+\langle\e^{-2\rho t}F(t+)\delta_{t}|g\rangle_{H^{-1}(\partial_{0,\rho})\times H^{1}(\partial_{0,\rho}^{\ast}).}
\end{align*}
The latter gives
\[
\chi_{\mathbb{R}_{\geq t}}(\m)f=\partial_{0,\rho}\chi_{\mathbb{R}_{\geq t}}(\m)\partial_{0,\rho}^{-1}f-\e^{-2\rho t}\left(\partial_{0,\rho}^{-1}f\right)(t+)\delta_{t}.
\]
The equality for $\chi_{\mathbb{R}_{\leq t}}(\m)f$ follows by arguing
analogously.
\end{proof}
The latter representation for the cut-off operators has the advantage
that it can be extended to certain elements in $H_{\rho}^{-1}(\mathbb{R};H)$
in a canonical way.
\begin{defn*}
Let $\rho>0,t\in\mathbb{R}.$ We define 
\begin{align*}
P_{t}:D(P_{t})\subseteq H_{\rho}^{-1}(\mathbb{R};H) & \to H_{\rho}^{-1}(\mathbb{R};H)\\
f & \mapsto\partial_{0,\rho}\chi_{\mathbb{R}_{\geq t}}(\m)\partial_{0,\rho}^{-1}f-\e^{-2\rho t}\left(\partial_{0,\rho}^{-1}f\right)(t+)\delta_{t}
\end{align*}
with%
\footnote{For a function $f\in L_{2,\mathrm{loc}}(\mathbb{R};H)$ we say that
$a\coloneqq f(t+)$ exists for some $t\in\mathbb{R},$ if 
\[
\forall\varepsilon>0\,\exists\delta>0:|f(s)-a|<\varepsilon\quad(s\in]t,t+\delta[\mbox{ a.e.}).
\]
Similarly, we say that $b\coloneqq f(t-)$ exists for some $t\in\mathbb{R},$
if 
\[
\forall\varepsilon>0\,\exists\delta>0:|f(s)-b|<\varepsilon\quad(s\in]t-\delta,t[\mbox{ a.e.}).
\]
} 
\[
D(P_{t})\coloneqq\left\{ f\in H_{\rho}^{-1}(\mathbb{R};H)\,;\,(\partial_{0,\rho}^{-1}f)(t+)\mbox{ exists}\right\} ,
\]
as well as
\begin{align*}
Q_{t}:D(Q_{t})\subseteq H_{\rho}^{-1}(\mathbb{R};H) & \to H_{\rho}^{-1}(\mathbb{R};H)\\
f & \mapsto\partial_{0,\rho}\chi_{\mathbb{R}_{\leq t}}(\m)\partial_{0,\rho}^{-1}f+\e^{-2\rho t}\left(\partial_{0,\rho}^{-1}f\right)(t-)\delta_{t}
\end{align*}
with 
\[
D(Q_{t})\coloneqq\left\{ f\in H_{\rho}^{-1}(\mathbb{R};H)\,;\,(\partial_{0,\rho}^{-1}f)(t-)\mbox{ exists}\right\} .
\]
\end{defn*}
\begin{rem}
By \prettyref{lem:cut_off_mal_anders} we have $P_{t}f=\chi_{\mathbb{R}_{\geq t}}(\m)f$
and $Q_{t}f=\chi_{\mathbb{R}_{\leq t}}(\m)f$ for $f\in H_{\rho}(\mathbb{R};H).$ 
\end{rem}
We collect some useful properties for the so introduced operators
$P_{t}$ and $Q_{t}$.
\begin{prop}
\label{prop:properties_cut_off}Let $\rho>0$ and $s,t\in\mathbb{R}.$

\begin{enumerate}[(a)]

\item We have $\delta_{s}\in D(P_{t})\cap D(Q_{t})$ and 
\[
P_{t}\delta_{s}=\begin{cases}
\delta_{s} & \mbox{ if }s>t,\\
0 & \mbox{ if }s\leq t,
\end{cases}\qquad Q_{t}\delta_{s}=\begin{cases}
0 & \mbox{ if }s\geq t,\\
\delta_{s} & \mbox{ if }s<t.
\end{cases}
\]

\item It holds 
\[
P_{t}P_{s}\subseteq P_{\max\{t,s\}}
\]
with equality if and only if $t\leq s.$

\item It holds 
\[
\tau_{s}P_{t}=P_{t-s}\tau_{s}.
\]

\item For $f\in D(P_{t})\cap D(Q_{t})$ we have that 
\[
f=P_{t}f+Q_{t}f+\e^{-2\rho t}\left(\left(\partial_{0,\rho}^{-1}f\right)(t+)-\left(\partial_{0,\rho}^{-1}f\right)(t-)\right)\delta_{t}.
\]

\item For $f\in H_{\rho}^{-1}(\mathbb{R};H)$ we have that $\spt f\subseteq\mathbb{R}_{\geq t}$
if and only if $f\in D(Q_{t})$ with $Q_{t}f=0.$ Moreover, $\spt f\subseteq\mathbb{R}_{\leq t}$
if and only if $f\in D(P_{t})$ with $P_{t}f=0.$

\end{enumerate}\end{prop}
\begin{proof}
\begin{enumerate}[(a)]

\item Since $\partial_{0,\rho}^{-1}\delta_{s}=\e^{2\rho s}\chi_{\mathbb{R}_{\geq s}},$
we infer $\delta_{s}\in D(P_{t})\cap D(Q_{t}).$ We compute 
\begin{align*}
P_{t}\delta_{s} & =\partial_{0,\rho}\chi_{\mathbb{R}_{\geq t}}(\m)\partial_{0,\rho}^{-1}\delta_{s}-\e^{-2\rho t}\left(\partial_{0,\rho}^{-1}\delta_{s}\right)(t+)\delta_{t}\\
 & =\partial_{0,\rho}\chi_{\mathbb{R}_{\geq t}}(\m)\e^{2\rho s}\chi_{\mathbb{R}_{\geq s}}-\e^{-2\rho t}\left(\e^{2\rho s}\chi_{\mathbb{R}_{\geq s}}\right)(t+)\delta_{t}\\
 & =\e^{2\rho s}\partial_{0,\rho}\chi_{\mathbb{R}_{\geq\max\{t,s\}}}-\e^{-2\rho(t-s)}\chi_{\mathbb{R}_{\geq s}}(t+)\delta_{t}\\
 & =\e^{2\rho(s-\max\{t,s\})}\delta_{\max\{t,s\}}-\e^{-2\rho(t-s)}\chi_{\mathbb{R}_{\geq s}}(t+)\delta_{t}.
\end{align*}
Hence, we obtain 
\[
P_{t}\delta_{s}=\begin{cases}
\delta_{s} & \mbox{ if }s>t,\\
0 & \mbox{ if }s\leq t.
\end{cases}
\]
The equality for $Q_{t}\delta_{s}$ follows by arguing analogously.

\item Let $f\in D(P_{t}P_{s})$. We have that 
\begin{equation}
\partial_{0,\rho}^{-1}P_{s}f=\chi_{\mathbb{R}_{\geq s}}(\m)\partial_{0,\rho}^{-1}f-\left(\partial_{0,\rho}^{-1}f\right)(s+)\chi_{\mathbb{R}_{\geq s}}\label{eq:P_sf}
\end{equation}
and since $\left(\partial_{0,\rho}^{-1}P_{s}f\right)(t+)$ exists,
we infer that $\left(\chi_{\mathbb{R}_{\geq s}}(\m)\partial_{0,\rho}^{-1}f\right)(t+)$
exists, too. Consequently, we compute 
\begin{align*}
P_{t}P_{s}f & =\partial_{0,\rho}\chi_{\mathbb{R}_{\geq t}}(\m)\partial_{0,\rho}^{-1}P_{s}f-\e^{-2\rho t}\left(\partial_{0,\rho}^{-1}P_{s}f\right)(t+)\delta_{t}\\
 & =\partial_{0,\rho}\chi_{\mathbb{R}_{\geq\max\{t,s\}}}(\m)\partial_{0,\rho}^{-1}f-\partial_{0,\rho}\left(\partial_{0,\rho}^{-1}f\right)(s+)\chi_{\mathbb{R}_{\geq\max\{t,s\}}}-\\
 & \quad-\e^{-2\rho t}\chi_{\mathbb{R}_{\geq s}}(t+)\left(\partial_{0,\rho}^{-1}f\right)(t+)\delta_{t}+\e^{-2\rho t}\left(\partial_{0,\rho}^{-1}f\right)(s+)\chi_{\mathbb{R}_{\geq s}}(t+)\delta_{t}\\
 & =\partial_{0,\rho}\chi_{\mathbb{R}_{\geq\max\{t,s\}}}(\m)\partial_{0,\rho}^{-1}f-\e^{-2\rho\max\{t,s\}}\left(\partial_{0,\rho}^{-1}f\right)(s+)\delta_{\max\{t,s\}}\\
 & \quad-\e^{-2\rho t}\chi_{\mathbb{R}_{\geq s}}(t+)\left(\left(\partial_{0,\rho}^{-1}f\right)(t+)-\left(\partial_{0,\rho}^{-1}f\right)(s+)\right)\delta_{t}.
\end{align*}
If $t<s,$ we get 
\begin{align*}
P_{t}P_{s}f & =\partial_{0,\rho}\chi_{\mathbb{R}_{\geq s}}(\m)\partial_{0,\rho}^{-1}f-\e^{-2\rho s}\left(\partial_{0,\rho}^{-1}f\right)(s+)\delta_{s}\\
 & =P_{s}f,
\end{align*}
while for $t\geq s$ we obtain 
\begin{align*}
P_{t}P_{s}f & =\partial_{0,\rho}\chi_{\mathbb{R}_{\geq t}}(\m)\partial_{0,\rho}^{-1}f-\e^{-2\rho t}\left(\partial_{0,\rho}^{-1}f\right)(s+)\delta_{t}-\e^{-2\rho t}\left(\left(\partial_{0,\rho}^{-1}f\right)(t+)-\left(\partial_{0,\rho}^{-1}f\right)(s+)\right)\delta_{t}\\
 & =P_{t}f.
\end{align*}
This shows the asserted inclusion $P_{t}P_{s}\subseteq P_{\max\{t,s\}}.$
We now prove that equality holds if and only if $t\leq s.$ Indeed,
if $t\leq s$ and $f\in D(P_{s})$ we obtain by \prettyref{eq:P_sf}
that
\[
\left(\partial_{0,\rho}^{-1}P_{s}f\right)(t+)=0
\]
and hence, $f\in D(P_{t}P_{s}).$ Assume now that $t>s.$ We give
an example for $f\in D(P_{t})$ such that $f\notin D(P_{t}P_{s}).$
For doing so, we define 
\[
g(x)\coloneqq\chi_{\mathbb{R}_{\geq s}}(x)\sin\left(\frac{1}{x-s}\right).
\]
Then clearly $g\in H_{\rho}(\mathbb{R})$. We set $f\coloneqq\partial_{0,\rho}g\in H_{\rho}^{-1}(\mathbb{R}).$
Then $f\in D(P_{t})$ since $\left(\partial_{0,\rho}^{-1}f\right)(t+)=\sin\left(\frac{1}{t-s}\right).$
However, $f\notin D(P_{s})$ since $\lim_{x\to s+}\sin\left(\frac{1}{x-s}\right)$
does not exist. 

\item Let $f\in H_{\rho}^{-1}(\mathbb{R};H).$ Then we have 
\begin{align*}
f\in D(P_{t}) & \Leftrightarrow\left(\partial_{0,\rho}^{-1}f\right)(t+)\mbox{ exists}\\
 & \Leftrightarrow\left(\tau_{s}\partial_{0,\rho}^{-1}f\right)((t-s)+)\mbox{ exists}\\
 & \Leftrightarrow\tau_{s}f\in D(P_{t-s}).
\end{align*}
Moreover, for $f\in D(P_{t})$ we have that 
\begin{align*}
\tau_{s}P_{t}f & =\tau_{s}\left(\partial_{0,\rho}\chi_{\mathbb{R}_{\geq t}}(\m)\partial_{0,\rho}^{-1}f-\e^{-2\rho t}\left(\partial_{0,\rho}^{-1}f\right)(t+)\delta_{t}\right)\\
 & =\partial_{0,\rho}\tau_{s}\chi_{\mathbb{R}_{\geq t}}(\m)\partial_{0,\rho}^{-1}f-\e^{-2\rho t}\left(\partial_{0,\rho}^{-1}f\right)(t+)\tau_{s}\delta_{t}.
\end{align*}
Using now $\tau_{s}\chi_{\mathbb{R}_{\geq t}}(\m)=\chi_{\mathbb{R}_{\geq t-s}}(\m)\tau_{s}$
as well as 
\begin{align*}
\tau_{s}\delta_{t} & =\e^{2\rho t}\partial_{0,\rho}\tau_{s}\chi_{\mathbb{R}_{\geq t}}\\
 & =\e^{2\rho t}\partial_{0,\rho}\chi_{\mathbb{R}_{\geq t-s}}\\
 & =\e^{2\rho s}\delta_{t-s},
\end{align*}
we derive that 
\begin{align*}
\tau_{s}P_{t}f & =\partial_{0,\rho}\chi_{\mathbb{R}_{\geq t-s}}(\m)\partial_{0,\rho}^{-1}\tau_{s}f-\e^{-2\rho\left(t-s\right)}\left(\partial_{0,\rho}^{-1}\tau_{s}f\right)\left(\left(t-s\right)+\right)\delta_{t-s}\\
 & =P_{t-s}\tau_{s}f.
\end{align*}

\item Let $f\in D(P_{t})\cap D(Q_{t}).$ Then we have that 
\begin{align*}
 & P_{t}f+Q_{t}f\\
 & =\partial_{0,\rho}\chi_{\mathbb{R}_{\geq t}}(\m)\partial_{0,\rho}^{-1}f-\e^{-2\rho t}\left(\partial_{0,\rho}^{-1}f\right)(t+)\delta_{t}+\partial_{0,\rho}\chi_{\mathbb{R}_{\leq t}}(\m)\partial_{0,\rho}^{-1}f+\e^{-2\rho t}\left(\partial_{0,\rho}^{-1}f\right)(t-)\delta_{t}\\
 & =f-\e^{-2\rho t}\left(\left(\partial_{0,\rho}^{-1}f\right)(t+)-\left(\partial_{0,\rho}^{-1}f\right)(t-)\right)\delta_{t}.
\end{align*}

\item Let $f\in H_{\rho}^{-1}(\mathbb{R};H).$ We first assume that
$\spt f\subseteq\mathbb{R}_{\geq t}$. Then by \prettyref{lem:support_H^-1}
we have $\spt\partial_{0,\rho}^{-1}f\subseteq\mathbb{R}_{\geq t}$
and thus 
\[
\left(\partial_{0,\rho}^{-1}f\right)(t-)=0.
\]
Hence, $f\in D(Q_{t})$ and 
\begin{align*}
Q_{t}f & =\partial_{0,\rho}\chi_{\mathbb{R}_{\leq t}}(\m)\partial_{0,\rho}^{-1}f+\e^{-2\rho t}\left(\partial_{0,\rho}^{-1}f\right)(t-)\delta_{t}=0.
\end{align*}
If on the other hand $f\in D(Q_{t})$ with $Q_{t}f=0,$ we compute
for $\varphi\in C_{c}^{\infty}(\mathbb{R};H)$ with $\spt\varphi\subseteq\mathbb{R}_{<t}$
\begin{align*}
\langle f|\varphi\rangle_{H^{-1}(\partial_{0,\rho})\times H^{1}(\partial_{0,\rho}^{\ast})} & =\langle\partial_{0,\rho}^{-1}f|\partial_{0,\rho}^{\ast}\varphi\rangle_{H_{\rho}(\mathbb{R};H)}\\
 & =\langle\chi_{\mathbb{R}_{\leq t}}(\m)\partial_{0,\rho}^{-1}f|\partial_{0,\rho}^{\ast}\varphi\rangle_{H_{\rho}(\mathbb{R};H)}\\
 & =\langle Q_{t}f|\varphi\rangle_{H^{-1}(\partial_{0,\rho})\times H^{1}(\partial_{0,\rho})}-\e^{-2\rho t}\left(\partial_{0,\rho}^{-1}f\right)(t-)\varphi(t)\\
 & =0.
\end{align*}
This gives $\spt f\subseteq\mathbb{R}_{\geq t}.$ \\
Assume now that $\spt f\subseteq\mathbb{R}_{\leq t}$. Consider the
space 
\[
V\coloneqq\left\{ \chi_{\mathbb{R}_{\geq t}}x\,;\, x\in H\right\} \subseteq H_{\rho}(\mathbb{R};H).
\]
Then $V$ is closed and for $g\in H_{\rho}(\mathbb{R};H)$ we have
that 
\begin{align*}
g\in V^{\bot} & \Leftrightarrow\forall x\in H:\intop_{t}^{\infty}\langle g(t)|x\rangle_{H}\e^{-2\rho t}\mbox{ d}t=0\\
 & \Leftrightarrow\intop_{t}^{\infty}g(t)\e^{-2\rho t}\mbox{ d}t=0.
\end{align*}
Let now $g\in V^{\bot}$. Then 
\[
\langle\chi_{\mathbb{R}_{\geq t}}(\m)\partial_{0,\rho}^{-1}f|g\rangle_{H_{\rho}(\mathbb{R};H)}=\langle f|\left(\partial_{0,\rho}^{\ast}\right)^{-1}\chi_{\mathbb{R}_{\geq t}}(\m)g\rangle_{H^{-1}(\partial_{0,\rho})\times H^{1}(\partial_{0,\rho}^{\ast})}
\]
and since 
\[
\left(\left(\partial_{0,\rho}^{\ast}\right)^{-1}\chi_{\mathbb{R}_{\geq t}}(\m)g\right)(s)=\intop_{s}^{\infty}\chi_{\mathbb{R}_{\geq t}}(r)g(r)\e^{-2\rho r}\mbox{ d}r\;\e^{2\rho s}=0
\]
for $s\leq t$, we infer that 
\[
\langle\chi_{\mathbb{R}_{\geq t}}(\m)\partial_{0,\rho}^{-1}f|g\rangle_{H_{\rho}(\mathbb{R};H)}=0.
\]
Hence, $\chi_{\mathbb{R}_{\geq t}}(\m)\partial_{0,\rho}^{-1}f\in V$
and thus, there is $x\in H$ with 
\[
\chi_{\mathbb{R}_{\geq t}}(\m)\partial_{0,\rho}^{-1}f=\chi_{\mathbb{R}_{\geq t}}x.
\]
The latter gives $f\in D(P_{t})$ and 
\[
P_{t}f=\partial_{0,\rho}\chi_{\mathbb{R}_{\geq t}}(\m)\partial_{0,\rho}^{-1}f-\e^{-2\rho t}\delta_{t}x=0.
\]
If $f\in D(P_{t})$ and $P_{t}f=0$ we compute for $\varphi\in C_{c}^{\infty}(\mathbb{R}_{>t};H)$
\begin{align*}
\langle f|\varphi\rangle_{H^{-1}(\partial_{0,\rho})\times H^{1}(\partial_{0,\rho}^{\ast})} & =\langle\partial_{0,\rho}^{-1}f|\partial_{0,\rho}^{\ast}\varphi\rangle_{H_{\rho}(\mathbb{R};H)}\\
 & =\langle\chi_{\mathbb{R}_{\geq t}}(\m)\partial_{0,\rho}^{-1}f|\partial_{0,\rho}^{\ast}\varphi\rangle_{H_{\rho}(\mathbb{R};H)}\\
 & =\langle P_{t}f|\varphi\rangle_{H^{-1}(\partial_{0,\rho})\times H^{1}(\partial_{0,\rho})}+\e^{-2\rho t}\left(\partial_{0,\rho}^{-1}f\right)(t+)\varphi(t)\\
 & =0,
\end{align*}
which completes the proof.

\end{enumerate}
\end{proof}

\section{Admissible initial values and histories}

It is the aim of this section to provide the ``correct'' spaces
for initial values and histories for a given well-posed evolutionary
problem. Throughout, let $H$ a Hilbert space and assume that $M:D(M)\subseteq\mathbb{C}\to L(H)$
is a linear material law and $A:D(A)\subseteq H\to H$ is densely
defined closed and linear, such that the evolutionary problem associated
with $M$ and $A$ is well-posed. We first aim to give the following
initial value problem a meaningful interpretation in the framework
of evolutionary problems:
\begin{align}
\left(\partial_{0,\rho}M(\partial_{0,\rho})+A\right)u & =0\mbox{ on }\mathbb{R}_{>0},\label{eq:problem_pos_times}\\
u|_{\mathbb{R}_{<0}} & =g.\label{eq:history}
\end{align}
Here, $\rho>s_{0}(M,A)$ and $g\in\chi_{\mathbb{R}_{\leq0}}(\m)[H_{\rho}^{1}(\mathbb{R};H)],$
i.e. $g=\chi_{\mathbb{R}_{\leq0}}(\m)w$ for some $w\in H_{\rho}^{1}(\mathbb{R};H)$
and we are seeking for a solution $u\in H_{\rho}^{1}(\mathbb{R};H)$
satisfying \prettyref{eq:problem_pos_times} and \prettyref{eq:history}.
We start by doing some heuristics and ignoring all domain constraints.
This will be used to motivate the definition of the space of admissible
histories $g.$ As $u$ is known on $\mathbb{R}_{<0}$ by \prettyref{eq:history},
we may decompose
\[
u=v+g,
\]
where $v\coloneqq\chi_{\mathbb{R}_{\geq0}}(\m)u\in H_{\rho}(\mathbb{R}_{\geq0};H).$
Then we have 
\begin{equation}
v-\chi_{\mathbb{R}_{\geq0}}g(0-)\in H_{\rho}^{1}(\mathbb{R};H).\label{eq:initial_value_stronger}
\end{equation}
We replace \prettyref{eq:problem_pos_times} by 
\[
P_{0}\left(\partial_{0,\rho}M(\partial_{0,\rho})+A\right)u=0,
\]
where $P_{0}$ is the cut-off operator defined in the previous section.
Using the decomposition of $u$ we infer 
\[
P_{0}\left(\partial_{0,\rho}M(\partial_{0,\rho})+A\right)v=-P_{0}\left(\partial_{0,\rho}M(\partial_{0,\rho})+A\right)g=-P_{0}\partial_{0,\rho}M(\partial_{0,\rho})g,
\]
where we have used $P_{0}A=AP_{0},$ since $A$ just acts in $H$.
The left-hand side in the latter equality can be written as 
\begin{align*}
P_{0}\partial_{0,\rho}M(\partial_{0,\rho})v+Av & =\partial_{0,\rho}\chi_{\mathbb{R}_{\geq0}}(\m)M(\partial_{0,\rho})v-\left(M(\partial_{0,\rho})v\right)(0+)\delta_{0}+Av\\
 & =\left(\partial_{0,\rho}M(\partial_{0,\rho})+A\right)v-\left(M(\partial_{0,\rho})v\right)(0+)\delta_{0}.
\end{align*}
Thus, we end up with an evolutionary problem for the unknown $v$
of the form 
\[
\left(\partial_{0,\rho}M(\partial_{0,\rho})+A\right)v=\left(M(\partial_{0,\rho})v\right)(0+)\delta_{0}-P_{0}\partial_{0,\rho}M(\partial_{0,\rho})g.
\]
We note that the right-hand side also depends on $v$ as the value
of $\left(M(\partial_{0,\rho})v\right)(0+)$ is unknown. We  replace
this term by an arbitrary $x\in H$, which will depend on the history
$g$. Since the evolutionary problem is assumed to be well-posed,
we get that 
\[
v=\left(\partial_{0,\rho}M(\partial_{0,\rho})+A\right)^{-1}\left(x\delta_{0}\right)-\left(\partial_{0,\rho}M(\partial_{0,\rho})+A\right)^{-1}P_{0}\partial_{0,\rho}M(\partial_{0,\rho})g.
\]
Taking condition \prettyref{eq:initial_value_stronger} into account,
we get that 
\[
\left(\partial_{0,\rho}M(\partial_{0,\rho})+A\right)^{-1}\left(x\delta_{0}-P_{0}\partial_{0,\rho}M(\partial_{0,\rho})g\right)-\chi_{\mathbb{R}_{\geq0}}g(0-)\in H_{\rho}^{1}(\mathbb{R};H)
\]
and this condition will be used for the definition of admissible initial
values and histories. In order to provide slightly shorter formulas,
we introduce the notation 
\[
S_{\rho}\coloneqq\left(\overline{\partial_{0,\rho}M(\partial_{0,\rho})+A}\right)^{-1}.
\]
We restrict ourselves to a slightly smaller class of material laws,
which are defined as follows.
\begin{defn*}
Let $M:D(M)\subseteq\mathbb{C}\to L(H)$ be a linear material law.
We call $M$ \emph{regularizing, }if there is $\rho_{0}\in\mathbb{R}_{>0}$
such that $\mathbb{C}_{\Re\geq\rho_{0}}\subseteq D(M)$, $M|_{\mathbb{C}_{\Re\geq\rho_{0}}}$
is bounded and 
\[
\forall\rho\geq\rho_{0},x\in H:\left(M(\partial_{0,\rho})\chi_{\mathbb{R}_{\geq0}}x\right)(0+)\mbox{ exists.}
\]
\end{defn*}
\begin{rem}
Note that $M$ is regularizing if and only if $\partial_{0,\rho}M(\partial_{0,\rho})\chi_{\mathbb{R}_{\geq0}}x\in D(P_{0})$
for each $x\in H,\rho\geq\rho_{0}$ for some $\rho_{0}>0.$ \end{rem}
\begin{lem}
\label{lem:char_regularizing_M}Let $M$ be a linear material law
and set 
\[
b(M)\coloneqq\inf\left\{ \rho_{0}\in\mathbb{R}_{\geq0}\,;\,\mathbb{C}_{\Re\geq\rho_{0}}\subseteq D(M),M|_{\mathbb{C}_{\Re\geq\rho_{0}}}\,\mathrm{bounded}\right\} .
\]
Then $M$ is regularizing if and only if $b(M)<\infty$ and 
\[
\exists\rho>b(M)\,\forall x\in H:\left(M(\partial_{0,\rho})\chi_{\mathbb{R}_{\geq0}}x\right)(0+)\mbox{ exists}.
\]
\end{lem}
\begin{proof}
If $M$ is regularizing, then the assertion holds trivially. Assume
now that $b(M)<\infty$ and 
\[
\exists\rho>b(M)\,\forall x\in H:\left(M(\partial_{0,\rho})\chi_{\mathbb{R}_{\geq0}}x\right)(0+)\mbox{ exists}.
\]
We need to prove that 
\[
\forall\rho>b(M),x\in H:\left(M(\partial_{0,\rho})\chi_{\mathbb{R}_{\geq0}}x\right)(0+)\mbox{ exists}.
\]
However, since $M|_{\mathbb{C}_{\Re>b(M)}}$ is analytic and bounded
on $\mathbb{C}_{\Re\geq b(M)+\varepsilon}$ for each $\varepsilon>0$,
we infer from \prettyref{thm:material_law_good} that $M(\partial_{0,\rho})=M(\partial_{0,\nu})$
on $H_{\rho}(\mathbb{R};H)\cap H_{\nu}(\mathbb{R};H)$ for each $\rho,\nu>b(M).$
Since $\chi_{\mathbb{R}_{\geq0}}x\in\bigcap_{\rho>0}H_{\rho}(\mathbb{R};H),$
the assertion follows.
\end{proof}
Next, we show that for a regularizing material law $M$ we have that
for large enough $\rho$ 
\[
\partial_{0,\rho}M(\partial_{0,\rho})g\in D(P_{0})
\]
for each $g\in\chi_{\mathbb{R}_{\leq0}}(\m)[H_{\rho}^{1}(\mathbb{R};H)]$
. For doing so, we need the following small observation.
\begin{lem}
\label{lem:char_chi_H^1} Let $\rho>0$ and $g\in H_{\rho}(\mathbb{R};H)$
with $\spt g\subseteq\mathbb{R}_{\leq0}$. Then $g\in\chi_{\mathbb{R}_{\leq0}}(\m)\left[H_{\rho}^{1}(\mathbb{R};H)\right]$
if and only if $g(0-)$ exists and $g+\chi_{\mathbb{R}_{\geq0}}g(0-)\in H_{\rho}^{1}(\mathbb{R};H).$ \end{lem}
\begin{proof}
Assume first that $g\in\chi_{\mathbb{R}_{\leq0}}(\m)[H_{\rho}^{1}(\mathbb{R};H)].$
Then there exists $w\in H_{\rho}^{1}(\mathbb{R};H)$ such that $g=\chi_{\mathbb{R}_{\leq0}}(\m)w.$
Since $w$ is continuous by the Sobolev embedding theorem (see \prettyref{prop:Sobolev}),
we infer that $g(0-)=w(0)$ exists. Moreover, we have that 
\begin{align*}
\partial_{0,\rho}\left(g+\chi_{\mathbb{R}_{\geq0}}g(0-)\right) & =\partial_{0,\rho}\chi_{\mathbb{R}_{\leq0}}(\m)w+\delta_{0}g(0-)\\
 & =\partial_{0,\rho}\chi_{\mathbb{R}_{\leq0}}(\m)\partial_{0,\rho}^{-1}\partial_{0,\rho}w+\delta_{0}w(0)\\
 & =Q_{0}\partial_{0,\rho}w\\
 & =\chi_{\mathbb{R}_{\leq0}}(\m)w\in H_{\rho}(\mathbb{R};H)
\end{align*}
and thus, $g+\chi_{\mathbb{R}_{\geq0}}g(0-)\in H_{\rho}^{1}(\mathbb{R};H).$

The reverse implication holds trivially, since $\chi_{\mathbb{R}_{\leq0}}(\m)\left(g+\chi_{\mathbb{R}_{\geq0}}g(0-)\right)=g.$ \end{proof}
\begin{lem}
\label{lem:regularizing}Let $M$ be a regularizing linear material
law and $\rho>b(M)$. Then for each $g\in\chi_{\mathbb{R}_{\leq0}}(\m)[H_{\rho}^{1}(\mathbb{R};H)]$
we have $\partial_{0,\rho}M(\partial_{0,\rho})g\in D(P_{0}).$\end{lem}
\begin{proof}
Let $g\in\chi_{\mathbb{R}_{\leq0}}(\m)[H_{\rho}^{1}(\mathbb{R};H)].$
We need to show that $\left(M(\partial_{0,\rho})g\right)(0+)$ exists.
By \prettyref{lem:char_chi_H^1} we have that $g(0-)$ exists and
$g+\chi_{\mathbb{R}_{\geq0}}g(0-)\in H_{\rho}^{1}(\mathbb{R};H).$
Thus, we have that 
\[
M(\partial_{0,\rho})g=M(\partial_{0,\rho})(g+\chi_{\mathbb{R}_{\geq0}}g(0-))-M(\partial_{0,\rho})\chi_{\mathbb{R}_{\geq0}}g(0-)
\]
and since the first summand on the right-hand side is in $H_{\rho}^{1}(\mathbb{R};H)$
and therefore continuous, and $M$ is regularizing, we infer that
$\left(M(\partial_{0,\rho})g\right)(0+)$ exists. 
\end{proof}
With these results at hand, we are now able to define the history
space. Throughout, we assume that $M$ is regularizing and we set
\begin{align*}
K_{\rho}:\chi_{\mathbb{R}_{\leq0}}(\m)[H_{\rho}^{1}(\mathbb{R};H)] & \to H_{\rho}^{-1}(\mathbb{R};H)\\
g & \mapsto P_{0}\partial_{0,\rho}M(\partial_{0,\rho})g.
\end{align*}

\begin{defn*}
Let $\rho>\max\{s_{0}(M,A),b(M)\}$. We define \nomenclature[B_150]{$\mathrm{His}_\rho(M,A)$}{the space of admissible histories for a material law $M$ and an operator $A$ in $H_\rho(\mathbb{R};H)$.}
\[
\mathrm{His}_{\rho}(M,A)\coloneqq\{g\in\chi_{\mathbb{R}_{\leq0}}(\m)[H_{\rho}^{1}(\mathbb{R};H)]\,;\,\exists x\in H:\: S_{\rho}\left(\delta_{0}x-K_{\rho}g\right)-\chi_{\mathbb{R}_{\geq0}}g(0-)\in H_{\rho}^{1}(\mathbb{R};H)\}
\]
the space of \emph{admissible histories for $M$ and $A$. }Moreover,
we define \nomenclature[B_160]{$\mathrm{IV}_\rho(M,A)$}{the space of admissible initial values for a material law $M$ and an operator $A$ in $H_\rho(\mathbb{R};H)$, the space $\{ g(0-)\, ;\, g\in \mathrm{His}_\rho(M,A) \}$.}
\[
\mathrm{IV}_{\rho}(M,A)\coloneqq\{g(0-)\,;\, g\in\mathrm{His}_{\rho}(M,A)\},
\]
the space of \emph{admissible initial values for $M$ and $A$.}
\end{defn*}
We first show that the element $x$ used in the definition of $\mathrm{His}_{\rho}(M,A)$
is uniquely determined.
\begin{lem}
\label{lem:Gamma}Let $g\in\mathrm{His}_{\rho}(M,A)$ and $x\in H$
such that 
\[
S_{\rho}\left(x\delta_{0}-K_{\rho}g\right)-\chi_{\mathbb{R}_{\geq0}}g(0-)\in H_{\rho}^{1}(\mathbb{R};H).
\]
Then 
\begin{align*}
x & =\left(M(\partial_{0,\rho})\chi_{\mathbb{R}_{\geq0}}g(0-)\right)(0+)\\
 & =\left(M(\partial_{0,\rho})g\right)(0-)-\left(M(\partial_{0,\rho})g\right)(0+).
\end{align*}
\end{lem}
\begin{proof}
Since 
\[
S_{\rho}\left(x\delta_{0}-K_{\rho}g\right)-\chi_{\mathbb{R}_{\geq0}}g(0-)\in H_{\rho}^{1}(\mathbb{R};H),
\]
we obtain 
\[
x\delta_{0}-K_{\rho}g-\left(\partial_{0,\rho}M(\partial_{0,\rho})+A\right)\chi_{\mathbb{R}_{\geq0}}g(0-)\in H_{\rho}(\mathbb{R};H^{-1}(|A^{\ast}|+1)).
\]
Since clearly $A\chi_{\mathbb{R}_{\geq0}}g(0-)\in H_{\rho}(\mathbb{R};H^{-1}(|A^{\ast}|+1))$
we infer 
\[
x\delta_{0}-K_{\rho}g-\partial_{0,\rho}M(\partial_{0,\rho})\chi_{\mathbb{R}_{\geq0}}g(0-)\in H_{\rho}(\mathbb{R};H^{-1}(|A^{\ast}|+1)).
\]
We note that due to causality $\partial_{0,\rho}M(\partial_{0,\rho})\chi_{\mathbb{R}_{\geq0}}g(0-)\in N(Q_{0})$
and hence, since we have that $\partial_{0,\rho}M(\partial_{0,\rho})\chi_{\mathbb{R}_{\geq0}}g(0-)\in D(P_{0})$
by assumption, \prettyref{prop:properties_cut_off} (d) yields 
\begin{align*}
\partial_{0,\rho}M(\partial_{0,\rho})\chi_{\mathbb{R}_{\geq0}}g(0-) & =P_{0}\partial_{0,\rho}M(\partial_{0,\rho})\chi_{\mathbb{R}_{\geq0}}g(0-)+\delta_{0}\left(M(\partial_{0,\rho})\chi_{\mathbb{R}_{\geq0}}g(0-)\right)(0+)\\
 & =K_{\rho}\chi_{\mathbb{R}_{\geq0}}g(0-)+\delta_{0}\left(M(\partial_{0,\rho})\chi_{\mathbb{R}_{\geq0}}g(0-)\right)(0+).
\end{align*}
Thus, we have that 
\begin{align*}
 & x\delta_{0}-K_{\rho}g-\partial_{0,\rho}M(\partial_{0,\rho})\chi_{\mathbb{R}_{\geq0}}g(0-)\\
 & =\delta_{0}(x-\left(M(\partial_{0,\rho})\chi_{\mathbb{R}_{\geq0}}g(0-)\right)(0+))-K_{\rho}\left(g+\chi_{\mathbb{R}_{\geq0}}g(0-)\right)\in H_{\rho}(\mathbb{R};H^{-1}(|A^{\ast}|+1)).
\end{align*}
Since $g+\chi_{\mathbb{R}_{\geq0}}g(0-)\in H_{\rho}^{1}(\mathbb{R};H)$
by \prettyref{lem:char_chi_H^1}, we have that $\partial_{0,\rho}M(\partial_{0,\rho})\left(g+\chi_{\mathbb{R}_{\geq0}}g(0-)\right)\in H_{\rho}(\mathbb{R};H)$
and thus, 
\[
P_{0}K_{\rho}\left(g+\chi_{\mathbb{R}_{\geq0}}g(0-)\right)=\chi_{\mathbb{R}_{\geq0}}(\m)K_{\rho}\left(g+\chi_{\mathbb{R}_{\geq0}}g(0-)\right)\in H_{\rho}(\mathbb{R};H).
\]
This, in turn, implies 
\[
\delta_{0}\left(x-\left(M(\partial_{0,\rho})\chi_{\mathbb{R}_{\geq0}}g(0-)\right)(0+)\right)\in H_{\rho}(\mathbb{R};H^{-1}(|A^{\ast}|+1))
\]
and hence, 
\[
x=\left(M(\partial_{0,\rho})\chi_{\mathbb{R}_{\geq0}}g(0-)\right)(0+).
\]
Moreover, we have that 
\begin{align*}
\left(M(\partial_{0,\rho})\chi_{\mathbb{R}_{\geq0}}g(0-)\right)(0+) & =\left(M(\partial_{0,\rho})\left(g+\chi_{\mathbb{R}_{\geq0}}g(0-)\right)\right)(0+)-\left(M(\partial_{0,\rho})g\right)(0+)\\
 & =\left(M(\partial_{0,\rho})\left(g+\chi_{\mathbb{R}_{\geq0}}g(0-)\right)\right)(0-)-\left(M(\partial_{0,\rho})g\right)(0+),
\end{align*}
since $M(\partial_{0,\rho})\left(g+\chi_{\mathbb{R}_{\geq0}}g(0-)\right)\in H_{\rho}^{1}(\mathbb{R};H).$
Hence, by causality of $M(\partial_{0,\rho})$, we end up with 
\begin{align*}
x & =\left(M(\partial_{0,\rho})\chi_{\mathbb{R}_{\geq0}}g(0-)\right)(0+)\\
 & =\left(M(\partial_{0,\rho})g\right)(0-)-\left(M(\partial_{0,\rho})g\right)(0+).\tag*{\qedhere}
\end{align*}
\end{proof}
\begin{defn*}
We define the operator 
\begin{align*}
\Gamma_{(M,A)}^{\rho}:\mathrm{His}_{\rho}(M,A) & \to H\\
g & \mapsto\left(M(\partial_{0,\rho})g\right)(0-)-\left(M(\partial_{0,\rho})g\right)(0+).
\end{align*}

\end{defn*}
We now provide a rigorous proof for the heuristics done at the beginning
of this section.
\begin{prop}
\label{prop:solution_ivp}Let $\rho>\max\{s_{0}(M,A),b(M)\}$ and
$g\in\mathrm{His}_{\rho}(M,A).$ We define 
\[
v\coloneqq S_{\rho}\left(\Gamma_{(M,A)}^{\rho}g\delta_{0}-K_{\rho}g\right).
\]
Then, for $u\coloneqq v+g$ we have that $u\in H_{\rho}^{1}(\mathbb{R};H)$
satisfies \prettyref{eq:history}. Moreover, $u$ satisfies \prettyref{eq:problem_pos_times}
in the sense that 
\[
\spt\left(\partial_{0,\rho}M(\partial_{0,\rho})+A\right)u\subseteq\mathbb{R}_{\leq0}.
\]
\end{prop}
\begin{proof}
We first show that $v\in H_{\rho}(\mathbb{R}_{\geq0};H).$ Indeed,
we have that $v-\chi_{\mathbb{R}_{\geq0}}g(0-)\in H_{\rho}^{1}(\mathbb{R};H)\subseteq H_{\rho}(\mathbb{R};H),$
which gives that $v\in H_{\rho}(\mathbb{R};H).$ Moreover, we have
that 
\begin{align*}
\partial_{0,\rho}^{-1}v & =S_{\rho}\partial_{0,\rho}^{-1}\left(\Gamma_{(M,A)}^{\rho}g\delta_{0}-K_{\rho}g\right)\\
 & =S_{\rho}\partial_{0,\rho}^{-1}\left(\Gamma_{(M,A)}^{\rho}g\delta_{0}-\left(\partial_{0,\rho}\chi_{\mathbb{R}_{\geq0}}(\m)M(\partial_{0,\rho})g-\delta_{0}\left(M(\partial_{0,\rho})g\right)(0+)\right)\right)\\
 & =S_{\rho}\left(\Gamma_{(M,A)}^{\rho}g\chi_{\mathbb{R}_{\geq0}}+\left(M(\partial_{0,\rho})g\right)(0+)\chi_{\mathbb{R}_{\geq0}}-\chi_{\mathbb{R}_{\geq0}}(\m)M(\partial_{0,\rho})g\right)
\end{align*}
and thus, due to causality of $S_{\rho}$ we deduce that $\spt\partial_{0,\rho}^{-1}v\subseteq\mathbb{R}_{\geq0}$.
This yields $\spt v\subseteq\mathbb{R}_{\geq0}.$ Hence, $u|_{\mathbb{R}_{<0}}=g$
and since $v(0+)=g(0-)$ we infer $u\in H_{\rho}^{1}(\mathbb{R};H)$.
Let now $\varphi\in C_{c}^{\infty}(\mathbb{R}_{>0};D(A^{\ast})).$
Then we have that 
\begin{align*}
 & \langle\left(\partial_{0,\rho}M(\partial_{0,\rho})+A\right)u|\varphi\rangle_{H_{\rho}(\mathbb{R};H^{-1}(|A^{\ast}|+1))\times H_{\rho}(\mathbb{R};H^{1}(|A^{\ast}|+1))}\\
 & =\langle u|\left(\partial_{0,\rho}M(\partial_{0,\rho})+A\right)^{\ast}\varphi\rangle_{H_{\rho}(\mathbb{R};H)}\\
 & =\langle v|\left(\partial_{0,\rho}M(\partial_{0,\rho})+A\right)^{\ast}\varphi\rangle_{H_{\rho}(\mathbb{R};H)}+\langle g|\left(\partial_{0,\rho}M(\partial_{0,\rho})+A\right)^{\ast}\varphi\rangle_{H_{\rho}(\mathbb{R};H)}\\
 & =\langle\Gamma_{(M,A)}^{\rho}g\delta_{0}-K_{\rho}g|\varphi\rangle_{H^{-1}(\partial_{0,\rho})\times H^{1}(\partial_{0,\rho}^{\ast})}+\langle g|\left(\partial_{0,\rho}M(\partial_{0,\rho})+A\right)^{\ast}\varphi\rangle_{H_{\rho}(\mathbb{R};H)}\\
 & =\langle\partial_{0,\rho}^{-1}\left(\Gamma_{(M,A)}^{\rho}g\delta_{0}-K_{\rho}g\right)|\partial_{0,\rho}^{\ast}\varphi\rangle_{H_{\rho}(\mathbb{R};H)}+\langle g|\left(\partial_{0,\rho}M(\partial_{0,\rho})\right)^{\ast}\varphi\rangle_{H_{\rho}(\mathbb{R};H)},
\end{align*}
where we have used $\spt A^{\ast}\varphi\subseteq\mathbb{R}_{\geq0}$
and $\spt g\subseteq\mathbb{R}_{\leq0}$. Recalling that 
\begin{align*}
 & \partial_{0,\rho}^{-1}\left(\Gamma_{(M,A)}^{\rho}g\delta_{0}-K_{\rho}g\right)\\
 & =\left(\left(M(\partial_{0,\rho})g\right)(0-)-\left(M(\partial_{0,\rho})g\right)(0+)\right)\chi_{\mathbb{R}_{\geq0}}+\left(M(\partial_{0,\rho})g\right)(0+)\chi_{\mathbb{R}_{\geq0}}-\chi_{\mathbb{R}_{\geq0}}(\m)M(\partial_{0,\rho})g\\
 & =\left(M(\partial_{0,\rho})g\right)(0-)\chi_{\mathbb{R}_{\geq0}}-\chi_{\mathbb{R}_{\geq0}}(\m)M(\partial_{0,\rho})g,
\end{align*}
we deduce that 
\begin{align*}
 & \langle\partial_{0,\rho}^{-1}\left(\Gamma_{(M,A)}^{\rho}g\delta_{0}-K_{\rho}g\right)|\partial_{0,\rho}^{\ast}\varphi\rangle_{H_{\rho}(\mathbb{R};H)}\\
 & =\langle\left(M(\partial_{0,\rho})g\right)(0-)\chi_{\mathbb{R}_{\geq0}}|\partial_{0,\rho}^{\ast}\varphi\rangle_{H_{\rho}(\mathbb{R};H)}-\langle\chi_{\mathbb{R}_{\geq0}}(\m)M(\partial_{0,\rho})g|\partial_{0,\rho}^{\ast}\varphi\rangle_{H_{\rho}(\mathbb{R};H)}\\
 & =-\langle\chi_{\mathbb{R}_{\geq0}}(\m)M(\partial_{0,\rho})g|\partial_{0,\rho}^{\ast}\varphi\rangle_{H_{\rho}(\mathbb{R};H)}\\
 & =-\langle M(\partial_{0,\rho})g|\partial_{0,\rho}^{\ast}\varphi\rangle_{H_{\rho}(\mathbb{R};H)}
\end{align*}
where we have used $\varphi(0)=0$ and $\spt\partial_{0,\rho}^{\ast}\varphi\subseteq\spt\varphi\subseteq\mathbb{R}_{>0}$.
Summarizing, we obtain that 
\begin{align*}
 & \langle\left(\partial_{0,\rho}M(\partial_{0,\rho})+A\right)u|\varphi\rangle_{H_{\rho}(\mathbb{R};H^{-1}(|A^{\ast}|+1))\times H_{\rho}(\mathbb{R};H^{1}(|A^{\ast}|+1))}\\
 & =-\langle M(\partial_{0,\rho})g|\partial_{0,\rho}^{\ast}\varphi\rangle_{H_{\rho}(\mathbb{R};H)}+\langle g|\left(\partial_{0,\rho}M(\partial_{0,\rho})\right)^{\ast}\varphi\rangle_{H_{\rho}(\mathbb{R};H)}\\
 & =0,
\end{align*}
which finishes the proof. 
\end{proof}
In the case that the evolutionary problem associated with $M$ and
$A$ does not have memory, the space $\mathrm{IV}_{\rho}(M,A)$ should
not depend on the admissible histories. More precisely, we should
have 
\[
\mathrm{His}_{\rho}(M,A)=\left\{ g\in\chi_{\mathbb{R}_{\leq0}}\left[H_{\rho}^{1}(\mathbb{R};H)\right]\,;\, g(0-)\in\mathrm{IV}_{\rho}(M,A)\right\} 
\]
and for $x\in\mathrm{IV}_{\rho}(M,A)$ any two functions $g,\tilde{g}\in\mathrm{His}_{\rho}(M,A)$
with $\tilde{g}(0-)=g(0-)=x$ should yield the same solution of the
initial value problem. Before we study this case, we need to define
what we mean by ``having no memory''.
\begin{defn*}
An operator $T:D(T)\subseteq H_{\rho}(\mathbb{R};H)\to H_{\rho}(\mathbb{R};H)$
is said to be \emph{amnesic, }if for all $f\in D(T)$ and $a\in\mathbb{R}$
\[
\spt f\subseteq\mathbb{R}_{\leq a}\Rightarrow\spt Tf\subseteq\mathbb{R}_{\leq a}.
\]

\end{defn*}
Apparently, amnesia and causality are two properties which are strongly
related. 
\begin{lem}
\label{lem:amnesic_vs_causal}Let $T:D(T)\subseteq H_{\rho}(\mathbb{R};H)\to H_{\rho}(\mathbb{R};H)$
and define 
\begin{align*}
\sigma_{-1}:L_{2,\mathrm{loc}}(\mathbb{R};H) & \to L_{2,\mathrm{loc}}(\mathbb{R};H)\\
f & \mapsto\left(t\mapsto f(-t)\right).
\end{align*}
Then $T$ is amnesic, if and only if $\sigma_{-1}T\sigma_{-1}:D(T\sigma_{-1})\subseteq H_{-\rho}(\mathbb{R};H)\to H_{-\rho}(\mathbb{R};H)$
is causal.\end{lem}
\begin{proof}
The proof is obvious.
\end{proof}
With this lemma at hand, we derive the following representation result
for amnesic operators.
\begin{prop}
\label{prop:representation_amnesic}Let $\rho\in\mathbb{R}$ and $T:H_{\rho}(\mathbb{R};H)\to H_{\rho}(\mathbb{R};H)$
bounded, translation invariant. Then $T$ is amnesic if and only if
there exists $N:\mathbb{C}_{\Re>0}\to L(H)$ analytic and bounded,
such that 
\begin{equation}
\widehat{Tf}(z)=N(-z+\rho)\hat{f}(z)\label{eq:repr_amnesic}
\end{equation}
for each $f\in H_{\rho}(\mathbb{R}_{\leq0};H),z\in\mathbb{C}_{\Re<\rho}.$\end{prop}
\begin{proof}
Assume that $T$ is amnesic. By \prettyref{lem:amnesic_vs_causal}
we know that $\sigma_{-1}T\sigma_{-1}\in L(H_{-\rho}(\mathbb{R};H))$
is causal. Moreover, it is clearly translation invariant and hence,
so is 
\begin{equation}
S\coloneqq\e^{\rho\m}\sigma_{-1}T\sigma_{-1}\left(\e^{\rho\m}\right)^{-1}:L_{2}(\mathbb{R};H)\to L_{2}(\mathbb{R};H),\label{eq:S}
\end{equation}
where $\e^{\rho\m}:H_{-\rho}(\mathbb{R};H)\to L_{2}(\mathbb{R};H)$
with $\left(\e^{\rho\m}f\right)(t)=\e^{\rho t}f(t)$. Now, by \prettyref{thm:weiss}
there exists an analytic and bounded function 
\[
N:\mathbb{C}_{\Re>0}\to L(H)
\]
such that
\[
\widehat{Sg}(z)=N(z)\hat{g}(z)
\]
for each $g\in L_{2}(\mathbb{R}_{\geq0};H)$, $z\in\mathbb{C}_{\Re>0}$.
Then we compute 
\begin{align*}
\widehat{Tf}(z) & =\left(\widehat{\sigma_{-1}\left(\e^{\rho\m}\right)^{-1}S\e^{\rho\m}\sigma_{-1}f}\right)(z)\\
 & =\left(\widehat{\left(\e^{\rho\m}\right)^{-1}S\e^{\rho\m}\sigma_{-1}f}\right)(-z)\\
 & =\left(\widehat{S\e^{\rho\m}\sigma_{-1}f}\right)(-z+\rho)\\
 & =N(-z+\rho)\left(\widehat{\e^{\rho\m}\sigma_{-1}f}\right)(-z+\rho)\\
 & =N(-z+\rho)\hat{f}(z)
\end{align*}
for each $f\in H_{\rho}(\mathbb{R}_{\leq0};H)$ and $z\in\mathbb{C}_{\Re<\rho}$,
which gives the assertion. Assume now that \prettyref{eq:repr_amnesic}
holds. We prove that $S$ given by \prettyref{eq:S} is causal, which
would yield the assertion. Let $g\in L_{2}(\mathbb{R}_{\geq0};H).$
We then have 
\begin{align*}
\widehat{Sg}(z) & =\left(\widehat{\e^{\rho\m}\sigma_{-1}T\sigma_{-1}\left(\e^{\rho\m}\right)^{-1}g}\right)(z)\\
 & =\left(\widehat{T\sigma_{-1}\left(\e^{\rho\m}\right)^{-1}g}\right)(-z+\rho)\\
 & =N(z)\left(\widehat{\sigma_{-1}\left(\e^{\rho\m}\right)^{-1}g}\right)(-z+\rho)\\
 & =N(z)\hat{g}(z)
\end{align*}
for each $z\in\mathbb{C}_{\Re>0}.$ Hence, $\widehat{Sg}\in\mathcal{\mathcal{H}}^{2}(\mathbb{C}_{\Re>0};H)$
and thus the assertion follows by \prettyref{thm:Paley-Wiener}.\end{proof}
\begin{prop}
\label{prop:amnesic_P_0}Let $\rho>0$ and $T:H_{\rho}(\mathbb{R};H)\to H_{\rho}^{-1}(\mathbb{R};H)$
bounded and translation invariant. Then the following statements are
equivalent: 

\begin{enumerate}[(i)]

\item For all $\varphi\in C_{c}^{\infty}(\mathbb{R}_{<0};H)$ we
have $T\varphi\in N(P_{0})$.

\item The operator $\left(\partial_{0,\rho}^{-1}\right)^{\ast}T$
is amnesic%
\footnote{Note that $\left(\partial_{0,\rho}^{-1}\right)^{\ast}:H_{\rho}^{-1}(\mathbb{R};H)\to H_{\rho}(\mathbb{R};H)$
is bounded, since $\partial_{0,\rho}\left(\partial_{0,\rho}^{^{-1}}\right)^{\ast}=\left(-\partial_{0,\rho}^{\ast}+2\rho\right)\left(\partial_{0,\rho}^{\ast}\right)^{-1}$
is bounded on $H_{\rho}(\mathbb{R};H)$ and $\partial_{0,\rho}$ is
normal.%
}. 

\item For all $g\in\chi_{\mathbb{R}_{\leq0}}(\m)\left[H_{\rho}^{1}(\mathbb{R};H)\right]$
we have $Tg\in N(P_{0})$.

\end{enumerate}\end{prop}
\begin{proof}
(i) $\Rightarrow$ (ii): Let $\varphi\in C_{c}^{\infty}(\mathbb{R}_{<0};H).$
Then by assumption 
\[
0=P_{0}T\varphi=\partial_{0,\rho}\chi_{\mathbb{R}_{\geq0}}(\m)\partial_{0,\rho}^{-1}T\varphi-\delta_{0}\left(\partial_{0,\rho}^{-1}T\varphi\right)(0+).
\]
Consequently, 
\[
\chi_{\mathbb{R}_{\geq0}}(\m)\partial_{0,\rho}^{-1}T\varphi=\chi_{\mathbb{R}_{\geq0}}x,
\]
where $x\coloneqq\left(\partial_{0,\rho}^{-1}T\varphi\right)(0+).$
The latter gives $\spt\left(\partial_{0,\rho}^{-1}\right)^{\ast}T\varphi\subseteq\mathbb{R}_{\leq0}.$
Indeed, let $\psi\in C_{c}^{\infty}(\mathbb{R}_{>0};H).$ Then we
compute 
\begin{align*}
\langle\left(\partial_{0,\rho}^{-1}\right)^{\ast}T\varphi|\psi\rangle_{H_{\rho}(\mathbb{R};H)} & =\langle\left(\partial_{0,\rho}^{-1}\right)^{\ast}\partial_{0,\rho}\partial_{0,\rho}^{-1}T\varphi|\psi\rangle_{H_{\rho}(\mathbb{R};H)}\\
 & =\langle\left(\partial_{0,\rho}^{-1}\right)^{\ast}\partial_{0,\rho}^{-1}T\varphi|\partial_{0,\rho}^{\ast}\psi\rangle_{H_{\rho}(\mathbb{R};H)}\\
 & =\langle\partial_{0,\rho}^{-1}T\varphi|\partial_{0,\rho}^{-1}\partial_{0,\rho}^{\ast}\psi\rangle_{H_{\rho}(\mathbb{R};H)}\\
 & =\langle\partial_{0,\rho}^{-1}T\varphi|-\psi+2\rho\partial_{0,\rho}^{-1}\psi\rangle_{H_{\rho}(\mathbb{R};H)}
\end{align*}
Setting now $\tilde{\psi}\coloneqq\partial_{0,\rho}^{-1}\psi\in H_{\rho}^{1}(\mathbb{R};H)$
and noting that due to the causality of $\partial_{0,\rho}^{-1}$
we have $\spt\tilde{\psi}\subseteq\mathbb{R}_{>0},$ we infer 
\begin{align*}
\langle\left(\partial_{0,\rho}^{-1}\right)^{\ast}T\varphi|\psi\rangle_{H_{\rho}(\mathbb{R};H)} & =\langle\partial_{0,\rho}^{-1}T\varphi|-\tilde{\psi}'+2\rho\tilde{\psi}\rangle_{H_{\rho}(\mathbb{R};H)}\\
 & =\langle\chi_{\mathbb{R}_{\geq0}}x|\partial_{0,\rho}^{\ast}\tilde{\psi}\rangle_{H_{\rho}(\mathbb{R};H)}\\
 & =\tilde{\psi}(0)=0.
\end{align*}
Using now the density of $C_{c}^{\infty}(\mathbb{R}_{<0};H)$ in $H_{\rho}(\mathbb{R}_{\leq0};H)$
and the boundedness and translation invariance of $\left(\partial_{0,\rho}^{-1}\right)^{\ast}T,$
we derive that $\left(\partial_{0,\rho}^{-1}\right)^{\ast}T$ is amnesic.\\
(ii) $\Rightarrow$ (iii): Let now $g\in\chi_{\mathbb{R}_{\leq0}}(\m)\left[H_{\rho}^{1}(\mathbb{R};H)\right]$.
We show that $\spt Tg\subseteq\mathbb{R}_{\leq0}.$ For doing so,
let $\varphi\in C_{c}^{\infty}(\mathbb{R}_{>0};H).$ Then 
\begin{align*}
\langle Tg|\varphi\rangle_{H^{-1}(\partial_{0,\rho})\times H^{1}(\partial_{0,\rho}^{\ast})} & =\langle\partial_{0,\rho}^{-1}Tg|\partial_{0,\rho}^{\ast}\varphi\rangle_{H_{\rho}(\mathbb{R};H)}\\
 & =\langle\partial_{0,\rho}^{-1}Tg|\partial_{0,\rho}^{-1}\partial_{0,\rho}^{\ast}\varphi'\rangle_{H_{\rho}(\mathbb{R};H)}\\
 & =\langle\left(\partial_{0,\rho}^{-1}\right)^{\ast}Tg|\varphi'\rangle_{H_{\rho}(\mathbb{R};H)}\\
 & =0,
\end{align*}
which gives $\spt Tg\subseteq\mathbb{R}_{\leq0}.$ The assertion now
follows from \prettyref{prop:properties_cut_off} (e).\\
(iii) $\Rightarrow$ (i): This is obvious.
\end{proof}
The latter two propositions now provide a characterization of those
material laws, where the history space in fact just depends on the
value $g(0-)$ and not on the whole history $g$.
\begin{prop}
\label{prop:material_law_amnesic}Let $\rho>\max\{s_{0}(M,A),b(M)\}$.
Then the following statements are equivalent:

\begin{enumerate}[(i)]

\item We have 
\[
\mathrm{His}_{\rho}(M,A)=\left\{ g\in\chi_{\mathbb{R}_{\leq0}}\left[H_{\rho}^{1}(\mathbb{R};H)\right]\,;\, g(0-)\in\mathrm{IV}_{\rho}(M,A)\right\} 
\]
and for each $g\in\mathrm{His}_{\rho}(M,A)$ with $g(0-)=0$ we have
that 
\[
S_{\rho}(\delta_{0}\Gamma_{(M,A)}^{\rho}g-K_{\rho}g)=0.
\]

\item There exist $M_{0},M_{1}\in L(H)$ such that 
\[
M(z)=M_{0}+z^{-1}M_{1}\quad(z\in D(M)\setminus\{0\}).
\]

\end{enumerate}

In this case we have $K_{\rho}g=0$ for each $g\in\chi_{\mathbb{R}_{\leq0}}(\m)\left[H_{\rho}^{1}(\mathbb{R};H)\right].$\end{prop}
\begin{proof}
(i) $\Rightarrow$ (ii): By assumption, we in particular have $C_{c}^{\infty}(\mathbb{R}_{<0};H)\subseteq\mathrm{His}_{\rho}(M,A)$,
since $0\in\mathrm{IV}_{\rho}(M,A).$ Consequently, 
\[
S_{\rho}(\delta_{0}\Gamma_{(M,A)}^{\rho}\varphi-K_{\rho}\varphi)=0
\]
for each $\varphi\in C_{c}^{\infty}(\mathbb{R}_{<0};H)$ by assumption.
Thus, 
\[
\delta_{0}\Gamma_{(M,A)}^{\rho}\varphi=K_{\rho}\varphi
\]
and hence, by \prettyref{prop:properties_cut_off} (b) and (a) 
\[
K_{\rho}\varphi=P_{0}\partial_{0,\rho}M(\partial_{0,\rho})\varphi=P_{0}\delta_{0}\Gamma_{(M,A)}^{\rho}\varphi=0
\]
for each $\varphi\in C_{c}^{\infty}(\mathbb{R}_{<0};H).$ Thus, $\left(\partial_{0,\rho}^{-1}\right)^{\ast}\partial_{0,\rho}M(\partial_{0,\rho})$
is amnesic by \prettyref{prop:amnesic_P_0}. Hence, according to \prettyref{prop:representation_amnesic}
there exists $N:\mathbb{C}_{\Re>0}\to L(H)$ analytic and bounded
such that 
\begin{equation}
\left(\widehat{\left(\partial_{0,\rho}^{-1}\right)^{\ast}\partial_{0,\rho}M(\partial_{0,\rho})f}\right)(z)=N(-z+\rho)\hat{f}(z)\label{eq:N}
\end{equation}
for each $f\in H_{\rho}(\mathbb{R}_{\leq0};H)$ and $z\in\mathbb{C}_{\Re<\rho}.$
Considering the analytic function 
\begin{align*}
T:\left\{ z\in\mathbb{C}\,;\, b(M)<\Re z<2\rho\right\}  & \to L(H)\\
z & \mapsto\frac{z}{-z+2\rho}M(z)
\end{align*}
we get that 
\[
\left(\partial_{0,\rho}^{-1}\right)^{\ast}\partial_{0,\rho}M(\partial_{0,\rho})=T(\partial_{0,\rho}).
\]
Using that $T$ is bounded on $\{z\in\mathbb{C}\,;\, b(M)+\varepsilon<\Re z<2\rho-\varepsilon\}$
for each $\varepsilon>0$, we can use \prettyref{lem:independent of rho}
to derive that 
\[
T(\partial_{0,\rho})f=T(\partial_{0,\mu})f
\]
for each $f\in H_{\rho}(\mathbb{R}_{\leq0};H)$ and $b(M)<\mu<\rho.$
Thus, choosing $f(t)\coloneqq\sqrt{2\pi}\chi_{\mathbb{R}_{\leq0}}(t)\e^{(\rho+1)t}x$
for arbitrary $x\in H$, we infer from \prettyref{eq:N} 
\[
\frac{1}{\rho+1-z}N(-z+\rho)x=\left(\widehat{T(\partial_{0,\rho})f}\right)(z)=\left(\widehat{T(\partial_{0,\Re z})f}\right)(z)=\frac{z}{-z+2\rho}\frac{1}{\rho+1-z}M(z)x
\]
for each $z\in\mathbb{C}$ with $b(M)<\Re z<\rho.$ Hence, 
\[
N(-z+\rho)=\frac{z}{-z+2\rho}M(z)\quad(b(M)<\Re z<\rho).
\]
Thus, 
\[
F(z)\coloneqq\begin{cases}
(-z+b(M))M(-z+b(M)) & \mbox{ if }\Re z<0,\\
\left(z+2\rho-b(M)\right)N(z+\rho-b(M)) & \mbox{ if }\Re z\geq0,
\end{cases}
\]
is an entire function. Moreover 
\[
|F(z)|\leq C|z|+D\quad(z\in\mathbb{C})
\]
for some $C,D\geq0.$ Hence, the function $G(z)\coloneqq\frac{1}{z}(F(z)-F(0))$
for $z\in\mathbb{C}\setminus\{0\}$ is bounded, which yields $G(z)=B$
for some $B\in L(H)$ and all $z\ne0$ by the Theorem of Liouville.
Hence, 
\[
F(z)=zB+F(0)\quad(z\in\mathbb{C}).
\]
For $\Re z<0$ we thus have 
\[
(-z+b(M))M(-z+b(M))=zB+F(0)
\]
and thus, 
\[
M(z)=\frac{b(M)-z}{z}B+\frac{1}{z}F(0)\quad(z\in\mathbb{C}_{\Re>b(M)}).
\]
Setting $M_{0}\coloneqq-B$ and $M_{1}\coloneqq b(M)B+F(0),$ the
assertion follows from the identity theorem.\\
(ii) $\Rightarrow$ (i): Let $g\in\chi_{\mathbb{R}_{\leq0}}(\m)\left[H_{\rho}^{1}(\mathbb{R};H)\right].$
Then $\left(\left(M_{0}+\partial_{0,\rho}^{-1}M_{1}\right)g\right)(0+)=\left(\partial_{0,\rho}^{-1}M_{1}g\right)(0+),$
which exists, since $\partial_{0,\rho}^{-1}M_{1}g\in H_{\rho}^{1}(\mathbb{R};H).$
Thus, $\partial_{0,\rho}M(\partial_{0,\rho})g\in D(P_{0})$ and we
compute 
\begin{align*}
K_{\rho}g & =P_{0}\partial_{0,\rho}M(\partial_{0,\rho})g\\
 & =\partial_{0,\rho}\chi_{\mathbb{R}_{\geq0}}(\m)\left(M_{0}+\partial_{0,\rho}^{-1}M_{1}\right)g-\delta_{0}\left(M_{0}+\partial_{0,\rho}^{-1}M_{1}g\right)(0+)\\
 & =\partial_{0,\rho}\chi_{\mathbb{R}_{\geq0}}(\m)\partial_{0,\rho}^{-1}M_{1}g-\delta_{0}\left(\partial_{0,\rho}^{-1}M_{1}g\right)(0+)\\
 & =P_{0}M_{1}g\\
 & =\chi_{\mathbb{R}_{\geq0}}(\m)M_{1}g\\
 & =0.
\end{align*}
Let now $x\in\mathrm{IV}_{\rho}(M,A).$ Then there exists $g\in\mathrm{His}_{\rho}(M,A)$
such that $g(0-)=x.$ Let now $h\in\chi_{\mathbb{R}_{\leq0}}(\m)\left[H_{\rho}^{1}(\mathbb{R};H)\right]$
with $h(0-)=x.$ We compute
\begin{align*}
S_{\rho}(\delta_{0}\Gamma_{(M,A)}^{\rho}g-K_{\rho}g) & =S_{\rho}\delta_{0}\Gamma_{(M,A)}^{\rho}g\\
 & =S_{\rho}\delta_{0}\left(\left(M(\partial_{0,\rho})g\right)(0-)-\left(M(\partial_{0,\rho})g\right)(0+)\right).
\end{align*}
Since 
\[
\left(M(\partial_{0,\rho})g\right)(0-)-\left(M(\partial_{0,\rho})g\right)(0+)=M_{0}g(0-)=M_{0}x
\]
we infer 
\[
S_{\rho}(\delta_{0}\Gamma_{(M,A)}^{\rho}g-K_{\rho}g)=S_{\rho}(\delta_{0}\Gamma_{(M,A)}^{\rho}h-K_{\rho}h),
\]
and thus, $h\in\mathrm{His}_{\rho}(M,A).$ Moreover, 
\[
S_{\rho}(\delta_{0}\Gamma_{(M,A)}^{\rho}g-K_{\rho}g)=0,
\]
for $g(0-)=0$, which gives (i). 
\end{proof}
The latter proposition shows that material laws of the form $M(\partial_{0,\rho})=M_{0}+\partial_{0,\rho}^{-1}M_{1}$
are precisely those, where it suffices to prescribe an initial value
$g(0-)$ and not a whole history $g$.

\section{$C_{0}$-semigroups associated with evolutionary problems}

This section is devoted to the regularity of initial value problems
as they were introduced in the previous section. More precisely, we
characterize those evolutionary problems which allow the definition
of an associated $C_{0}$-semigroup on a suitable Hilbert space. The
key tool will be the Widder-Arendt Theorem (see \cite{Arendt1987}
and \prettyref{thm:Widder-Arendt} of this thesis). First of all,
we recall the definition of a $C_{0}$-semigroup.
\begin{defn*}
Let $X$ be a Banach space. A \emph{$C_{0}$-semigroup on $X$} is
a mapping $T:\mathbb{R}_{\geq0}\to L(X)$ such that 

\begin{enumerate}[(a)]

\item $T(0)=1$ and for all $t,s\in\mathbb{R}_{\geq0}$ we have that
$T(t+s)=T(t)T(s).$

\item For all $x\in X$ we have that $T(\cdot)x$ is continuous on
$\mathbb{R}_{\geq0}.$

\end{enumerate}
\end{defn*}
We want to associate a $C_{0}$-semigroup with a well-posed evolutionary
problem. For this, we need to find the right Hilbert space $X$, where
the semigroup should act on. Throughout, we assume that $M:D(M)\subseteq\mathbb{C}\to L(H)$
is a regularizing linear material law, $A:D(A)\subseteq H\to H$ is
densely defined closed and linear and the associated evolutionary
problem is well-posed. Again, we set 
\[
S_{\rho}\coloneqq\overline{\left(\partial_{0,\rho}M(\partial_{0,\rho})+A\right)^{-1}}
\]
for $\rho>s_{0}(M,A).$ We begin with the following proposition.
\begin{prop}
\label{prop:space_invariant} Let $\rho>\max\{s_{0}(M,A),b(M)\}$
and $g\in\mathrm{His}_{\rho}(M,A).$ Moreover, we set 
\[
v\coloneqq S_{\rho}\left(\delta_{0}\Gamma_{(M,A)}^{\rho}g-K_{\rho}g\right)
\]
and $u\coloneqq v+g.$ For $t>0$ we set $\tilde{g}\coloneqq\chi_{\mathbb{R}_{\leq0}}(\m)\tau_{t}u$
and $\tilde{v}\coloneqq\chi_{\mathbb{R}_{\geq0}}(\m)\tau_{t}u.$ Then
we have $\tilde{g}\in\mathrm{His}_{\rho}(M,A)$ and 
\[
\tilde{v}=S_{\rho}\left(\delta_{0}\Gamma_{(M,A)}^{\rho}\tilde{g}-K_{\rho}\tilde{g}\right).
\]
\end{prop}
\begin{proof}
{} We first prove that 
\begin{equation}
\spt(\partial_{0,\rho}M(\partial_{0,\rho})+A)\tau_{t}u\subseteq\mathbb{R}_{\leq0}.\label{eq:spt_tau_u}
\end{equation}
For doing so, let $\varphi\in C_{c}^{\infty}(\mathbb{R}_{>0};D(A^{\ast})).$
Then we have 
\begin{align*}
 & \langle(\partial_{0,\rho}M(\partial_{0,\rho})+A)\tau_{t}u|\varphi\rangle_{H_{\rho}(\mathbb{R};H^{-1}(|A|+1))\times H_{\rho}(\mathbb{R};H^{1}(|A|+1))}\\
 & =\langle\tau_{t}u|\left(\partial_{0,\rho}M(\partial_{0,\rho})+A\right)^{\ast}\varphi\rangle_{H_{\rho}(\mathbb{R};H)}\\
 & =\langle u|\tau_{-t}\e^{2\rho t}\left(\partial_{0,\rho}M(\partial_{0,\rho})+A\right)^{\ast}\varphi\rangle_{H_{\rho}(\mathbb{R};H)}\\
 & =\e^{2\rho t}\langle u|\left(\partial_{0,\rho}M(\partial_{0,\rho})+A\right)^{\ast}\tau_{-t}\varphi\rangle_{H_{\rho}(\mathbb{R};H)}\\
 & =\e^{2\rho t}\langle\left(\partial_{0,\rho}M(\partial_{0,\rho})+A\right)u|\tau_{-t}\varphi\rangle_{H_{\rho}(\mathbb{R};H^{-1}(|A|+1))\times H_{\rho}(\mathbb{R};H^{1}(|A|+1))}\\
 & =0,
\end{align*}
by \prettyref{prop:solution_ivp}, since $\tau_{-t}\varphi\in C_{c}^{\infty}(\mathbb{R}_{>0};D(A^{\ast}))$.
Hence, \prettyref{eq:spt_tau_u} holds and thus, employing the causality
of $M(\partial_{0,\rho})$ we derive 
\begin{align*}
 & (\partial_{0,\rho}M(\partial_{0,\rho})+A)\tau_{t}u\\
 & =\chi_{\mathbb{R}_{\leq0}}(\m)(\partial_{0,\rho}M(\partial_{0,\rho})+A)\tau_{t}u\\
 & =\chi_{\mathbb{R}_{\leq0}}(\m)\partial_{0,\rho}M(\partial_{0,\rho})\tau_{t}u+A\tilde{g}\\
 & =\partial_{0,\rho}\chi_{\mathbb{R}_{\leq0}}(\m)M(\partial_{0,\rho})\tau_{t}u+\delta_{0}\left(M(\partial_{0,\rho})\tau_{t}u\right)(0-)+A\tilde{g}\\
 & =\partial_{0,\rho}\chi_{\mathbb{R}_{\leq0}}(\m)M(\partial_{0,\rho})\tilde{g}+\delta_{0}\left(M(\partial_{0,\rho})\tilde{g}\right)(0-)+A\tilde{g}\\
 & =Q_{0}\partial_{0,\rho}M(\partial_{0,\rho})\tilde{g}+A\tilde{g},
\end{align*}
where we implicitly have shown that $\partial_{0,\rho}M(\partial_{0,\rho})\tilde{g}\in D(Q_{0}).$
Since $\tau_{t}u=\tilde{v}+\tilde{g}$ we infer 
\[
\partial_{0,\rho}M(\partial_{0,\rho})\tau_{t}u+A\tilde{v}=Q_{0}\partial_{0,\rho}M(\partial_{0,\rho})\tilde{g}.
\]
Now, we note that $\partial_{0,\rho}M(\partial_{0,\rho})\tilde{g}\in D(P_{0})$
by \prettyref{lem:regularizing} and hence, \prettyref{prop:properties_cut_off}
(d) together with the equality above yields 
\begin{align*}
\partial_{0,\rho}M(\partial_{0,\rho})\tilde{v} & =\partial_{0,\rho}M(\partial_{0,\rho})(\tau_{t}u-\tilde{g})\\
 & =\partial_{0,\rho}M(\partial_{0,\rho})\tau_{t}u-\left(K_{\rho}\tilde{g}+Q_{0}\partial_{0,\rho}M(\partial_{0,\rho})\tilde{g}\right)-\\
 & \quad-\delta_{0}\left(\left(M(\partial_{0,\rho})\tilde{g}\right)(0+)-\left(M(\partial_{0,\rho})\tilde{g}\right)(0-)\right)\\
 & =-A\tilde{v}-K_{\rho}\tilde{g}+\delta_{0}\left(\left(M(\partial_{0,\rho})\tilde{g}\right)(0-)-\left(M(\partial_{0,\rho})\tilde{g}\right)(0+)\right).
\end{align*}
Hence, 
\[
\left(\partial_{0,\rho}M(\partial_{0,\rho})+A\right)\tilde{v}=-K_{\rho}\tilde{g}+\delta_{0}\left(\left(M(\partial_{0,\rho})\tilde{g}\right)(0-)-\left(M(\partial_{0,\rho})\tilde{g}\right)(0+)\right).
\]
The latter gives $\tilde{g}\in\mathrm{His}_{\rho}(M,A).$ Indeed,
we have that 
\begin{align*}
 & S_{\rho}\left(-K_{\rho}\tilde{g}+\delta_{0}\left(\left(M(\partial_{0,\rho})\tilde{g}\right)(0-)-\left(M(\partial_{0,\rho})\tilde{g}\right)(0+)\right)\right)-\chi_{\mathbb{R}_{\geq0}}\tilde{g}(0-)\\
 & =\tilde{v}-\chi_{\mathbb{R}_{\geq0}}\tilde{g}(0-)\in H_{\rho}^{1}(\mathbb{R};H),
\end{align*}
since $\tilde{g}(0-)=\tilde{v}(0+)$ and $\tilde{v}\in\chi_{\mathbb{R}_{\geq0}}(\m)[H_{\rho}^{1}(\mathbb{R};H)].$
Hence, we also have that 
\[
\tilde{v}=S_{\rho}\left(\delta_{0}\Gamma_{(M,A)}^{\rho}\tilde{g}-K_{\rho}\tilde{g}\right).\tag*{\qedhere}
\]

\end{proof}
We are now able to define a semigroup associated with $M$ and $A$.
\begin{defn*}
Let $\rho>\max\{s_{0}(M,A),b(M)\}$ and $g\in\mathrm{His}_{\rho}(M,A).$
Set 
\[
v\coloneqq S_{\rho}\left(\delta_{0}\Gamma_{(M,A)}^{\rho}g-K_{\rho}g\right)
\]
and $u\coloneqq v+g.$ Then, for $t\geq0$ we define the operator\nomenclature[Z_070]{$T_{(M,A)}$}{the $C_0$-semigroup associated with a material law $M$ and an operator $A$.}
\[
T_{(M,A)}(t):D_{\rho}\subseteq\mathrm{IV}_{\rho}(M,A)\times\mathrm{His}_{\rho}(M,A)\to\mathrm{IV}_{\rho}(M,A)\times\mathrm{His}_{\rho}(M,A)
\]
with 
\[
D_{\rho}\coloneqq\left\{ (g(0-),g)\,;\, g\in\mathrm{His}_{\rho}(M,A)\right\} 
\]
by 
\[
T_{(M,A)}(t)(g(0-),g)\coloneqq\left(v(t+),\chi_{\mathbb{R}_{\leq0}}(\m)\tau_{t}u\right)=\left(\tau_{t}u(0),\chi_{\mathbb{R}_{\leq0}}(\m)\tau_{t}u\right).
\]
Moreover we set 
\begin{align*}
T_{(M,A)}^{(1)}(t)\left(g(0-),g\right) & \coloneqq v(t+)\\
T_{(M,A)}^{(2)}(t)(g(0-),g) & \coloneqq\chi_{\mathbb{R}_{\leq0}}(\m)\tau_{t}u.
\end{align*}
\end{defn*}
\begin{rem}
\label{rem:T^1 und T^2} Of course, both operator families $T_{(M,A)}^{(1)}$
and $T_{(M,A)}^{(2)}$ can be defined in terms of the other one. In
fact, we have that 
\[
T_{(M,A)}^{(1)}(t)(g(0-),g)=\left(T_{(M,A)}^{2}(t)\left(g(0-),g\right)\right)(0-)
\]
and 
\[
\left(T_{(M,A)}^{(2)}(t)\left(g(0-),g\right)\right)(s)=\begin{cases}
g(t+s) & \mbox{ if }s\leq-t,\\
T_{(M,A)}^{(1)}(t+s)\left(g(0-),g\right) & \mbox{ if }-t<s\leq0,
\end{cases}\quad(s\leq0).
\]

\end{rem}
We first show that $T_{(M,A)}$ satisfies the properties of a $C_{0}$-semigroup
with respect to the topology on $H\times H_{\mu}(\mathbb{R}_{\leq0};H)$
for each $\mu\leq\rho.$ 
\begin{prop}
\label{prop:T is sg}Let $\rho>\max\{s_{0}(M,A),b(M)\}$, $t\geq0$
and $T_{(M,A)}(t)$ like above. Then, $T_{(M,A)}(t)$ is linear and
we have that 
\[
T_{(M,A)}(t)(g(0-),g)\to T_{(M,A)}(0)(g(0-),g)=(g(0-),g)
\]
in $H\times H_{\mu}(\mathbb{R}_{\leq0};H)$ for all $g\in\mathrm{His}_{\rho}(M,A)$
and $\mu\leq\rho$ as $t\to0$ and 
\[
T_{(M,A)}(t+s)=T_{(M,A)}(t)T_{(M,A)}(s)
\]
 for each $t,s\geq0.$ In particular, $T_{(M,A)}(\cdot)(g(0-),g):\mathbb{R}_{\geq0}\to H\times H_{\mu}(\mathbb{R}_{\leq0};H)$
is continuous for each $g\in\mathrm{His}_{\rho}(M,A)$ and $\mu\leq\rho.$\end{prop}
\begin{proof}
The linearity of $T_{(M,A)}(t)$ is obvious since all operators involved
are linear. Moreover, we have that 
\[
T_{(M,A)}(0)(g(0-),g)=\left(v(0+),\chi_{\mathbb{R}_{\leq0}}(\m)u\right)=(g(0-),g)
\]
for $g\in\mathrm{His}_{\rho}(M,A)$ by \prettyref{prop:solution_ivp}
and 
\begin{align*}
 & |T_{(M,A)}(t)(g(0-),g)-(g(0-),g)|_{H\times H_{\rho}(\mathbb{R}_{\leq0};H)}^{2}\\
 & =|v(t+)-g(0-)|_{H}^{2}+|\chi_{\mathbb{R}_{\leq0}}(\m)\tau_{t}u-\chi_{\mathbb{R}_{\leq0}}(\m)u|_{H_{\rho}(\mathbb{R};H)},\\
 & \leq|v(t+)-g(0-)|_{H}^{2}+|\tau_{t}u-u|_{H_{\rho}(\mathbb{R};H)}\to0\quad(t\to0)
\end{align*}
by right continuity of $v$, $v(0+)=g(0-)$ and the strong continuity
of $t\mapsto\tau_{t}\in L(H_{\rho}(\mathbb{R};H)).$ Moreover, since
$H_{\rho}(\mathbb{R}_{\leq0};H)\hookrightarrow H_{\mu}(\mathbb{R}_{\leq0};H)$
for $\mu\leq\rho,$ the convergence also holds in $H\times H_{\mu}(\mathbb{R}_{\leq0};H).$
Let now $t,s\geq0$ and $g\in\mathrm{His}_{\rho}(M,A)$ and $u$ like
above. Then we have by \prettyref{prop:space_invariant} 
\begin{align*}
\chi_{\mathbb{R}_{\geq0}}(\m)\tau_{s}u & =S_{\rho}\left(\delta_{0}\Gamma_{(M,A)}^{\rho}\chi_{\mathbb{R}_{\leq0}}(\m)\tau_{s}u-K_{\rho}\chi_{\mathbb{R}_{\leq0}}(\m)\tau_{s}u\right)
\end{align*}
and thus, 
\[
T_{(M,A)}(t)T_{(M,A)}(s)(g(0-),g)=\left(\tau_{t+s}u(0),\chi_{\mathbb{R}_{\leq0}}(\m)\tau_{t+s}u\right)=T_{(M,A)}(t+s)g.\tag*{\qedhere}
\]

\end{proof}
The question which arises now is: can we extend the so-defined $T_{(M,A)}$
to a $C_{0}$-semigroup on 
\[
X_{\rho}^{\mu}\coloneqq\overline{D}_{\rho}^{H\times H_{\mu}(\mathbb{R}_{\leq0};H)}
\]
\nomenclature[B_180]{$X_\rho^\mu$}{the closure of $D_\rho$ in $H\times H_\mu(\mathbb{R}_{\leq 0};H)$, the natural domain of $T_{(M,A)}$.}\nomenclature[B_170]{$D_\rho$}{the space $\{(g(0-),g) \, ;\, g\in \mathrm{His}_{\rho}(M,A)\}$.}for
some $\mu\leq\rho$? Before we can answer this question, we need the
following pre-requisites.
\begin{lem}
\label{lem:r_g}Let $\rho>\max\{s_{0}(M,A),b(M)\}$ and $g\in\mathrm{His}_{\rho}(M,A)$.
We define 
\[
r_{g}:\mathbb{R}_{>\rho}\to H
\]
by 
\[
r_{g}(\lambda)\coloneqq\left(\lambda M(\lambda)+A\right)^{-1}\left(\frac{1}{\sqrt{2\pi}}\Gamma_{(M,A)}^{\rho}g-\left(\mathcal{L}_{\lambda}K_{\rho}g\right)(0)\right)\quad(\lambda>\rho).
\]
Then $r_{g}\in C^{\infty}(\mathbb{R}_{>g};H).$ More precisely 
\[
r_{g}(\lambda)=\mathcal{L}_{\lambda}\left(T_{(M,A)}^{(1)}(\cdot)(g(0-),g)\right)(0)\quad(\lambda>\rho).
\]
 \end{lem}
\begin{proof}
We set 
\[
v\coloneqq S_{\rho}\left(\delta_{0}\Gamma_{(M,A)}^{\rho}g-K_{\rho}g\right)=T_{(M,A)}^{(1)}(\cdot)(g(0-),g).
\]
By definition of $\mathrm{His}_{\rho}(M,A)$ we know that $v\in H_{\rho}(\mathbb{R}_{\geq0};H).$
Thus, by \prettyref{cor:Paley_Wiener_unitary} we have that $\hat{v}\in\mathcal{H}^{2}(\mathbb{C}_{\Re>\rho};H)$
and so, in particular, $\hat{v}|_{\mathbb{R}_{>\rho}}\in C^{\infty}(\mathbb{R}_{>\rho};H)$.
For $\lambda>\rho$ we compute 
\begin{align*}
\hat{v}(\lambda) & =\left(\mathcal{L}_{\lambda}v\right)(0)\\
 & =\left(\mathcal{L}_{\lambda}S_{\rho}\left(\delta_{0}\Gamma_{(M,A)}^{\rho}g-K_{\rho}g\right)\right)(0).
\end{align*}
Using \prettyref{prop:independence of rho H^-1} we get that 
\[
S_{\rho}\left(\delta_{0}\Gamma_{(M,A)}^{\rho}g-K_{\rho}g\right)\sim S_{\lambda}\left(\delta_{0}\Gamma_{(M,A)}^{\rho}g-K_{\rho}g\right)
\]
and thus, 
\begin{align*}
\hat{v}(\lambda) & =(\lambda M(\lambda)+A)^{-1}\left(\mathcal{L}_{\lambda}\left(\delta_{0}\Gamma_{(M,A)}^{\rho}g-K_{\rho}g\right)\right)(0)\\
 & =(\lambda M(\lambda)+A)^{-1}\left(\frac{1}{\sqrt{2\pi}}\Gamma_{(M,A)}^{\rho}g-\left(\mathcal{L}_{\lambda}K_{\rho}g\right)(0)\right),
\end{align*}
which yields the assertion. \end{proof}
\begin{lem}
\label{lem:boundedness T1 implies T2}Let $\rho>\max\{s_{0}(M,A),b(M)\}$.
Assume that for some $\mu\leq\rho$ and $\omega\geq\rho$ we have
that 
\[
T_{(M,A)}^{(1)}:D_{\rho}\subseteq X_{\rho}^{\mu}\to C_{\omega}(\mathbb{R}_{\geq0};H)
\]
is bounded. Then for each $\varepsilon>0$ the operator 
\[
T_{(M,A)}^{(2)}:D_{\rho}\subseteq X_{\rho}^{\mu}\to C_{\omega+\varepsilon}(\mathbb{R}_{\geq0};H_{\mu}(\mathbb{R}_{\leq0};H))
\]
is bounded.\end{lem}
\begin{proof}
Let $\varepsilon>0$ and $g\in\mathrm{His}_{\rho}(M,A)$. Then by
\prettyref{rem:T^1 und T^2} we have that 
\begin{align*}
 & \left|T_{(M,A)}^{(2)}(t)(g(0-),g)\right|_{H_{\mu}(\mathbb{R}_{\leq0};H)}^{2}\\
 & =\intop_{-\infty}^{-t}|g(t+s)|_{H}^{2}\e^{-2\mu s}\mbox{ d}s+\intop_{-t}^{0}|T_{(M,A)}^{(1)}(t+s)(g(0-),g)|_{H}^{2}\e^{-2\mu s}\mbox{ d}s\\
 & =\left(\:\intop_{-\infty}^{0}|g(s)|_{H}^{2}\e^{-2\mu s}\mbox{ d}s\ +\intop_{0}^{t}|T_{(M,A)}^{(1)}(s)(g(0-),g)|_{H}^{2}\e^{-2\mu s}\mbox{ d}s\right)\e^{2\mu t}\\
 & \leq\left(|g|_{H_{\mu}(\mathbb{R}_{\leq0};H)}^{2}+C|(g(0-),g)|_{X_{\rho}^{\mu}}^{2}\intop_{0}^{t}\e^{2(\omega-\mu)s}\mbox{ d}s\right)\e^{2\mu t}\\
 & \leq\e^{2\mu t}|g|_{H_{\mu}(\mathbb{R}_{\leq0};H)}^{2}+C|(g(0-),g)|_{X_{\rho}^{\mu}}^{2}t\e^{2\omega t}\\
 & \leq\tilde{C}\e^{2\left(\omega+\varepsilon\right)t}|(g(0-),g)|_{X_{\rho}^{\mu}}^{2}
\end{align*}
for all $t\geq0$, where $\tilde{C}\coloneqq1+C\frac{1}{2\varepsilon\e}.$
Using that $T_{(M,A)}^{(2)}(\cdot)\left(g(0-),g\right):\mathbb{R}_{\geq0}\to H_{\mu}(\mathbb{R}_{\leq0};H)$
is continuous by \prettyref{prop:T is sg}, we have shown that 
\[
T_{(M,A)}^{(2)}:D_{\rho}\subseteq X_{\rho}^{\mu}\to C_{\omega+\varepsilon}(\mathbb{R}_{\geq0};H_{\mu}(\mathbb{R}_{\leq0};H))
\]
is bounded.
\end{proof}
We can now prove a Hille-Yosida type result for the semigroup $T_{(M,A)}.$ 
\begin{thm}
\label{thm:sg} Let $\rho>\max\{s_{0}(M,A),b(M)\}$ and consider for
$t\geq0$ the mapping $T_{(M,A)}(t):D_{\rho}\subseteq H\times H_{\rho}(\mathbb{R}_{\leq0};H)\to H\times H_{\rho}(\mathbb{R}_{\leq0};H).$
For $g\in\mathrm{His}_{\rho}(M,A)$ we define
\[
r_{g}:\mathbb{R}_{>\rho}\to H
\]
by 
\[
r_{g}(\lambda)\coloneqq\left(\lambda M(\lambda)+A\right)^{-1}\left(\frac{1}{\sqrt{2\pi}}\Gamma_{(M,A)}^{\rho}g-\left(\mathcal{L}_{\lambda}K_{\rho}g\right)(0)\right)\quad(\lambda>\rho).
\]
Then $T_{(M,A)}$ can be extended to a $C_{0}$-semigroup on $X_{\rho}^{\mu}\coloneqq\overline{D_{\rho}}^{H\times H_{\mu}(\mathbb{R}_{\leq0};H)}$
for some $\mu\leq\rho$ if and only if there is $M\geq1,\omega\geq\rho$
such that 
\[
\frac{1}{n!}|r_{g}^{(n)}(\lambda)|\leq\frac{M}{\left(\lambda-\omega\right)^{n+1}}\left(|g(0-)|+|g|_{H_{\mu}(\mathbb{R}_{\leq0};H)}\right)\quad(n\in\mathbb{N},\lambda>\omega).
\]
In this case we have that
\begin{align*}
T_{(M,A)}^{(1)}:X_{\rho}^{\mu} & \to C_{\omega}(\mathbb{R}_{\geq0};H)\\
T_{(M,A)}^{(2)}:X_{\rho}^{\mu} & \to C_{\omega+\varepsilon}(\mathbb{R}_{\geq0};H_{\mu}(\mathbb{R}_{\leq0};H))
\end{align*}
are bounded for each $\varepsilon>0$. \end{thm}
\begin{proof}
First assume that we can extend $T_{(M,A)}$ to a $C_{0}$-semigroup
on $X_{\rho}^{\mu}$ for some $\mu\leq\rho.$ Then there exists $M\geq1$
and $\omega\in\mathbb{R}$ such that 
\[
|T_{(M,A)}(t)(g(0-),g)|_{H\times H_{\mu}(\mathbb{R}_{\leq0};H)}\leq M\e^{\omega t}\left|(g(0-),g)\right|_{H\times H_{\mu}(\mathbb{R}_{\leq0};H)}\quad(t\geq0)
\]
for all $g\in\mathrm{His}_{\rho}(M,A).$ In particular, we have that
\[
v\coloneqq T_{(M,A)}^{(1)}(\cdot)(g(0-),g)
\]
satisfies 
\[
|v(t)|\leq M\e^{\omega t}|(g(0-),g)|_{H\times H_{\mu}(\mathbb{R}_{\leq0};H)}.
\]
Using \prettyref{lem:r_g} we derive 
\begin{align*}
r_{g}(\lambda) & =\left(\mathcal{L}_{\lambda}v\right)(0)\\
 & =\frac{1}{\sqrt{2\pi}}\intop_{\mathbb{R}}\e^{-\lambda s}v(s)\mbox{ d}s\\
 & =\frac{1}{\sqrt{2\pi}}\intop_{0}^{\infty}\e^{-\lambda s}v(s)\mbox{ d}s
\end{align*}
for each $\lambda>\max\{\omega,\rho\}$ and consequently 
\[
r_{g}^{(n)}(\lambda)=\frac{1}{\sqrt{2\pi}}\intop_{0}^{\infty}\e^{-\lambda s}\left(-s\right)^{n}v(s)\mbox{ d}s.
\]
Hence, we have that 
\begin{align*}
|r_{g}^{(n)}(\lambda)| & \leq\frac{1}{\sqrt{2\pi}}M|(g(0-),g)|_{H\times H_{\mu}(\mathbb{R}_{\leq0};H)}\intop_{0}^{\infty}\e^{\left(\omega-\lambda\right)s}s^{n}\mbox{ d}s\\
 & =\frac{1}{\sqrt{2\pi}}M|(g(0-),g)|_{H\times H_{\mu}(\mathbb{R}_{\leq0};H)}n!\frac{1}{(\lambda-\omega)^{n}}\intop_{0}^{\infty}\e^{(\omega-\lambda)s}\mbox{ d}s\\
 & =\frac{1}{\sqrt{2\pi}}M|g(0-),g)|_{H\times H_{\mu}(\mathbb{R}_{\leq0};H)}\frac{n!}{(\lambda-\omega)^{n+1}},
\end{align*}
which gives the desired estimate. Assume now that there is $M\geq1,\omega\geq\rho$
such that 
\[
\frac{1}{n!}|r_{g}^{(n)}(\lambda)|\leq\frac{M}{(\lambda-\omega)^{n+1}}|(g(0-),g)|_{H\times H_{\mu}(\mathbb{R}_{\leq0};H)}\quad(\lambda>\omega).
\]
Then, the function $\tilde{r}:\mathbb{R}_{>0}\to H$ with $\tilde{r}(\lambda)=r(\lambda+\omega)$
satisfies the conditions of the Widder-Arendt \prettyref{thm:Widder-Arendt}
and thus, there is $f\in L_{\infty}(\mathbb{R}_{\geq0};H)$ with $|f|_{L_{\infty}}\leq M|(g(0-),g)|_{H\times H_{\mu}(\mathbb{R}_{\leq0};H)}$
such that 
\[
\tilde{r}_{g}(\lambda)=\intop_{0}^{\infty}\e^{-\lambda s}f(s)\mbox{ d}s\quad(\lambda>0).
\]
Hence, we have 
\[
\hat{v}(\lambda)=r_{g}(\lambda)=\tilde{r}_{g}(\lambda-\omega)=\intop_{0}^{\infty}\e^{-\left(\lambda-\omega\right)s}f(s)\mbox{ d}s\quad(\lambda>\omega)
\]
and by analytic continuation, we infer that 
\[
\hat{v}(z)=\intop_{0}^{\infty}\e^{-zs}\e^{\omega s}f(s)\mbox{ d}s\quad(z\in\mathbb{C}_{\Re>\omega}).
\]
By the injectivity of the Fourier-Laplace transform, we derive 
\[
v(t)=\sqrt{2\pi}\e^{\omega t}f(t)\quad(t\in\mathbb{R})
\]
and so, 
\[
|v(t)|_{H}=\sqrt{2\pi}\e^{\omega t}|f(t)|_{H}\leq\sqrt{2\pi}M\e^{\omega t}|\left(g(0-),g\right)|_{H\times H_{\mu}(\mathbb{R}_{\leq0};H)}.
\]
In other words, we have that 
\[
T_{(M,A)}^{(1)}:D_{\rho}\subseteq X_{\rho}^{\mu}\to C_{\omega}(\mathbb{R};H)
\]
is bounded and thus, can be extended to a bounded operator defined
on $X_{\rho}^{\mu}$. Moreover, for each $\varepsilon>0$ we have
by \prettyref{lem:boundedness T1 implies T2} 
\[
T_{(M,A)}^{(2)}:D_{\rho}\subseteq X_{\rho}^{\mu}\to C_{\omega+\varepsilon}(\mathbb{R}_{\geq0};H_{\mu}(\mathbb{R}_{\leq0};H))
\]
is bounded and hence, it can be extended to $X_{\rho}^{\mu}.$ Summarizing,
we have shown that $T_{(M,A)}$ can be extended to a $C_{0}$-semigroup
on $X_{\rho}^{\mu}.$ 
\end{proof}
We conclude this section by a result which relates the classical exponential
stability of $T_{(M,A)}$ with the exponential stability of the corresponding
evolutionary problem in the sense of Chapter 2. More precisely, we
show that for a certain class of material laws, the growth bound of
the semigroup $T_{(M,A)}$ can be estimated by $s_{0}(M,A),$ which
would yield the exponential stability of the semigroup if the evolutionary
problem is exponentially stable by \prettyref{thm:char_exp_stab}.
This theorem can be seen as a generalization of the well-known Gearhart-Prüß
Theorem and its proof indeed follows the lines of the proof of the
classical version presented in \cite[Theorem 5.2.1]{ABHN_2011}.
\begin{thm}
\label{thm:G-P}Let $\rho>\max\{s_{0}(M,A),b(M)\}$ and $s_{0}(M,A)<\mu\leq\rho$.
We assume that $T_{(M,A)}$ defines a $C_{0}$-semigroup on $X_{\rho}^{\mu}$,
that $\mathbb{C}_{\Re>\mu}\setminus D(M)$ is discrete and for each
$\lambda>0$ the mapping 
\[
\Phi_{\lambda}:\mathbb{C}_{\Re>\mu}\cap D(M)\ni z\mapsto\left(zM(z)-(z+\lambda)M(z+\lambda)\right)\in L(H)
\]
is bounded, Moreover, assume that 
\[
K_{\rho}:\mathrm{His}_{\rho}(M,A)\subseteq H_{\rho}(\mathbb{R}_{\leq0};H)\to H_{\mu}(\mathbb{R}_{\geq0};H)
\]
is well-defined and bounded. Then $\omega(T_{(M,A)})\leq\mu,$ where
$\omega(T_{(M,A)})$ denotes the growth bound of $T_{(M,A)}.$ \nomenclature[Z_080]{$\omega(T)$}{the growth type of a $C_0$-semigroup $T$.}\end{thm}
\begin{proof}
By Datko's Lemma (see \cite{Datko_1970} or \cite[Chapter V, Theorem 1.8]{engel2000one})
it suffices to prove 
\begin{equation}
\intop_{0}^{\infty}|T_{(M,A)}(t)(x,f)|_{X_{\rho}^{\mu}}^{2}\e^{-2(\mu+\varepsilon)t}\mbox{ d}t<\infty\label{eq:Datko}
\end{equation}
for each $(x,f)\in X_{\rho}^{\mu}$ and each $\varepsilon>0.$ For
doing so, let $\varepsilon>0$. We define 
\begin{align*}
S:\mathbb{C}_{\Re>\mu} & \to L(H)\\
z & \mapsto(zM(z)+A)^{-1}
\end{align*}
which is analytic and bounded since $\mbox{\ensuremath{\mu}>}s_{0}(M,A)$.
Let $g\in\mathrm{His}_{\rho}(M,A).$ As $T_{(M,A)}$ is a $C_{0}$-semigroup
on $X_{\rho}^{\mu}$ there exists $\omega\geq\rho$ such that 
\[
T_{(M,A)}^{(1)}:X_{\rho}^{\mu}\to C_{\omega}(\mathbb{R}_{\geq0};H)
\]
is continuous. Hence, we have $T_{(M,A)}^{(1)}(\cdot)(g(0-),g)=S_{\rho}\left(\delta_{0}\Gamma_{(M,A)}^{\rho}g-K_{\rho}g\right)\in C_{\omega}(\mathbb{R}_{\geq0};H)\hookrightarrow H_{\omega+1}(\mathbb{R}_{\geq0};H)$
and thus, 
\[
\left(z\mapsto S(z)\left(\frac{1}{\sqrt{2\pi}}\Gamma_{(M,A)}^{\rho}g-\widehat{K_{\rho}g}(z)\right)\right)\in\mathcal{H}^{2}(\mathbb{C}_{\Re>\omega+1};H)
\]
with 
\begin{align*}
\left\Vert S(\cdot)\left(\frac{1}{\sqrt{2\pi}}\Gamma_{(M,A)}^{\rho}g-\widehat{K_{\rho}g}\right)\right\Vert _{\mathcal{H}^{2}(\mathbb{C}_{\Re>\omega+1};H)} & =\|T_{(M,A)}^{(1)}(\cdot)(g(0-),g)\|_{H_{\omega+1}(\mathbb{R}_{\geq0};H)}\\
 & \leq C\|(g(0-),g)\|_{X_{\rho}^{\mu}}
\end{align*}
for some $C>0$ by the Theorem of Paley-Wiener (see \prettyref{cor:Paley_Wiener_unitary}).
Moreover, since $S\in\mathcal{H}^{\infty}(\mathbb{C}_{\Re>\mu};L(H))$
and $\widehat{K_{\rho}g}\in\mathcal{H}^{2}(\mathbb{C}_{\Re>\mu};H)$
(since $K_{\rho}g\in H_{\mu}(\mathbb{R}_{\geq0};H)$ by hypothesis),
we infer that
\[
\left(z\mapsto S(z)\widehat{K_{\rho}g}(z)\right)\in\mathcal{H}^{2}(\mathbb{C}_{\Re>\mu};H)\subseteq\mathcal{H}^{2}(\mathbb{C}_{\Re>\omega+1};H)
\]
and thus, 
\[
\left(z\mapsto S(z)\Gamma_{(M,A)}^{\rho}g\right)\in\mathcal{H}^{2}(\mathbb{C}_{\Re>\omega+1};H).
\]
We estimate for each $\mu<\kappa\leq\omega+1$ and $\lambda>0$ 
\begin{align*}
 & \left(\intop_{\mathbb{R}}\left|S(\i t+\kappa)\Gamma_{(M,A)}^{\rho}g\right|^{2}\mbox{ d}t\right)^{\frac{1}{2}}\\
\leq & \left(\intop_{\mathbb{R}}\left|S(\i t+\kappa+\lambda)\Gamma_{(M,A)}^{\rho}g\right|^{2}\mbox{ d}t\right)^{\frac{1}{2}}+\\
 & +\left(\intop_{\mathbb{R}}\left|\left(S(\i t+\kappa)-S(\i t+\kappa+\lambda)\right)\Gamma_{(M,A)}^{\rho}g\right|^{2}\mbox{ d}t\right)^{\frac{1}{2}}.
\end{align*}
Using 
\begin{align*}
 & S(\i t+\kappa)-S(\i t+\kappa+\lambda)\\
 & =\left(\left(\i t+\kappa\right)M(\i t+\kappa)+A\right)^{-1}-\left(\left(\i t+\kappa+\lambda\right)M(\i t+\kappa+\lambda)+A\right)^{-1}\\
 & =-\left(\left(\i t+\kappa\right)M(\i t+\kappa)+A\right)^{-1}\Phi_{\lambda}(\i t+\kappa)\left(\left(\i t+\kappa+\lambda\right)M(\i t+\kappa+\lambda)+A\right)^{-1}\\
 & =-S(\i t+\kappa)\Phi_{\lambda}(\i t+\kappa)S(\i t+\kappa+\lambda)
\end{align*}
 for each $t\in\mathbb{R}$, we can estimate the last integral by
\begin{align*}
 & \left(\intop_{\mathbb{R}}\left|\left(S(\i t+\kappa)-S(\i t+\kappa+\lambda)\right)\Gamma_{(M,A)}^{\rho}g\right|^{2}\mbox{ d}t\right)^{\frac{1}{2}}\\
\leq & \|S\|_{\mathcal{H}^{\infty}(\mathbb{C}_{\Re>\mu};L(H))}\|\Phi_{\lambda}\|_{\infty}\left(\intop_{\mathbb{R}}\left|S(\i t+\kappa+\lambda)\Gamma_{(M,A)}^{\rho}g\right|^{2}\mbox{ d}t\right)^{\frac{1}{2}}
\end{align*}
and hence, for $\lambda>\omega+1-\mu$ 
\[
\left(\intop_{\mathbb{R}}\left|S(\i t+\kappa)\Gamma_{(M,A)}^{\rho}g\right|^{2}\mbox{ d}t\right)^{\frac{1}{2}}\leq\tilde{C}|S(\cdot)\Gamma_{(M,A)}^{\rho}g|_{\mathcal{H}^{2}(\mathbb{C}_{\Re>\omega+1};H)},
\]
with $\tilde{C}\coloneqq1+\|S\|_{\mathcal{H}^{\infty}(\mathbb{C}_{\Re>\mu};L(H))}\|\Phi_{\lambda}\|_{\infty}.$
Since $\kappa$ was arbitrary, we infer $S(\cdot)\Gamma_{(M,A)}^{\rho}g\in\mathcal{H}^{2}(\mathbb{C}_{\Re>\mu};H)$
and thus, 
\[
\left(z\mapsto S(z)\left(\frac{1}{\sqrt{2\pi}}\Gamma_{(M,A)}^{\rho}g-\widehat{K_{\rho}g}(z)\right)\right)\in\mathcal{H}^{2}(\mathbb{C}_{\Re>\mu};H).
\]
Again, by Paley-Wiener we have that $T_{(M,A)}^{(1)}(\cdot)(g(0-),g)\in H_{\mu}(\mathbb{R}_{\geq0};H)$
with 
\begin{align*}
 & \left|T_{(M,A)}^{(1)}(\cdot)(g(0-),g)\right|_{H_{\mu}(\mathbb{R}_{\geq0};H)}\\
 & =\left|S(\cdot)\left(\frac{1}{\sqrt{2\pi}}\Gamma_{(M,A)}^{\rho}g-\widehat{K_{\rho}g}\right)\right|_{\mathcal{H}^{2}(\mathbb{C}_{\Re>\mu};H)}\\
 & \leq\frac{1}{\sqrt{2\pi}}\left|S(\cdot)\Gamma_{(M,A)}^{\rho}g\right|_{\mathcal{H}^{2}(\mathbb{C}_{\Re>\mu};H)}+\left|S(\cdot)\widehat{K_{\rho}g}\right|_{\mathcal{H}^{2}(\mathbb{C}_{\Re>\mu};H)}\\
 & \leq\frac{1}{\sqrt{2\pi}}\tilde{C}|S(\cdot)\Gamma_{(M,A)}^{\rho}g|_{\mathcal{H}^{2}(\mathbb{C}_{\Re>\omega+1};H)}+\|S\|_{\mathcal{H}^{\infty}(\mathbb{C}_{\Re>\mu};L(H))}|K_{\rho}g|_{H_{\mu}(\mathbb{R}_{\geq0};H)}\\
 & \leq\frac{1}{\sqrt{2\pi}}\tilde{C}\left|S(\cdot)\left(\Gamma_{(M,A)}^{\rho}g-\widehat{K_{\rho}g}\right)\right|_{\mathcal{H}^{2}(\mathbb{C}_{\Re>\omega+1};H)}+\\
 & \phantom{aaaa}+\left(\frac{\tilde{C}}{\sqrt{2\pi}}+1\right)\|S\|_{\mathcal{H}^{\infty}(\mathbb{C}_{\Re>\mu};L(H))}\|K_{\rho}\||g|_{H_{\rho}(\mathbb{R}_{\leq0};H)}\\
 & \leq\left(\frac{\tilde{C}}{\sqrt{2\pi}}C+\left(\frac{\tilde{C}}{\sqrt{2\pi}}+1\right)\|S\|_{\mathcal{H}^{\infty}(\mathbb{C}_{\Re>\mu};L(H))}\|K_{\rho}\|\right)|(g(0-),g)|_{H\times H_{\rho}(\mathbb{R}_{\leq0};H)}
\end{align*}
or in other words 
\[
T_{(M,A)}^{(1)}:D_{\rho}\subseteq X_{\rho}^{\mu}\to H_{\mu}(\mathbb{R}_{\geq0};H)
\]
is well-defined and bounded. Thus, using \prettyref{rem:T^1 und T^2},
we obtain 
\begin{align*}
 & \intop_{0}^{\infty}|T_{(M,A)}^{(2)}(t)(g(0-),g)|_{H_{\mu}(\mathbb{R}_{\leq0};H)}^{2}\e^{-2(\mu+\varepsilon)t}\mbox{ d}t\\
 & =\intop_{0}^{\infty}\left(\:\intop_{-\infty}^{-t}|g(t+s)|^{2}\e^{-2\mu s}\mbox{ d}s+\intop_{-t}^{0}|T_{(M,A)}^{(1)}(t+s)(g(0-),g)|^{2}\e^{-2\mu s}\mbox{ d}s\right)\e^{-2\left(\mu+\varepsilon\right)t}\mbox{ d}t\\
 & =\frac{1}{2\varepsilon}\left(|g|_{H_{\mu}(\mathbb{R}_{\leq0};H)}^{2}+|T_{(M,A)}^{(1)}(\cdot)(g(0-),g)|_{H_{\mu}(\mathbb{R}_{\geq0};H)}^{2}\right)\\
 & \leq C|(g(0-),g)|_{X_{\rho}^{\mu}}^{2}
\end{align*}
which proves the well-definedness and boundedness of 
\[
T_{(M,A)}^{(2)}:D_{\rho}\subseteq X_{\rho}^{\mu}\to H_{\mu+\varepsilon}(\mathbb{R}_{\geq0};H_{\mu}(\mathbb{R}_{\leq0};H)).
\]
Thus, 
\[
T_{(M,A)}:D_{\rho}\subseteq X_{\rho}^{\mu}\to H_{\mu+\varepsilon}(\mathbb{R}_{\geq0};X_{\rho}^{\mu})
\]
is bounded and hence, \prettyref{eq:Datko} follows by continuous
extension. 
\end{proof}

\section{Applications\label{sec:Applications}}

In this section, we apply the results of the three previous sections
to several examples. More precisely, we  discuss how the results can
be applied to differential-algebraic equations, delay equations and
integro-differential equations. Throughout, we assume that $A:D(A)\subseteq H\to H$
is a densely defined closed linear operator and that the evolutionary
problem associated with $M$ and $A$ is well-posed, where the material
law $M$ will be specified in each case. For each of these examples
we need to verify certain conditions for the material law in order
to apply our results. For the ease of readability we recall these
assumptions in the following. \\
First of all we recall the definition 
\[
b(M)\coloneqq\inf\{\rho\geq0\,;\,\mathbb{C}_{\Re\geq\rho}\subseteq D(M),M|_{\mathbb{C}_{\Re\geq\rho}}\mbox{ bounded}\}
\]
In order to define the space of admissible histories, we assume that
$M$ is regularizing, i.e. 
\begin{equation}
\left(b(M)<\infty\right)\wedge\left(\forall\rho>b(M),x\in H:\left(M(\partial_{0,\rho})\chi_{\mathbb{R}_{\geq0}}x\right)(0+)\right)\mbox{ exists.}\tag{REG}\label{eq:M_reg}
\end{equation}
We fix $\rho>b(M)$ and recall the definitions
\begin{align*}
K_{\rho}:\chi_{\mathbb{R}_{\leq0}}\left[H_{\rho}^{1}(\mathbb{R};H)\right] & \to H_{\rho}^{-1}(\mathbb{R};H)\\
g & \mapsto P_{0}\partial_{0,\rho}M(\partial_{0,\rho})g
\end{align*}
and 
\begin{align*}
\Gamma_{(M,A)}^{\rho}:\mathrm{His}_{\rho}(M,A) & \to H\\
g & \mapsto\left(M(\partial_{0,\rho})g\right)(0-)-\left(M(\partial_{0,\rho})g\right)(0+),
\end{align*}
where we remark that 
\[
\Gamma_{(M,A)}^{\rho}g=\left(M(\partial_{0,\rho})\chi_{\mathbb{R}_{\geq0}}g(0-)\right)(0+)\quad(g\in\mathrm{His}_{\rho}(M,A))
\]
by \prettyref{lem:Gamma}. Furthermore, we recall 
\begin{align*}
r_{g}:\mathbb{R}_{>\rho} & \to H\\
\lambda & \mapsto\left(\lambda M(\lambda)+A\right)^{-1}\left(\Gamma_{(M,A)}^{\rho}g-\left(\mathcal{L}_{\lambda}K_{\rho}g\right)(0)\right),
\end{align*}
where $M$ satisfies \prettyref{eq:M_reg}, $\rho>\max\{s_{0}(M,A),b(M)\}$
and $g\in\mathrm{His}_{\rho}(M,A)$. By \prettyref{thm:sg}, this
functions needs to verify the Hille-Yosida type condition, i.e. there
exist $M\geq1,\omega>\rho$ such that 
\begin{equation}
\forall g\in\mathrm{His}_{\rho}(M,A),n\in\mathbb{N},\lambda>\omega:\frac{1}{n!}|r_{g}^{(n)}(\lambda)|_{H}\leq\frac{M}{(\lambda-\omega)^{n+1}}|(g(0-),g)|_{H\times H_{\mu}(\mathbb{R}_{\leq0};H)}\tag{HY}\label{eq:Hille_Yosida}
\end{equation}
for some $\mu\leq\rho$ in order to obtain a strongly continuous semigroup
$T_{(M,A)}$ on the Hilbert space $X_{\rho}^{\mu}\coloneqq\overline{D_{\rho}}^{H\times H_{\mu}(\mathbb{R}_{\leq0};H)}$,
where 
\[
D_{\rho}\coloneqq\left\{ (g(0-),g)\,;\, g\in\mathrm{His}_{\rho}(M,A)\right\} .
\]
Finally, we recall the conditions for the generalized Gearhart-Prüß
Theorem \prettyref{thm:G-P}, which provides an upper bound for the
growth-type of the semigroup on $X_{\rho}^{\mu}$ of the form $\omega(T_{(M,A)})\leq\mu$
for some $s_{0}(M,A)<\mu\leq\rho$:
\begin{equation}
\Phi_{\lambda}:\mathbb{C}_{\Re>\mu}\cap D(M)\ni z\mapsto\left(zM(z)-(z+\lambda)M(z+\lambda)\right)\in L(H)\mbox{ is bounded}\tag{GP-1}\label{eq:GP1}
\end{equation}
for each $\lambda>0$ and 
\begin{equation}
K_{\rho}:\mathrm{His}_{\rho}(M,A)\subseteq H_{\rho}(\mathbb{R}_{\leq0};H)\to H_{\mu}(\mathbb{R}_{\geq0};H)\mbox{ is well-defined and bounded}.\tag{GP-2}\label{eq:GP2}
\end{equation}

\subsection{Differential-algebraic equations}

According to \prettyref{prop:material_law_amnesic} evolutionary problems
without memory are precisely those, where the material law $M$ is
given by 
\[
M(z)\coloneqq M_{0}+z^{-1}M_{1}\quad(z\in\mathbb{C}\setminus\{0\})
\]
for some bounded linear operators $M_{0},M_{1}\in L(H).$ Of course,
those operators satisfy \prettyref{eq:M_reg}. Indeed, we have that
$b(M)=0$ and for each $\rho>0$ 
\[
\left(M(\partial_{0,\rho})\chi_{\mathbb{R}_{\geq0}}x\right)(0+)=M_{0}x
\]
and consequently, 
\[
\Gamma_{(M,A)}^{\rho}g=M_{0}g(0-)
\]
for all $g\in\mathrm{His}_{\rho}(M,A)$ with $\rho>\max\{s_{0}(M,A),0\}$,
which is given by 
\[
\mathrm{His}_{\rho}(M,A)=\left\{ g\in\chi_{\mathbb{R}_{\leq0}}(\m)\left[H_{\rho}^{1}(\mathbb{R};H)\right]\,;\, S_{\rho}\left(\delta_{0}M_{0}g(0-)\right)-\chi_{\mathbb{R}_{\geq0}}g(0-)\in H_{\rho}^{1}(\mathbb{R};H)\right\} 
\]
according to \prettyref{prop:material_law_amnesic}. The latter gives
\[
\mathrm{IV}_{\rho}(M,A)=\left\{ u_{0}\in H\,;\, S_{\rho}\left(\delta_{0}M_{0}u_{0}\right)-\chi_{\mathbb{R}_{\geq0}}u_{0}\in H_{\rho}^{1}(\mathbb{R};H)\right\} .
\]

\begin{prop}
\label{prop:IV_diff_alg}Assume that $M_{0}$ has closed range. Denote
by $\iota_{0}:N(M_{0})^{\bot}\to H$ and by $\iota_{1}:R(M_{0})\to H$
the canonical embeddings of the orthogonal complement of the null-space
and the range of $M_{0},$ respectively. Then, $\iota_{1}^{\ast}M_{0}\iota_{0}:N(M_{0})^{\bot}\to R(M_{0})$
is boundedly invertible and 
\[
\left\{ u_{0}\in D(A)\,;\,(M_{1}+A)u_{0}\in R(M_{0}),\iota_{0}\left(\iota_{1}^{\ast}M_{0}\iota_{0}\right)^{-1}\iota_{1}^{\ast}(M_{1}+A)u_{0}\in D(A)\right\} \subseteq\mathrm{IV}_{\rho}(M,A)
\]
for each $\rho>\max\{s_{0}(M,A),0\}.$ \end{prop}
\begin{proof}
Note that $\iota_{1}^{\ast}M_{0}\iota_{0}$ is a bijective closed
operator. Thus, the bounded invertibility follows by the closed graph
theorem. Let now $u_{0}\in D(A)$ such that $(M_{1}+A)u_{0}\in R(M_{0})$
and $\iota_{0}\left(\iota_{1}^{\ast}M_{0}\iota_{0}\right)^{-1}\iota_{1}^{\ast}(M_{1}+A)u_{0}\in D(A)$.
Then we have that\nomenclature[O_200]{$S_\rho$}{the solution operator $(\overline{\partial_{0,\rho}M(\partial_{0,\rho})+A})^{-1}$.}
\begin{align*}
S_{\rho}\left(\delta_{0}M_{0}u_{0}\right)-\chi_{\mathbb{R}_{\geq0}}u_{0} & =S_{\rho}\left(\delta_{0}M_{0}u_{0}-\left(\partial_{0,\rho}M_{0}+M_{1}+A\right)\chi_{\mathbb{R}_{\geq0}}u_{0}\right)\\
 & =-S_{\rho}\left(\left(M_{1}+A\right)\chi_{\mathbb{R}_{\geq0}}u_{0}\right).
\end{align*}
Since $(M_{1}+A)u_{0}\in R(M_{0})$ we have that 
\[
(M_{1}+A)u_{0}=M_{0}\iota_{0}\left(\iota_{1}^{\ast}M_{0}\iota_{0}\right)^{-1}\iota_{1}^{\ast}(M_{1}+A)u_{0}.
\]
We set $x\coloneqq\iota_{0}\left(\iota_{1}^{\ast}M_{0}\iota_{0}\right)^{-1}\iota_{1}^{\ast}(M_{1}+A)u_{0}$
and derive, 
\begin{align*}
\partial_{0,\rho}\left(S_{\rho}\left(\delta_{0}M_{0}u_{0}\right)-\chi_{\mathbb{R}_{\geq0}}u_{0}\right) & =-S_{\rho}\left(\partial_{0,\rho}\left(M_{1}+A\right)\chi_{\mathbb{R}_{\geq0}}u_{0}\right)\\
 & =-S_{\rho}\left(\partial_{0,\rho}M_{0}\chi_{\mathbb{R}_{\geq0}}x\right).
\end{align*}
Using that $x\in D(A)$ by assumption, we conclude 
\[
\partial_{0,\rho}\left(S_{\rho}\left(\delta_{0}M_{0}u_{0}\right)-\chi_{\mathbb{R}_{\geq0}}u_{0}\right)=-\chi_{\mathbb{R}_{\geq0}}x+S_{\rho}\left(\chi_{\mathbb{R}_{\geq0}}\left(M_{1}+A\right)x\right)\in H_{\rho}(\mathbb{R};H),
\]
which shows $u_{0}\in\mathrm{IV}_{\rho}(M,A).$ \end{proof}
\begin{rem}
\label{rem:IV_evolution}$\,$

\begin{enumerate}[(a)]

\item In the case of a classical evolution equation, i.e. in the
case $M_{0}=1$ and $M_{1}=0$, the latter proposition yields that
\[
D(A^{2})\subseteq\mathrm{IV}_{\rho}(M,A).
\]
In case of an ordinary differential-algebraic equation, i.e. $A=0$
in our setting, we can even show more.

\item Another approach to obtain solutions for a class of differential-algebraic
equations is used in \cite{wehowski2015}, where even nonlinear problems
are considered. There, one applies a backward difference scheme to
obtain a discretised version of the equation, which turns out to be
well-posed. Then one shows that these approximating solutions converge
to a solution of the original problem in a suitable sense. Using this
approach, a condition for admissible initial values is derived, even
in the case, when $R(M_{0})$ is not closed. 

\end{enumerate}\end{rem}
\begin{prop}
Assume that $M_{0}$ has closed range. Then 
\[
\left\{ u_{0}\in H\,;\, M_{1}u_{0}\in R(M_{0})\right\} =\mathrm{IV}_{\rho}(M,0)
\]
for each $\rho>\max\{s_{0}(M,0),0\}.$ In particular, $\mathrm{IV}_{\rho}(M,0)$
is closed and independent of $\rho.$ \end{prop}
\begin{proof}
According to \prettyref{prop:IV_diff_alg} we have that 
\[
\left\{ u_{0}\in H\,;\, M_{1}u_{0}\in R(M_{0})\right\} \subseteq\mathrm{IV}_{\rho}(M,0).
\]
Assume now $u_{0}\in\mathrm{IV}_{\rho}(M,0).$ Then 
\[
u\coloneqq\left(\partial_{0,\rho}M_{0}+M_{1}\right)^{-1}\delta_{0}M_{0}u_{0}
\]
satisfies\nomenclature[B_100]{$H_\rho^1(\mathbb{R};H)$}{the domain of $\partial_{0,\rho}$ equipped with the norm $\vert\partial_{0,\rho}\cdot\vert_{H_\rho(\mathbb{R};H)}$.}\nomenclature[B_110]{$H_{\rho}^{-1}(\mathbb{R};H)$}{the first extrapolation space of $\partial_{0,\rho}$, the completion of $H_\rho(\mathbb{R};H)$ with respect to the norm $\vert\partial^{-1}_{0,\rho}\cdot\vert_{H_\rho(\mathbb{R};H)}$.}
\[
u-\chi_{\mathbb{R}_{\geq0}}u_{0}\in H_{\rho}^{1}(\mathbb{R};H),
\]
by definition of $\mathrm{IV}_{\rho}(M,0).$ Thus, 
\[
M_{1}u=\partial_{0,\rho}M_{0}(u-\chi_{\mathbb{R}_{\geq0}}u_{0})=M_{0}\partial_{0,\rho}(u-\chi_{\mathbb{R}_{\geq0}}u_{0}),
\]
which implies 
\[
M_{1}u(t)\in R(M_{0})\quad(t\in\mathbb{R}\mbox{ a.e.}).
\]
Hence, there is a sequence $(t_{n})_{n\in\mathbb{N}}$ in $\mathbb{R}_{>0}$
such that $t_{n}\to0$ and $M_{1}u(t_{n})\in R(M_{0})$ for each $n\in\mathbb{N}.$
Since $u(t_{n})\to u_{0}$ due to continuity, we infer 
\[
M_{1}u_{0}\in R(M_{0})
\]
by the closedness of $R(M_{0}).$ \end{proof}
\begin{example}[DAEs in finite dimensions ]
\label{exa:consistent_IV}Assume that $H=\mathbb{C}^{n}$ and $M_{0},M_{1}\in\mathbb{C}^{n\times n}$
and that $s_{0}(M,0)<\infty$ with $M(z)=M_{0}+z^{-1}M_{1}.$ By the
previous proposition we know that 
\[
\mathrm{IV}_{\rho}(M,0)=\left\{ u_{0}\in\mathbb{C}^{n}\,;\, M_{1}u_{0}\in R(M_{0})\right\} 
\]
for each $\rho\geq\max\{s_{0}(M,0),0\}$. The class of ordinary differential-algebraic
equations in finite dimensions have been widely studied and the set
of so-called \emph{consistent initial values} is well-known and given
as follows: By the Weierstraß normal form for the matrix pair $(M_{0},M_{1}),$
there exist regular matrices $P,Q\in\mathbb{C}^{n\times n}$ such
that 
\[
PM_{0}Q=\left(\begin{array}{cc}
1 & 0\\
0 & 0
\end{array}\right),\quad PM_{1}Q=\left(\begin{array}{cc}
J & 0\\
0 & 1
\end{array}\right),
\]
where $J\in\mathbb{C}^{d\times d}$ is a Jordan matrix (see e.g. \cite[Theorem 2.7]{Mehrmann2006}
and note that $s_{0}(M,0)<\infty$ implies $N=0$). The set of consistent
initial values is given by (see \cite[Theorem 2.12]{Mehrmann2006})
\[
B\coloneqq\left\{ u_{0}\in\mathbb{C}^{n}\,;\, Q^{-1}u_{0}\in\mathbb{C}^{d}\times\{0\}\right\} .
\]
We show that $B=\mathrm{IV}_{\rho}(M,0).$ If $u_{0}\in B$ then 
\[
M_{1}u_{0}=P^{-1}PM_{1}QQ^{-1}u_{0}=P^{-1}\left(\begin{array}{cc}
J & 0\\
0 & 1
\end{array}\right)Q^{-1}u_{0}=P^{-1}\left(\begin{array}{c}
Jy\\
0
\end{array}\right),
\]
where $(y,0)\coloneqq Q^{-1}u_{0}\in\mathbb{C}^{d}\times\{0\}.$ Setting
$x\coloneqq Q\left(\begin{array}{c}
Jy\\
0
\end{array}\right)$ we compute 
\[
M_{0}x=P^{-1}\left(\begin{array}{cc}
1 & 0\\
0 & 0
\end{array}\right)Q^{-1}x=P^{-1}\left(\begin{array}{c}
Jy\\
0
\end{array}\right)=M_{1}u_{0}
\]
and thus, $u_{0}\in\mathrm{IV}_{\rho}(M,0).$ Let now $u_{0}\in\mathrm{IV}_{\rho}(M,0).$
Then there is $x\in H$ such that 
\[
M_{1}u_{0}=M_{0}x
\]
or equivalently, 
\[
\left(\begin{array}{cc}
J & 0\\
0 & 1
\end{array}\right)Q^{-1}u_{0}=\left(\begin{array}{cc}
1 & 0\\
0 & 0
\end{array}\right)Q^{-1}x,
\]
which yields $u_{0}\in B.$ 
\end{example}
In the special setting of differential-algebraic equations we are
able to strengthen the Hille-Yosida type result \prettyref{thm:sg}
in the following way.

\newpage{}
\begin{prop}
\label{prop:HY}Let $\rho>\max\{s_{0}(M,A),0\}$. Then the following
statements are equivalent:

\begin{enumerate}[(i)]

\item There exists $\mu\leq\rho$ such that $T_{(M,A)}$ extends
to a $C_{0}$-semigroup on $X_{\rho}^{\mu}$.

\item There exists $\mu\leq\rho$ such that \prettyref{eq:Hille_Yosida}
holds.

\item There exist $M\geq1$ and $\omega\geq\rho$, such that for
each $x\in\mathrm{IV}_{\rho}(M,A),n\in\mathbb{N}$ we have 
\begin{equation}
\frac{1}{n!}|s_{x}^{(n)}(\lambda)|_{H}\leq\frac{M}{(\lambda-\omega)^{n+1}}|x|_{H}\quad(\lambda>\omega),\label{eq:HY_diff_alg}
\end{equation}
where 
\begin{align*}
s_{x}:\mathbb{R}_{>\omega} & \to H\\
\lambda & \to(\lambda M_{0}+M_{1}+A)^{-1}(M_{0}x).
\end{align*}

\item $T_{(M,A)}$ extends to a $C_{0}$-semigroup on $X_{\rho}^{\mu}$
for each $\mu\leq\rho.$

\item $T_{(M,A)}^{(1)}(\cdot,0)$ extends to a $C_{0}$-semigroup
on $\overline{\mathrm{IV}_{\rho}(M,A)}^{H}$.

\end{enumerate}\end{prop}
\begin{proof}
(i) $\Rightarrow$ (ii) holds by \prettyref{thm:sg}.\\
(ii) $\Rightarrow$(iii): There exist $\mu\leq\rho$, $M\geq1,\omega\geq\rho$
such that 
\[
\frac{1}{n!}|r_{g}^{(n)}(\lambda)|_{H}\leq\frac{M}{(\lambda-\omega)^{n+1}}|(g(0-),g)|_{X_{\rho}^{\mu}}\quad(\lambda>\omega,n\in\mathbb{N})
\]
for each $g\in\mathrm{His}_{\rho}(M,A).$ Let $x\in\mathrm{IV}_{\rho}(M,A)$
and define $g_{k}(t)\coloneqq\chi_{[-\frac{1}{k},0]}(t)(kt+1)x$ for
$k\in\mathbb{N},t\leq0.$ Then $g_{k}\in\mathrm{His}_{\rho}(M,A)$
by \prettyref{prop:material_law_amnesic} and $g_{k}(0-)=x.$ Recall
that 
\[
r_{g_{k}}(\lambda)=(\lambda M_{0}+M_{1}+A)^{-1}(M_{0}x)=s_{x}(\lambda)\quad(\lambda>\omega)
\]
by \prettyref{prop:material_law_amnesic} for each $k\in\mathbb{N}.$
Thus, we have that 
\[
\frac{1}{n!}|s_{x}^{(n)}(\lambda)|_{H}=\lim_{k\to\infty}\frac{1}{n!}|r_{g_{k}}^{(n)}(\lambda)|_{H}\leq\lim_{k\to\infty}\frac{M}{(\lambda-\omega)^{n+1}}|(x,g_{k})|_{X_{\rho}^{\mu}}=\frac{M}{(\lambda-\omega)^{n+1}}|x|_{H}
\]
for every $n\in\mathbb{N}$ and $\lambda>\omega$.\\
(iii) $\Rightarrow$(iv): Let $\mu\leq\rho.$ We recall that for each
$g\in\mathrm{His}_{\rho}(M,A)$ we have that 
\[
r_{g}=s_{g(0-)}.
\]
Thus, 
\[
\frac{1}{n!}|r_{g}^{(n)}(\lambda)|_{H}=\frac{1}{n!}|s_{g(0-)}^{(n)}(\lambda)|_{H}\leq\frac{M}{(\lambda-\omega)^{n+1}}|g(0-)|_{H}\leq\frac{M}{(\lambda-\omega)^{n+1}}|\left(g(0-),g\right)|_{H\times H_{\mu}(\mathbb{R}_{\leq0};H)}
\]
and hence, the assertion follows from \prettyref{thm:sg}.\\
(iv) $\Rightarrow$(v): Let $\mu\leq\rho$. We first prove that $(x,0)\in X_{\rho}^{\mu}$
for each $x\in\overline{\mathrm{IV}_{\rho}(M,A)}$. For doing so,
let $x\in\overline{\mathrm{IV}_{\rho}(M,A)}$ and choose a sequence
$(x_{n})_{n\in\mathbb{N}}$ in $\mathrm{IV}_{\rho}(M,A)$ such that
$x_{n}\to x$ as $n\to\infty.$ Define $g_{n}(t)\coloneqq\chi_{[-\frac{1}{k},0]}(t)(kt+1)x_{n}$
for $n\in\mathbb{N},t\leq0.$ Then $g_{n}\in\mathrm{His}_{\rho}(M,A)$
and thus, $(x_{n},g_{n})\in D_{\rho}.$ Since $g_{n}\to0$ in $H_{\mu}(\mathbb{R}_{\leq0};H)$
we infer that $(x,0)\in X_{\rho}^{\mu}.$ Hence, $T_{(M,A)}^{(1)}(\cdot,0)$
is a well-defined semigroup on $\overline{\mathrm{IV}_{\rho}(M,A)},$
which is strongly continuous, since $T_{(M,A)}$ is strongly continuous.
\\
(v) $\Rightarrow$(i): It suffices to prove that the operator $T_{(M,A)}:D_{\rho}\subseteq X_{\rho}^{\mu}\to C_{\omega}(\mathbb{R}_{\geq0};H\times H_{\mu}(\mathbb{R}_{\leq0};H))$
is bounded for some $\omega\in\mathbb{R}.$ By assumption, we have
that 
\[
T_{(M,A)}^{(1)}(\cdot,0):\overline{\mathrm{IV}_{\rho}(M,A)}\to C_{\omega'}(\mathbb{R}_{\geq0};H)
\]
is bounded for some $\omega'\in\mathbb{R}$. We assume without loss
of generality that $\omega'\geq\rho$ and set $\omega\coloneqq\omega'+1.$
Let $g\in\mathrm{His}_{\rho}(M,A)$ and recall that $K_{\rho}g=0$
by \prettyref{prop:material_law_amnesic}. Then we have that 
\[
T_{(M,A)}^{(1)}(\cdot)(g(0-),g)=S_{\rho}\left(\delta_{0}\Gamma_{(M,A)}^{\rho}g-K_{\rho}g\right)=S_{\rho}(\delta_{0}M_{0}g(0-))=T_{(M,A)}^{(1)}(\cdot)(g(0-),0)
\]
and thus,
\[
T_{(M,A)}^{(1)}:D_{\rho}\subseteq X_{\rho}^{\mu}\to C_{\omega'}(\mathbb{R}_{\geq0};H)
\]
is bounded. By \prettyref{lem:boundedness T1 implies T2} we infer
that 
\[
T_{(M,A)}^{(2)}:D_{\rho}\subseteq X_{\rho}^{\mu}\to C_{\omega}(\mathbb{R}_{\geq0};H)
\]
is bounded, which gives the assertion.\end{proof}
\begin{rem}
\label{rem:HY_classic}\begin{enumerate}[(a)]

\item An easy computation gives that 
\[
s_{x}^{(n)}(\lambda)=(-1)^{n}n!\left(\left(\lambda M_{0}+M_{1}+A\right)^{-1}M_{0}\right)^{n+1}x
\]
for each $n\in\mathbb{N},x\in H,\lambda\geq\rho,$ where $s_{x}$
is given as in \prettyref{prop:HY} (iii). Hence, \prettyref{eq:HY_diff_alg}
holds if and only if 
\[
\left|\left(\left(\lambda M_{0}+M_{1}+A\right)^{-1}M_{0}\right)^{n+1}x\right|_{H}\leq\frac{M}{(\lambda-\omega)^{n+1}}|x|_{H}\quad(\lambda>\omega,n\in\mathbb{N},x\in\mathrm{IV}_{\rho}(M,A))
\]
for some $M\geq1,\omega\geq\rho.$ 

\item If $M_{0}=1,M_{1}=0$ the equivalence (iii)$\Leftrightarrow$
(v) is the well-known Hille-Yosida Theorem (see e.g. \cite[Ch. II, Th. 3.8]{engel2000one}).
Indeed, we have that 
\[
s_{x}(\lambda)=(\lambda+A)^{-1}x
\]
for $x\in\mathrm{IV}_{\rho}(M,A)$ and hence, (iii) in \prettyref{prop:HY}
gives that 
\[
|(\lambda+A)^{-n}x|_{H}\leq\frac{M}{(\lambda-\omega)^{n}}|x|_{H}
\]
 for some $M\geq1$ and $\omega\geq\rho.$ Using $D(A^{2})\subseteq\mathrm{IV}_{\rho}(M,A)$
by \prettyref{rem:IV_evolution} we derive 
\[
\|(\lambda+A)^{-n}\|\leq\frac{M}{(\lambda-\omega)^{n}},
\]
which is the classical Hille-Yosida condition for $-A$ to be a generator
of a $C_{0}$-semigroup which is nothing but $T_{(1,A)}^{(1)}(\cdot,0)$
on $\overline{\mathrm{IV}_{\rho}(M,A)}=H$. 

\end{enumerate}\end{rem}
\begin{prop}
\label{prop:GP_amnesic} Let $\rho>\max\{s_{0}(M,A),0\}$ and assume
that \prettyref{eq:Hille_Yosida} holds. Then $\omega(T_{(M,A)})\leq s_{0}(M,A).$ \end{prop}
\begin{proof}
Let $s_{0}(M,A)<\mu\leq\rho.$ We apply \prettyref{thm:G-P}. First
we note that $\mathbb{C}_{\Re>\mu}\setminus D(M)\subseteq\{0\}$ and
hence, it is discrete. Moreover, by \prettyref{prop:HY} $T_{(M,A)}$
extends to a $C_{0}$-semigroup on $X_{\rho}^{\mu}.$ We have to show
that \prettyref{eq:GP1} and \prettyref{eq:GP2} hold. This, however
is clear, since 
\[
\Phi_{\lambda}(z)=zM(z)-(z+\lambda)M(z+\lambda)=-\lambda M_{0}
\]
for each $z\in\mathbb{C}_{\Re>\mu}\setminus\{0\},\lambda>0$ and hence,
\prettyref{eq:GP1} is satisfied. Furthermore, $K_{\rho}=0$ by \prettyref{prop:material_law_amnesic}
and hence, \prettyref{eq:GP2} holds trivially. Thus, \prettyref{thm:G-P}
is applicable and hence, $\omega(T_{(M,A)})\leq\mu.$ Since $\mu>s_{0}(M,A)$
was arbitrary, we infer $\omega(T_{(M,A)})\leq s_{0}(M,A).$ \end{proof}
\begin{rem}
In the case $M_{0}=1,M_{1}=0$ we get $\omega(T)\leq s_{0}(1,A)$,
which is the well-known Gearhart-Prüß Theorem for $C_{0}$-semigroups
on Hilbert spaces (see e.g. \cite{Pruss1984} or \cite[Ch. V,Th. 1.11]{engel2000one}). 
\end{rem}
We conclude this subsection, by discussing an abstract version of
the heat equation in the presented framework.
\begin{example}
\label{exa:diff_alg}Let $H_{0},H_{1}$ be Hilbert spaces and $C:D(C)\subseteq H_{0}\to H_{1}$
densely defined closed and linear. We consider the evolutionary equation
\[
\left(\partial_{0,\rho}\left(\begin{array}{cc}
\eta & 0\\
0 & 0
\end{array}\right)+\left(\begin{array}{cc}
0 & 0\\
0 & k^{-1}
\end{array}\right)+\left(\begin{array}{cc}
0 & -C^{\ast}\\
C & 0
\end{array}\right)\right)\left(\begin{array}{c}
\vartheta\\
q
\end{array}\right)=F,
\]
where $\eta\in L(H_{0}),k\in L(H_{1})$ are selfadjoint, strictly
$m$-accretive operators with 
\begin{align*}
\langle\eta x_{0}|x_{0}\rangle_{H_{0}} & \geq c_{0}|x_{0}|_{H_{0}}^{2},\\
\langle kx_{1}|x_{1}\rangle_{H_{1}} & \geq c_{1}|x_{1}|_{H_{1}}^{2}
\end{align*}
for some $c_{0},c_{1}>0$ and each $(x_{0},x_{1})\in H_{0}\oplus H_{1}\eqqcolon H$.
Thus, in the setting of differential-algebraic equations we have 
\[
M_{0}=\left(\begin{array}{cc}
\eta & 0\\
0 & 0
\end{array}\right),\quad M_{1}=\left(\begin{array}{cc}
0 & 0\\
0 & k^{-1}
\end{array}\right),\quad A=\left(\begin{array}{cc}
0 & -C^{\ast}\\
C & 0
\end{array}\right).
\]
We note that the evolutionary problem is well-posed by \prettyref{prop:classical_material_law}.
We have that $s_{0}(M,A)\leq0.$ Indeed, for $\rho>0$ and $z\in\mathbb{C}_{\Re>\rho}$
we get 
\[
\Re\langle(zM_{0}+M_{1})x|x\rangle_{H}\geq\min\left\{ \rho c_{0},c_{1}^{-1}\right\} |x|_{H}^{2}
\]
for each $x\in H$. Hence $(zM_{0}+M_{1})$ is strictly accretive
and by \prettyref{prop:pert_accretive}, so is $zM_{0}+M_{1}+A$ and
we get 
\[
\|(zM_{0}+M_{1}+A)^{-1}\|\leq\frac{1}{\min\left\{ \rho c_{0},c_{1}^{-1}\right\} }.
\]
This proves $s_{0}(M,A)\leq0.$ If $C$ is boundedly invertible, we
even obtain $s_{0}(M,A)<0$ by \prettyref{prop:exp_decay_para}.\\
We want to determine $\overline{\mathrm{IV}_{\rho}(M,A)}.$ For doing
so, let $\rho>0$ and $u_{0}=(v,w)\in\mathrm{IV}_{\rho}(M,A)$. Then
it follows that 
\[
S_{\rho}\left(\delta_{0}M_{0}u_{0}\right)-\chi_{\mathbb{R}_{\geq0}}u_{0}\in H_{\rho}^{1}(\mathbb{R};H).
\]
We set 
\[
\left(\begin{array}{c}
\vartheta\\
q
\end{array}\right)\coloneqq S_{\rho}\left(\delta_{0}M_{0}u_{0}\right)=\left(\overline{\partial_{0,\rho}\left(\begin{array}{cc}
\eta & 0\\
0 & 0
\end{array}\right)+\left(\begin{array}{cc}
0 & 0\\
0 & k^{-1}
\end{array}\right)+\left(\begin{array}{cc}
0 & -C^{\ast}\\
C & 0
\end{array}\right)}\right)^{-1}\delta_{0}\left(\begin{array}{c}
\eta v\\
0
\end{array}\right),
\]
which by definition needs to belong to $H_{\rho}(\mathbb{R};H).$
The equations read 
\begin{align*}
\partial_{0,\rho}\eta\vartheta-C^{\ast}q & =\delta_{0}\eta v,\\
k^{-1}q+C\vartheta & =0,
\end{align*}
and thus, the second equation yields $\vartheta\in D(C).$ Furthermore,
since $q(t)\to w$ as $t\to0+,$ we derive that $kC\vartheta(t)=-q(t)\to-w$
as $t\to0+.$ Since also $\vartheta(t)\to v$ as $t\to0+,$ we infer
from the closedness of $C$ that $v\in D(C)$ and $w=-kCv.$ Thus,
\[
\mathrm{IV}_{\rho}(M,A)\subseteq\left\{ (v,-kCv)\,;\, v\in D(C)\right\} .
\]
To find an appropriate subset of $\mathrm{IV}_{\rho}(M,A)$ we apply
\prettyref{prop:IV_diff_alg}. Let $u_{0}=(v,w)\in D(A).$ Then $(M_{1}+A)u_{0}\in R(M_{0})$
yields 
\[
k^{-1}w+Cv=0
\]
and hence, $w=-kCv.$ Since $w\in D(C^{\ast})$ we infer $v\in D(C^{\ast}kC).$
Moreover, since we have that $\iota_{0}\left(\iota_{1}^{\ast}M_{0}\iota_{0}\right)^{-1}\iota_{1}^{\ast}(M_{1}+A)u_{0}\in D(A)$,
we derive 
\[
\eta^{-1}C^{\ast}w\in D(C)
\]
 and thus, $v\in D(C\eta^{-1}C^{\ast}kC).$ Thus, we have that 
\[
\left\{ (v,-kCv)\,;\, v\in D(C\eta^{-1}C^{\ast}kC)\right\} \subseteq\mathrm{IV}_{\rho}(M,A).
\]
Both inclusions yield 
\[
\overline{\mathrm{IV}_{\rho}(M,A)}=\left\{ (v,-kCv)\,;\, v\in D(C)\right\} \eqqcolon G.
\]
If we show that \prettyref{eq:HY_diff_alg} is satisfied, the semigroup
$T_{(M,A)}^{(1)}(\cdot,0)$ extends to a $C_{0}$-semigroup on $G$.
Let $v\in D(C),\lambda>0$ and define 
\[
\left(\begin{array}{c}
u_{n}\\
v_{n}
\end{array}\right)\coloneqq\left(\left(\begin{array}{cc}
\lambda\eta & -C^{\ast}\\
C & k^{-1}
\end{array}\right)^{-1}\left(\begin{array}{cc}
\eta & 0\\
0 & 0
\end{array}\right)\right)^{n}\left(\begin{array}{c}
v\\
-kCv
\end{array}\right)
\]
for $n\in\mathbb{N}.$ Then, 
\begin{align*}
u_{n} & =\left(\left(\lambda\eta+C^{\ast}kC\right)^{-1}\eta\right)^{n}v\\
v_{n} & =-kC\left(\left(\lambda\eta+C^{\ast}kC\right)^{-1}\eta\right)^{n}v.
\end{align*}
First, we estimate $u_{n}$ in terms of $v$. For doing so, we compute
\[
\left(\lambda\eta+C^{\ast}kC\right)^{-1}=\sqrt{\eta^{-1}}\left(\lambda+\sqrt{\eta^{-1}}C^{\ast}kC\sqrt{\eta^{-1}}\right)^{-1}\sqrt{\eta^{-1}}
\]
and consequently 
\begin{align*}
\left(\left(\lambda\eta+C^{\ast}kC\right)^{-1}\eta\right)^{n} & =\left(\sqrt{\eta^{-1}}\left(\lambda+\sqrt{\eta^{-1}}C^{\ast}kC\sqrt{\eta^{-1}}\right)^{-1}\sqrt{\eta}\right)^{n}\\
 & =\sqrt{\eta^{-1}}\left(\lambda+\sqrt{\eta^{-1}}C^{\ast}kC\sqrt{\eta^{-1}}\right)^{-n}\sqrt{\eta}.
\end{align*}
Since $\sqrt{\eta^{-1}}C^{\ast}kC\sqrt{\eta^{-1}}$ is a selfadjoint
accretive operator on $H_{0},$ we get that 
\[
|u_{n}|_{H_{0}}\leq\frac{\|\sqrt{\eta}\|\sqrt{c_{0}^{-1}}}{\lambda^{n}}|v|_{H_{0}}.
\]
Moreover, by the above computations, we have that 
\begin{align*}
-kC\left(\left(\lambda\eta+C^{\ast}kC\right)^{-1}\eta\right)^{n} & =-kC\sqrt{\eta^{-1}}\left(\lambda+\sqrt{\eta^{-1}}C^{\ast}kC\sqrt{\eta^{-1}}\right)^{-n}\sqrt{\eta}\\
 & \supseteq\sqrt{k}(\lambda+\sqrt{k}C\eta^{-1}C^{\ast}\sqrt{k})^{-n}(-\sqrt{k}C)
\end{align*}
and thus, 
\[
v_{n}=\sqrt{k}\left(\lambda+\sqrt{k}C\eta^{-1}C^{\ast}\sqrt{k}\right)^{-n}(-\sqrt{k}Cv).
\]
Since $\sqrt{k}C\eta^{-1}C^{\ast}\sqrt{k}$ is a selfadjoint accretive
operator, we infer 
\[
|v_{n}|_{H_{1}}\leq\frac{\|\sqrt{k}\|}{\lambda^{n}}|\sqrt{k}Cv|_{H_{1}}\leq\frac{\|\sqrt{k}\||\sqrt{c_{1}^{-1}}}{\lambda^{n}}|kCv|_{H_{1}}.
\]
Summarizing, we have shown that 
\[
\left|\left(\left(\begin{array}{cc}
\lambda\eta & -C^{\ast}\\
C & k^{-1}
\end{array}\right)^{-1}\left(\begin{array}{cc}
\eta & 0\\
0 & 0
\end{array}\right)\right)^{n}\left(\begin{array}{c}
v\\
-kCv
\end{array}\right)\right|_{H}\leq\frac{M}{\lambda^{n}}\left|\left(\begin{array}{c}
v\\
-kCv
\end{array}\right)\right|_{H}
\]
for some $M\geq1$, which is \prettyref{eq:HY_diff_alg} according
to \prettyref{rem:HY_classic} (a). Moreover $\omega(T_{(M,A)}^{(1)}(\cdot,0))\leq s_{0}(M,A)$
by \prettyref{prop:GP_amnesic}.\end{example}
\begin{rem}
If we choose $H_{0}=L_{2}(\Omega)$ and $H_{1}=R(\grad_{0})$ for
some bounded domain $\Omega\subseteq\mathbb{R}^{3}$ and $C\coloneqq\iota_{R(\grad_{0})}^{\ast}\grad_{0},$
we end up with the heat equation. Thus, we can associate a $C_{0}-$semigroup
on the Hilbert space $G=\left\{ (v,-k\grad_{0}v)\,;\, v\in D(\grad_{0})\right\} \subseteq L_{2}(\Omega)\oplus R(\grad_{0})$,
which is exponentially stable, since $C$ is boundedly invertible
due to the Poincare inequality. We emphasize that the $C_{0}$-semigroup
is not the classical one, since we require the continuity with respect
to time for both unknowns, the temperature and the heat flux. This
is why we need the Hilbert space $G$ instead of the usual space $L_{2}(\Omega).$ 
\end{rem}

\subsection{Partial differential equations with finite delay }

We generalize the setting of the previous subsection and consider
now material laws of the form 
\[
M(z)\coloneqq M_{0}+z^{-1}(M_{1}+\sum_{i=2}^{n}M_{i}\e^{-zh_{i}})\quad(z\in\mathbb{C}\setminus\{0\}),
\]
where $M_{0},M_{1},M_{i}\in L(H)$ and $h_{i}>0,\: i\in\{2,\ldots,n\}.$
Then, $M$ satisfies \prettyref{eq:M_reg} with $b(M)=0$ and for
$\rho>0,x\in H$ we get 
\[
M(\partial_{0,\rho})\chi_{\mathbb{R}_{\geq0}}x=\left(t\mapsto\chi_{\mathbb{R}_{\geq0}}(t)M_{0}x+t\chi_{\mathbb{R}_{\geq0}}(t)M_{1}x+\sum_{i=2}^{n}(t-h_{i})\chi_{\mathbb{R}_{\geq h_{i}}}(t)M_{i}x\right)
\]
which yields 
\[
\left(M(\partial_{0,\rho})\chi_{\mathbb{R}_{\geq0}}x\right)(0+)=M_{0}x.
\]
Moreover, we have that 
\begin{align*}
K_{\rho}g & =P_{0}\partial_{0,\rho}M(\partial_{0,\rho})g\\
 & =\chi_{\mathbb{R}_{\geq0}}(\m)\sum_{i=2}^{n}M_{i}\tau_{-h_{i}}g
\end{align*}
for each $g\in\chi_{\mathbb{R}_{\leq0}}(\m)\left[H_{\rho}^{1}(\mathbb{R};H)\right]$.
Thus, a function $g\in\chi_{\mathbb{R}_{\leq0}}(\m)[H_{\rho}^{1}(\mathbb{R};H)]$
for some $\rho>\max\{s_{0}(M,A),0\}$ belongs the space of admissible
histories $\mathrm{His}_{\rho}(M,A)$ if and only if 
\[
S_{\rho}\left(\delta_{0}M_{0}g(0-)-\chi_{\mathbb{R}_{\geq0}}(\m)\sum_{i=2}^{n}M_{i}\tau_{-h_{i}}g\right)-\chi_{\mathbb{R}_{\geq0}}g(0-)\in H_{\rho}^{1}(\mathbb{R};H).
\]
Similar to the previous section, we describe a subset of $\mathrm{His}_{\rho}(M,A)$.
\begin{prop}
\label{prop:His_delay} Let $R(M_{0})$ be closed and denote by $\iota_{0}:N(M_{0})^{\bot}\to H$
and $\iota_{1}:R(M_{0})\to H$ the canonical embeddings of $N(M_{0})^{\bot}$
and $R(M_{0}),$ respectively. Then, $\iota_{1}^{\ast}M_{0}\iota_{0}:N(M_{0})^{\bot}\to R(M_{0})$
is boundedly invertible and each function $g\in\chi_{\mathbb{R}_{\leq0}}(\m)\left[H_{\rho}^{1}(\mathbb{R};H)\right],\rho>0$
satisfying

\begin{itemize}

\item $g(0-)\in D(A),$

\item $(M_{1}+A)g(0-)\in R(M_{0})$,

\item $\iota_{0}(\iota_{1}^{\ast}M_{0}\iota_{0})^{-1}\iota_{1}^{\ast}(M_{1}+A)g(0-)\in D(A),$

\item $M_{i}g(-h_{i})=0$ for each $i\in\{2,\ldots,n\}$, 

\end{itemize}

belongs to $\mathrm{His}_{\rho}(M,A).$\end{prop}
\begin{proof}
The bounded invertibility of $\iota_{1}^{\ast}M_{0}\iota_{0}$ follows
by the closed graph theorem. Let now $g\in\chi_{\mathbb{R}_{\leq0}}(\m)\left[H_{\rho}^{1}(\mathbb{R};H)\right],\rho>0$
satisfying the conditions. We need to prove that 
\[
S_{\rho}\left(\delta_{0}M_{0}g(0-)-\chi_{\mathbb{R}_{\geq0}}(\m)\sum_{i=2}^{n}M_{i}\tau_{-h_{i}}g\right)-\chi_{\mathbb{R}_{\geq0}}g(0-)\in H_{\rho}^{1}(\mathbb{R};H).
\]
Since $g(0-)\in D(A)$ we get that 
\begin{align*}
 & S_{\rho}\left(\delta_{0}M_{0}g(0-)-\chi_{\mathbb{R}_{\geq0}}(\m)\sum_{i=2}^{n}M_{i}\tau_{-h_{i}}g\right)-\chi_{\mathbb{R}_{\geq0}}g(0-)\\
= & -S_{\rho}\left(\chi_{\mathbb{R}_{\geq0}}(\m)\sum_{i=2}^{n}M_{i}\tau_{-h_{i}}g+\chi_{\mathbb{R}_{\geq0}}(M_{1}+A)g(0-)+\sum_{i=2}^{n}M_{i}\tau_{-h_{i}}\chi_{\mathbb{R}_{\geq0}}g(0-)\right).
\end{align*}
Note that by assumption there is $x\in D(A)$ such that $M_{0}x=(M_{1}+A)g(0-).$
Moreover, $\spt\tau_{-h_{i}}\chi_{\mathbb{R}_{\geq0}}g(0-)\subseteq\mathbb{R}_{\geq h_{i}}$
and thus, 
\begin{align*}
 & S_{\rho}\left(\delta_{0}M_{0}g(0-)-\chi_{\mathbb{R}_{\geq0}}(\m)\sum_{i=2}^{n}M_{i}\tau_{-h_{i}}g\right)-\chi_{\mathbb{R}_{\geq0}}g(0-)\\
= & -S_{\rho}\left(\chi_{\mathbb{R}_{\geq0}}M_{0}x\right)-S_{\rho}\left(\chi_{\mathbb{R}_{\geq0}}(\m)\sum_{i=2}^{n}M_{i}\tau_{-h_{i}}(g+\chi_{\mathbb{R}_{\geq0}}g(0-))\right).
\end{align*}
Note that $g+\chi_{\mathbb{R}_{\geq0}}g(0-)\in H_{\rho}^{1}(\mathbb{R};H)$
and thus, $M_{i}\tau_{-h_{i}}(g+\chi_{\mathbb{R}_{\geq0}}g(0-))\in H_{\rho}^{1}(\mathbb{R};H).$
Furthermore, by hypothesis we have that 
\[
\left(M_{i}\tau_{-h_{i}}(g+\chi_{\mathbb{R}_{\geq0}}g(0-))\right)(0)=M_{i}g(-h_{i})=0
\]
and hence, 
\[
\chi_{\mathbb{R}_{\geq0}}(\m)\sum_{i=2}^{n}M_{i}\tau_{-h_{i}}(g+\chi_{\mathbb{R}_{\geq0}}g(0-))\in H_{\rho}^{1}(\mathbb{R};H).
\]
Thus, it suffices to check that 
\[
S_{\rho}\left(\chi_{\mathbb{R}_{\geq0}}M_{0}x\right)\in H_{\rho}^{1}(\mathbb{R};H).
\]
This, however, follows by 
\[
\partial_{0,\rho}S_{\rho}(\chi_{\mathbb{R}_{\geq0}}M_{0}x)=\chi_{\mathbb{R}_{\geq0}}x-S_{\rho}\left(\chi_{\mathbb{R}_{\geq0}}(M_{1}+A)x+\sum_{i=2}^{n}M_{i}\tau_{-h_{i}}\chi_{\mathbb{R}_{\geq0}}x\right)\in H_{\rho}(\mathbb{R};H),
\]
where we have used $x\in D(A).$ 
\end{proof}
The main difficulty lies in the verification of \prettyref{eq:Hille_Yosida}
in order to obtain a $C_{0}$-semigroup on $X_{\rho}^{\mu}$ for some
$\rho>s_{0}(M,A)$ and $\mu\leq\rho$. However, if \prettyref{eq:Hille_Yosida}
is satisfied, Conditions \prettyref{eq:GP1} and \prettyref{eq:GP2}
follow and hence, we can estimate the growth-type of the semigroup
by $\mu.$
\begin{prop}
\label{prop:GP_delay} Let $\rho>\max\{s_{0}(M,A),0\}$ and $s_{0}(M,A)<\mu\leq\rho.$
Assume that $T_{(M,A)}$ extends to a $C_{0}$-semigroup on $X_{\rho}^{\mu}.$
Then $\omega(T_{(M,A)})\leq\mu.$ \end{prop}
\begin{proof}
By \prettyref{thm:G-P} it suffices to prove \prettyref{eq:GP1} and
\prettyref{eq:GP2}. We observe that for $z\in\mathbb{C}_{\Re>\mu}$
and $\lambda>0$ we have 
\begin{align*}
\|\Phi_{\lambda}(z)\| & =\|zM_{0}+M_{1}+\sum_{i=2}^{n}M_{i}\e^{-zh_{i}}-\left(\left(z+\lambda\right)M_{0}+M_{1}+\sum_{i=2}^{n}M_{i}\e^{-(z+\lambda)h_{i}}\right)\|\\
 & =\|-\lambda M_{0}+\sum_{i=2}^{n}M_{i}\e^{-zh_{i}}(1-\e^{-\lambda h_{i}})\|\\
 & \leq\lambda\|M_{0}\|+\sum_{i=2}^{n}\|M_{i}\|\e^{-\mu h_{i}}
\end{align*}
which shows \prettyref{eq:GP1}. Moreover, for $g\in\chi_{\mathbb{R}_{\leq0}}(\m)\left[H_{\rho}^{1}(\mathbb{R};H)\right]$
we estimate 
\begin{align*}
\left(\intop_{0}^{\infty}|K_{\rho}g(t)|_{H}^{2}\e^{-2\mu t}\mbox{ d}t\right)^{\frac{1}{2}} & =\left(\intop_{0}^{\infty}|\sum_{i=2}^{n}M_{i}g(t-h_{i})|_{H}^{2}\e^{-2\mu t}\mbox{ d}t\right)^{\frac{1}{2}}\\
 & \leq\sum_{i=2}^{n}\|M_{i}\|\e^{-\mu h_{i}}\left(\:\intop_{-h_{i}}^{0}|g(s)|_{H}^{2}\e^{-2\mu s}\mbox{ d}s\right)^{\frac{1}{2}}\\
 & \leq\sum_{i=2}^{n}\|M_{i}\|^{2}\e^{-\mu h_{i}}|g|_{H_{\rho}(\mathbb{R}_{\leq0};H)},
\end{align*}
which proves \prettyref{eq:GP2}. Hence, the assertion follows. 
\end{proof}
We now discuss a more concrete example.

\subsubsection*{On a Cauchy-type problem with finite delay}

Let $A:D(A)\subseteq H\to H$ such that $-A$ is the generator of
a $C_{0}$-semigroup $T$ on $H$ and $M_{i}\in L(H),\, i\in\{2,\ldots,n\}.$
We consider the following evolutionary problem 
\[
\left(\partial_{0,\rho}+\sum_{i=2}^{n}M_{i}\tau_{-h_{i}}+A\right)u=f,
\]
where $h_{i}>0$ for each $i\in\{2,\ldots,n\}.$ This problem is well-posed,
since for $z\in\mathbb{C}_{\Re>\rho}$ with $\rho>\max\{0,\omega(T)\}$
sufficiently large, we have that 
\[
z+\sum_{i=2}^{n}M_{i}\e^{-zh_{i}}+A=(z+A)\left(1+(z+A)^{-1}\sum_{i=2}^{n}M_{i}\e^{-zh_{i}}\right),
\]
and 
\[
\|(z+A)^{-1}\sum_{i=2}^{n}M_{i}\e^{-zh_{i}}\|\leq\frac{M\sum_{i=2}^{n}\|M_{i}\|\e^{-\Re zh_{i}}}{\Re z-\omega(T)}\leq\frac{M\sum_{i=2}^{n}\|M_{i}\|}{\rho-\omega(T)}\e^{-\rho h},
\]
for some $M\geq1,$ where $h\coloneqq\min_{2\leq i\leq n}h_{i}.$
Thus, choosing $\rho$ large enough such that 
\begin{equation}
\frac{M\sum_{i=2}^{n}\|M_{i}\|}{\rho-\omega(T)}\e^{-\rho h}<1,\label{eq:rho_large_enough_delay}
\end{equation}
we derive the well-posedness by employing the Neumann-series. 
\begin{lem}
For $\rho>\max\{s_{0}(M,A),0\}$ we have that 
\[
X_{\rho}^{\mu}=H\times H_{\mu}(\mathbb{R}_{\leq0};H)
\]
for each $\mu\leq\rho$. \end{lem}
\begin{proof}
We have to prove that $D_{\rho}=\left\{ (g(0-),g)\,;\, g\in\mathrm{His}_{\rho}(M,A)\right\} $
is dense in $H\times H_{\mu}(\mathbb{R}_{\leq0};H)$. Let $v\in H$
and $f\in H_{\mu}(\mathbb{R}_{\leq0};H).$ For $\varepsilon>0$ we
take $x\in D(A^{2})$ such that 
\[
|x-v|_{H}<\varepsilon
\]
and $\varphi\in C_{c}^{\infty}(\mathbb{R}_{<0};H)$ such that 
\[
|\varphi-f|_{H_{\mu}(\mathbb{R}_{\leq0};H)}<\varepsilon.
\]
Moreover, we choose $\psi\in C_{c}^{\infty}(\mathbb{R};H)$ with $|\psi|_{H_{\mu}(\mathbb{R};H)}<\varepsilon$
such that $\psi(0)=x$ and $\psi(-h_{i})=-\varphi(-h_{i})$ for each
$i\in\{2,\ldots,n\}.$ We define 
\[
g\coloneqq\chi_{\mathbb{R}_{\leq0}}(\m)(\varphi+\psi)\in\chi_{\mathbb{R}_{\leq0}}(\m)\left[H_{\rho}^{1}(\mathbb{R};H)\right]
\]
and show that $g\in\mathrm{His}_{\rho}(M,A).$ For doing so, we apply
\prettyref{prop:His_delay} and show that $g$ satisfies the four
conditions stated there. We have that 

\begin{itemize}

\item $g(0-)=\psi(0)=x\in D(A),$

\item $(M_{1}+A)g(0-)=Ax\in H=R(M_{0}),$

\item $\iota_{0}(\iota_{1}^{\ast}M_{0}\iota_{0})^{-1}\iota_{1}^{\ast}(M_{1}+A)g(0-)=Ax\in D(A),$

\item $M_{i}g(-h_{i})=M_{i}(\varphi(-h_{i})+\psi(-h_{i}))=0$ for
$i\in\{2,\ldots,n\}.$

\end{itemize}

Hence, we indeed have $g\in\mathrm{His}_{\rho}(M,A).$ Moreover, 
\begin{align*}
|g(0-)-v|_{H} & =|x-v|_{H}<\varepsilon,\\
|g-f|_{H_{\mu}(\mathbb{R}_{\leq0};H)} & \leq|\varphi-f|_{H_{\mu}(\mathbb{R}_{\leq0};H)}+|\psi|_{H_{\mu}(\mathbb{R};H)}<2\varepsilon,
\end{align*}
which proves the density. \end{proof}
\begin{prop}
Let $\rho>\max\{s_{0}(M,A),0\}$ such that \prettyref{eq:rho_large_enough_delay}
is satisfied. Then for each $\mu\leq\rho$, $T_{(M,A)}$ can be extended
to a $C_{0}$-semigroup on $X_{\rho}^{\mu}$. \end{prop}
\begin{proof}
By \prettyref{prop:T is sg} and \prettyref{lem:boundedness T1 implies T2}
it suffices to prove that there is some $\omega\in\mathbb{R}$ such
that 
\[
T_{(M,A)}^{(1)}:D_{\rho}\subseteq X_{\rho}^{\mu}\to C_{\omega}(\mathbb{R}_{\geq0};H)
\]
is bounded. We observe that for $\rho>\omega(T)$ 
\begin{equation}
(\partial_{0,\rho}+A)^{-1}:H_{\rho}(\mathbb{R}_{\geq0};H)\to C_{\rho}(\mathbb{R}_{\geq0};H)\label{eq:VDK}
\end{equation}
is bounded. Indeed, for $f\in C_{c}^{\infty}(\mathbb{R}_{\geq0};H)$
we estimate 
\begin{align*}
\left|\left(\left(\partial_{0,\rho}+A\right)^{-1}f\right)(t)\right|_{H}\e^{-\rho t} & =\left|\intop_{0}^{t}T(t-s)f(s)\mbox{ d}s\right|_{H}\e^{-\rho t}\\
 & \leq M\e^{(\omega(T)-\rho)t}\intop_{0}^{t}|f(s)|_{H}\e^{-\omega(T)s}\mbox{ d}s\\
 & \leq M\e^{(\omega(T)-\rho)t}\sqrt{\frac{\e^{2(\rho-\omega(T))t}-1}{2(\rho-\omega(T))}}|f|_{H_{\rho}(\mathbb{R}_{\geq0};H)}\\
 & \leq M\frac{1}{\sqrt{2(\rho-\omega(T))}}|f|_{H_{\rho}(\mathbb{R}_{\geq0};H)},
\end{align*}
which shows the claim. Let now $\rho>\omega(T)$ satisfying \prettyref{eq:rho_large_enough_delay}
and $\mu\leq\rho.$ For $(x,g)\in X_{\rho}^{\mu}$ we set 
\begin{align*}
u & \coloneqq\left(\partial_{0,\rho}+\sum_{i=2}^{n}M_{i}\tau_{h_{i}}+A\right)^{-1}\left(\delta_{0}x-\chi_{\mathbb{R}_{\geq0}}(\m)\sum_{i=2}^{n}M_{i}\tau_{h_{i}}g\right)\\
 & =\left(1+(\partial_{0,\rho}+A)^{-1}\sum_{i=2}^{n}M_{i}\tau_{h_{i}}\right)^{-1}(\partial_{0,\rho}+A)^{-1}\delta_{0}x+\\
 & \phantom{aaaa}-\left(\partial_{0,\rho}+\sum_{i=2}^{n}M_{i}\tau_{h_{i}}+A\right)^{-1}\chi_{\mathbb{R}_{\geq0}}(\m)\sum_{i=2}^{n}M_{i}\tau_{h_{i}}g.
\end{align*}
Since $(\partial_{0,\rho}+A)^{-1}\delta_{0}x\in H_{\rho}(\mathbb{R}_{\geq0};H)$
and 
\[
\chi_{\mathbb{R}_{\geq0}}(\m)\tau_{h_{i}}g=\chi_{[0,h_{i}]}(\m)\tau_{h_{i}}g\in H_{\mu}([0,h_{i}];H)\hookrightarrow H_{\rho}(\mathbb{R};H),
\]
we derive that $u\in H_{\rho}(\mathbb{R}_{\geq0};H)$ and 
\[
|u|_{H_{\rho}(\mathbb{R};H)}\leq C|(x,g)|_{X_{\rho}^{\mu}}
\]
for some $C>0.$ We thus have shown that 
\[
T_{(M,A)}^{(1)}:X_{\rho}^{\mu}\to H_{\rho}(\mathbb{R}_{\geq0};H)
\]
is bounded. Moreover, we have that 
\[
\left(\partial_{0,\rho}+\sum_{i=2}^{n}M_{i}\tau_{h_{i}}+A\right)u=\delta_{0}x-\chi_{\mathbb{R}_{\geq0}}(\m)\sum_{i=2}^{n}M_{i}\tau_{h_{i}}g,
\]
which gives 
\[
u=(\partial_{0,\rho}+A)^{-1}\delta_{0}x-(\partial_{0,\rho}+A)^{-1}\left(\chi_{\mathbb{R}_{\geq0}}(\m)\sum_{i=2}^{n}M_{i}\tau_{h_{i}}g-\sum_{i=2}^{n}M_{i}\tau_{h_{i}}u\right).
\]
Since $-A$ generates a $C_{0}$-semigroup, the first term on the
right-hand side belongs to $C_{\rho}(\mathbb{R}_{\geq0};H).$ Moreover,
employing \prettyref{eq:VDK} we infer that also the second term lies
in $C_{\rho}(\mathbb{R}_{\geq0};H)$ and hence, so does $u$. Hence,
we have shown 
\[
T_{(M,A)}^{(1)}\left[X_{\rho}^{\mu}\right]\subseteq C_{\rho}(\mathbb{R}_{\geq0};H).
\]
Since $T_{(M,A)}^{(1)}\in L(X_{\rho}^{\mu};H_{\rho}(\mathbb{R}_{\geq0};H))\subseteq L(X_{\rho}^{\mu};H_{\rho+1}(\mathbb{R}_{\geq0};H))$
and $C_{\rho}(\mathbb{R}_{\geq0};H)\hookrightarrow H_{\rho+1}(\mathbb{R}_{\geq0};H),$
we infer that $T_{(M,A)}^{(1)}:X_{\rho}^{\mu}\to C_{\rho}(\mathbb{R}_{\geq0};H)$
is bounded by the closed graph theorem. \end{proof}
\begin{cor}
Assume that $A-c$ is $m$-accretive for some $c>0.$ Moreover, assume
that $M_{i}$ is selfadjoint and non-negative for each $i\in\{2,\ldots,n\}.$
Then $s_{0}(M,A)\leq-c$ and consequently, the semigroup $T_{(M,A)}$
on $X_{\rho}^{\mu}$ for $s_{0}(M,A)<\mu<0$ is exponentially stable.\end{cor}
\begin{proof}
For $x\in H$ and $z\in\mathbb{C}$ we have 
\[
\Re\langle\left(z+\sum_{i=2}^{n}M_{i}\e^{-zh_{i}}+A\right)x|x\rangle\geq\left(\Re z+c\right)|x|_{H}^{2}
\]
and hence, $s_{0}(M,A)\leq-c.$ The second assertion follows from
\prettyref{prop:GP_delay}.
\end{proof}

\subsection{Integro-differential equations}

A further class we want to inspect with the methods provided in this
chapter is a particular class of integro-differential equations. For
simplicity, we restrict ourselves to integro-differential equations
of parabolic-type.

Let $H$ be a Hilbert space and $A:D(A)\subseteq H\to H$ an $m$-accretive
operator. Moreover, let $k:\mathbb{R}_{\geq0}\to L(H)$ be a given
kernel satisfying $k\in L_{1,\mu}(\mathbb{R}_{\geq0};L(H))$ with
$|k|_{L_{1,\mu}}<1$ for some $\mu\in\mathbb{R}$ (recall the definitions
given in \prettyref{sub:Integro-differential-equations}). We consider
an integro-differential equation of the form 
\begin{equation}
\left(\partial_{0,\rho}\left(1-k\ast\right)^{-1}+A\right)u=f,\label{eq:integro_selfadjoint}
\end{equation}
for $\rho\geq\mu.$ This is an evolutionary equation with a material
law given by 
\[
M(z)\coloneqq(1-\sqrt{2\pi}\hat{k}(z))^{-1}\quad(z\in\mathbb{C}_{\Re>\mu}).
\]

To avoid technicalities we assume that $H$ is separable. We first
check that $M$ satisfies \prettyref{eq:M_reg}. Clearly, we have
$b(M)\leq\max\{0,\mu\}$ due to the Neumann series. In order to prove
that 
\[
\left((1-k\ast)^{-1}\chi_{\mathbb{R}_{\geq0}}x\right)(0+)
\]
exists for each $x\in H,$ we need the following lemma.%

\begin{lem}
\label{lem:conv_cont}For $\rho\in\mathbb{R}$ consider the space
\[
CH_{\rho}(\mathbb{R};H)\coloneqq C_{\rho}(\mathbb{R};H)\cap H_{\rho}(\mathbb{R};H)
\]
equipped with the norm 
\[
|f|_{CH_{\rho}}\coloneqq|f|_{\rho,\infty}+|f|_{H_{\rho}(\mathbb{R};H)}\quad(f\in CH_{\rho}(\mathbb{R};H)).
\]
Then $CH_{\rho}(\mathbb{R};H)$ is a Banach space and for $\ell\in L_{1,\rho}(\mathbb{R}_{\geq0};L(H))$
the operator 
\[
\ell\ast:CH_{\rho}(\mathbb{R};H)\to CH_{\rho}(\mathbb{R};H)
\]
is well-defined and bounded with 
\[
\|\ell\ast\|\leq|\ell|_{L_{1,\rho}}.
\]
\end{lem}
\begin{proof}
If $(f_{n})_{n\in\mathbb{N}}$ is a Cauchy sequence in $CH_{\rho}(\mathbb{R};H)$
it converges to $f$ and $\tilde{f}$ in $C_{\rho}(\mathbb{R};H)$
and $H_{\rho}(\mathbb{R};H)$, respectively. Then there is a subsequence
$\left(f_{n_{k}}\right)_{k\in\mathbb{N}}$ such that $f_{n_{k}}(t)\to\tilde{f}(t)$
for almost every $t\in\mathbb{R}.$ Since $f_{n}(t)\to f(t)$ for
each $t\in\mathbb{R}$ we infer $f(t)=\tilde{f}(t)$ for almost every
$t\in\mathbb{R},$ which proves that $f\in CH_{\rho}(\mathbb{R};H)$
is the limit of $(f_{n})_{n\in\mathbb{N}}$ in $CH_{\rho}(\mathbb{R};H).$
Let now $\ell\in L_{1,\rho}(\mathbb{R}_{\geq0};L(H)).$ We first show
that $\ell\ast f$ is continuous for $f\in CH_{\rho}(\mathbb{R};H).$
Recall that by \prettyref{lem:integral_formula_conv} 
\[
\left(\ell\ast f\right)(t)=\intop_{0}^{\infty}\ell(s)f(t-s)\mbox{ d}s\quad(t\in\mathbb{R}\mbox{ a.e.}).
\]
Let $t,t'\in\mathbb{R}$ with $|t-t'|\leq1.$ Then we have that 
\[
\left|\intop_{0}^{\infty}\ell(s)\left(f(t-s)-f(t'-s)\right)\mbox{ d}s\right|_{H}\leq\intop_{0}^{\infty}\|\ell(s)\|\left|f(t-s)-f(t'-s)\right|_{H}\mbox{ d}s.
\]
By continuity of $f$, we have $f(t'-s)\to f(t-s)$ for each $s\in\mathbb{R}$
as $t'\to t$. Moreover, we have that
\begin{align*}
\|\ell(s)\|\left|f(t-s)-f(t'-s)\right|_{H} & \leq\|\ell(s)\|\e^{-\rho s}\left(\left|f(t-s)\right|_{H}\e^{-\rho(t-s)}+|f(t'-s)|_{H}\e^{-\rho(t-s)}\right)\e^{\rho t}\\
 & \leq\|\ell(s)\|\e^{-\rho s}\left(1+\e^{-\rho(t-t')}\right)|f|_{\rho,\infty}\e^{\rho t}\\
 & \leq\|\ell(s)\|\e^{-\rho s}\left(1+\e^{|\rho|}\right)|f|_{\rho,\infty}\e^{\rho t}\eqqcolon g(s)
\end{align*}
for each $s\in\mathbb{R}$ and since $g\in L_{1}(\mathbb{R}_{\geq0})$
we obtain 
\[
\intop_{0}^{\infty}\|\ell(s)\|\left|f(t-s)-f(t'-s)\right|_{H}\mbox{ d}s\to0\quad(t'\to t)
\]
by dominated convergence. Hence $\left(t\mapsto\intop_{0}^{\infty}\ell(s)f(t-s)\mbox{ d}s\right)$
is a continuous representer of $\ell\ast f$. Moreover, 
\[
\left|\intop_{0}^{\infty}\ell(s)f(t-s)\mbox{ d}s\right|_{H}\e^{-\rho t}\leq\intop_{0}^{\infty}\|\ell(s)\|\e^{-\rho s}|f(t-s)|_{H}\e^{-\rho(t-s)}\mbox{ d}s\leq|\ell|_{L_{1,\rho}}|f|_{\rho,\infty}
\]
for each $t\in\mathbb{R}.$ Together with \prettyref{lem:convolution_bd}
this shows that $\ell\ast:CH_{\rho}(\mathbb{R};H)\to CH_{\rho}(\mathbb{R};H)$
is bounded with $\|\ell\ast\|\leq|\ell|_{L_{1,\rho}}.$ \end{proof}
\begin{prop}
For each $x\in H$ we have that 
\[
\left((1-k\ast)^{-1}\chi_{\mathbb{R}_{\geq0}}x\right)(0+)=x.
\]
\end{prop}
\begin{proof}
Let $\rho>\max\{0,\mu\}$ and $x\in H$. We first prove that $k\ast\chi_{\mathbb{R}_{\geq0}}x\in CH_{\rho}(\mathbb{R};H).$
For $t,t'\in\mathbb{R}$ we have that 
\[
\intop_{0}^{\infty}k(s)\chi_{\mathbb{R}_{\geq0}}(t'-s)x\mbox{ d}s=\intop_{0}^{\infty}\chi_{[0,t']}(s)k(s)x\mbox{ d}s\to\intop_{0}^{\infty}\chi_{[0,t]}(s)k(s)x\mbox{ d}s=\intop_{0}^{\infty}k(s)\chi_{\mathbb{R}_{\geq0}}(t-s)x\mbox{ d}s
\]
as $t'\to t$ by dominated convergence. Hence, by \prettyref{lem:integral_formula_conv}
$\left(t\mapsto\intop_{0}^{\infty}k(s)\chi_{\mathbb{R}_{\geq0}}(t-s)x\mbox{ d}s\right)$
is a continuous representer of $k\ast\chi_{\mathbb{R}_{\geq0}}x.$
Since 
\[
\left|\intop_{0}^{\infty}k(s)\chi_{\mathbb{R}_{\geq0}}(t-s)x\mbox{ d}s\right|_{H}\e^{-\rho t}\leq\intop_{0}^{\infty}\|k(s)\|\e^{-\rho s}\mbox{ d}s|x|_{H}\quad(t\in\mathbb{R}),
\]
we obtain $f\coloneqq k\ast\chi_{\mathbb{R}_{\geq0}}x\in CH_{\rho}(\mathbb{R};H)$.
By the Neumann series we have 
\[
(1-k\ast)^{-1}\chi_{\mathbb{R}_{\geq0}}x=\chi_{\mathbb{R}_{\geq0}}x+\sum_{n=0}^{\infty}\left(k\ast\right)^{n}f.
\]
We note that the series also converges in $CH_{\rho}(\mathbb{R};H),$
since by \prettyref{lem:conv_cont}
\[
|\left(k\ast\right)^{n}f|_{CH_{\rho}}\leq|k|_{L_{1,\rho}}^{n}|f|_{CH_{\rho}}
\]
for each $n\in\mathbb{N}$ and $CH_{\rho}(\mathbb{R};H)$ is complete.
Thus, $\sum_{n=0}^{\infty}\left(k\ast\right)^{n}f\in CH_{\rho}(\mathbb{R};H)$
and since $\spt f\subseteq\mathbb{R}_{\geq0}$ and $k\ast$ is causal,
we have that $\spt\sum_{n=0}^{\infty}\left(k\ast\right)^{n}f\subseteq\mathbb{R}_{\geq0}$
and hence, using the continuity 
\[
\left(\sum_{n=0}^{\infty}\left(k\ast\right)^{n}f\right)(0+)=\left(\sum_{n=0}^{\infty}\left(k\ast\right)^{n}f\right)(0-)=0.
\]
Hence, we get that 
\[
\left((1-k\ast)^{-1}\chi_{\mathbb{R}_{\geq0}}x\right)(0+)=x+\left(\sum_{n=0}^{\infty}\left(k\ast\right)^{n}f\right)(0+)=x.\tag*{\qedhere}
\]

\end{proof}
Our next goal is to find a suitable subset of $\mathrm{His}_{\rho}(M,A).$
For doing so, we need to impose more regularity for the kernel $k$.
More precisely, we assume the following: There exists a kernel $k'\in L_{1,\tilde{\mu}}(\mathbb{R}_{\geq0};L(H))$
for some $\tilde{\mu}\in\mathbb{R}$ such that for each $x,y\in H,t\in\mathbb{R}_{\geq0}$
we have 
\begin{equation}
\langle x|(k(t)-k(0))y\rangle_{H}=\intop_{0}^{t}\langle x|k'(s)y\rangle_{H}\mbox{ d}s.\label{eq:regularity_k}
\end{equation}

\begin{lem}
\label{lem:convolution_regular_kernel}Assume that $k$ satisfies
\prettyref{eq:regularity_k} and let $u\in H_{\rho}(\mathbb{R};H)$
for some $\rho\geq\max\{\mu,\tilde{\mu}\}.$ Then $k\ast u\in H_{\rho}^{1}(\mathbb{R};H)$
and $\partial_{0,\rho}\left(k\ast u\right)=k'\ast u+k(0)u.$ Hence,
\[
k\ast:H_{\rho}(\mathbb{R};H)\to H_{\rho}^{1}(\mathbb{R};H)
\]
is a bounded linear operator with norm less than or equal to $|k'|_{L_{1,\rho}}+\|k(0)\|$
and 
\[
z\hat{k}(z)=\hat{k'}(z)+\frac{1}{\sqrt{2\pi}}k(0)\quad(z\in\mathbb{C}_{\Re\geq\max\{\mu,\tilde{\mu}\}}).
\]
\end{lem}
\begin{proof}
Let $\varphi\in C_{c}^{\infty}(\mathbb{R};H)$ and compute 
\begin{align*}
 & \langle k\ast u|\partial_{0,\rho}^{\ast}\varphi\rangle_{H_{\rho}(\mathbb{R};H)}\\
 & =-\intop_{\mathbb{R}}\left\langle \left.\intop_{-\infty}^{t}k(t-s)u(s)\mbox{ d}s\right|(\varphi\e^{-2\rho\cdot})'(t)\right\rangle _{H}\mbox{ d}t\\
 & =-\intop_{\mathbb{R}}\intop_{s}^{\infty}\langle k(t-s)u(s)|(\varphi\e^{-2\rho\cdot})'(t)\rangle_{H}\mbox{ d}t\mbox{ d}s\\
 & =-\intop_{\mathbb{R}}\intop_{s}^{\infty}\langle\intop_{0}^{t-s}k'(r)u(s)\mbox{ d}r|(\varphi\e^{-2\rho\cdot})'(t)\rangle_{H}\mbox{ d}t\mbox{ d}s-\intop_{\mathbb{R}}\langle k(0)u(s)|\intop_{s}^{\infty}(\varphi\e^{-2\rho\cdot})'(t)\mbox{ d}t\rangle_{H}\mbox{ d}s\\
 & =\intop_{\mathbb{R}}\intop_{s}^{\infty}\langle k'(t-s)u(s)|\varphi(t)\e^{-2\rho t}\rangle_{H}\mbox{ d}t\mbox{ d}s+\intop_{\mathbb{R}}\langle k(0)u(s)|\varphi(s)\rangle_{H}\e^{-2\rho s}\mbox{ d}s\\
 & =\langle k'\ast u+k(0)u|\varphi\rangle_{H_{\rho}(\mathbb{R};H)}.
\end{align*}
This yields $\partial_{0,\rho}(k\ast u)=k'\ast u+k(0)u$. The boundedness
of $k\ast:H_{\rho}(\mathbb{R};H)\to H_{\rho}^{1}(\mathbb{R};H)$ with
the asserted norm bound follows from \prettyref{lem:convolution_bd},
while the last assertion is a consequence of \prettyref{lem:kernel_Fourier}.\end{proof}
\begin{prop}
\label{prop:integro_well-posed} If $k$ satisfies \prettyref{eq:regularity_k},
then the evolutionary problem associated with $A$ and $M(z)=(1-\sqrt{2\pi}\hat{k}(z))^{-1}$
is well-posed. Moreover, for $\rho\geq\max\{\mu,\tilde{\mu}\}$ we
have that 
\begin{align*}
K_{\rho}:H_{\rho}(\mathbb{R}_{\leq0};H) & \to H_{\mu\vee\tilde{\mu}}(\mathbb{R}_{\geq0};H)\\
g & \mapsto P_{0}\partial_{0}M(\partial_{0})g
\end{align*}
is well-defined and bounded and, for $\mu\vee\tilde{\mu}\leq\nu\leq\rho$
and $g\in H_{\rho}(\mathbb{R}_{\leq0};H)\subseteq H_{\nu}(\mathbb{R}_{\leq0};H)$
we have that 
\[
K_{\rho}g=K_{\nu}g.
\]
\end{prop}
\begin{proof}
Since $|k|_{L_{1,\mu}}<1$ we have that 
\[
M(z)=1+\sum_{n=1}^{\infty}\left(\sqrt{2\pi}\:\hat{k}(z)\right)^{n},\quad(z\in\mathbb{C}_{\Re>\mu})
\]
due to the Neumann series. From \prettyref{lem:convolution_regular_kernel},
we therefore infer that 
\[
zM(z)=z+\left(\sqrt{2\pi}\hat{k'}(z)+k(0)\right)\sum_{n=0}^{\infty}\left(\sqrt{2\pi}\hat{k}(z)\right)^{n}\eqqcolon z+\tilde{M}(z)\quad(z\in\mathbb{C}_{\Re>\max\{\mu,\tilde{\mu}\}}).
\]
Noting that 
\[
\left\Vert \tilde{M}(z)\right\Vert \leq\left(|k'|_{L_{1,\tilde{\mu}}}+\|k(0)\|\right)\frac{1}{1-|k|_{L_{1,\mu}}}\eqqcolon\rho_{0}
\]
for each $z\in\mathbb{C}_{\Re>\max\{\mu,\tilde{\mu}\}}$ we obtain
that 
\[
zM(z)+A=z+A+\tilde{M}(z)=(z+A)(1+(z+A)^{-1}\tilde{M}(z))
\]
is boundedly invertible for each $z\in\mathbb{C}_{\Re>\max\{0,\rho_{0},\mu,\tilde{\mu}\}}$
with 
\[
\|(zM(z)+A)^{-1}\|\leq\frac{1}{\Re z-\rho_{0}}.
\]
This shows the well-posedness of the evolutionary problem associated
with $M$ and $A$.\\
Let now $g\in H_{\rho}(\mathbb{R}_{\leq0};H)$ for some $\rho\geq\max\{\mu,\tilde{\mu}\}.$
Then we have by \prettyref{lem:convolution_regular_kernel}
\begin{align*}
\partial_{0,\rho}M(\partial_{0,\rho})g & =\partial_{0,\rho}g+\partial_{0,\rho}\sum_{n=1}^{\infty}\left(k\ast\right)^{n}g\\
 & =\partial_{0,\rho}g+\left(k'\ast+k(0)\right)\sum_{n=0}^{\infty}(k\ast)^{n}g
\end{align*}
and thus, 
\begin{align*}
K_{\rho}g & =\chi_{\mathbb{R}_{\geq0}}(\m)\left(k'\ast+k(0)\right)\sum_{n=0}^{\infty}(k\ast)^{n}g\\
 & =\chi_{\mathbb{R}_{\geq0}}(\m)\left(k'\ast\right)\sum_{n=0}^{\infty}(k\ast)^{n}g+k(0)\chi_{\mathbb{R}_{\geq0}}(\m)\sum_{n=1}^{\infty}(k\ast)^{n}g.
\end{align*}
By \prettyref{lem:convolution_bd} we know that $k\ast,k'\ast:H_{\mu\vee\tilde{\mu}}(\mathbb{R};H)\to H_{\mu\vee\tilde{\mu}}(\mathbb{R};H)$
are bounded operators. Since by the choice of $\rho$ we have that
$H_{\rho}(\mathbb{R}_{\leq0};H)\hookrightarrow H_{\mu\vee\mu'}(\mathbb{R};H)$,
it follows that $K_{\rho}g\in H_{\mu\vee\mu'}(\mathbb{R}_{\geq0};H)$
and that $K_{\rho}:H_{\rho}(\mathbb{R}_{\leq0};H)\to H_{\mu\vee\tilde{\mu}}(\mathbb{R}_{\geq0};H)$
is bounded. Moreover, the computation above yields the independence
of $\rho\geq\mu\vee\tilde{\mu}$. 
\end{proof}
In order to give a suitable subset of $\mathrm{His}_{\rho}(M,A),$
we also need to restrict the class of possible operators $A$.
\begin{lem}
\label{lem:solution_op_regular} Let $\rho>\max\{s_{0}(M,A),0,\mu,\tilde{\mu}\}$
and assume that $(\partial_{0,\rho}+A)^{-1}\left[H_{\rho}(\mathbb{R};H)\right]\subseteq H_{\rho}^{1}(\mathbb{R};H)$
and that $k$ satisfies \prettyref{eq:regularity_k}. Then, $S_{\rho}:H_{\rho}(\mathbb{R};H)\to H_{\rho}^{1}(\mathbb{R};H)$
is well-defined and bounded.\end{lem}
\begin{proof}
By the closed graph theorem, it suffices to check that $S_{\rho}\left[H_{\rho}(\mathbb{R};H)\right]\subseteq H_{\rho}^{1}(\mathbb{R};H).$
So, let $f\in H_{\rho}(\mathbb{R};H)$ and set $u\coloneqq S_{\rho}f.$
Then, using \prettyref{lem:convolution_regular_kernel} we infer that
\begin{align*}
f & =\partial_{0,\rho}(1-k\ast)^{-1}u+Au\\
 & =\partial_{0,\rho}u+\partial_{0,\rho}\sum_{n=1}^{\infty}\left(k\ast\right)^{n}u+Au\\
 & =\partial_{0,\rho}u+\left((k'\ast)+k(0)\right)\sum_{n=0}^{\infty}\left(k\ast\right)^{n}u+Au
\end{align*}
and hence, 
\[
u=(\partial_{0,\rho}+A)^{-1}(f-\left((k'\ast)+k(0)\right)\sum_{n=0}^{\infty}\left(k\ast\right)^{n}u).
\]
Since $u\in H_{\rho}(\mathbb{R};H),$ we derive that $f-\left((k'\ast)+k(0)\right)\sum_{n=0}^{\infty}\left(k\ast\right)^{n}u\in H_{\rho}(\mathbb{R};H)$
and thus, $u\in H_{\rho}^{1}(\mathbb{R};H)$ by assumption. \end{proof}
\begin{rem}
The additional assumption imposed on $A$ is called \emph{maximal
regularity} of the associated semigroup. It was proved in \cite{DeSimon1964}
that this is equivalent to the fact that $-A$ generates an analytic
semigroup. Note that this characterization is false for Banach spaces
\cite{Kalton2000}. For results in Banach spaces (more precisely UMD-spaces)
we refer to \cite{Dore1987,Weis2001}. For an approach to maximal
regularity for a class of evolutionary problems we refer to \cite{Picard2016_maxreg}.\end{rem}
\begin{cor}
\label{cor:His_integro} Let $\rho>\max\{s_{0}(M,A),0,\mu,\tilde{\mu}\}$
and assume that $(\partial_{0,\rho}+A)^{-1}\left[H_{\rho}(\mathbb{R};H)\right]\subseteq H_{\rho}^{1}(\mathbb{R};H)$
and that $k$ satisfies \prettyref{eq:regularity_k}. Then each $g\in\chi_{\mathbb{R}_{\leq0}}(\m)\left[H_{\rho}^{1}(\mathbb{R};H)\right]$
with $g(0-)\in D(A)$ belongs to $\mathrm{His}_{\rho}(M,A).$ Hence,
\[
X_{\rho}^{\nu}=H\times H_{\nu}(\mathbb{R}_{\leq0};H)
\]
for each $\nu\leq\rho.$ \end{cor}
\begin{proof}
Let $g\in\chi_{\mathbb{R}_{\leq0}}(\m)\left[H_{\rho}^{1}(\mathbb{R};H)\right]$
with $g(0-)\in D(A).$ %
{} We compute
\begin{align*}
 & S_{\rho}\left(\delta_{0}g(0-)-K_{\rho}g\right)-\chi_{\mathbb{R}_{\geq0}}g(0-)\\
 & =S_{\rho}\left(\delta_{0}g(0-)-K_{\rho}g-\partial_{0,\rho}(1-k\ast)^{-1}\chi_{\mathbb{R}_{\geq0}}g(0-)-\chi_{\mathbb{R}_{\geq0}}Ag(0-)\right)\\
 & =-S_{\rho}\left(\left(k'\ast+k(0)\right)\sum_{j=0}^{\infty}\left(k\ast\right)^{j}\chi_{\mathbb{R}_{\geq0}}g(0-)+K_{\rho}g+\chi_{\mathbb{R}_{\geq0}}Ag(0-)\right).
\end{align*}
The assertion now follows from \prettyref{lem:solution_op_regular}.
\end{proof}
Now we are in the position to prove, that we can associate a $C_{0}$-semigroup
on $X_{\rho}^{\nu}$ for suitable $\rho$ and $\nu.$
\begin{prop}
\label{prop:semigroup_integro}Let $\rho>\max\{s_{0}(M,A),0,\mu,\tilde{\mu}\}$.
Assume that $(\partial_{0,\rho}+A)^{-1}\left[H_{\rho}(\mathbb{R};H)\right]\subseteq H_{\rho}^{1}(\mathbb{R};H)$
and that $k$ satisfies \prettyref{eq:regularity_k}. Then $T_{(M,A)}$
extends to a $C_{0}$-semigroup on $X_{\rho}^{\nu}=H\times H_{\nu}(\mathbb{R}_{\leq0};H)$
for each $\max\{\mu,\tilde{\mu}\}\leq\nu\leq\rho.$ \end{prop}
\begin{proof}
Let $\max\{\mu,\tilde{\mu}\}\leq\nu\leq\rho$. By \prettyref{prop:T is sg}
and \prettyref{lem:boundedness T1 implies T2} it suffices to prove
that 
\[
T_{(M,A)}^{(1)}:D_{\rho}\subseteq X_{\rho}^{\nu}\to C_{\rho}(\mathbb{R}_{\geq0};H)
\]
is bounded. So let $g\in\mathrm{His}_{\rho}(M,A)$ and set $u\coloneqq T_{(M,A)}^{(1)}(g(0-),g)$,
i.e. 
\[
u=S_{\rho}(\delta_{0}g(0-)-K_{\rho}g)=S_{\rho}(\delta_{0}g(0-)-K_{\nu}g).
\]
We note that according to \prettyref{prop:integro_well-posed} $K_{\nu}:H_{\nu}(\mathbb{R}_{\leq0};H)\to H_{\mu\vee\tilde{\mu}}(\mathbb{R}_{\geq0};H)\hookrightarrow H_{\rho}(\mathbb{R};H)$
is bounded. By \prettyref{lem:solution_op_regular} we have that $S_{\rho}:H_{\rho}(\mathbb{R};H)\to H_{\rho}^{1}(\mathbb{R};H)$
is bounded, which in turn gives that $S_{\rho}:H_{\rho}^{-1}(\mathbb{R};H)\to H_{\rho}(\mathbb{R};H)$
is bounded. Thus, $u\in H_{\rho}(\mathbb{R};H)$ with 
\begin{equation}
|u|_{H_{\rho}(\mathbb{R};H)}\leq C|(g(0-),g)|_{X_{\rho}^{\nu}}\label{eq:norm_bound_u}
\end{equation}
 for some $C\geq0.$ Moreover, we have that 
\begin{align*}
\left(\partial_{0,\rho}+A\right)u+(k'\ast+k(0))\sum_{n=0}^{\infty}(k\ast)^{n}u & =(\partial_{0,\rho}(1-k\ast)^{-1}+A)u\\
 & =\delta_{0}g(0-)-K_{\nu}g
\end{align*}
and hence 
\begin{align*}
u & =(\partial_{0,\rho}+A)^{-1}\left(\delta_{0}g(0-)-K_{\nu}g-(k'\ast+k(0))\sum_{n=0}^{\infty}(k\ast)^{n}u\right)\\
 & =T_{-A}g(0-)-(\partial_{0,\rho}+A)^{-1}\left(K_{\nu}g+(k'\ast+k(0))\sum_{n=0}^{\infty}(k\ast)^{n}u\right),
\end{align*}
where $T_{-A}$ denotes the contractive $C_{0}$-semigroup generated
by $-A$ (recall that $A$ is $m$-accretive). Since $K_{\nu}g,u\in H_{\rho}(\mathbb{R};H)$
we have that 
\[
(\partial_{0,\rho}+A)^{-1}\left(K_{\nu}g+(k'\ast+k(0))\sum_{n=0}^{\infty}(k\ast)^{n}u\right)\in H_{\rho}^{1}(\mathbb{R};H)\hookrightarrow C_{\rho}(\mathbb{R};H),
\]
according to \prettyref{prop:Sobolev} and thus, $u\in C_{\rho}(\mathbb{R}_{\geq0};H)$
with 
\begin{align*}
 & |u|_{C_{\rho}(\mathbb{R}_{\geq0};H)}\\
 & \leq|g(0-)|_{H}+M\left|K_{\nu}g+(k'\ast+k(0))\sum_{n=0}^{\infty}(k\ast)^{n}u\right|_{H_{\rho}(\mathbb{R};H)}\\
 & \leq|g(0-)|_{H}+M\left(\|K_{\nu}\||g|_{H_{\nu}(\mathbb{R}_{\leq0};H)}+\left(|k'|_{L_{1,\tilde{\mu}}}+\|k(0)\|\right)\frac{1}{1-|k|_{L_{1,\mu}}}|u|_{H_{\rho}(\mathbb{R};H)}\right)\\
 & \leq|g(0-)|_{H}+M\left(\|K_{\nu}\||g|_{H_{\nu}(\mathbb{R}_{\leq0};H)}+\left(|k'|_{L_{1,\tilde{\mu}}}+\|k(0)\|\right)\frac{1}{1-|k|_{L_{1,\mu}}}C|(g(0-),g)|_{X_{\rho}^{\nu}}\right),
\end{align*}
where we have used \prettyref{eq:norm_bound_u} and where $M$ denotes
the operator norm of $(\partial_{0,\rho}+A)^{-1}:H_{\rho}(\mathbb{R};H)\to H_{\rho}^{1}(\mathbb{R};H).$ 
\end{proof}
We conclude this section by an estimate for the growth bound of $T_{(M,A)}.$ 
\begin{prop}
\label{prop:G-P_integro}Let $\rho>\max\{s_{0}(M,A),0,\mu,\tilde{\mu}\}$.
Assume that $(\partial_{0,\rho}+A)^{-1}\left[H_{\rho}(\mathbb{R};H)\right]\subseteq H_{\rho}^{1}(\mathbb{R};H)$
and that $k$ satisfies \prettyref{eq:regularity_k}. Then $\omega(T_{(M,A)})\leq\max\{s_{0}(M,A),\mu,\tilde{\mu}\}.$ \end{prop}
\begin{proof}
We need to show the assumptions of \prettyref{thm:G-P}, that is we
need to prove \prettyref{eq:GP1} and \prettyref{eq:GP2}. We note
that \prettyref{eq:GP2} was already shown in \prettyref{prop:integro_well-posed}
and thus, it suffices to check \prettyref{eq:GP1}. For doing so,
let $\lambda>0$, $\rho>\max\{s_{0}(M,A),\mu,\tilde{\mu}\}$ and consider
the mapping 
\[
\Phi:\mathbb{C}_{\Re>\rho}\ni z\mapsto zM(z)-(z+\lambda)M(z+\lambda)\in L(H).
\]
We compute 
\begin{align*}
 & \Phi(z)\\
 & =z(1-\sqrt{2\pi}\hat{k}(z))^{-1}-(z+\lambda)(1-\sqrt{2\pi}\hat{k}(z+\lambda))^{-1}\\
 & =\sqrt{2\pi}z(1-\sqrt{2\pi}\hat{k}(z))^{-1}(\hat{k}(z+\lambda)-\hat{k}(z))\left(1-\sqrt{2\pi}\hat{k}(z+\lambda)\right)^{-1}-\lambda(1-\sqrt{2\pi}\hat{k}(z+\lambda))^{-1}.
\end{align*}
Since 
\[
\|(1-\sqrt{2\pi}\hat{k}(z))^{-1}\|\leq\frac{1}{1-|k|_{L_{1,\mu}}}\quad(z\in\mathbb{C}_{\Re>\rho}),
\]
it suffices to prove that $z\mapsto z(\hat{k}(z+\lambda)-\hat{k}(z))$
is bounded on $\mathbb{C}_{\Re>\rho}$. Using \prettyref{lem:convolution_regular_kernel}
we have that 
\begin{align*}
z(\hat{k}(z+\lambda)-\hat{k}(z)) & =(z+\lambda)\hat{k}(z+\lambda)-z\hat{k}(z)-\lambda\hat{k}(z+\lambda)\\
 & =\hat{k'}(z+\lambda)-\hat{k'}(z)-\lambda\hat{k}(z+\lambda),
\end{align*}
which yields that 
\[
\|z(\hat{k}(z+\lambda)-\hat{k}(z))\|\leq2|k'|_{L_{1,\tilde{\mu}}}+\lambda|k|_{L_{1,\mu}}\quad(z\in\mathbb{C}_{\Re>\rho}).
\]
This completes the proof. \end{proof}
\begin{rem}
The latter proposition yields the exponential stability of $T_{(M,A)}$,
provided that $s_{0}(M,A),\mu,\tilde{\mu}<0.$ For sufficient assumptions
on the kernel $k$ and the operator $A$, which yield $s_{0}(M,A)<0$
we refer to \prettyref{sub:Integro-differential-equations-1}.
\end{rem}

\section{Notes}

We have provided a way to incorporate initial values and histories
in the framework of evolutionary problems. The key observation was
that suitable right-hand sides belonging to the extrapolation space
$H_{\rho}^{-1}$ together with the causality of the solution operator,
yield that the solution indeed satisfies the desired initial conditions.
This idea was already used in \cite{Picard_McGhee}, however just
for the case of pure initial values. A similar way to incorporate
histories in the case of delay equations was suggested in \cite{Kalauch2011}.
The main improvement of the strategy presented here, is that we are
able to define a space of ``admissible'' initial values and histories
for a given evolutionary problem. Here, admissible means that these
spaces of initial values and histories allow for the definition of
a $C_{0}$-semigroup. To the best of the author's knowledge, so far
there was no intrinsic way to define such spaces in that generality.
Of course, for each of the classes of differential equations discussed
in \prettyref{sec:Applications}, suitable choices for those initial
values were known and the theory of $C_{0}$-semigroups was already
established for those.\\
Indeed, for the case of differential-algebraic equations, the notion
of \emph{consistent} initial values is well-established (see e.g.
\cite{Mehrmann2006} in finite dimensions or \cite{Reis2007} for
a class of DAEs in infinite dimensions). As we have shown in \prettyref{exa:diff_alg},
these consistent initial values coincide with our space $\mathrm{IV}_{\rho}(M,A).$
It should be noted that in the theory of differential-algebraic equations
it is common to assume that the corresponding evolutionary problem
is \emph{not} well-posed in the sense that the solution operator $S_{\rho}$
is bounded from $H_{\rho}$ to $H_{\rho},$ but just bounded from
$H_{\rho}$ to $H_{\rho}^{-k}$ for some $k\in\mathbb{N}.$ Thus,
those problems are not covered by our abstract results so far and
it is postponed to future studies to investigate, how our results
could be generalized to evolutionary problems whose solution operators
are just bounded from $H_{\rho}$ to $H_{\rho}^{-k}.$ Moreover, we
have restricted ourselves to homogeneous problems, i.e. evolutionary
problems of the form 
\[
\left(\partial_{0,\rho}M(\partial_{0,\rho})+A\right)u=0\mbox{ on }\mathbb{R}_{>0}.
\]
It is well-known that in general a non-vanishing source term on the
right-hand side has an influence on the possible choices of initial
values. Again, the study of those problems is postponed to future
work.\\
In the theory of delay equations there are several choices for ``admissible''
histories. We just mention \cite{Webb1976,Batkai_2005}, where the
histories belong to some $L_{p}$-space, or \cite{Hale1971,hale1993introduction}
for histories in the space of continuous functions. Especially in
the theory of nonlinear delay differential equations and equations
with state dependent delay, a restriction of admissible histories
is needed to obtain a solution. For this topic we refer to \cite{Ruess2009}
for nonlinear problems and to \cite{Walther2003,Walther2009} for
state dependent delay equations.\\
Finally, there exists a large amount of articles considering $C_{0}$-semigroups
associated with integro-differential equations. We just mention \cite{Gripenberg1990_Volterra}
for finite dimensions and \cite{Batkai_2005,Kunisch1983,Grimmer1982}
for infinite dimensions. Another approach to integro-differential
equations is provided in \cite{pruss1993evolutionary}, where not
a $C_{0}$-semigroup is associated to the problem but a so-called
resolvent family, which allows for a slightly weaker notion of solutions.
This approach was successfully applied to several problems.\\
A further class of differential equations, which was not addressed
in \prettyref{sec:Applications} is the class of fractional differential
equations. Although, such equations are covered by the framework of
evolutionary problems (see \cite{Picard2013_fractional}) they do
not seem to fit in the framework developed in this chapter. The main
reason is that the space $\mathrm{His}_{\rho}(M,A)$ does not seem
to be the right choice for those problems. Indeed, since the Sobolev
embedding theorem also holds for $H_{\rho}^{\alpha}$ with $\alpha>\frac{1}{2}$
one could weaken the definition of $\mathrm{His}_{\rho}(M,A)$ in
the sense that the solution should belong to some $H_{\rho}^{\alpha}$
for $\alpha>\frac{1}{2}.$ So it would be a valuable project to inspect,
how the framework could be generalized in order to cover fractional
differential equations. \\
We conclude this section by an open problem. In \prettyref{thm:G-P}
we proved that $\omega(T_{(M,A)})\leq s_{0}(M,A)$ for suitable material
laws $M.$ As it was already pointed out, this result reduces to the
famous Gearhart-Prüß Theorem if $M=1.$ However, in this case it is
obvious that even equality holds, since $s_{0}(1,A)\leq\omega(T_{(1,A)})$
holds trivially. So the question arises, whether there are examples
of material laws $M$, where the estimate $\omega(T_{(M,A)})<s_{0}(M,A)$
holds. There are good reasons to believe that there is such an example,
even in the simple case $M(z)=M_{0}+z^{-1}M_{1}$ with a non-invertible
$M_{0},$ since in this case $\mathrm{IV}_{\rho}(M,A)$ is not dense
in $H$ in general. Thus, the corresponding semigroup acts on a proper
subspace of $H$, while the abscissa of boundedness $s_{0}(M,A)$
is defined using the whole space $H$, and hence, the growth bound
could be strictly less than the abscissa of boundedness. However,
so far the author was not able to construct an example. 

\newpage{}

$\,$

\thispagestyle{empty}

\appendix
\counterwithin{thm}{chapter}
\chapter{The Fourier-Laplace transform}

In this section we briefly recall some well-known facts about the
Fourier- and the Fourier-Laplace transform. We present the statements
just for scalar-valued functions, although we will use them in the
Hilbert space-valued case. However, employing the tensor product structure
of $L_{2}(\mathbb{R};H,\mu)=L_{2}(\mathbb{R},\mu)\otimes H$ for a
Hilbert space $H$ and a Borel-measure $\mu$, we immediately get
that the results carry over to the Hilbert space case. For the theory
of tensor products of Hilbert spaces and operators we refer to \cite{Weidmann,berezanskii1986selfadjoint},
where the case of selfadjoint operators is considered, and to \cite{Picard_McGhee,Trostorff_2011}
for the general case.\\
We start our considerations by studying the Fourier transform.
\begin{defn*}
Let $f\in L_{1}(\mathbb{R}).$ Then we define the \emph{Fourier transform}
$\mathcal{F}f$ of $f$ by 
\[
\left(\mathcal{F}f\right)(x)\coloneqq\frac{1}{\sqrt{2\pi}}\intop_{\mathbb{R}}\e^{-\i xt}f(t)\mbox{ d}t\quad(x\in\mathbb{R}).
\]
\end{defn*}
\begin{rem}
\label{rem:Riemann-Lebesgue}Obviously, $\mathcal{F}f$ is continuous
(by dominated convergence) and bounded by $\frac{1}{\sqrt{2\pi}}|f|_{L_{1}},$
and hence, the Fourier transform $\mathcal{F}$ is a bounded linear
operator from $L_{1}(\mathbb{R})$ to $C_{b}(\mathbb{R}).$ Choosing
$f\in C_{c}^{\infty}(\mathbb{R})$ one obtains $\mathcal{F}f\in C_{0}(\mathbb{R})$.
Indeed, one has 
\[
\mathcal{F}f(x)=\frac{1}{\sqrt{2\pi}}\intop_{\mathbb{R}}\e^{-\i xt}f(t)\mbox{ d}t=\frac{1}{\i x}\frac{1}{\sqrt{2\pi}}\intop_{\mathbb{R}}\e^{-\i xt}f'(t)\mbox{ d}t\quad(x\ne0)
\]
and hence, 
\[
\limsup_{|x|\to\infty}|\mathcal{F}f(x)|\leq\limsup_{|x|\to\infty}\frac{1}{|x|}\frac{1}{\sqrt{2\pi}}|f'|_{L_{1}}=0.
\]
By continuous extension, we thus get $\mathcal{F}f\in C_{0}(\mathbb{R})$
for each $f\in L_{1}(\mathbb{R}).$ This statement is known as the
Lemma of Riemann-Lebesgue.
\end{rem}
Next, we will consider the Fourier transform on a particular subspace
of $L_{1}(\mathbb{R}),$ namely the Schwartz space $\mathcal{S}(\mathbb{R})$
of rapidly decreasing functions. 
\begin{prop}
\label{prop:Fourier_Schwartz}The Fourier transform is a bijection
on $\mathcal{S}(\mathbb{R}),$ i.e. $\mathcal{F}|_{\mathcal{S}(\mathbb{R})}:\mathcal{S}(\mathbb{R})\to\mathcal{S}(\mathbb{R})$
is bijective. Moreover, for $f\in L_{1}(\mathbb{R})$ with $\mathcal{F}f\in L_{1}(\mathbb{R})$
we have 
\begin{equation}
f(x)=\left(\mathcal{F}^{\ast}\mathcal{F}f\right)(x)\coloneqq\frac{1}{\sqrt{2\pi}}\intop_{\mathbb{R}}\e^{\i xt}\mathcal{F}f(t)\,\mathrm{d}t=\mathcal{F}f(-x)\quad(x\in\mathbb{R}\mbox{ a.e.}),\label{eq:inverse Fourier}
\end{equation}
and we have for all $f,g\in L_{1}(\mathbb{R}):$
\begin{equation}
\intop_{\mathbb{R}}\left(\mathcal{F}f\right)(x)^{\ast}g(x)\,\mathrm{d}x=\intop_{\mathbb{R}}f(x)^{\ast}\left(\mathcal{F}^{\ast}g\right)(x)\,\mathrm{d}x.\label{eq:adjoint_Fourier}
\end{equation}
\end{prop}
\begin{proof}
Let $\phi\in\mathcal{S}(\mathbb{R}).$ We first show that $\mathcal{F}\phi\in\mathcal{S}(\mathbb{R}).$
For doing so, let $k,n\in\mathbb{N}$ and compute 
\begin{align*}
|x|^{n}\left|\left(\partial^{k}\mathcal{F}\phi\right)(x)\right| & =\frac{1}{\sqrt{2\pi}}\left|\intop_{\mathbb{R}}(-\i x)^{n}(-\i t)^{k}\e^{-\i xt}\phi(t)\mbox{ d}t\right|\\
 & =\frac{1}{\sqrt{2\pi}}\left|\intop_{\mathbb{R}}(-\i t)^{k}\e^{-\i xt}\phi^{(n)}(t)\mbox{ d}t\right|\\
 & \leq\frac{1}{\sqrt{2\pi}}\left|(t\mapsto t^{k}\phi^{(n)}(t))\right|_{L_{1}}
\end{align*}
for each $x\in\mathbb{R},$ where we have used differentiation under
the integral and integration by parts. The latter estimate shows $\mathcal{F}\phi\in\mathcal{S}(\mathbb{R})$.
In order to show \prettyref{eq:inverse Fourier}, we first prove \prettyref{eq:adjoint_Fourier}.
Let $f,g\in L_{1}(\mathbb{R}).$ Then we have by Fubini's theorem
\begin{align*}
\intop_{\mathbb{R}}\left(\mathcal{F}f\right)(x)^{\ast}g(x)\mbox{ d}x & =\intop_{\mathbb{R}}\left(\frac{1}{\sqrt{2\pi}}\intop_{\mathbb{R}}\e^{-\i xt}f(t)\mbox{ d}t\right)^{\ast}\; g(x)\mbox{ d}x\\
 & =\frac{1}{\sqrt{2\pi}}\intop_{\mathbb{R}}\intop_{\mathbb{R}}\e^{\i xt}f(t)^{\ast}g(x)\mbox{ d}t\mbox{ d}x\\
 & =\intop_{\mathbb{R}}f(t)^{\ast}\left(\mathcal{F}^{\ast}g\right)(t)\mbox{ d}t.
\end{align*}
Consider now $\gamma\in\mathcal{S}(\mathbb{R})$ given by 
\[
\gamma(x)\coloneqq\e^{-\frac{x^{2}}{2}}\quad(x\in\mathbb{R}).
\]
Then, $\gamma(0)=1$ and $\gamma'(x)+x\gamma(x)=0$ for each $x\in\mathbb{R}.$
Moreover, $\mathcal{F}\gamma(0)=\frac{1}{\sqrt{2\pi}}\intop_{\mathbb{R}}\e^{-\frac{x^{2}}{2}}\mbox{ d}x=1$
and 
\begin{align*}
\left(\mathcal{F}\gamma\right)'(x) & =\frac{1}{\sqrt{2\pi}}\intop_{\mathbb{R}}(-\i t)\e^{-\i xt}\gamma(t)\mbox{ d}t\\
 & =\frac{1}{\sqrt{2\pi}}\intop_{\mathbb{R}}\e^{-\i xt}\i\gamma'(t)\mbox{ d}t\\
 & =-\frac{1}{\sqrt{2\pi}}x\intop_{\mathbb{R}}\e^{-\i xt}\gamma(t)\mbox{ d}t=-x\mathcal{F}\gamma(x)
\end{align*}
for each $x\in\mathbb{R}.$ Thus, $\mathcal{F}\gamma$ and $\gamma$
are both solutions of the initial value problem $y'(x)+xy(x)=0$ and
$y(0)=1,$ and so, they coincide. Let now $x\in\mathbb{R},a>0$ and
$f\in L_{1}(\mathbb{R})$ with $\mathcal{F}f\in L_{1}(\mathbb{R}).$
Since $\mathcal{F}^{\ast}f=\mathcal{F}\sigma_{-1}f,$ where $\left(\sigma_{-1}f\right)(x)\coloneqq f(-x),$
we compute, using \prettyref{eq:adjoint_Fourier}, 
\begin{align}
\frac{1}{\sqrt{2\pi}}\intop_{\mathbb{R}}\e^{\i xt}\gamma(at)\left(\mathcal{F}f\right)(t)\mbox{ d}t & =\frac{1}{\sqrt{2\pi}}\intop_{\mathbb{R}}\left(\e^{-\i xt}\gamma(at)\right)^{\ast}\left(\mathcal{F}^{\ast}\sigma_{-1}f\right)(t)\mbox{ d}t\nonumber \\
 & =\frac{1}{\sqrt{2\pi}}\intop_{\mathbb{R}}\left(\mathcal{F}\e^{-\i x\cdot}\gamma(a\cdot)\right)(t)^{\ast}f(-t)\mbox{ d}t\nonumber \\
 & =\frac{1}{2\pi}\intop_{\mathbb{R}}\left(\intop_{\mathbb{R}}\e^{-\i ty}\e^{-\i xy}\gamma(ay)\mbox{ d}y\right)^{\ast}f(-t)\mbox{ d}t\nonumber \\
 & =\frac{1}{\sqrt{2\pi}}\intop_{\mathbb{R}}\frac{1}{\sqrt{2\pi}}\intop_{\mathbb{R}}\e^{-\i\left(t-x\right)y}\gamma(ay)\mbox{ d}y\; f(t)\mbox{ d}t\nonumber \\
 & =\frac{1}{\sqrt{2\pi}}\intop_{\mathbb{R}}f(t)\frac{1}{a}\mathcal{F}\gamma\left(\frac{t-x}{a}\right)\mbox{ d}t\nonumber \\
 & =\frac{1}{\sqrt{2\pi}}\intop_{\mathbb{R}}f(az+x)\gamma\left(z\right)\mbox{ d}z.\label{eq:Fourier_Gauss}
\end{align}
By dominated convergence we obtain 
\begin{equation}
\lim_{a\to0}\frac{1}{\sqrt{2\pi}}\intop_{\mathbb{R}}\e^{\i xt}\gamma(at)\left(\mathcal{F}f\right)(t)\mbox{ d}t=\frac{1}{\sqrt{2\pi}}\intop_{\mathbb{R}}\left(\mathcal{F}f\right)(t)\e^{\i xt}\mbox{ d}t.\label{eq:Fourier_Gauss_lhs}
\end{equation}
For the term on the right hand side of \prettyref{eq:Fourier_Gauss},
we consider the linear operators 
\begin{align*}
S_{a}:L_{1}(\mathbb{R}) & \to L_{1}(\mathbb{R})\\
f & \mapsto\left(x\mapsto\intop_{\mathbb{R}}f(az+x)\gamma\left(z\right)\mbox{ d}z\right).
\end{align*}
Then 
\[
|S_{a}f|_{L_{1}}=\intop_{\mathbb{R}}\left|\intop_{\mathbb{R}}f(az+x)\gamma(z)\mbox{ d}z\right|\mbox{ d}x\leq\intop_{\mathbb{R}}\intop_{\mathbb{R}}|f(az+x)|\mbox{ d}x\;\gamma(z)\mbox{ d}z\leq|f|_{L_{1}}|\gamma|_{L_{1}},
\]
and so $\|S_{a}\|\leq|\gamma|_{L_{1}}$ for each $a>0.$ Moreover,
since $S_{a}\psi\to\psi|\gamma|_{L_{1}}$ in $L_{1}(\mathbb{R})$
as $a\to0$ for each $\psi\in C_{c}^{\infty}(\mathbb{R}),$ we deduce
$S_{a}f\to f|\gamma|_{L_{1}}$ in $L_{1}(\mathbb{R})$ as $a\to0$
for each $f\in L_{1}(\mathbb{R})$. Thus, choosing a suitable sequence
$(a_{n})_{n\in\mathbb{N}}$ of positive reals with $a_{n}\to0,$ we
get that 
\[
\intop_{\mathbb{R}}f(a_{n}z+x)\gamma(z)\mbox{ d}z\to f(x)|\gamma|_{L_{1}}\quad(n\to\infty)
\]
for almost every $x\in\mathbb{R}.$ Thus, \prettyref{eq:Fourier_Gauss},
\prettyref{eq:Fourier_Gauss_lhs} and $|\gamma|_{L_{1}}=\sqrt{2\pi}$
imply 
\[
\frac{1}{\sqrt{2\pi}}\intop_{\mathbb{R}}\e^{\i xt}\left(\mathcal{F}f\right)(t)\mbox{ d}t=f(x)
\]
for almost every $x\in\mathbb{R},$ which is \prettyref{eq:inverse Fourier}.
In particular, we obtain for $\phi\in\mathcal{S}(\mathbb{R})$ 
\[
\phi(x)=(\mathcal{F}^{\ast}\mathcal{F}\phi)(x)\quad(x\in\mathbb{R}),
\]
which gives that $\mathcal{F}|_{\mathcal{S}(\mathbb{R})}$ is one-to-one.
Moreover, 
\[
\left(\mathcal{F}\mathcal{F}^{\ast}\phi\right)(x)=\left(\mathcal{F}\mathcal{F}\sigma_{-1}\phi\right)(x)=\left(\mathcal{F}^{\ast}\mathcal{F}\sigma_{-1}\phi\right)(-x)=\phi(x)\quad(x\in\mathbb{R}),
\]
which shows that $\mathcal{F}|_{\mathcal{S}(\mathbb{R})}$ is onto.
\end{proof}
The latter proposition yields that $\mathcal{F}$ can be established
as a unitary operator on $L_{2}(\mathbb{R}).$
\begin{thm}[Plancherel]
\label{thm:Plancherel} The operator $\mathcal{F}|_{\mathcal{S}(\mathbb{R})}:\mathcal{S}(\mathbb{R})\subseteq L_{2}(\mathbb{R})\to L_{2}(\mathbb{R})$
has a unique unitary extension, again denoted by $\mathcal{F}:L_{2}(\mathbb{R})\to L_{2}(\mathbb{R}).$\end{thm}
\begin{proof}
By \prettyref{prop:Fourier_Schwartz} we have $\mathcal{F}[\mathcal{S}(\mathbb{R})]=\mathcal{S}(\mathbb{R}),$
which shows, that $\mathcal{F}|_{\mathcal{S}(\mathbb{R})}$ has dense
range. Moreover, by \prettyref{eq:adjoint_Fourier} we have for $\phi\in\mathcal{S}(\mathbb{R})$
\begin{align*}
|\mathcal{F}\phi|_{L_{2}(\mathbb{R})}^{2} & =\langle\mathcal{F}\phi|\mathcal{F}\phi\rangle_{L_{2}(\mathbb{R})}=\langle\phi|\mathcal{F}^{\ast}\mathcal{F}\phi\rangle_{L_{2}(\mathbb{R})}=|\phi|_{L_{2}(\mathbb{R})}^{2},
\end{align*}
which shows, that $\mathcal{F}|_{\mathcal{S}(\mathbb{R})}$ is an
isometry on $L_{2}(\mathbb{R}).$ Both things imply that $\mathcal{F}|_{\mathcal{S}(\mathbb{R})}$
has a unique unitary extension.\end{proof}
\begin{defn*}
Let $f\in L_{1,\rho}(\mathbb{R})\coloneqq\left\{ g:\mathbb{R}\to\mathbb{C}\,;\, g\mbox{ measurable, }\intop_{\mathbb{R}}|g|\e^{-\rho t}\mbox{ d}t<\infty\right\} $
for some $\rho\in\mathbb{R}.$ Then we define 
\[
\mathcal{L}_{\rho}f(t)\coloneqq\mathcal{F}\left(\e^{-\rho\m}f\right)(t)=\frac{1}{\sqrt{2\pi}}\intop_{\mathbb{R}}\e^{-(\i t+\rho)x}f(x)\mbox{ d}x\quad(t\in\mathbb{R}),
\]
the \emph{Fourier-Laplace transform of $f$. }\nomenclature[O_110]{$\mathcal{F}$}{the Fourier transformation.}\nomenclature[O_120]{$\mathcal{L}_\rho$}{the Fourier-Laplace transformation.}Moreover,
we introduce the notation 
\[
\hat{f}(\i t+\rho)\coloneqq\mathcal{L}_{\rho}f(t)\quad(t\in\mathbb{R}).
\]
\end{defn*}
\begin{cor}
Let $\rho\in\mathbb{R}.$ Then the Fourier-Laplace transform can be
established as a unitary operator $\mathcal{L}_{\rho}:H_{\rho}(\mathbb{R})\to L_{2}(\mathbb{R}).$\end{cor}
\begin{proof}
We note that $\mathcal{L}_{\rho}=\mathcal{F}\e^{-\rho\m}$. Since
$\e^{-\rho\m}:H_{\rho}(\mathbb{R})\to L_{2}(\mathbb{R})$ is obviously
unitary, we derive that $\mathcal{L}_{\rho}$ is unitary by \prettyref{thm:Plancherel}.
\end{proof}
We also provide an adapted version of the Riemann-Lebesgue Lemma.
\begin{prop}
\label{prop:Riemann-Lebesgue_professional} Let $f\in L_{1}(\mathbb{R}_{\geq0}),$
such that $\left(t\mapsto tf(t)\right)\in L_{1}(\mathbb{R}_{\geq0}).$
Moreover, let $K\subseteq\mathbb{R}_{\geq0}$ compact. Then 
\[
\forall\varepsilon>0\,\exists M>0\,\forall|t|\geq M,\rho\in K:|\mathcal{L}_{\rho}f(t)|<\varepsilon.
\]
\end{prop}
\begin{proof}
We first show, that 
\[
\mathbb{R}_{\geq0}\ni\rho\mapsto\mathcal{L}_{\rho}f\in C_{b}(\mathbb{R})
\]
is Lipschitz-continuous. Indeed, for $\rho,\rho'\in\mathbb{R}_{\geq0}$
we have that 
\begin{align*}
|\mathcal{L}_{\rho}f(t)-\mathcal{L}_{\rho'}f(t)| & \leq\frac{1}{\sqrt{2\pi}}\intop_{\mathbb{R}_{\geq0}}|f(s)|\left|\e^{-\rho s}-\e^{-\rho's}\right|\mbox{ d}s\\
 & \leq\frac{1}{\sqrt{2\pi}}\intop_{\mathbb{R}_{\geq0}}|sf(s)|\mbox{ d}s\,|\rho-\rho'|
\end{align*}
for each $t\in\mathbb{R},$ which shows the Lipschitz-continuity.
Consequently, 
\[
\left\{ \mathcal{L}_{\rho}f\,;\,\rho\in K\right\} \subseteq C_{b}(\mathbb{R})
\]
is compact. Let now $\varepsilon>0$ and choose $K'\subseteq K$ finite
such that 
\[
\left\{ \mathcal{L}_{\rho}f\,;\,\rho\in K\right\} \subseteq\bigcup_{\rho'\in K'}B(\mathcal{L}_{\rho'}f,\varepsilon).
\]
According to the Lemma of Riemann-Lebesgue (\prettyref{rem:Riemann-Lebesgue}),
there is some $M>0$ such that 
\[
\forall|t|\geq M,\rho'\in K':|\mathcal{L}_{\rho'}f(t)|<\varepsilon.
\]
Let now $\rho\in K,|t|\geq M.$ Then there is $\rho'\in K'$ such
that $\|\mathcal{L}_{\rho}f-\mathcal{L}_{\rho'}f\|_{\infty}<\varepsilon$
and we conclude
\[
|\mathcal{L}_{\rho}f(t)|\leq\|\mathcal{L}_{\rho}f-\mathcal{L}_{\rho'}f\|+|\mathcal{L}_{\rho'}f(t)|<2\varepsilon.\tag*{\qedhere}
\]
 
\end{proof}
As a further property of the Fourier-Laplace transform we want to
state a theorem of Paley and Wiener, characterizing the $L_{2}$-functions
supported on the positive real axis by their Fourier-Laplace transform.
We begin with the following lemma.
\begin{lem}
\label{lem:positive_support}Let $f\in\bigcap_{\rho>0}H_{\rho}(\mathbb{R})$
with $\sup_{\rho>0}|f|_{H_{\rho}}<\infty.$ Then $f\in L_{2}(\mathbb{R}_{\geq0})$
and $|f|_{L_{2}}=\sup_{\rho>0}|f|_{H_{\rho}}=\lim_{\rho\to0}|f|_{H_{\rho}}$.\end{lem}
\begin{proof}
We first show that $f$ is supported on $\mathbb{R}_{\geq0}.$ For
doing so, let $a<b<0$. Then for each $\rho>0$ we estimate 
\begin{align*}
\intop_{a}^{b}|f(t)|\mbox{ d}t & =\intop_{a}^{b}\e^{\rho t}\e^{-\rho t}|f(t)|\mbox{ d}t\\
 & \leq\left(\intop_{a}^{b}\e^{2\rho t}\mbox{ d}t\right)^{\frac{1}{2}}|f|_{H_{\rho}}\\
 & =\left(\frac{1}{2\rho}\left(\e^{2\rho b}-\e^{2\rho a}\right)\right)^{\frac{1}{2}}|f|_{H_{\rho}}\\
 & \leq\e^{\rho b}\sqrt{(b-a)}|f|_{H_{\rho}},
\end{align*}
by the mean value theorem. Letting $\rho\to\infty$ and using that
$b<0$ as well as $\sup_{\rho>0}|f|_{H_{\rho}}<\infty$ we derive
\[
\intop_{a}^{b}|f(t)|\mbox{ d}t=0.
\]
As $a<b<0$ were chosen arbitrarily, we get $\spt f\subseteq\mathbb{R}_{\geq0}$.
Moreover, we estimate for each $n\in\mathbb{N}$ and $\rho>0$
\[
\intop_{0}^{n}|f(t)|^{2}\mbox{ d}t=\intop_{0}^{n}\e^{2\rho t}e^{-2\rho t}|f(t)|^{2}\mbox{ d}t\leq\e^{2\rho n}|f|_{H_{\rho}}^{2}\leq\e^{2\rho n}\sup_{\mu>0}|f|_{H_{\mu}}^{2}.
\]
Letting $\rho\to0$ we derive $\intop_{0}^{n}|f(t)|^{2}\mbox{ d}t\leq\sup_{\mu>0}|f|_{H_{\mu}}^{2}$
for each $n\in\mathbb{N}$ and so, $f\in L_{2}(\mathbb{R}_{\geq0})$
with $|f|_{L_{2}}\leq\sup_{\rho>0}|f|_{H_{\rho}}.$ Moreover, since
$\spt f\subseteq\mathbb{R}_{\geq0}$ we get $|f|_{H_{\rho}}\leq|f|_{H_{\mu}}$
for $\rho>\mu\geq0,$ and hence, $|f|_{L_{2}}\geq\sup_{\rho>0}|f|_{H_{\rho}}=\lim_{\rho\to0}|f|_{H_{\rho}},$
which completes the proof. 
\end{proof}
We are now able to state the Paley-Wiener Theorem. We mainly follow
the proof presented in \cite[19.2 Theorem]{rudin1987real}.
\begin{thm}[Paley-Wiener, \cite{Paley_Wiener}]
\label{thm:Paley-Wiener} Let $g\in\mathcal{H}^{2}(\mathbb{C}_{>0}).$
\nomenclature[B_070]{$\mathcal{H}^2(\mathbb{C}_{\Re>a})$}{the Hardy space of analytic functions $f$ on the half plane $\mathbb{C}_{\Re>a}$ with $\sup_{\rho>a}\intop_{\mathbb{R}} \vert f(\i t+\rho)\vert^2\, \dd t<\infty$.}Then
there exists $f\in L_{2}(\mathbb{R}_{\geq0})$ such that 
\[
g(\i t+\rho)=\hat{f}(\i t+\rho)=\left(\mathcal{L}_{\rho}f\right)(t)\quad(t\in\mathbb{R},\rho>0).
\]
Moreover, $|f|_{L_{2}}=|g|_{\mathcal{H}^{2}}.$\end{thm}
\begin{proof}
For $\rho>0$ we define $g_{\rho}=\left(t\mapsto g(\i t+\rho)\right).$
By assumption, $g_{\rho}\in L_{2}(\mathbb{R})$ for each $\rho>0$
and $\sup_{\rho>0}|g_{\rho}|_{L_{2}}<\infty.$ For $\rho>0$ we define
$f_{\rho}\coloneqq\mathcal{F}^{\ast}g_{\rho}\in L_{2}(\mathbb{R}).$
Moreover, we set $f(x)\coloneqq\e^{x}f_{1}(x)$ for $x\in\mathbb{R}.$
We  prove that $f\in\bigcap_{\rho>0}H_{\rho}(\mathbb{R}).$ For doing
so, let $\rho>0.$ Moreover we fix $a>0$ and $x\in\mathbb{R}$. By
Cauchy's integral theorem we have 
\begin{align}
\i\intop_{-a}^{a}\e^{(\i t+1)x}g(\i t+1)\mbox{ d}t-\intop_{\rho}^{1}\e^{(\i a+\kappa)x}g(\i a+\kappa)\mbox{ d}\kappa-\i\intop_{-a}^{a}\e^{(\i t+\rho)x}g(\i t+\rho)\mbox{ d}t+\label{eq:cauchy_integral}\\
+\intop_{\rho}^{1}\e^{(-\i a+\kappa)x}g(-\i a+\kappa)\mbox{ d}\kappa & =0.\nonumber 
\end{align}
Moreover, since 
\[
\intop_{\mathbb{R}}\left|\intop_{\rho}^{1}\e^{(\pm\i a+\kappa)x}g(\pm\i a+\kappa)\mbox{ d}\kappa\right|^{2}\mbox{ d}a\leq|1-\rho||g|_{\mathcal{H}^{2}}^{2}\intop_{\rho}^{1}\e^{2\kappa x}\mbox{ d}\kappa<\infty
\]
we find a sequence $(a_{n})_{n\in\mathbb{N}}$ such that $a_{n}\to\infty$
and 
\[
\intop_{\rho}^{1}\e^{(\pm ia_{n}+\kappa)x}g(\pm\i a_{n}+\kappa)\mbox{ d}\kappa\to0\quad(n\to\infty).
\]
Hence, \prettyref{eq:cauchy_integral} implies 
\[
\intop_{-a_{n}}^{a_{n}}\e^{(\i t+1)x}g(\i t+1)\mbox{ d}t-\intop_{-a_{n}}^{a_{n}}\e^{(\i t+\rho)x}g(\i t+\rho)\mbox{ d}t\to0\quad(n\to\infty).
\]
Since on the other hand we have 
\[
\lim_{n\to\infty}\intop_{-a_{n}}^{a_{n}}\e^{\i tx}g(\i t+\mu)\mbox{ d}t=\lim_{n\to\infty}\mathcal{F}^{\ast}\left(\chi_{[-a_{n},a_{n}]}(\m)g_{\mu}\right)\to f_{\mu}\quad\mbox{ in }L_{2}(\mathbb{R})
\]
for each $\mu>0,$ we may choose a subsequence (again labeled by $n$),
such that 
\[
\intop_{-a_{n}}^{a_{n}}\e^{(\i t+1)x}g(\i t+1)\mbox{ d}t-\intop_{-a_{n}}^{a_{n}}\e^{(\i t+\rho)x}g(\i t+\rho)\mbox{ d}t\to\e^{x}f_{1}(x)-\e^{\rho x}f_{\rho}(x)\quad(n\to\infty)
\]
for almost every $x\in\mathbb{R}.$ Thus, we have 
\[
f(x)=\e^{x}f_{1}(x)=\e^{\rho x}f_{\rho}(x)\quad(x\in\mathbb{R}\mbox{ a.e.})
\]
and since $f_{\rho}\in L_{2}(\mathbb{R})$ we obtain $f\in H_{\rho}(\mathbb{R}).$
Moreover, the latter equality together with \prettyref{thm:Plancherel}
gives 
\[
\sup_{\rho>0}|f|_{H_{\rho}}=\sup_{\rho>0}|f_{\rho}|_{L_{2}}=\sup_{\rho>0}|g_{\rho}|_{L_{2}}=|g|_{\mathcal{H}^{2}.}
\]
The assertion now follows from \prettyref{lem:positive_support}.
\end{proof}
As a consequence of the latter theorem we obtain the following corollary.
\begin{cor}
\label{cor:Paley_Wiener_unitary}Let $\rho_{0}\in\mathbb{R}$. Then
the mapping 
\begin{align*}
\mathcal{L}:H_{\rho_{0}}(\mathbb{R}_{\geq0}) & \to\mathcal{H}^{2}(\mathbb{C}_{\Re>\rho_{0}})\\
f & \mapsto\hat{f}
\end{align*}
is unitary.\end{cor}
\begin{proof}
Let $f\in H_{\rho_{0}}(\mathbb{R}_{\geq0}).$ Then $\left(t\mapsto\e^{-\rho t}f(t)\right)\in L_{1}(\mathbb{R}_{\geq0})\cap L_{2}(\mathbb{R}_{\geq0})$
for each $\rho>\rho_{0}$ and hence 
\[
\hat{f}(z)=\frac{1}{\sqrt{2\pi}}\intop_{0}^{\infty}\e^{-zs}f(s)\mbox{ d}s\quad(z\in\mathbb{C}_{\Re>\rho_{0}}).
\]
From this representation, we see that $\hat{f}$ is analytic. Moreover,
due to \prettyref{thm:Plancherel} we have 
\[
|\hat{f}(\i\cdot+\rho)|_{L_{2}}=|\e^{-\rho\cdot}f|_{L_{2}}=|f|_{H_{\rho}}\quad(\rho>\rho_{0}).
\]
Moreover, since $\spt f\subseteq\mathbb{R}_{\geq0}$ we have that
$\rho\mapsto|f|_{H_{\rho}}$ is monotone decreasing and by monotone
convergence we obtain $|f|_{H_{\rho_{0}}}=\lim_{\rho\to\rho_{0}}|f|_{H_{\rho}}=\sup_{\rho>\rho_{0}}|f|_{H_{\rho}}$.
Thus, we have 
\[
|\hat{f}\,|_{\mathcal{H}^{2}}=\sup_{\rho>\rho_{0}}|\hat{f}(\i\cdot+\rho)|_{L_{2}}=\sup_{\rho>\rho_{0}}|f|_{H_{\rho}}=|f|_{H_{\rho_{0}}},
\]
which shows that $\mathcal{L}$ is isometric. To show that it is onto,
let $g\in\mathcal{H}^{2}(\mathbb{C}_{\Re>\rho_{0}})$. Then we define
$\tilde{g}\coloneqq g(\cdot+\rho_{0})\in\mathcal{H}^{2}(\mathbb{C}_{\Re>0})$
and by \prettyref{thm:Paley-Wiener} there is $h\in L_{2}(\mathbb{R}_{\geq0})$
such that $\hat{h}=\tilde{g}$ on $\mathbb{C}_{\Re>0}.$ Setting $f\coloneqq\e^{\rho_{0}\cdot}h\in H_{\rho_{0}}(\mathbb{R}_{\geq0})$
we obtain 
\begin{align*}
\hat{f}(\i t+\rho) & =\frac{1}{\sqrt{2\pi}}\intop_{0}^{\infty}\e^{-(\i t+\rho)s}e^{\rho_{0}s}h(s)\mbox{ d}s\\
 & =\hat{h}(\i t+\rho-\rho_{0})=\tilde{g}(\i t+\rho-\rho_{0})=g(\i t+\rho)\quad(t\in\mathbb{R},\rho>\rho_{0}),
\end{align*}
which finishes the proof.
\end{proof}
Using this result, we are able to prove the following statement, which
will be used later in the appendix.
\begin{cor}
\label{cor:exp_dense}Let $\Lambda\subseteq\mathbb{R}_{>0}$ be a
set with an accumulation point. Then the set $\left\{ \e^{-\lambda\cdot}\,;\,\lambda\in\Lambda\right\} $
is total in $L_{1}(\mathbb{R}_{\geq0}).$\end{cor}
\begin{proof}
Since $L_{1}(\mathbb{R}_{\geq0})'=L_{\infty}(\mathbb{R}_{\geq0}),$
it suffices to show that if $f\in L_{\infty}(\mathbb{R}_{\geq0})$
with 
\[
\intop_{0}^{\infty}\e^{-\lambda t}f(t)\mbox{ d}t=0\quad(\lambda\in\Lambda),
\]
then $f=0.$ Since $L_{\infty}(\mathbb{R}_{\geq0})\subseteq\bigcap_{\rho>0}H_{\rho}(\mathbb{R}_{\geq0}),$
we get by the latter equality $\hat{f}(\lambda)=0$ for each $\lambda\in\Lambda.$
Since $\hat{f}$ is holomorphic on $\mathbb{C}_{\Re>0}$ and zero
on a set with accumumlation point, it follows that $\hat{f}=0$ and
hence, by \prettyref{cor:Paley_Wiener_unitary}, we obtain $f=0.$ \end{proof}

\chapter{The Theorem of Rademacher}

In this section we prove Rademacher's Theorem for Lipschitz-continuous
functions on $\mathbb{R}$ with values in a reflexive Banach space.
The theorem, proved by Rademacher in 1919 (\cite{Rademacher1919})
for functions $F:U\subseteq\mathbb{R}^{n}\to\mathbb{R}^{m}$, states
that a Lipschitz-continuous function is differentiable almost everywhere
and its derivative is bounded. \\
First, we begin to prove the scalar-valued case. For doing so, we
need Lebesgue's famous differentiation theorem, which will be shown
in the first part of this appendix. We follow the argumentation presented
in \cite[Chapter 7]{rudin1987real} and we start by stating Vitalis
Covering Lemma. For doing so, we need the following auxiliary result.
\begin{lem}
\label{lem:subcollection} Let $X$ be a metric space and $\mathcal{U}\subseteq\mathcal{P}(X)$
a collection of subsets with non-empty interior. Then there exists
a maximal disjoint subcollection $\mathcal{S}\subseteq\mathcal{U}$
(i.e. all elements in $\mathcal{S}$ are pairwise disjoint). Moreover,
if $X$ is separable, then $\mathcal{S}$ is at most countable.\end{lem}
\begin{proof}
We apply Zorn's Lemma to $\mathcal{M}\coloneqq\left\{ \mathcal{S}\subseteq\mathcal{U}\,;\,\forall U,V\in\mathcal{S}:U\cap V\ne\emptyset\Rightarrow U=V\mbox{ }\right\} $
ordered by inclusion. Clearly, every totally ordered subset of $\mathcal{M}$
has an upper bound given by its union and hence, the first assertion
follows. Assume now that $X$ is separable and let $(x_{n})_{n\in\mathbb{N}}$
be a dense sequence in $X$. Then the mapping 
\begin{align*}
F:\left\{ x_{n}\,;\, n\in\mathbb{N}\right\} \cap\bigcup\mathcal{S} & \to\mathcal{S}\\
x & \mapsto U(x),
\end{align*}
where $U(x)\in\mathcal{S}$ is the unique set with $x\in U(x)$ (recall
that $\mathcal{S}$ consists of pairwise disjoint sets), is onto.
Indeed, for each $U\in\mathcal{S}$ there exists $n\in\mathbb{N}$
with $x_{n}\in\interior U,$ since $\interior U$ is non-empty and
$\{x_{n}\,;\, n\in\mathbb{N}\}$ is dense. Hence, $\mathcal{S}$ is
at most countable. \end{proof}
\begin{prop}[Vitali Covering Lemma]
\label{prop:vitali} Let $n\in\mathbb{N}_{>0},$ $\mathcal{B}\subseteq\mathcal{P}(\mathbb{R}^{n})$
a collection of non-empty balls such that 
\[
R\coloneqq\sup_{U\in\mathcal{B}}\diam U<\infty.
\]
Then there exists an at most countable disjoint subcollection $\mathcal{\tilde{B}}\subseteq\mathcal{B}$
such that\nomenclature[Z_030]{$\lambda(A)$}{the Lebesgue measure of a Borel set $A$.}
\[
\lambda\left(\bigcup\mathcal{B}\right)\leq5^{n}\lambda\left(\bigcup\tilde{\mathcal{B}}\right).
\]
\end{prop}
\begin{proof}
For $k\in\mathbb{N}$ we define 
\[
\mathcal{B}_{k}\coloneqq\left\{ U\in\mathcal{B}\,;\,2^{-(k+1)}R<\diam U\leq2^{-k}R\right\} .
\]
 We set $\tilde{\mathcal{B}}_{0}\subseteq\mathcal{B}_{0}$ as a maximal
disjoint subcollection and define recursively $\mathcal{\tilde{B}}_{k+1}$
to be a maximal disjoint subcollection of 
\[
\mathcal{C}_{k+1}\coloneqq\left\{ U\in\mathcal{B}_{k+1}\,;\,\forall V\in\bigcup_{j=0}^{k}\tilde{\mathcal{B}}_{j}:U\cap V=\emptyset\right\} .
\]
Then $\tilde{\mathcal{B}}\coloneqq\bigcup_{k\in\mathbb{N}}\tilde{\mathcal{B}}_{k}$
is a disjoint subcollection of $\mathcal{B}$, which is at most countable
by \prettyref{lem:subcollection}. Moreover, for each $U\in\mathcal{B}_{k}$
for some $k\in\mathbb{N},$ there exists $V\in\bigcup_{j=0}^{k}\tilde{\mathcal{B}}_{j}$
with $V\cap U\ne\emptyset.$ Indeed, if $U\cap V=\emptyset$ for each
$V\in\bigcup_{j=0}^{k}\tilde{\mathcal{B}}_{j}$, then $U\in\mathcal{C}_{k}$
and $U\cap V=\emptyset$ for each $V\in\tilde{\mathcal{B}}_{k}.$
Then, $\tilde{\mathcal{B}}_{k}\cup\{U\}$ would be a disjoint subcollection
of $\mathcal{C}_{k},$ which contradicts the maximality of $\tilde{\mathcal{B}}_{k}.$
So for $U\in\mathcal{B}_{k}$ there is $V\in\bigcup_{j=0}^{k}\tilde{\mathcal{B}}_{j}$
such that $U\cap V\ne\emptyset.$ Then, $U\subseteq5V,$ where $5V$
denotes the ball with the same center as $V$ but with 5 times its
radius. Indeed, since $\diam U\leq2^{-k}R$ and $\diam V>2^{-(k+1)}R,$
the triangle inequality yields the assertion. Summarizing we get 
\[
\lambda\left(\bigcup\mathcal{B}\right)=\lambda\left(\bigcup_{k\in\mathbb{N}}\bigcup\mathcal{B}_{k}\right)\leq\lambda\left(\bigcup_{k\in\mathbb{N}}\bigcup_{V\in\tilde{\mathcal{B}}_{k}}5V\right)\leq5^{n}\lambda\left(\bigcup\tilde{\mathcal{B}}\right).\tag*{\qedhere}
\]

\end{proof}
We will use Vitali's covering lemma to prove the weak 1-1 estimate
for the Hardy-Littlewood maximal function, which is defined as follows.
\begin{defn*}
Let $X$ be a Banach space, $n\in\mathbb{N}_{>0},$ $f\in L_{1,\mathrm{loc}}(\mathbb{R}^{n};X)$.
We denote the collection of all open and closed balls by $\mathcal{B}\coloneqq\left\{ B(y,r)\,;\, y\in\mathbb{R}^{n},r>0\right\} \cup\left\{ B[y,r]\,;\, y\in\mathbb{R}^{n},r>0\right\} $.
Then the \emph{Hardy-Littlewood maximal function} applied to $f$
is given by 
\[
\left(Mf\right)(x)\coloneqq\sup\left\{ \frac{1}{\lambda(B)}\intop_{B}|f(y)|\mbox{ d}y\,;\, B\in\mathcal{B},x\in B\right\} \quad(x\in\mathbb{R}^{n}).
\]
\end{defn*}
\begin{lem}
Let $X$ be a Banach space, $n\in\mathbb{N}_{>0},$ $f\in L_{1,\mathrm{loc}}(\mathbb{R}^{n};X)$.
Then $Mf$ is lower-semicontinuous.\end{lem}
\begin{proof}
Let $t\geq0$ and consider the set 
\[
U\coloneqq\left\{ x\in\mathbb{R}^{n}\,;\,(Mf)(x)>t\right\} .
\]
Let $x\in U.$ Then there exists $y\in\mathbb{R}^{n},r>0$ with $|x-y|\leq r$
such that 
\[
t<\frac{1}{\lambda\left(B(y,r)\right)}\intop_{B(y,r)}|f(z)|\mbox{ d}z.
\]
Choose now $r'>r$ such that
\[
t<\frac{1}{\lambda(B(y,r'))}\intop_{B(y,r)}|f(z)|\mbox{ d}z.
\]
Then for each $x'$ with $|x'-x|<r'-r$ we have $x'\in B(y,r')$ and
thus,
\[
t<\frac{1}{\lambda(B(y,r'))}\intop_{B(y,r)}|f(z)|\mbox{ d}z\leq\frac{1}{\lambda(B(y,r'))}\intop_{B(y,r')}|f(z)|\mbox{ d}z\leq(Mf)(x'),
\]
hence $B(x,r'-r)\subseteq U.$ Thus, $U$ is open and hence, $Mf$
is lower-semicontinuous.\end{proof}
\begin{prop}
\label{prop:Hardy-Littlewood}Let $X$ be a Banach space, $n\in\mathbb{N}_{>0},f\in L_{1}(\mathbb{R}^{n};X)$
and $t>0.$ Then
\[
\lambda\left(\left\{ x\in\mathbb{R}^{n}\,;\,\left(Mf\right)(x)>t\right\} \right)\leq\frac{5^{n}}{t}|f|_{L_{1}}.
\]
\end{prop}
\begin{proof}
Set $\mathcal{C}\coloneqq\left\{ B\in\mathcal{B}\,;\,\frac{1}{\lambda(B)}\intop_{B}|f|>t\right\} .$
Then $\lambda(B)\leq\frac{|f|_{1}}{t}$ for each $B\in\mathcal{C}$,
which in particular implies $\sup_{B\in\mathcal{C}}\diam B<\infty.$
By \prettyref{prop:vitali} there exists a countable disjoint subcollection
$\tilde{\mathcal{C}}$ of $\mathcal{C}$ such that 
\[
\lambda(\bigcup\mathcal{C})\leq5^{n}\lambda(\bigcup\tilde{\mathcal{C}}).
\]
Hence, we can estimate
\begin{align*}
\lambda\left(\left\{ x\in\mathbb{R}^{n}\,;\,\left(Mf\right)(x)>t\right\} \right) & \leq\lambda\left(\bigcup\mathcal{C}\right)\\
 & \leq5^{n}\sum_{B\in\tilde{\mathcal{C}}}\lambda(B)\\
 & \leq5^{n}\sum_{B\in\tilde{\mathcal{C}}}\frac{1}{t}\intop_{B}|f|\\
 & \leq\frac{5^{n}}{t}|f|_{L_{1}}.\tag*{\qedhere}
\end{align*}

\end{proof}
With the help of the latter estimate we are able to prove the Lebesgue
Differentiation Theorem.
\begin{thm}[Lebesgue Differentiation Theorem]
\label{thm:Lebesgue} Let $X$ be a Banach space, $n\in\mathbb{N}_{>0},f\in L_{1,\mathrm{loc}}(\mathbb{R}^{n};X).$
For $x\in\mathbb{R}^{n}$ we define $\mathcal{B}_{x}\coloneqq\left\{ B\in\mathcal{B}\,;\, x\in B\right\} $.
Then $(\mathcal{B}_{x},\supseteq)$ is a directed set and for almost
every $x\in\mathbb{R}^{n}$
\begin{equation}
\lim_{B\in\mathcal{B}_{x}}\frac{1}{\lambda(B)}\intop_{B}\left|f(y)-f(x)\right|\,\mathrm{d}y=0.\label{eq:Lebesgue}
\end{equation}
\end{thm}
\begin{proof}
As \prettyref{eq:Lebesgue} is a local property, we may assume without
loss of generality that $f\in L_{1}(\mathbb{R}^{n};X).$ It suffices
to prove that the set 
\[
M_{\delta}\coloneqq\left\{ x\in\mathbb{R}^{n}\,;\,\limsup_{B\in\mathcal{B}_{x}}\frac{1}{\lambda(B)}\intop_{B}|f(y)-f(x)|\mbox{ d}y>\delta\right\} 
\]
has measure $0$ for each $\delta>0$. For doing so, let $\varepsilon,\delta>0$
and choose $g\in C_{c}^{\infty}(\mathbb{R}^{n};X)$ with $|f-g|_{L_{1}}<\varepsilon.$
Then for $x\in\mathbb{R}^{n}$ we have 
\begin{align*}
 & \limsup_{B\in\mathcal{B}_{x}}\frac{1}{\lambda(B)}\intop_{B}|f(y)-f(x)|\mbox{ d}y\\
 & \leq\limsup_{B\in\mathcal{B}_{x}}\frac{1}{\lambda(B)}\intop_{B}|f(y)-g(y)|\mbox{ d}y+\limsup_{B\in\mathcal{B}_{x}}\frac{1}{\lambda(B)}\intop_{B}|g(y)-g(x)|\mbox{ d}y+|g(x)-f(x)|\\
 & \leq M(f-g)(x)+|g(x)-f(x)|,
\end{align*}
and thus, 
\begin{align*}
\lambda(M_{\delta}) & \leq\lambda\left(\left\{ x\in\mathbb{R}^{n}\,;\, M(f-g)(x)>\frac{\delta}{2}\right\} \right)+\lambda\left(\left\{ x\in\mathbb{R}^{n}\,;\,|f(x)-g(x)|>\frac{\delta}{2}\right\} \right)\\
 & \leq5^{n}\frac{2}{\delta}|f-g|_{L_{1}}+\frac{2}{\delta}|f-g|_{L_{1}}\\
 & <\frac{2(5^{n}+1)}{\delta}\varepsilon,
\end{align*}
by \prettyref{prop:Hardy-Littlewood}, which yields the assertion.
\end{proof}
Using the latter theorem, we are able to prove Rademacher's Theorem
in the scalar-valued case.
\begin{thm}[Rademacher's Theorem; scalar-valued]
\label{thm:Radmacher1} Let $F:I\subseteq\mathbb{R}\to\mathbb{K}$
Lipschitz-continuous, where $I$ is an (possibly unbounded) interval
and $\mathbb{K}\in\{\mathbb{R},\mathbb{C}\}$. Then $F$ is differentiable
almost everywhere with $F'\in L_{\infty}(I)$ and $|F'|{}_{L_{\infty}}=|F|_{\mathrm{Lip}}.$
Moreover 
\[
F(t)-F(s)=\intop_{s}^{t}F'
\]
for all $t,s\in I,s<t.$\end{thm}
\begin{proof}
We define the linear functional $\varphi:\lin\left\{ \chi_{[s,t[}\,;\, s,t\in I,s<t\right\} \subseteq L_{1}(I)\to\mathbb{K}$
by 
\[
\varphi\left(\sum_{i=1}^{n}\alpha_{i}\chi_{[s_{i},t_{i}[}\right)\coloneqq\sum_{i=1}^{n}\alpha_{i}\left(F(t_{i})-F(s_{i})\right),
\]
for $n\in\mathbb{N}_{\geq1},$ $\alpha_{i}\in\mathbb{K},s_{i},t_{i}\in I,$
with $s_{i}<t_{i}$ for $i\in\{1,\ldots,n\}.$ Then, for pairwise
disjoint intervals $[s_{i},t_{i}[,$ we have 
\[
\left|\varphi\left(\sum_{i=1}^{n}\alpha_{i}\chi_{[s_{i},t_{i}[}\right)\right|\leq|F|_{\mathrm{Lip}}\sum_{i=1}^{n}\left|\alpha_{i}\right|(t_{i}-s_{i})=|F|_{\mathrm{Lip}}\left|\sum_{i=1}^{n}\alpha_{i}\chi_{[s_{i},t_{i}[}\right|_{L_{1}}
\]
and thus, $\varphi$ can be extended to an element in $L_{1}(I)'$,
again denoted by $\varphi,$ with $\|\varphi\|\leq|F|_{\mathrm{Lip}}.$
Since 
\[
|F(t)-F(s)|=|\varphi(\chi_{[s,t[})|\leq\|\varphi\||\chi_{[s,t[}|_{L_{1}}=\|\varphi\|(t-s),
\]
for $s,t\in I,s<t,$ we also obtain $|F|_{\mathrm{Lip}}\leq\|\varphi\|.$
Since $L_{1}(I)'\cong L_{\infty}(I)$, there exists $f\in L_{\infty}(I)$
with $|f|_{L_{\infty}}=\|\varphi\|=|F|_{\mathrm{Lip}}$ such that
\[
\varphi(g)=\intop_{I}fg\quad(g\in L_{1}(I)).
\]
In particular, we have 
\[
F(t)-F(s)=\intop_{s}^{t}f
\]
for $s,t\in I,s<t$ and hence, the assertion follows by \prettyref{thm:Lebesgue}.
\end{proof}
In order to prove Rademacher's Theorem for Banach space-valued functions,
we need the following result due to Pettis (\cite{Pettis1938}). We
follow the proof given in \cite[Theorem 3.5.3]{hille1957functional}.
\begin{thm}[Pettis]
\label{thm:Pettis}Let $X$ be a Banach space over $\mathbb{K}\in\{\mathbb{R},\mathbb{C}\}$,
$(\Omega,\mathcal{F},\mu)$ a $\sigma$-finite measure space and $f:\Omega\to X.$
Then $f$ is measurable%
\footnote{Recall that $f$ is called measurable, if there exists a sequence
of simple functions, which converge to $f$ almost everywhere. A simple
function is a finitely-valued function $g:\Omega\to X,$ such that
$g^{-1}[\{x\}]\in\mathcal{F}$ with $\mu\left(g^{-1}[\{x\}]\right)<\infty$
for each $x\in X$. %
} if and only if

\begin{enumerate}[(a)]

\item $f$ is weakly measurable, i.e. $x'\circ f:\Omega\to\mathbb{K}$
is measurable for each $x'\in X'$, and

\item $f$ is almost separably valued, i.e. $\overline{\lin f[\Omega\setminus N]}$
is separable for some $N\in\mathcal{F}$ with $\mu(N)=0$. 

\end{enumerate}\end{thm}
\begin{proof}
If $f$ is measurable, then clearly it is weakly measurable. Moreover,
as $f$ is the almost everywhere limit of simple functions, it is
almost separably-valued. Assume now conversely that $f$ satisfies
(a) and (b). We define $Y\coloneqq\overline{\lin f[\Omega\setminus N]}$,
which is a separable Banach space by (b). Thus, there exists a sequence
$(x_{n}')_{n\in\mathbb{N}}$ in $X'$ such that%
\footnote{Choose a dense sequence $(y_{n})_{n\in\mathbb{N}}$ in $Y$ and let
$x_{n}'\in X'$ with $\|x_{n}'\|=1$ and $|x_{n}'(y_{n})|=|y_{n}|$
for $n\in\mathbb{N}.$%
} 
\[
|y|=\sup_{n\in\mathbb{N}}|x_{n}'(y)|\quad(y\in Y).
\]
Thus, by (a) we get that $\left(\Omega\setminus N\ni x\mapsto|f(x)|\in\mathbb{R}\right)$
is measurable, as it is the supremum of a sequence of measurable functions.
Hence, the set $F\coloneqq\left\{ x\in\Omega\setminus N\,;\,|f(x)|>0\right\} $
is measurable. Let $\varepsilon>0$, $(y_{n})_{n\in\mathbb{N}}$ a
dense sequence in $Y$ and set 
\[
E_{n}\coloneqq\left\{ x\in F\,;\,|f(x)-y_{n}|<\varepsilon\right\} \quad(n\in\mathbb{N}).
\]
Then $E_{n}$ is measurable for each $n\in\mathbb{N}$ and $\bigcup_{n\in\mathbb{N}}E_{n}=F$
by the density of $\{y_{n}\,;\, n\in\mathbb{N}\}$. Setting $F_{0}=E_{0}$
and $F_{n+1}=E_{n+1}\setminus\bigcup_{k=0}^{n}F_{k}$ for $n\in\mathbb{N}$,
we obtain a sequence of pairwise disjoint measurable sets $(F_{n})_{n\in\mathbb{N}}$
with $\bigcup_{n\in\mathbb{N}}F_{n}=F.$ We set 
\[
g\coloneqq\sum_{k=0}^{\infty}y_{k}\chi_{F_{k}}
\]
and obtain $|f(x)-g(x)|<\varepsilon$ for each $x\in\Omega\setminus N.$
Hence, if $g$ is measurable, then so is $f$. For showing the measurability
of $g$, let $(\Omega_{k})_{k\in\mathbb{N}}$ be a sequence of pairwise
disjoint measurable sets such that $\bigcup_{k\in\mathbb{N}}\Omega_{k}=\Omega$
and $\mu(\Omega_{k})<\infty$ for each $k\in\mathbb{N}.$ For $n\in\mathbb{N}$
we set 
\[
g_{n}\coloneqq\sum_{k,j=0}^{n}y_{k}\chi_{F_{k}\cap\Omega_{j}}.
\]
Then $\left(g_{n}\right)_{n\in\mathbb{N}}$ is a sequence of simple
functions with $g_{n}\to g$ pointwise as $n\to\infty$ and thus,
$g$ is measurable.
\end{proof}
We are now able to prove Rademacher's Theorem for Lipschitz-continuous
functions with values in a reflexive Banach space. In fact, one can
prove the Theorem for a larger class of Banach spaces, namely those
having the Radon-Nikodym property, see e.g. \cite[Theorem 2, p. 106]{DiestelUhl}.
More precisely, the validity of Rademacher's Theorem is equivalent
to the fact that $X$ satisfies the Radon-Nikodym property (cf. \cite[Theorem 1.5]{Arendt1987}).
But since we mainly deal with Hilbert spaces, we may restrict ourselves
to this simpler situation.
\begin{thm}[Rademacher's Theorem; Banach space-valued]
\label{thm:Rademacher2} Let $F:I\subseteq\mathbb{R}\to X$ Lipschitz-continuous,
where $I$ is an (possibly unbounded) interval and $X$ is a reflexive
Banach space. Then $F$ is differentiable almost everywhere with $F'\in L_{\infty}(I;X)$
and $|F'|{}_{L_{\infty}}=|F|_{\mathrm{Lip}}.$ Moreover 
\[
F(t)-F(s)=\intop_{s}^{t}F'
\]
for all $t,s\in I,s<t.$\end{thm}
\begin{proof}
We consider the separable subspace $Y\coloneqq\overline{\lin F[I]},$
which is reflexive as it is a closed subspace of a reflexive space.
For $y'\in Y'$ the function $y'\circ F:I\to\mathbb{K}$ is Lipschitz-continuous
with $|y'\circ F|_{\mathrm{Lip}}\leq\|y'\||F|_{\mathrm{Lip}}$ and
hence, there exists a unique $f_{y'}\in L_{\infty}(I)$ with $|f_{y'}|_{L_{\infty}}\leq\|y'\||F|_{\mathrm{Lip}}$
and 
\[
y'(F(t)-F(s))=\intop_{s}^{t}f_{y'}
\]
for each $s,t\in I$ with $s<t$ by \prettyref{thm:Radmacher1}. Since
$Y$ is reflexive and separable, so is $Y'.$ Let $A\subseteq S_{Y'}(0,1)$
be a linear independent countable subset, which is dense in $S_{Y'}(0,1)$
and set $Z\coloneqq\lin_{\mathbb{Q}}A.$ Then $Z$ is dense in $Y'$.
For each $z'\in Z$ there exists a nullset $N_{z'}^{(1)}\subseteq I$
such that $|f_{z'}(x)|\leq|f_{z'}|_{L_{\infty}}\leq\|z'\||F|_{\mathrm{Lip}}$
for each $x\in I\setminus N_{z'}^{(1)}$. Moreover, for $z'=\sum_{i=1}^{n}\alpha_{i}y_{i}'$
with $\alpha_{1},\ldots,\alpha_{n}\in\mathbb{Q},y_{1}',\ldots y_{n}'\in A,n\in\mathbb{N}$
we have that 
\[
z'(F(t)-F(s))=\intop_{s}^{t}\sum_{i=1}^{n}\alpha_{i}f_{y_{i}'},
\]
and hence, there exists a nullset $N_{z'}^{(2)}\subseteq I$ such
that $f_{z'}(x)=\sum_{i=1}^{n}\alpha_{i}f_{y_{i}'}(x)$ for each $x\in I\setminus N_{z'}^{(2)}.$
Thus, setting $N\coloneqq\bigcup_{z'\in Z}N_{z'}^{(1)}\cup N_{z'}^{(2)},$
we infer that for each $x\in I\setminus N$ and $z'_{1},z'_{2}\in Z,\lambda\in\mathbb{Q}$
we have that 
\begin{align*}
|f_{z_{1}'}(x)| & \leq\|z_{1}'\||F|_{\mathrm{Lip}},\\
\lambda f_{z_{1}'}(x)+f_{z_{2}'}(x) & =f_{\lambda z_{1}'+z_{2}'}(x).
\end{align*}
Hence, for each $x\in I\setminus N$ the mapping 
\begin{align*}
\varphi_{x}:Z\subseteq Y' & \to\mathbb{K}\\
z' & \mapsto f_{z'}(x)
\end{align*}
is a linear bounded functional, which, due to the density of $Z$,
can be extended to an element in $Y''=Y.$ We define the mapping 
\begin{align*}
f:I & \to Y\\
x & \mapsto\begin{cases}
\varphi_{x} & \mbox{ if }x\in I\setminus N,\\
0 & \mbox{ else}.
\end{cases}
\end{align*}
We first show that $f$ is measurable. For doing so, we note that
$f$ is separably valued. Moreover, for each $y'\in Y'$ there is
a sequence $\left(z_{n}'\right)_{n\in\mathbb{N}}$ in $Z$ with $z_{n}'\to y'$
in $Y'$ as $n\to\infty.$ Thus, we have that 
\[
\left(y'\circ f\right)(x)=\lim_{n\to\infty}z_{n}'(f(x))=\lim_{n\to\infty}\varphi_{x}(z_{n}')=\lim_{n\to\infty}f_{z_{n}'}(x)\quad(x\in I\setminus N),
\]
which implies the measurability of $y'\circ f.$ Thus, by \prettyref{thm:Pettis}
we derive the measurability of $f$. Moreover, we have that 
\[
|f(x)|=\sup_{z'\in A}|z'(f(x))|=\sup_{z'\in A}|f_{z'}(x)|\leq|F|_{\mathrm{Lip}}\quad(x\in I\setminus N),
\]
which shows that $f\in L_{\infty}(I;X)$ and $|f|_{L_{\infty}}\leq|F|_{\mathrm{Lip}}.$
Furthermore, for $z'\in Z$ we have 
\[
z'(F(t)-F(s))=\intop_{s}^{t}f_{z'}(x)\mbox{ d}x=\intop_{s}^{t}z'(f(x))\mbox{ d}x=z'\left(\intop_{s}^{t}f(x)\mbox{ d}x\right),
\]
which yields 
\[
F(t)-F(s)=\intop_{s}^{t}f(x)\mbox{ d}x
\]
for $s,t\in I$ with $s<t.$ The latter implies $|F(t)-F(s)|\leq|f|_{L_{\infty}}|t-s|,$
which shows $|f|_{L_{\infty}}=|F|_{\mathrm{Lip}}.$ That $f$ is indeed
the almost everywhere derivative of $F$ follows from \prettyref{thm:Lebesgue}.
\end{proof}

\newpage{}

$\:$

\thispagestyle{empty}
\chapter{The Widder-Arendt Theorem}

In this part of the appendix we provide a proof of the Widder-Arendt
Theorem. This theorem, which was proved by Widder for the scalar-valued
case (\cite{Widder1934},\cite[Theorem 8, p. 157]{Widder1971}) and
generalized by Arendt to the Banach space-valued case (\cite{Arendt1987}),
characterizes those functions, which are representable as the Laplace
transform of a bounded measurable function. We mainly  follow the
rationale given in \cite[Chapter 2]{ABHN_2011}. 
\begin{lem}
\label{lem:delta-sequence}We define 
\[
\rho_{k,t}(s)\coloneqq\frac{1}{k!}\left(\frac{k}{t}\right)^{k+1}\e^{-\frac{k}{t}s}s^{k}\quad(s,t>0,k\in\mathbb{N}_{\geq1}).
\]
Then for each $t>0$ the sequence $\left(\rho_{k,t}\right)_{k\in\mathbb{N}_{\geq1}}$
is a mollifier centered at $t$.\end{lem}
\begin{proof}
Let $t>0$. Obviously, $\rho_{k,t}\in C^{\infty}(\mathbb{R}_{>0})$
and $\rho_{k,t}>0$ for each $k$. Moreover, for $k\in\mathbb{N}_{\geq1},$
\begin{align*}
\intop_{0}^{\infty}\rho_{k,t}(s)\mbox{ d}s & =\frac{1}{k!}\left(\frac{k}{t}\right)^{k+1}\intop_{0}^{\infty}\e^{-\frac{k}{t}s}s^{k}\mbox{ d}s\\
 & =\frac{k}{t}\intop_{0}^{\infty}\e^{-\frac{k}{t}s}\mbox{ d}s\\
 & =1,
\end{align*}
by integration by parts. We define $g_{k}(s)\coloneqq\frac{k^{k+1}}{k!}s^{k}\e^{-ks}=t\rho_{k,t}(st)$
for $k\in\mathbb{N}_{\geq1},s>0.$ Then 
\begin{align*}
\intop_{0}^{t-\varepsilon}\rho_{k,t}(s)\mbox{ d}s & =\intop_{0}^{1-\frac{\varepsilon}{t}}g_{k}(s)\mbox{ d}s,\\
\intop_{t+\varepsilon}^{\infty}\rho_{k,t}(s)\mbox{ d}s & =\intop_{1+\frac{\varepsilon}{t}}^{\infty}g_{k}(s)\mbox{ d}s
\end{align*}
for each $k\in\mathbb{N}_{\geq1},\varepsilon>0.$ So, it suffices
to prove that $\intop_{0}^{1-\delta}g_{k}(s)\mbox{ d}s\to0\quad(k\to\infty)$
and $\intop_{1+\delta}^{\infty}g_{k}(s)\mbox{ d}s\to0\quad(k\to\infty)$
for each $\delta>0$. For doing so, we define $h(s)\coloneqq s\e^{1-s}$
for $s>0.$ Then, since $\frac{k^{k}}{k!}<\e^{k},$ we obtain $g_{k}(s)\leq k\left(h(s)\right)^{k}$
for each $k\in\mathbb{N}_{\ge1},s>0.$ We observe that $h$ strictly
increases on $]0,1[$. Thus, we derive 
\[
\intop_{0}^{1-\delta}g_{k}(s)\mbox{ d}s\leq k\intop_{0}^{1-\delta}h(1-\delta)^{k}\mbox{ d}s\leq kh(1-\delta)^{k}(1-\delta)\to0\quad(k\to\infty)
\]
since $h(1-\delta)<h(1)=1.$ Moreover, we observe that $h(s)\e^{as}\leq\frac{1}{1-a}$
for each $s>0,0<a<1.$ Hence, 
\begin{align*}
\intop_{1+\delta}^{\infty}g_{k}(s)\mbox{ d}s & \leq k\left(\frac{1}{1-a}\right)^{k}\intop_{1+\delta}^{\infty}\e^{-aks}\mbox{ d}s\\
 & =\left(\frac{1}{1-a}\right)^{k}\frac{1}{a}\e^{-ak(1+\delta)}.
\end{align*}
Choosing now $a\coloneqq\frac{\delta}{1+\delta}$, we derive 
\[
\intop_{1+\delta}^{\infty}g_{k}(s)\mbox{ d}s\leq\frac{\left(1+\delta\right)^{k+1}}{\delta}\e^{-k\delta}\to0\quad(k\to\infty).\tag*{\qedhere}
\]

\end{proof}
Next we state the Post-Widder inversion formula for bounded continuous
functions. We note that this result also holds for locally integrable
functions (see e.g. \cite[Theorem 1.7.7]{ABHN_2011}), however, as
the proof for the more abstract result is rather technical, we stick
to the simpler case.
\begin{lem}[Post-Widder Inversion Formula]
\label{lem:Post-Widder}Let $X$ be a Banach space and $f\in C_{b}(\mathbb{R}_{\geq0};X).$
Then, for each $t>0$ we have 
\[
f(t)=\lim_{k\to\infty}(-1)^{k}\frac{1}{k!}\left(\frac{k}{t}\right)^{k+1}\sqrt{2\pi}\hat{f}^{(k)}\left(\frac{k}{t}\right).
\]
\end{lem}
\begin{proof}
We have that 
\[
\hat{f}^{(k)}(z)=\frac{1}{\sqrt{2\pi}}\intop_{0}^{\infty}\e^{-zs}\left(-s\right)^{k}f(s)\mbox{ d}s\quad(k\in\mathbb{N},z\in\mathbb{C}_{\Re>0}).
\]
In particular, we get for $k\in\mathbb{N}_{\geq1}$ and $t>0$ 
\[
(-1)^{k}\frac{1}{k!}\left(\frac{k}{t}\right)^{k+1}\sqrt{2\pi}\hat{f}^{(k)}\left(\frac{k}{t}\right)=\intop_{0}^{\infty}(-1)^{k}\frac{1}{k!}\left(\frac{k}{t}\right)^{k+1}\e^{-\frac{k}{t}s}\left(-s\right)^{k}f(s)\mbox{ d}s=\intop_{0}^{\infty}\rho_{k,t}(s)f(s)\mbox{ d}s,
\]
where $\rho_{k,t}$ is given as in \prettyref{lem:delta-sequence}.
For $\varepsilon>0,t>0$ we choose $\delta>0$ such that $|f(s)-f(t)|<\varepsilon$
for $|s-t|<\delta$. Then, we obtain 
\begin{align*}
\left|(-1)^{k}\frac{1}{k!}\left(\frac{k}{t}\right)^{k+1}\sqrt{2\pi}\hat{f}^{(k)}\left(\frac{k}{t}\right)-f(t)\right| & \leq\intop_{0}^{\infty}\rho_{k,t}(s)\left|f(s)-f(t)\right|\mbox{ d}s\\
 & \leq2|f|_{\infty}\left(\intop_{0}^{t-\delta}\rho_{k,t}(s)\mbox{ d}s+\intop_{t+\delta}^{\infty}\rho_{k,t}(s)\mbox{ d}s\right)+\varepsilon\\
 & \leq2\varepsilon
\end{align*}
for sufficiently large $k$ according to \prettyref{lem:delta-sequence}. \end{proof}
\begin{thm}[Riesz-Stieltjes representation]
\label{thm:Riesz-Stieltjes}Let $X$ be a Banach space and $T:L_{1}(\mathbb{R}_{\geq0})\to X$
linear and bounded. We define 
\[
F(t)=T\chi_{[0,t[}\quad(t\geq0).
\]
Then $F:\mathbb{R}_{\geq0}\to X$ is Lipschitz-continuous with $F(0)=0$
and $|F|_{\mathrm{Lip}}=\|T\|$ and for each continuous $g\in L_{1}(\mathbb{R}_{\geq0}),$
we have%
\footnote{We note that on a compact interval, we have that $F$ is of bounded
variation and hence, the integral $\intop_{0}^{t}g(s)\mbox{ d}F(s)$
exists for continuous functions $g$.%
} 
\[
Tg=\lim_{t\to\infty}\intop_{0}^{t}g(s)\,\mathrm{d}F(s).
\]
\end{thm}
\begin{proof}
We have 
\[
|F(t)-F(s)|=\left|T\left(\chi_{[0,t[}-\chi_{[0,s[}\right)\right|\leq\|T\||t-s|\quad(t,s\geq0),
\]
and hence, $F$ is Lipschitz-continuous with $|F|_{\mathrm{Lip}}\leq\|T\|$.
For $g=\sum_{j=1}^{n}c_{j}\chi_{[\alpha_{j},\alpha_{j+1}[,}$ where
$c_{1},\ldots,c_{n}\in\mathbb{R}$ and $0\leq\alpha_{1}<\ldots<\alpha_{n}$
we have that 
\begin{align*}
Tg & =\sum_{j=1}^{n}c_{j}T\chi_{[\alpha_{j},\alpha_{j+1}[}\\
 & =\sum_{j=1}^{n}c_{j}\left(T\chi_{[0,\alpha_{j+1}[}-T\chi_{[0,\alpha_{j}[}\right)\\
 & =\sum_{j=1}^{n}c_{j}\left(F(\alpha_{j+1})-F(\alpha_{j})\right)\\
 & =\intop_{0}^{\infty}g(s)\mbox{ d}F(s).
\end{align*}
In particular, $|Tg|\leq|F|_{\mathrm{Lip}}\sum_{j=1}^{n}c_{j}|\alpha_{j+1}-\alpha_{j}|=|F|_{\mathrm{Lip}}|g|_{L_{1}},$
and since simple functions are dense in $L_{1}(\mathbb{R}_{\geq0}),$
we derive $\|T\|\leq|F|_{\mathrm{Lip}}$. Let now $g\in L_{1}(\mathbb{R}_{\geq0})$
continuous and $t>0$. Then there exists a sequence of simple functions
$(g_{n})_{n\in\mathbb{N}}$ such that $g_{n}|_{[0,t]}\to g|_{[0,t]}$
in $L_{1}(\mathbb{R})$ and point-wise. Thus, 
\[
\intop_{0}^{t}g(s)\mbox{ d}F(s)=\lim_{n\to\infty}\intop_{0}^{t}g_{n}(s)\mbox{ d}F(s)=\lim_{n\to\infty}T\left(g_{n}|_{[0,t]}\right)=T\left(g|_{[0,t]}\right).
\]
Moreover, since $g|_{[0,t]}\to g$ in $L_{1}(\mathbb{R}_{\geq0})$
as $t\to\infty$ we obtain the last assertion. 
\end{proof}
Before we come to our main result, we need the following technical
lemma.
\begin{lem}
\label{lem:Post-Widder, die 2.}Let $X$ be a Banach space, $r\in C^{\infty}(\mathbb{R}_{>0};X)$
such that 
\[
M\coloneqq\sup_{k\in\mathbb{N},\lambda>0}\left|\frac{r^{(k)}(\lambda)\lambda^{k+1}}{k!}\right|<\infty.
\]
Then, for $\lambda>0$ we have 
\[
\intop_{0}^{\infty}\e^{-\lambda t}(-1)^{k}\frac{1}{k!}\left(\frac{k}{t}\right)^{k+1}r^{(k)}\left(\frac{k}{t}\right)\,\mathrm{d}t\to r(\lambda)\quad(k\to\infty).
\]
\end{lem}
\begin{proof}
Let $\lambda>0$. We first prove that the integrals exist. For doing
so, we estimate 
\[
\intop_{0}^{\infty}\left|\e^{-\lambda t}(-1)^{k}\frac{1}{k!}\left(\frac{k}{t}\right)^{k+1}r^{(k)}\left(\frac{k}{t}\right)\right|\mbox{ d}t\leq M\intop_{0}^{\infty}e^{-\lambda t}\mbox{ d}t<\infty,
\]
since $\lambda>0.$ We compute 
\begin{align*}
\intop_{0}^{\infty}\e^{-\lambda t}(-1)^{k}\frac{1}{k!}\left(\frac{k}{t}\right)^{k+1}r^{(k)}\left(\frac{k}{t}\right)\mbox{ d}t & =(-1)^{k}\frac{1}{\left(k-1\right)!}\intop_{0}^{\infty}\e^{-\lambda\frac{k}{s}}s^{k-1}r^{(k)}(s)\mbox{ d}s\\
 & =(-1)^{k}\frac{\left(\lambda k\right)^{k-1}}{(k-1)!}\intop_{0}^{\infty}G_{k}(\lambda k,s)r^{(k)}(s)\mbox{ d}s,
\end{align*}
where $G_{k}(x,s)\coloneqq\e^{-\frac{x}{s}}\left(\frac{s}{x}\right)^{k-1}$
for $x,s>0.$ Integration by parts then yields 
\begin{align*}
 & \intop_{0}^{\infty}\e^{-\lambda t}(-1)^{k}\frac{1}{k!}\left(\frac{k}{t}\right)^{k+1}r^{(k)}\left(\frac{k}{t}\right)\mbox{ d}t\\
= & \frac{(-1)^{k}(\lambda k)^{k-1}}{(k-1)!}\left(\sum_{j=0}^{k-1}(-1)^{j}\left(\partial_{2}^{j}G_{k}\right)(\lambda k,s)r^{(k-j-1)}(s)|_{s=0}^{\infty}+(-1)^{k}\intop_{0}^{\infty}\left(\partial_{2}^{k}G_{k}\right)(\lambda k,s)r(s)\mbox{ d}s\right).
\end{align*}
We want to compute the derivatives $\partial_{2}^{j}G_{k}.$ First,
we observe that $G_{k}(tx,ts)=G(x,s)$ for each $x,s,t>0.$ Hence,
\[
0=\partial\left(t\mapsto G_{k}(tx,ts)\right)(1)=x\left(\partial_{1}G_{k}\right)(x,s)+s\left(\partial_{2}G_{k}\right)(x,s)
\]
or equivalently 
\[
\frac{1}{x}\left(\partial_{2}G_{k}\right)(x,s)=-\frac{1}{s}\left(\partial_{1}G_{k}\right)(x,s)
\]
for each $x,s>0.$ We now prove by induction, that 
\[
\frac{1}{x}\left(\partial_{2}^{j}G_{k}\right)(x,s)=\frac{1}{s}(-1)^{j}\left(\partial_{1}^{j}G_{k-j+1}\right)(x,s)\quad(x,s>0,j\in\{1,\ldots,k\}).
\]
The case $j=1$ we have proved above. Assume now that the formula
holds for $j.$ Then 
\begin{align*}
 & \frac{1}{x}\left(\partial_{2}^{j+1}G_{k}\right)(x,s)\\
 & =\partial_{2}\left(\frac{1}{\m_{2}}(-1)^{j}\partial_{1}^{j}G_{k-j+1}\right)(x,s)\\
 & =(-1)^{j+1}\frac{1}{s^{2}}\left(\partial_{1}^{j}G_{k-j+1}\right)(x,s)+\frac{1}{s}(-1)^{j}\left(\partial_{1}^{j}\partial_{2}G_{k-j+1}\right)(x,s)\\
 & =\frac{1}{s}\left(-1\right)^{j+1}\left(\frac{1}{s}\left(\partial_{1}^{j}G_{k-j+1}\right)(x,s)-\left(\partial_{1}^{j}\partial_{2}G_{k-j+1}\right)(x,s)\right)\\
 & =\frac{1}{s}\left(-1\right)^{j+1}\left(\frac{1}{s}\left(\partial_{1}^{j}G_{k-j+1}\right)(x,s)+\left(\partial_{1}^{j}\left(\frac{\m_{1}}{\m_{2}}\partial_{1}G_{k-j+1}\right)\right)(x,s)\right)\\
 & =\frac{1}{s}\left(-1\right)^{j+1}\left(\frac{1}{s}\partial_{1}^{j}G_{k-j+1}+\partial_{1}^{j+1}\left(\frac{\m_{1}}{\m_{2}}G_{k-j+1}\right)-\partial_{1}^{j}\left(\frac{1}{\m_{2}}G_{k-j+1}\right)\right)(x,s)\\
 & =\frac{1}{s}\left(-1\right)^{j+1}\left(\partial_{1}^{j+1}\left(\frac{\m_{1}}{\m_{2}}G_{k-j+1}\right)\right)(x,s)\\
 & =\frac{1}{s}\left(-1\right)^{j+1}\left(\partial_{1}^{j+1}G_{k-j}\right)(x,s),
\end{align*}
which proves the assertion. From this we derive 
\begin{align*}
 & \left|\sum_{j=0}^{k-1}(-1)^{j}\left(\partial_{2}^{j}G_{k}\right)(\lambda k,s)r^{(k-j-1)}(s)\right|\\
 & =\left|\sum_{j=0}^{k-1}\left(\partial_{1}^{j}G_{k-j+1}\right)(\lambda k,s)\frac{\lambda k}{s}r^{(k-j-1)}(s)\right|\\
 & =(\lambda k)\left|\sum_{j=0}^{k-1}\partial^{j}\left(x\mapsto\e^{-\frac{x}{s}}\left(\frac{1}{x}\right)^{k-j}\right)(\lambda k)s^{k-j-1}r^{(k-j-1)}(s)\right|\\
 & \leq(\lambda k)M\sum_{j=0}^{k-1}\left|\frac{1}{s}\partial^{j}\left(x\mapsto\e^{-\frac{x}{s}}x^{-k+j}\right)(\lambda k)(k-j-1)!\right|.
\end{align*}
Since $\left|\partial^{j}\left(x\mapsto\e^{-\frac{x}{s}}x^{-k+j}\right)(\lambda k)\right|\leq C|p(s^{-1})|\e^{-\frac{\lambda k}{s}}$
for some constant $C$ and some polynomial $p$ we get 
\[
\sum_{j=0}^{k-1}(-1)^{j}\left(\partial_{2}^{j}G_{k}\right)(\lambda k,s)r^{(k-j-1)}(s)|_{s=0}^{\infty}=0.
\]
Thus, we have that
\begin{align*}
\intop_{0}^{\infty}\e^{-\lambda t}(-1)^{k}\frac{1}{k!}\left(\frac{k}{t}\right)^{k+1}r^{(k)}\left(\frac{k}{t}\right)\mbox{ d}t & =\frac{(\lambda k)^{k-1}}{(k-1)!}\intop_{0}^{\infty}\left(\partial_{2}^{k}G_{k}\right)(\lambda k,s)r(s)\mbox{ d}s\\
 & =\frac{(\lambda k)^{k}}{(k-1)!}\intop_{0}^{\infty}(-1)^{k}\frac{1}{s}\left(\partial_{1}^{k}G_{1}\right)(\lambda k,s)r(s)\mbox{ d}s\\
 & =\frac{(\lambda k)^{k}}{(k-1)!}\intop_{0}^{\infty}\frac{1}{s^{k+1}}\e^{-\frac{\lambda k}{s}}r(s)\mbox{ d}s\\
 & =\frac{(\lambda k)^{k}}{(k-1)!}\intop_{0}^{\infty}t^{k-1}\e^{-\lambda kt}r\left(\frac{1}{t}\right)\mbox{ d}t\\
 & =\frac{\left(\lambda k\right){}^{k}}{(k-1)!}\intop_{0}^{\infty}t^{k}\e^{-\lambda kt}f\left(t\right)\mbox{ d}t\\
 & =\frac{\left(\lambda k\right)^{k}}{(k-1)!}\sqrt{2\pi}(-1)^{k}\hat{f}^{(k)}(\lambda k),
\end{align*}
where $f(t)\coloneqq\frac{1}{t}r\left(\frac{1}{t}\right)$ for $t>0$.
Since $f\in C_{b}(\mathbb{R}_{>0};X)$ we obtain by \prettyref{lem:Post-Widder}
(choose $t=\frac{1}{\lambda}$) 
\[
\frac{\left(\lambda k\right)^{k}}{(k-1)!}\sqrt{2\pi}(-1)^{k}\hat{f}^{(k)}(\lambda k)=\frac{1}{\lambda}\frac{\left(\lambda k\right)^{k+1}}{k!}\sqrt{2\pi}(-1){}^{k}\hat{f}^{(k)}(\lambda k)\to\frac{1}{\lambda}f(\frac{1}{\lambda})=r(\lambda)\quad(k\to\infty).\tag*{\qedhere}
\]
\end{proof}
\begin{thm}
\label{thm:Pre-widder}Let $X$ be a Banach space, $r\in C^{\infty}(\mathbb{R}_{>0};X)$
such that 
\[
M\coloneqq\sup_{\lambda>0,k\in\mathbb{N}}\left|\frac{r^{(k)}(\lambda)\lambda^{k+1}}{k!}\right|<\infty.
\]
Then there exists $F:\mathbb{R}_{\geq0}\to X$ Lipschitz-continuous
with $F(0)=0$ such that 
\[
r(\lambda)=\intop_{0}^{\infty}\e^{-\lambda t}\,\mathrm{d}F(t)\quad(\lambda>0).
\]
Moreover $|F|_{\mathrm{Lip}}=M.$\end{thm}
\begin{proof}
For $k\in\mathbb{N}$ we define the operator $T_{k}:L_{1}(\mathbb{R}_{\geq0})\to X$
by 
\[
T_{k}f\coloneqq\intop_{0}^{\infty}f(t)(-1)^{k}\frac{1}{k!}\left(\frac{k}{t}\right)^{k+1}r^{(k)}\left(\frac{k}{t}\right)\mbox{ d}t\quad(f\in L_{1}(\mathbb{R}_{\geq0})).
\]
Then $T_{k}$ is linear and since $|T_{k}f|\leq M|f|_{L_{1}}$ for
$f\in L_{1}(\mathbb{R}_{\geq0})$, it is also bounded. By \prettyref{lem:Post-Widder, die 2.}
we have 
\[
T_{k}\e^{-\lambda\cdot}\to r(\lambda)\quad(k\to\infty)
\]
for each $\lambda>0.$ As $\left\{ \e^{-\lambda\cdot}\,;\,\lambda>0\right\} $
is a total set in $L_{1}(\mathbb{R}_{\geq0})$ (see \prettyref{cor:exp_dense})
and $\sup_{k\in\mathbb{N}}\|T_{k}\|\leq M,$ we derive that there
exists a bounded linear operator $T:L_{1}(\mathbb{R}_{\geq0})\to X$
with $\|T\|\leq M$ and 
\[
T_{k}f\to Tf\quad(k\to\infty)
\]
for each $f\in L_{1}(\mathbb{R}_{\geq0}).$ In particular, 
\[
T\e^{-\lambda\cdot}=r(\lambda)
\]
for each $\lambda>0.$ Applying now \prettyref{thm:Riesz-Stieltjes},
we find $F:\mathbb{R}_{\geq0}\to X$ Lipschitz with $F(0)=0,$ $|F|_{\mathrm{Lip}}=\|T\|\leq M$
and 
\[
\intop_{0}^{\infty}\e^{-\lambda t}\mbox{ d}F(t)=T\e^{-\lambda\cdot}=r(\lambda)\quad(\lambda>0).
\]
It is left to show $M\leq|F|_{\mathrm{Lip}}.$ We observe that 
\begin{align*}
|r^{(k)}(\lambda)| & =\left|\intop_{0}^{\infty}\e^{-\lambda t}(-t)^{k}\mbox{ d}F(t)\right|\\
 & \leq\|T\|\intop_{0}^{\infty}t^{k}\e^{-\lambda t}\mbox{ d}t\\
 & =|F|_{\mathrm{Lip}}k!\frac{1}{\lambda^{k+1}}\quad(k\in\mathbb{N},\lambda>0)
\end{align*}
and thus, $M=\sup_{\lambda>0,k\in\mathbb{N}}\left|\frac{r^{(k)}(\lambda)\lambda^{k+1}}{k!}\right|\leq|F|_{\mathrm{Lip}}$.
\end{proof}
Using Rademacher's Theorem (\prettyref{thm:Rademacher2}), we are
now able to prove the Theorem of Widder-Arendt. Again we just state
the theorem for reflexive Banach spaces and note that it also holds
for Banach spaces with the Radon-Nikodym property (in fact, it is
even equivalent to $X$ having the Radon-Nikodym property, see \cite[Theorem 1.4]{Arendt1987}).
\begin{thm}[Widder-Arendt]
 \label{thm:Widder-Arendt}Let $X$ be a reflexive Banach space and
$r\in C^{\infty}(\mathbb{R}_{>0};X)$ such that 
\[
M\coloneqq\sup_{\lambda>0,k\in\mathbb{N}}\left|\frac{r^{(k)}(\lambda)\lambda^{k+1}}{k!}\right|<\infty.
\]
Then there exists $f\in L_{\infty}(\mathbb{R}_{>0};X)$ such that
$|f|_{\infty}=M$ and 
\[
r(\lambda)=\intop_{0}^{\infty}\e^{-\lambda t}f(t)\,\mathrm{d}t.
\]
\end{thm}
\begin{proof}
By \prettyref{thm:Pre-widder} there is $F:\mathbb{R}_{\geq0}\to X$
Lipschitz continuous with $F(0)=0,|F|_{\mathrm{Lip}}=M$ and 
\[
r(\lambda)=\intop_{0}^{\infty}\e^{-\lambda t}\mbox{ d}F(t).
\]
By \prettyref{thm:Rademacher2}, $F$ is differentiable almost everywhere
with a bounded derivative $f\coloneqq F'$ with $|f|_{\infty}=|F|_{\mathrm{Lip}}=M$.
Moreover, 
\[
F(t)-F(s)=\intop_{s}^{t}f=\intop_{0}^{\infty}\chi_{[s,t]}f
\]
for each $0\leq s<t$. By continuous extension we get 
\[
\intop_{0}^{\infty}g\mbox{ d}F=\intop_{0}^{\infty}gf
\]
for each continuous function $g\in L_{1}(\mathbb{R}_{>0};X)$ and
thus, especially 
\[
r(\lambda)=\intop_{0}^{\infty}\e^{-\lambda t}\mbox{ d}F(t)=\intop_{0}^{\infty}\e^{-\lambda t}f(t)\mbox{ d}t\quad(\lambda>0).\tag*{\qedhere}
\]
\end{proof}

\newpage{}\newcommand{\etalchar}[1]{$^{#1}$}

\end{document}